\documentclass[11pt,oneside,reqno]{amsart}
\usepackage[latin9]{inputenc}
\usepackage{geometry}
\usepackage{mathrsfs}
\usepackage{amstext}
\usepackage{amsthm}
\usepackage{amssymb}
\usepackage{graphicx}
\usepackage{setspace}
\usepackage[unicode=true,
 bookmarks=false,
 breaklinks=false,pdfborder={0 0 1},backref=section,colorlinks=false]
 {hyperref}

\makeatletter

\newcommand{\lyxdot}{.}

\numberwithin{equation}{section}
\numberwithin{figure}{section}
\theoremstyle{plain}
\newtheorem{thm}{\protect\theoremname}[section]
\newtheorem{lem}[thm]{\protect\lemmaname}
\newtheorem{prop}[thm]{\protect\propositionname}

\theoremstyle{remark}
\newtheorem{rem}[thm]{\protect\remarkname}
\newtheorem{notation}[thm]{\protect\notationname}
\newtheorem*{notation*}{\protect\notationname}

\theoremstyle{definition}
\newtheorem{defn}[thm]{\protect\definitionname}
\theoremstyle{remark}
\newtheorem*{acknowledgement*}{\protect\acknowledgementname}

\usepackage{amsthm}\usepackage{mathrsfs}\usepackage{nameref}\usepackage{upgreek}
\usepackage[noadjust]{cite}
\usepackage{stmaryrd}
\usepackage{etoolbox}
\usepackage{enumitem}
\setlist[itemize]{leftmargin=*}
\setlist[enumerate]{leftmargin=*}

\DeclareFontFamily{U}{matha}{\hyphenchar\font45}
\DeclareFontShape{U}{matha}{m}{n}{
      <5> <6> <7> <8> <9> <10> gen * matha
      <10.95> matha10 <12> <14.4> <17.28> <20.74> <24.88> matha12
      }{}
\DeclareSymbolFont{matha}{U}{matha}{m}{n}
\DeclareFontSubstitution{U}{matha}{m}{n}

\DeclareFontFamily{U}{mathx}{\hyphenchar\font45}
\DeclareFontShape{U}{mathx}{m}{n}{
      <5> <6> <7> <8> <9> <10>
      <10.95> <12> <14.4> <17.28> <20.74> <24.88>
      mathx10
      }{}
\DeclareSymbolFont{mathx}{U}{mathx}{m}{n}
\DeclareFontSubstitution{U}{mathx}{m}{n}

\DeclareMathDelimiter{\vvvert}{0}{matha}{"7E}{mathx}{"17}

\DeclareMathAlphabet{\scal}{U}{dutchcal}{m}{n}

\numberwithin{equation}{section}

\def\th@plain{\thm@notefont{}\itshape}
\def\th@definition{\thm@notefont{}\normalfont}

\makeatother

\providecommand{\acknowledgementname}{Acknowledgement}
\providecommand{\definitionname}{Definition}
\providecommand{\lemmaname}{Lemma}
\providecommand{\notationname}{Notation}
\providecommand{\propositionname}{Proposition}
\providecommand{\remarkname}{Remark}
\providecommand{\theoremname}{Theorem}

\begin{document}
\title[Approximation Method and Metastability of Non-rev. 2D Ising/Potts
Models]{Approximation Method to Metastability: An Application to Non-reversible,
Two-dimensional Ising and Potts Models Without External Fields}
\author{Seonwoo Kim and Insuk Seo}
\address{S. Kim. Department of Mathematical Sciences, Seoul National University,
Republic of Korea.}
\email{ksw6leta@snu.ac.kr}
\address{I. Seo. Department of Mathematical Sciences and R.I.M., Seoul National
University, Republic of Korea.}
\email{insuk.seo@snu.ac.kr}
\begin{abstract}
The main contribution of the current study is two-fold. First, we
investigate the energy landscape of the Ising and Potts models on
finite two-dimensional lattices without external fields in the low
temperature regime. The complete analysis of the energy landscape
of these models was unknown because of its complicated \emph{plateau}
saddle structure between the ground states. We characterize this structure
completely in terms of a random walk on the set of sub-trees of a
ladder graph. Second, we provide a considerable simplification of
the well-known potential-theoretic approach to metastability. In particular,
by replacing the role of variational principles such as the Dirichlet
and Thomson principles with an \emph{$H^{1}$-approximation of the
equilibrium potential}, we develop a new method that can be applied
to \emph{non-reversible dynamics} as well in a simple manner. As an
application of this method, we analyze metastable behavior of not
only the reversible Metropolis--Hastings dynamics, but also of several
interesting non-reversible dynamics associated with the low-temperature
Ising and Potts models explained above, and derive the Eyring--Kramers
law and the Markov chain model reduction of these models.
\end{abstract}

\maketitle

\section{\label{sec1}Introduction}

Metastability is a universal low-temperature phenomenon that appears
in stochastic models with multiple locally stable states. Important
models that exhibit this phenomenon include ferromagnetic spin systems
at low temperatures (e.g., \cite{BAC,BBI,BM,LS Potts,NZ,NS Ising1,NS Ising2}),
interacting particle systems with attracting/sticky interactions (e.g.,
\cite{BL ZRP,BDG,Kim,KS1,Landim TAZRP,Seo NRZRP}) and small random
perturbation of dynamical systems (e.g., \cite{BEGK,FW RPDS,LMS DT,LeeS NR1,LeeS NR2,RezSeo}).
We also refer to the monographs \cite{BdenH Meta,OV meta} and references
therein for the extensive literature on the mathematical study of
metastability.

Among the multitude of such models, this study is concerned with the
two-dimensional Ising and Potts models \emph{without external fields
}in the low temperature regime. These are metastable spin systems
on which analyses of the energy landscape and metastable behavior
have only partially been conducted.

\subsection{Energy landscape and metastability of Ising and Potts models}

An in-depth understanding of the energy landscape is required for
the precise analysis of the metastability. Hence, we first analyze
the energy landscape of the Ising and Potts Hamiltonian without external
fields, which is defined later in \eqref{e_Ham}. 

Existing literature states that, in the very low temperature regime
(i.e., the regime when the temperature tends to $0$), the monochromatic
configurations (i.e., the configurations such that all spins are the
same) are metastable (e.g., \cite{NZ,NS Ising1}). A crucial aspect
of the analysis of the energy landscape lies in the understanding
of the\emph{ saddle structure} between these ground states.

\subsubsection*{Energy landscape: case of non-zero external field}

The saddle structure has been studied comprehensively for the case
when a \emph{non-zero} external field is acting on the system. In
particular, the saddle structure of the Ising model with a non-zero
external field was fully investigated in \cite{NS Ising1,NS Ising2}
and \cite{AC,BAC} for two- and three-dimensional lattices, respectively.
The saddle structures of these models are characterized solely by
the appearances of a certain form of \emph{critical droplets}, and
thus, the energy landscapes exhibit sharp and simple structures (although,
of course, the proof thereof is complicated). We note that the corresponding
analysis of the $d$-dimensional Ising model is open for $d\ge4$.
A similar analysis was recently carried out on the two-dimensional
Potts model \cite{BGN Neg,BGN Pos}, which verified that the Potts
model with an external field acting on one spin exhibits the same
type of saddle structure.

\subsubsection*{Energy landscape: case of zero external field}

The Ising and Potts models with a\emph{ zero external field} are known
to have a highly complicated saddle structure, and thus, no complete
analysis has been conducted to date. In particular, unlike the non-zero
external field model, which has a sharp and simple saddle structure,
the zero external field model has a \emph{complicated and huge plateau
saddle structure}. The energy barrier between the metastable valleys
was calculated in \cite{NZ} and the gate structure was characterized
in \cite{BGN} for this model. This level of analysis already provides
large deviation-type results. However, to establish a precise metastability
analysis, such as the so-called Eyring--Kramers law (see Section
\ref{sec3.2} for further details), it is necessary to characterize
all configurations that belong to the plateau and to understand their
connectivity structure fully. The completion of these tasks is the
first main contribution of the current study. We reveal that the bulk
and edge parts of the saddle plateau can be described in terms of
a standard random walk and a random walk on a space of sub-trees of
a periodic ladder graph, respectively.

\subsubsection*{Analysis of metastability}

A complete understanding on the saddle structure enables the quantitative
analysis of the metastable behavior of the stochastic Ising and Potts
models, which are Markovian dynamics in which the invariant measure
is the Gibbs measure associated with the Hamiltonian.

For the model with a non-zero external field, a precise quantitative
metastability result known as the Eyring--Kramers law was obtained
in \cite{BM} for the well-known Metropolis--Hastings dynamics associated
with the two- and three- dimensional Ising models, based on the characterization
of the saddle structure in terms of the critical droplets.

In this study, we prove the Eyring--Kramers law for the zero external
field model based on the precise characterization of the plateau structure.
Moreover, since all monochromatic configurations are ground states
in the zero external field model, we expect successive metastable
transitions among these. In this case, these transitions can be described
in terms of a Markov chain. This result is known as the Markov chain
model reduction of metastable behavior (cf. \cite{BL TM,BL TM2,BL MG,Landim MMC}).
We obtain this result for the Ising and Potts models. In this study,
we consider the Metropolis--Hastings dynamics and several \emph{non-reversible}
dynamics as the stochastic systems.

\subsection{New strategy for analyzing metastability}

One of the most successful and classical methodologies in the quantitative
study of metastability is the potential-theoretic approach first developed
in \cite{BEGK}, which is comprehensively summarized in the monograph
\cite{BdenH Meta}. Another method is the martingale approach that
was developed in \cite{BL TM,BL TM2} and is summarized in \cite{Landim MMC}.
Both methods rely heavily on the estimate of a potential-theoretic
notion known as the \emph{capacity} between metastable sets; thus,
the estimation of the capacity is the central task in the application
of these approaches.

A standard means of carrying out this task is to use variational expressions
for the capacity. More precisely, the Dirichlet and Thomson principles
express the capacity as the values of minimization and maximization
problems, which can be used to derive the upper and lower bounds of
the capacity by approximating the optimizers of the two problems.
This strategy provides a sharp estimate of the capacity if we can
construct nice proxies of the optimizers so that the upper and lower
bounds are sharp.

\subsubsection*{Difficulty of current problem}

A deep understanding of the saddle structure and the behavior of the
dynamics thereon is required for the construction of proxies of optimizers
of the Dirichlet and Thomson principles. This task is particularly
difficult when the dynamics is non-reversible. The two principles
were extended to the non-reversible case in a more general form in
\cite{GL,LMS DT,Seo NRZRP,Slowik}, and several important non-reversible
models were also successfully analyzed based on this approach; for
example, \cite{Berglund,Landim TAZRP,LMS DT,LS Potts,LS NR1}. The
main difficulty of the non-reversible model arises from the fact that
these principles are constructed on the so-called flow space, which
is difficult to handle, especially when the energy landscape has a
complicated structure, as in the model that is considered in this
study.

\subsubsection*{Capacity estimate via $H^{1}$-approximation}

To overcome the above difficulty, we develop an intuitive and simple
means of estimating the capacity via the $H^{1}$-approximation, which
is the second main contribution of the study. 

The idea is quite simple. Suppose that a Markov process on a state
space $\mathfrak{S}$ is associated with a Dirichlet form $D(f)$
that acts on functions $f:\mathfrak{S}\rightarrow\mathbb{R}$. The
Dirichlet form $D(f)$ can be regarded as the (square of) the associated
$H^{1}$-norm of the function $f$. Subsequently, the capacity between
two disjoint sets $A,\,B\subset\mathfrak{S}$ is then defined as $D(h_{A,\,B})$,
where $h_{A,\,B}:\mathfrak{S}\rightarrow\mathbb{R}$ is the function
known as the \emph{equilibrium potential} between $A$ and $B$ (refer
to Section \ref{sec4.1} for details). Thus, the capacity estimate
can be formulated by an asymptotic relation 
\begin{equation}
D(h_{A,\,B})\simeq M\;.\label{e_H1.1}
\end{equation}
We observe that the proof of \eqref{e_H1.1} is reduced to constructing
a test function $\widetilde{h}$ that approximates $h_{A,\,B}$ in
the sense of the $H^{1}$-norm, that is,
\begin{equation}
D(\widetilde{h}-h_{A,\,B})\simeq0\;\;\;\;\text{and}\;\;\;\;D(\widetilde{h})\simeq M\;.\label{e_H1.2}
\end{equation}
An in-depth understanding of the energy landscape is necessary for
the construction of the function $\widetilde{h}$. The verification
of the second condition of \eqref{e_H1.2} is performed through a
direct computation, whereas that of the first condition requires large
deviation-type estimates on each metastable set. Note that the non-reversibility
does not play any role in this method, and therefore, its implementation
in the non-reversible model does not have particular difficulty, in
comparison with to that of the reversible one. Furthermore, we emphasize
that the flow space structure is not involved, even in the non-reversible
case. This new method is explained in further detail in Section \ref{sec4.5}.
Subsequently, we apply the method to the reversible and non-reversible
dynamics associated with the Ising and Potts models.

\section{\label{sec2}Models}

In this section, we introduce the Ising and Potts models and the associated
Markovian dynamics that exhibit metastability in the low temperature
regime in detail. 

\subsection{\label{sec2.1}Ising and Potts models}

\subsubsection*{Lattice}

For two positive integers $K$ and $L$, we denote by $\Lambda=\Lambda_{K,\,L}$
a two-dimensional rectangular lattice of size $K\times L$ on which
the spin system is defined. For simplicity, we only consider the case
in which $\Lambda$ is given periodic boundary conditions so that
we can write 
\begin{equation}
\Lambda=\mathbb{T}_{K}\times\mathbb{T}_{L}\;,\label{e_Lambda}
\end{equation}
where $\mathbb{T}_{n}:=\mathbb{Z}/(n\mathbb{Z})$ represents a discrete
one-dimensional torus. \emph{We remark that our arguments can also
be directly applied to the same model with open boundary conditions;
see Section \ref{sec3.5} for a summary.}

\subsubsection*{Spin configurations}

For an integer $q\ge2$, we denote by $S=\{1,\,\dots,\,q\}$ the collection
of $q$ spins and by $\mathcal{X}:=S^{\Lambda}$ the space of the
spin configurations on the lattice $\Lambda$. We express a configuration
\begin{equation}
\sigma\in\mathcal{X}\;\;\;\;\text{as}\;\;\;\;\sigma=(\sigma(x))_{x\in\Lambda}\;,\label{e_sigma}
\end{equation}
where $\sigma(x)\in S$ represents the spin of configuration $\sigma$
at site $x\in\Lambda$. 

\subsubsection*{Hamiltonian and Gibbs measure}

We define a Hamiltonian $H:\mathcal{X}\rightarrow\mathbb{R}$ as
\begin{equation}
H(\sigma)=\sum_{\{x,\,y\}\subset\Lambda:\,x\sim y}\mathbf{1}\{\sigma(x)\ne\sigma(y)\}\;\;\;\;;\;\sigma\in\mathcal{X}\;,\label{e_Ham}
\end{equation}
where $x\sim y$ if $x$ and $y$ are nearest neighbors in $\Lambda$.
Note that according to \eqref{e_Ham}, \emph{no external magnetic
field exists.}

Denote by $\mu_{\beta}$ the Gibbs measure on $\mathcal{X}$, which
realizes the energy function $H$ that is subjected to the inverse
temperature $\beta>0$:
\begin{equation}
\mu_{\beta}(\sigma)=\frac{1}{Z_{\beta}}e^{-\beta H(\sigma)}\;\;\;\;;\;\sigma\in\mathcal{X}\;,\label{e_mu}
\end{equation}
where $Z_{\beta}$ is the partition function defined by 
\begin{equation}
Z_{\beta}=\sum_{\zeta\in\mathcal{X}}e^{-\beta H(\zeta)}\;.\label{e_Zbeta}
\end{equation}
This probability measure $\mu_{\beta}$ is referred to as the Ising
measure if $q=2$ and the Potts measure if $q\ge3$.

\subsubsection*{Ground states}

For each $a\in S$, we use the bold-faced letter $\mathbf{a}\in\mathcal{X}$
to denote the monochromatic configuration in which all spins are $a$;
that is, (cf. \eqref{e_sigma})
\[
\mathbf{a}(x)=a\;\;\;\;\text{for all }x\in\Lambda\;.
\]
Subsequently, we collect the $q$ monochromatic configurations in
\begin{equation}
\mathcal{S}:=\{\mathbf{1},\,\mathbf{2},\,\dots,\,\mathbf{q}\}\;.\label{e_S}
\end{equation}
According to definition \eqref{e_Ham} of the Hamiltonian, $\mathcal{S}$
is precisely the collection of the minimizers, that is, the \textit{ground
states}, of the Hamiltonian $H$. The next theorem asserts that the
Gibbs measure $\mu_{\beta}$ is concentrated on the ground states
in the low temperature regime $\beta\rightarrow\infty$.
\begin{thm}
\label{t_mu}It holds that $Z_{\beta}=q+O(e^{-2\beta})$\footnote{We state that $f(\beta)=O(g(\beta))$ if there exists a constant $C$
independent of $\beta$ such that $|f(\beta)|\le Cg(\beta)$ for all
$\beta>0$.}. In turn,
\[
\lim_{\beta\to\infty}\mu_{\beta}(\mathcal{S})=1\;\;\;\text{and}\;\;\;\lim_{\beta\to\infty}\mu_{\beta}(\mathbf{s})=\frac{1}{q}\text{ for all }\mathbf{s}\in\mathcal{S}\;.
\]
\end{thm}

The proof is elementary, and we omit the details.
\begin{rem}
We present some remarks on the companion articles.
\begin{enumerate}
\item The same model on a three-dimensional lattice is considered in \cite{KS 3D}. 
\item In the current study and in \cite{KS 3D}, the regime $\beta\rightarrow\infty$
on the fixed lattice (i.e., fixed $K$ and $L$) is considered. The
model such that $K$ and $L$ increase together with $\beta$ is considered
in \cite{KS Growing}. 
\end{enumerate}
\end{rem}

\subsection{\label{sec2.2}Reversible and non-reversible Metropolis-type dynamics}

We introduce the stochastic Ising and Potts models, for which the
metastable behavior in the regime $\beta\rightarrow\infty$ is the
main concern of this study. These dynamics are continuous-time Markov
processes, the invariant measure of which is represented by the Gibbs
measure $\mu_{\beta}$, as introduced in \eqref{e_mu}. 

\subsubsection*{Reversible Metropolis--Hastings dynamics }

For $x\in\Lambda$, $a\in S$ and $\sigma\in\mathcal{X}$, we denote
by $\sigma^{x,\,a}\in\mathcal{X}$ the configuration obtained from
$\sigma$ by updating the spin at site $x$ to $a$. Then, we define
a continuous-time Markov chain $\{\sigma_{\beta}^{\mathrm{MH}}(t)\}_{t\ge0}$
on the configuration space $\mathcal{X}$, whose jump rates are given
by
\[
r_{\beta}^{\mathrm{MH}}(\sigma,\,\zeta)=\begin{cases}
e^{-\beta[H(\zeta)-H(\sigma)]_{+}} & \text{if }\zeta=\sigma^{x,\,a}\ne\sigma\text{ for some }x\in\Lambda\text{ and }a\in S\;,\\
0 & \text{otherwise}\;,
\end{cases}
\]
where $[t]_{+}=\max\{t,\,0\}$. This dynamics is known as the continuous-time
\emph{Metropolis--Hastings (MH) dynamics}, which is the standard
dynamics (cf. \cite{NS Ising1,NS Ising2}) in the context of the metastability
of spin systems in the very low temperature regime. Note that the
Markov process $\{\sigma_{\beta}^{\mathrm{MH}}(t)\}_{t\ge0}$ is reversible
with respect to the Gibbs measure $\mu_{\beta}$ since the following
detailed balance condition holds: 
\begin{equation}
\mu_{\beta}(\sigma)r_{\beta}^{\mathrm{MH}}(\sigma,\,\zeta)=\mu_{\beta}(\zeta)r_{\beta}^{\mathrm{MH}}(\zeta,\,\sigma)=\begin{cases}
\min\{\mu_{\beta}(\sigma),\,\mu_{\beta}(\zeta)\} & \text{if }\zeta=\sigma^{x,\,a}\ne\sigma\;,\\
0 & \text{otherwise}\;.
\end{cases}\label{e_cdtMH}
\end{equation}
In particular, $\mu_{\beta}$ is the invariant measure of the MH dynamics. 

\subsubsection*{Non-reversible cyclic dynamics}

We next introduce a \emph{non-reversible} spin-flipping dynamics,
also of a Metropolis type, under which the spin updates occur in a
cyclic manner. To define this dynamics, we denote by $\tau_{x}\sigma\in\mathcal{X}$
the configuration obtained from $\sigma$ by \emph{rotating} the spin
at $x$: 
\[
(\tau_{x}\sigma)(y)=\begin{cases}
\sigma(x)+1 & \text{if }y=x\;,\\
\sigma(y) & \text{otherwise}\;,
\end{cases}
\]
where we use the convention that $q+1=1$. Moreover, denote by $\{\sigma_{\beta}^{\mathrm{cyc}}(t)\}_{t\ge0}$
the continuous-time Markov process on $\mathcal{X}$ with jump rates
\[
r_{\beta}^{\mathrm{cyc}}(\sigma,\,\zeta)=\begin{cases}
R_{\beta}(\sigma;\,x) & \text{if }\zeta=\tau_{x}\sigma\text{ for some }x\in\Lambda\;,\\
0 & \text{otherwise}\;,
\end{cases}
\]
where
\begin{equation}
R_{\beta}(\sigma;\,x):=\exp\big\{-\beta\big[\max_{a\in S}H(\sigma^{x,\,a})-H(\sigma)\big]\big\}\;.\label{e_Rbeta}
\end{equation}
We refer to $\sigma_{\beta}^{\mathrm{cyc}}(\cdot)$ as the \emph{cyclic
dynamics}. Later, it is verified in Proposition \ref{p_muinv} that
the Gibbs measure $\mu_{\beta}(\cdot)$ is the invariant measure of
the model. 

For $q=2$, that is, for the Ising model, we obtain $\sigma_{\beta}^{\mathrm{cyc}}(\cdot)\equiv\sigma_{\beta}^{\mathrm{MH}}(\cdot)$.
However, for $q\ge3$, the cyclic dynamics is\textbf{\emph{ }}\emph{non-reversible}
as a rotation of a spin cannot be reversed, and thus, the cyclic dynamics
behaves in a fundamentally different manner to the reversible MH dynamics.
In summary, we can regard the cyclic dynamics as a \emph{non-reversible
generalization of the standard MH dynamics of the Potts model.} 

Finally, we remark that the metastability of a similar cyclic (non-reversible)
dynamics for the mean-field Potts model was studied in \cite{LS Potts,Lee}.

\subsubsection*{Notation of two dynamics}

We omit the explicit labels in $\sigma_{\beta}^{\mathrm{MH}}(\cdot)$
(resp. $r_{\beta}^{\textup{MH}}(\cdot,\,\cdot)$) and $\sigma_{\beta}^{\mathrm{cyc}}(\cdot)$
(resp. $r_{\beta}^{\textup{cyc}}(\cdot,\,\cdot)$) and simply denote
$\sigma_{\beta}(\cdot)$ (resp. $r_{\beta}(\cdot,\,\cdot)$) for both
dynamics when no risk of confusion arises. Moreover, we denote by
$\mathbb{P}_{\sigma}=\mathbb{P}_{\sigma}^{\beta}$ and $\mathbb{E}_{\sigma}=\mathbb{E}_{\sigma}^{\beta}$
the law and corresponding expectation of the process $\sigma_{\beta}(\cdot)$
starting from $\sigma\in\mathcal{X}$.

\subsubsection*{Other models}

All metastability results that are explained in the following section
for the MH/cyclic dynamics can be extended to several other interesting
non-reversible models. An introduction of these models and an explanation
of the corresponding main results are presented in Section \ref{sec5}
without detailed proofs.

\section{\label{sec3}Main Results}

In this section, we explain our main results. We fix the size $K,\,L$
of the lattice box (cf. \eqref{e_Lambda}) and assume, without loss
of generality, that $K\le L$. We also assume that $K,\,L\ge11$ for
a few technical reasons.

\subsection{\label{sec3.1}Energy barrier}

To investigate the metastable behavior of the dynamical system, first
one has to quantify the energy barrier between metastable sets. Since
the ground states operate as metastable states in the $\beta\rightarrow\infty$
regime, we compute the energy barrier between them.

\subsubsection*{Edge structure}

For $\sigma,\,\zeta\in\mathcal{X}$, we write $\sigma\sim\zeta$ if
the dynamics can jump from $\sigma$ to $\zeta$ or vice versa, that
is,
\begin{equation}
r_{\beta}(\sigma,\,\zeta)+r_{\beta}(\zeta,\,\sigma)>0\;.\label{e_edge}
\end{equation}
Note that $\sigma\sim\zeta$ if and only if $\zeta\sim\sigma$. This
relation depends on the selection of the dynamics, but not on the
value of $\beta>0$.

\subsubsection*{Height of paths}

We first define \textit{height} $H(\sigma,\,\zeta)$ between $\sigma$
and $\zeta$ such that $\sigma\sim\zeta$ as 
\begin{equation}
H(\sigma,\,\zeta)=\begin{cases}
\max\{H(\sigma),\,H(\zeta)\} & \text{for the MH dynamics}\;,\\
\max_{a\in S}H(\sigma^{x,\,a})\;\;\text{where }\zeta=\tau_{x}\sigma\text{ or }\sigma=\tau_{x}\zeta & \text{for the cyclic dynamics}\;.
\end{cases}\label{e_Hsigmazeta}
\end{equation}
Note that this height is defined so that for both dynamics and for
$\sigma,\,\zeta\in\mathcal{X}$ with $r_{\beta}(\sigma,\,\zeta)>0$,
\begin{equation}
\mu_{\beta}(\sigma)r_{\beta}(\sigma,\,\zeta)=Z_{\beta}^{-1}e^{-\beta H(\sigma,\,\zeta)}\;.\label{e_Hcdt}
\end{equation}

A sequence $\omega=(\omega_{n})_{n=0}^{N}$ of configurations is defined
as a \textit{path} from $\omega_{0}$ to $\omega_{N}$ if $r_{\beta}(\omega_{n},\,\omega_{n+1})>0$
for each $n\in\llbracket0,\,N-1\rrbracket$ (According to \eqref{e_edge},
this is a stronger requirement than $\omega_{n}\sim\omega_{n+1}$).
We write $\omega:\sigma\to\zeta$ if $\omega$ is a path from $\sigma$
to $\zeta$. Subsequently, the \textit{height} $\Phi_{\omega}$ of
a path $\omega=(\omega_{n})_{n=0}^{N}$ is defined as
\begin{equation}
\Phi_{\omega}:=\begin{cases}
\max_{n\in\llbracket0,\,N-1\rrbracket}H(\omega_{n},\,\omega_{n+1}) & \text{if }N\ge1\;,\\
H(\omega_{0}) & \text{if }N=0\;.
\end{cases}\label{e_height}
\end{equation}
We note that, particularly in the MH dynamics, it holds that
\begin{equation}
\Phi_{\omega}=\max_{n\in\llbracket0,\,N-1\rrbracket}\max\{H(\omega_{n}),\,H(\omega_{n+1})\}=\max_{n\in\llbracket0,\,N\rrbracket}H(\omega_{n})\;,\label{e_MHheight}
\end{equation}
but this result does not hold in the cyclic dynamics.

\subsubsection*{Communication height and energy barrier}

For $\sigma,\,\zeta\in\mathcal{X}$ (not necessarily $\sigma\sim\zeta$),
the \textit{communication height} $\Phi(\sigma,\,\zeta)$ from $\sigma$
to $\zeta$ is defined as
\begin{equation}
\Phi(\sigma,\,\zeta):=\min_{\omega:\sigma\to\zeta}\Phi_{\omega}\;.\label{e_commheight}
\end{equation}
It can be easily checked that
\begin{equation}
\Phi(\sigma,\,\sigma)=H(\sigma)\;\;\;\;\text{for all }\sigma\in\mathcal{X}\;.\label{e_commheight2}
\end{equation}
Moreover, for two disjoint subsets $\mathcal{P}$ and $\mathcal{Q}$
we define
\[
\Phi(\mathcal{P},\,\mathcal{Q}):=\min_{\sigma\in\mathcal{P}}\min_{\zeta\in\mathcal{Q}}\Phi(\sigma,\,\zeta)\;,
\]
where we use the convention $\Phi(\{\sigma\},\,\mathcal{Q})=\Phi(\sigma,\,\mathcal{Q})$
and $\Phi(\mathcal{P},\,\{\zeta\})=\Phi(\mathcal{P},\,\zeta)$. For
each $\mathbf{a}\in\mathcal{S}$, we write 
\begin{equation}
\breve{\mathbf{a}}:=\mathcal{S}\setminus\{\mathbf{a}\}\;.\label{e_brevedef}
\end{equation}
Then, the \textit{energy barrier} $\Gamma$ between ground states
is defined as 
\[
\Gamma:=\Phi(\mathbf{a},\,\breve{\mathbf{a}})\;\;\;\;;\;\mathbf{a}\in\mathcal{S}\;.
\]
Note that $\Gamma$ does not depend on the selection of $\mathbf{a}\in\mathcal{S}$
owing to the symmetry of the models. The following result characterizes
the energy barrier.
\begin{thm}
\label{t_EB}For all $\mathbf{a},\,\mathbf{b}\in\mathcal{S}$, it
holds that 
\begin{equation}
\Gamma=\Phi(\mathbf{a},\,\mathbf{b})=\begin{cases}
2K+2 & \text{for MH dynamics}\;,\\
2K+4 & \text{for cyclic dynamics with }q\ge3\;.
\end{cases}\label{e_GammaPhi}
\end{equation}
\end{thm}

The proof of this theorem for the MH dynamics has already been obtained
in \cite[Theorem 2.1]{NZ}. The proof for the cyclic dynamics is presented
in Section \ref{sec8.1}. We note that the relation $\Gamma=\Phi(\mathbf{a},\,\mathbf{b})$
for any $\mathbf{a},\,\mathbf{b}\in\mathcal{S}$ follows directly
from the symmetry of the dynamics in the case of the MH dynamics.
On the other hand, this needs to be proved separately in the case
of the cyclic dynamics.

\subsection{\label{sec3.2}Eyring--Kramers law}

In view of Theorem \ref{t_mu}, when $\beta$ is very large the Gibbs
measure is concentrated on the ground states in $\mathcal{S}$, and
therefore, the associated Markovian dynamics that was introduced in
the previous section is expected to exhibit metastable behavior. More
precisely, the process $\sigma_{\beta}(\cdot)$ starting from a ground
state is expected to make a transition to another one on a large time
scale. To describe this hopping dynamics among the ground states precisely,
first, the mean of the transition time from one ground state to another
must be calculated. This result is referred to as the\emph{ Eyring--Kramers
(EK) law}. 

Our first main result for the metastability of the dynamics defined
above is the EK law. For a set $\mathcal{P}\subseteq\mathcal{X}$,
we denote by $\tau_{\mathcal{P}}$ the hitting time of the set $\mathcal{P}$.
When $\mathcal{P}=\{\sigma\}$ is a singleton, we write $\tau_{\sigma}$
instead of $\tau_{\{\sigma\}}$. We again emphasize that $K\le L$
is assumed. The following theorem holds for both the MH and cyclic
dynamics.
\begin{thm}[Eyring--Kramers law]
\label{t_EK} There exists a constant $\kappa=\kappa(K,\,L)>0$ such
that for all $a,\,b\in S$ (cf. \eqref{e_brevedef} and \eqref{e_GammaPhi}),
\begin{equation}
\lim_{\beta\to\infty}e^{-\Gamma\beta}\mathbb{E}_{\mathbf{a}}[\tau_{\breve{\mathbf{a}}}]=\frac{\kappa}{q-1}\;\;\;\;\text{and}\;\;\;\;\lim_{\beta\to\infty}e^{-\Gamma\beta}\mathbb{E}_{\mathbf{a}}[\tau_{\mathbf{b}}]=\kappa\;.\label{e_EK}
\end{equation}
Moreover, the constant $\kappa$ satisfies
\begin{equation}
\kappa=\frac{\nu_{0}}{4}+o_{K}(1)\;,\label{e_kappaest}
\end{equation}
where $o_{K}(1)$ is a term that vanishes as $K\rightarrow\infty$
and where $\nu_{0}$ is a constant defined as
\begin{equation}
\nu_{0}=\begin{cases}
1 & \text{if }K<L\;,\\
1/2 & \text{if }K=L\;.
\end{cases}\label{e_nu0}
\end{equation}
\end{thm}

\begin{notation*}
In the remainder of this paper, any appearance of $\nu_{0}$ refers
to definition \eqref{e_nu0}.
\end{notation*}
We remark that in the case of $K<L$, the metastable transition occurs
in only one direction of the lattice, whereas in the case of $K=L$,
two possible (horizontal and vertical) directions exist. This difference
induces definition \eqref{e_nu0} of $\nu_{0}$. Furthermore, the
constants $\kappa(K,\,L)$ for the MH and cyclic dynamics differ (although
their limits in the regime $K\rightarrow\infty$ are equal). Refer
to Section \ref{sec3.4} for comments on the proof of Theorem \ref{t_EK}. 

\subsection{\label{sec3.3}Markov chain model reduction}

The aforementioned EK law provides a sharp estimate of the mean transition
time from one ground state to another. Furthermore, since the hopping
transitions between ground states occur successively, these sequential
jumps need to be understood in a more systematic manner. This can
be achieved using the method of Markov chain model reduction developed
in \cite{BL TM,BL TM2}, and this characterization is the next main
result of this study.

\subsubsection*{Approximation by trace process}

Since it can be guessed from the EK law (cf. Theorem \ref{t_EK})
that the transitions among the ground states occur in the time scale
$e^{\Gamma\beta}$, the original process first needs to be accelerated
by this factor to observe the inter-ground state jumps in the ordinary
time scale. Hence, we define $\widetilde{\sigma}_{\beta}(\cdot)=\sigma_{\beta}(e^{\Gamma\beta}\cdot)$.
Moreover, let $(Y_{\beta}(t))_{t\ge0}$ be the\emph{ trace process
}of $\widetilde{\sigma}_{\beta}(\cdot)$ on the set $\mathcal{S}$
of ground states, which is a continuous-time Markov process obtained
from the process $\widetilde{\sigma}_{\beta}(\cdot)$ by\emph{ turning
off the clock} when the process does not stay on $\mathcal{S}$. For
a more detailed explanation regarding this process, refer to \cite{BL TM}.

The following estimate verifies that the accelerated process $\widetilde{\sigma}_{\beta}(\cdot)$
does not spend meaningful time outside of $\mathcal{S}$, and therefore,
the trace process $Y_{\beta}(\cdot)$ successfully captures the inter-ground
state dynamics of the process $\widetilde{\sigma}_{\beta}(\cdot)$. 
\begin{thm}[Negligibility of excursions]
\label{t_neg} For both the MH and cyclic dynamics, we have 
\begin{equation}
\lim_{\beta\rightarrow\infty}\max_{\mathbf{a}\in\mathcal{S}}\mathbb{E}_{\mathbf{a}}\Big[\int_{0}^{T}\mathbf{1}\{\widetilde{\sigma}_{\beta}(u)\notin\mathcal{S}\}du\Big]=0\;\;\;\;\text{for all }T>0\;.\label{e_neg}
\end{equation}
\end{thm}

\begin{proof}
\noindent Denote by $\mathbb{P}_{\mu_{\beta}}$ the law of the process
$\sigma_{\beta}(\cdot)$ with starting distribution $\mu_{\beta}$.
Then for any $u>0$, we obtain
\[
\mathbb{P}_{\mathbf{a}}[\sigma_{\beta}(u)\notin\mathcal{S}]\le\frac{1}{\mu_{\beta}(\mathbf{a})}\mathbb{P}_{\mu_{\beta}}[\sigma_{\beta}(u)\notin\mathcal{S}]=\frac{\mu_{\beta}(\mathcal{X}\setminus\mathcal{S})}{\mu_{\beta}(\mathbf{a})}\;,
\]
where the final identity holds because $\mu_{\beta}$ is the invariant
distribution. Hence, by the Fubini theorem, 
\[
\mathbb{E}_{\mathbf{a}}\Big[\int_{0}^{t}\mathbf{1}\{\widetilde{\sigma}_{\beta}(u)\notin\mathcal{S}\}du\Big]=\int_{0}^{t}\mathbb{P}_{\mathbf{a}}[\sigma_{\beta}(e^{\Gamma\beta}u)\notin\mathcal{S}]du\le t\cdot\frac{\mu_{\beta}(\mathcal{X}\setminus\mathcal{S})}{\mu_{\beta}(\mathbf{a})}\;,
\]
which vanishes as $\beta\rightarrow\infty$ according to Theorem \ref{t_mu}. 
\end{proof}

\subsubsection*{Convergence of trace process}

We prove that the process $Y_{\beta}(\cdot)$ converges to a certain
continuous-time Markov process $Y(\cdot)$ on $\mathcal{S}$. This
limiting Markov chain $\{Y(t)\}_{t\ge0}$ is defined as a continuous-time
Markov chain on $\mathcal{S}$ with a uniform jump rate $r_{Y}(\cdot,\,\cdot)$
given by 
\begin{equation}
r_{Y}(\mathbf{a},\,\mathbf{b})=\kappa^{-1}\;\;\;\;\text{for all }\mathbf{a},\,\mathbf{b}\in\mathcal{S}\;,\label{e_LMC}
\end{equation}
where $\kappa=\kappa(K,\,L)$ is the constant that appears in Theorem
\ref{t_EK}. We subsequently obtain the following convergence result
for both the MH and cyclic dynamics. 
\begin{thm}[Markov chain model reduction]
\label{t_MC} Suppose that the original process $\sigma_{\beta}(\cdot)$
starts from $\mathbf{a}\in\mathcal{S}$. Then, the law of the Markov
chain $Y_{\beta}(\cdot)$ converges to the law of the limiting Markov
chain $Y(\cdot)$ starting from $\mathbf{a}$ as $\beta\rightarrow\infty$,
where the convergence occurs in the sense of the usual Skorokhod topology.
\end{thm}

By combining this result with Theorem \ref{t_neg}, we can describe
the metastable behavior (i.e., the inter-ground state) dynamics of
$\sigma_{\beta}(\cdot)$ on the time scale $e^{\beta\Gamma}$ using
the Markovian movement expressed by $Y(\cdot)$. Therefore, the package
of results in Theorems \ref{t_neg} and \ref{t_MC} is known in the
literature as the Markov chain model reduction of the metastable behavior.
This method provides a powerful explanation for the metastable behavior,
especially when multiple ground states exist, as in the current model.
Refer to \cite{Landim MMC} for further details on this methodology.
\begin{rem}[Coarse-graining effect]
 In the cyclic dynamics, the spin can be updated in only one direction
$1\rightarrow2\rightarrow\cdots\rightarrow q\rightarrow1$. However,
Theorem \ref{t_MC} states that the macroscopic jumps among the ground
states $\mathbf{1},\,\mathbf{2},\,\dots,\,\mathbf{q}$ occur in a
uniform manner. Thus, it can be inferred that a \emph{coarse-graining
effect} removes the cyclic feature of the microscopic world on the
macroscopic scale. 
\end{rem}

The proof of this type of result for non-reversible dynamics is much
more difficult than that for reversible dynamics, and to the authors'
knowledge, rigorous results have only been established for four models
\cite{KS1,LS NR1,LeeS NR2,Seo NRZRP}. In this study, we obtain the
fifth result.

\subsection{\label{sec3.4}Comments on proof}

It was demonstrated in \cite{BEGK} and \cite{BL TM2} that the proof
of the EK laws for both the reversible and non-reversible dynamics
is reduced to the estimation of a potential-theoretic notion known
as the capacity, which is defined in Section \ref{sec4.1}. The proof
of the Markov chain model reduction for reversible dynamics also relies
on the estimation of the capacity (cf. \cite{BL TM}), whereas \cite{LS NR1}
shows that the non-reversible case additionally requires the estimate
of the equilibrium potential on each metastable set.

In Section \ref{sec4}, we present a novel and robust strategy to
derive the estimates of the capacity as well as the equilibrium potential
based on the $H^{1}$-approximation. According to this strategy, the
proofs of Theorems \ref{t_EK} and \ref{t_MC} are reduced to constructions
of certain \emph{test functions}. This is a decent simplification
of the already-known methodologies, in that:
\begin{enumerate}
\item Classical potential-theoretic approaches (e.g., \cite{LS NR1,Seo NRZRP})
require the construction of a (divergence-free) test flow, which is
in general very difficult to achieve. 
\item This new method provides the estimate of the equilibrium potential
as a by-product (which is crucial in the proof of the Markov chain
model reduction for non-reversible dynamics).
\end{enumerate}
We expect that this strategy will be applicable to a wide range of
non-reversible models. The details of this method are discussed in
Section \ref{sec4.5}.

\subsubsection*{Analysis of energy landscape }

A deep understanding of the energy landscape as well as the behavior
of the dynamics thereon is required to construct a test function to
employ our strategy. The main difficulty of the zero external field
model that is considered in this study lies on the complexity of the
saddle structure. This structure consists of a large plateau with
a high magnitude of complexity (compared to the non-zero external
field model, in which the saddle structure is sharp and fully characterized
by the existence of a special form of critical droplets). This is
another major difficulty that we aim to overcome in this study.

However, it is impossible to describe the saddle structure without
the definition of numerous notions. Hence, we provide a detailed explanation
of the saddle structure for the MH and cyclic dynamics in Sections
\ref{sec6} and \ref{sec8}, respectively. Thereafter, we construct
the test functions for the MH and cyclic dynamics in Sections \ref{sec7}
and \ref{sec9}, respectively.

\subsection{\label{sec3.5}Remarks on open-boundary models}

A careful investigation of our proof reveals that all results explained
above can also be extended to the open-boundary model. In the open-boundary
model, the energy barrier is given by 
\[
\Gamma=\begin{cases}
K+1 & \text{for MH dynamics}\;,\\
K+3 & \text{for cyclic dynamics with }q\ge3\;,
\end{cases}
\]
and the constant $\kappa$ that appears in Theorem \ref{t_EK} satisfies
$\kappa(K,\,L)=\frac{1}{4KL}[\nu_{0}+o_{K}(1)]$. With these modifications,
Theorems \ref{t_EB}, \ref{t_EK}, \ref{t_neg} and \ref{t_MC} also
hold with these modifications.

\section{\label{sec4}General Strategy Based on $H^{1}$-Approximation}

In this section, we introduce the new strategy to prove the EK law
and Markov chain model reduction. We explain our strategy in terms
of the process $\sigma_{\beta}(\cdot)$ which is either the MH or
cyclic dynamics.

\subsection{\label{sec4.1}Preliminaries}

We first provide a brief but self-contained summary of the essential
background on the potential theory of Markov processes. For more detailed
explanation, we refer to \cite{BL TM2,GL,Slowik}. Note again that
we do not assume reversibility of the underlying process $\sigma_{\beta}(\cdot)$.

\subsubsection*{Equilibrium potential and capacity}

Let $\mathcal{P}$ and $\mathcal{Q}$ be disjoint and non-empty subsets
of $\mathcal{X}$.
\begin{enumerate}
\item The \textit{equilibrium potential} between $\mathcal{P}$ and $\mathcal{Q}$
is the function $h_{\mathcal{P},\,\mathcal{Q}}=h_{\mathcal{P},\,\mathcal{Q}}^{\beta}:\mathcal{X}\rightarrow[0,\,1]$
defined by 
\begin{equation}
h_{\mathcal{P},\,\mathcal{Q}}(\sigma)=\mathbb{P}_{\sigma}[\tau_{\mathcal{P}}<\tau_{\mathcal{Q}}]\;\;\;\;;\;\sigma\in\mathcal{X}\;.\label{e_eqp}
\end{equation}
This function can also be characterized as the unique solution of
the following equation: 
\begin{equation}
\begin{cases}
h_{\mathcal{P},\,\mathcal{Q}}\equiv1 & \text{on }\mathcal{P}\;,\\
h_{\mathcal{P},\,\mathcal{Q}}\equiv0 & \text{on }\mathcal{Q}\;,\\
\mathcal{L}_{\beta}h_{\mathcal{P},\,\mathcal{Q}}\equiv0 & \text{on }(\mathcal{P}\cup\mathcal{Q})^{c}\;,
\end{cases}\label{e_eqpsol}
\end{equation}
where $\mathcal{L}_{\beta}$ denotes the \emph{infinitesimal generator}
associated with the process $\sigma_{\beta}(\cdot)$ that acts on
each $f:\mathcal{X}\to\mathbb{R}$ as 
\begin{equation}
(\mathcal{L}_{\beta}f)(\sigma)=\sum_{\zeta\in\mathcal{X}}r_{\beta}(\sigma,\,\zeta)[f(\zeta)-f(\sigma)]\;.\label{e_gen}
\end{equation}
\item Denote by $D_{\beta}(\cdot)$ the \textit{Dirichlet form} associated
with the process $\sigma_{\beta}(\cdot)$; i.e., for each $f:\mathcal{X}\to\mathbb{R}$,
\begin{equation}
D_{\beta}(f)=\langle f,\,-\mathcal{L}_{\beta}f\rangle_{\mu_{\beta}}=\frac{1}{2}\sum_{\sigma,\,\zeta\in\mathcal{X}}\mu_{\beta}(\sigma)r_{\beta}(\sigma,\,\zeta)[f(\zeta)-f(\sigma)]^{2}\;,\label{e_Diri}
\end{equation}
where $\langle\cdot,\,\cdot\rangle_{\mu_{\beta}}$ denotes the $L^{2}(\mu_{\beta})$-inner
product defined by 
\[
\langle f,\,g\rangle_{\mu_{\beta}}:=\sum_{\sigma\in\mathcal{X}}f(\sigma)g(\sigma)\mu_{\beta}(\sigma)\;.
\]
\item The \textit{capacity} between $\mathcal{P}$ and $\mathcal{Q}$ is
then defined as
\begin{equation}
\mathrm{Cap}_{\beta}(\mathcal{P},\,\mathcal{Q}):=D_{\beta}(h_{\mathcal{P},\,\mathcal{Q}})\;.\label{e_Capdef}
\end{equation}
\end{enumerate}

\subsubsection*{Adjoint dynamics}

We denote by $r_{\beta}^{*}(\cdot,\,\cdot)$ the adjoint jump rate
with respect to the original transition rate $r_{\beta}(\cdot,\,\cdot)$;
i.e., for all $\sigma,\,\zeta\in\mathcal{X}$, 
\begin{equation}
r_{\beta}^{*}(\sigma,\,\zeta):=\frac{\mu_{\beta}(\zeta)r_{\beta}(\zeta,\,\sigma)}{\mu_{\beta}(\sigma)}\;.\label{e_adjrate}
\end{equation}
The adjoint dynamics $(\sigma_{\beta}^{*}(t))_{t\ge0}$ is the continuous-time
Markov process on $\mathcal{X}$ with jump rate $r_{\beta}^{*}(\cdot,\,\cdot)$.
The process $\sigma_{\beta}^{*}(\cdot)$ also admits $\mu_{\beta}$
as its unique invariant distribution (cf. \cite[Section 2]{GL}).
Note that if the process $\sigma_{\beta}(\cdot)$ is reversible, we
have $r_{\beta}^{*}\equiv r_{\beta}$ by the detailed balance condition
and hence it holds that $\sigma_{\beta}(\cdot)\equiv\sigma_{\beta}^{*}(\cdot)$.

For disjoint and non-empty subsets $\mathcal{P}$ and $\mathcal{Q}$
of $\mathcal{X}$, we can define the equilibrium potential $h_{\mathcal{P},\,\mathcal{Q}}^{*}$
and the capacity $\mathrm{Cap}_{\beta}^{*}(\mathcal{P},\,\mathcal{Q})$
with respect to the adjoint dynamics $\sigma_{\beta}^{*}(\cdot)$
in the same manner as above. It is well known, see e.g. \cite[display (2.4)]{GL},
that $\mathrm{Cap}_{\beta}(\mathcal{P},\,\mathcal{Q})=\mathrm{Cap}_{\beta}^{*}(\mathcal{P},\,\mathcal{Q})$. 

\subsection{\label{sec4.2}Reduction to potential-theoretic estimates}

In this subsection, we explain the relation between the potential-theoretic
notions explained above and our main theorems (namely, the EK law
and the Markov chain model reduction).

\subsubsection*{General strategy for EK law}

The proof of the EK law relies on the so-called magic formula discovered
for the reversible case in \cite{BEGK} and then extended to the non-reversible
case in \cite[Proposition A.2]{BL TM2}. This formula asserts that
for $\mathbf{a}\in\mathcal{S}$ and $\mathcal{Q}\subseteq\breve{\mathbf{a}}$
(cf. \eqref{e_brevedef}), we have the following expression for the
mean transition time: 
\begin{equation}
\mathbb{E}_{\mathbf{a}}[\tau_{\mathcal{Q}}]=\frac{1}{\mathrm{Cap}_{\beta}(\mathbf{a},\,\mathcal{Q})}\sum_{\zeta\in\mathcal{X}}\mu_{\beta}(\zeta)h_{\mathbf{a},\,\mathcal{Q}}^{*}(\zeta)\;.\label{e_Etau}
\end{equation}
By Theorem \ref{t_mu} and a trivial bound $0\le h_{\mathbf{a},\,\mathcal{Q}}^{*}\le1$
(cf. \eqref{e_eqp}), we have\footnote{Writing $f(\beta)=o(g(\beta))$ implies that $\lim_{\beta\to\infty}f(\beta)/g(\beta)=0$.}
\[
0\le\sum_{\zeta\in\mathcal{X}\setminus\mathcal{S}}\mu_{\beta}(\zeta)h_{\mathbf{a},\,\mathcal{Q}}^{*}(\zeta)\le\mu_{\beta}(\mathcal{X}\setminus\mathcal{S})=o(1)\;.
\]
Inserting this to \eqref{e_Etau} and applying Theorem \ref{t_mu}
again, we obtain
\begin{equation}
\mathbb{E}_{\mathbf{a}}[\tau_{\mathcal{Q}}]=\frac{1}{\mathrm{Cap}_{\beta}(\mathbf{a},\,\mathcal{Q})}\Big[\sum_{\mathbf{s}\in\mathcal{S}}\frac{1}{q}h_{\mathbf{a},\,\mathcal{Q}}^{*}(\mathbf{s})+o(1)\Big]\;.\label{e_Etau2}
\end{equation}
Therefore, in order to prove the EK law, \textbf{\emph{it suffices
to deduce sharp asymptotics of $\mathrm{Cap}_{\beta}(\mathbf{a},\,\mathcal{Q})$
and $h_{\mathbf{a},\,\mathcal{Q}}^{*}(\mathbf{s})$ for each $\mathbf{s}\in\mathcal{S}$.}}

\subsubsection*{General strategy for Markov chain model reduction}

Denote by $r_{Y_{\beta}}(\cdot,\,\cdot):\mathcal{S}\times\mathcal{S}\rightarrow[0,\,\infty)$
the jump rate of the trace process $Y_{\beta}(\cdot)$ which is indeed
a Markov process on $\mathcal{S}$. Then, to prove the convergence
of $Y_{\beta}(\cdot)$ to the process $Y(\cdot)$ with jump rate \eqref{e_LMC},
it suffices to prove that 
\begin{equation}
\lim_{\beta\rightarrow\infty}r_{Y_{\beta}}(\mathbf{a},\,\mathbf{b})=\frac{1}{\kappa}\;.\label{e_rYbeta}
\end{equation}
It has been discovered in \cite{BL TM2,LS NR1} that
\begin{equation}
r_{Y_{\beta}}(\mathbf{a},\,\mathbf{b})=e^{\beta\Gamma}\frac{\mathrm{Cap}_{\beta}(\mathbf{a},\,\breve{\mathbf{a}})}{\mu_{\beta}(\mathbf{a})}h_{\mathbf{b},\,\mathcal{S}\setminus\{\mathbf{a},\,\mathbf{b}\}}(\mathbf{a})\;\;\;\;;\;\mathbf{a},\,\mathbf{b}\in\mathcal{S}\;.\label{e_MCMR}
\end{equation}
Thus, in view of Theorem \ref{t_mu}, to prove \eqref{e_rYbeta},
\textbf{\emph{it suffices to estimate $\mathrm{Cap}_{\beta}(\mathbf{a},\,\breve{\mathbf{a}})$
and $h_{\mathbf{b},\,\mathcal{S}\setminus\{\mathbf{a},\,\mathbf{b}\}}(\mathbf{a})$
for $\mathbf{a},\,\mathbf{b}\in\mathcal{S}$.}}

Summing up, we can observe that Theorems \ref{t_EK} and \ref{t_MC}
are consequences of sharp estimates of $\text{Cap}_{\beta}(\mathcal{P},\,\mathcal{Q})$,
$h_{\mathcal{P},\,\mathcal{Q}}(\mathbf{s})$ and $h_{\mathcal{P},\,\mathcal{Q}}^{*}(\mathbf{s})$
for disjoint, non-empty subsets $\mathcal{P},\,\mathcal{Q}\subset\mathcal{S}$
and $\mathbf{s}\in\mathcal{S}$. These estimates will be established
in the next subsection (cf. Theorems \ref{t_Cap} and \ref{t_eqp}). 

\subsection{Estimates of capacity and equilibrium potential}

\subsubsection*{Auxiliary constants}

To state the main estimates, we first need to introduce the bulk and
edge constants which stand for the bulk and edge parts, respectively,
of the saddle structure.

The \emph{bulk constant} $\mathfrak{b}=\mathfrak{b}(K,\,L)$ is defined
by (cf. \eqref{e_nu0})
\begin{equation}
\mathfrak{b}:=\begin{cases}
\nu_{0}\frac{(K+2)(L-4)}{4KL} & \text{for MH dynamics}\;,\\
\nu_{0}\frac{(K-2)(L-4)}{4KL} & \text{for cyclic dynamics with }q\ge3\;,
\end{cases}\label{e_bdef}
\end{equation}
while the \emph{edge constant} $\mathfrak{e}=\mathfrak{e}(K,\,L)$
is defined by
\begin{equation}
\mathfrak{e}:=\frac{\nu_{0}}{L}\mathfrak{e}_{0}\;,\label{e_edef}
\end{equation}
where a complicated expression for the constant $\mathfrak{e}_{0}$
will be given in \eqref{e_e0def} and \eqref{e_e0cycdef} for the
MH and cyclic dynamics, respectively. Finally, the constant $\kappa$
that appears in Theorem \ref{t_EK} is defined by 
\begin{equation}
\kappa:=\mathfrak{b}+2\mathfrak{e}\;.\label{e_kappadef}
\end{equation}

\begin{lem}
\label{l_kappa}The constant $\kappa$ satisfies \eqref{e_kappaest}.
\end{lem}

\begin{proof}
By Propositions \ref{p_e0est} and \ref{p_e0barabest}, we have $0<\mathfrak{e}\le\frac{1}{L}$
and $0<\mathfrak{e}<\frac{2}{L}$ for the MH and cyclic dynamics,
respectively. Thus, $\mathfrak{e}=o_{K}(1)$ and therefore the conclusion
of the lemma follows immediately from the expression \eqref{e_bdef}
of the bulk constant.
\end{proof}

\subsubsection*{Main estimates}

We start from the capacity estimate. 
\begin{thm}
\label{t_Cap} For two non-empty, disjoint subsets $\mathcal{P}$
and $\mathcal{Q}$ of $\mathcal{S}$, it holds that 
\begin{equation}
\mathrm{Cap}_{\beta}(\mathcal{P},\,\mathcal{Q})=(1+o(1))\cdot c_{0}(\mathcal{P},\,\mathcal{Q})e^{-\Gamma\beta}\;,\label{e_CapMH}
\end{equation}
where $c_{0}(\mathcal{P},\,\mathcal{Q}):=\frac{|\mathcal{P}||\mathcal{Q}|}{\kappa(|\mathcal{P}|+|\mathcal{Q}|)}$.
\end{thm}

Next, we state the estimate of equilibrium potentials. 
\begin{thm}
\label{t_eqp}For non-empty, disjoint subsets $\mathcal{P}$ and $\mathcal{Q}$
of $\mathcal{S}$ and $\mathbf{s}\in\mathcal{S}\setminus(\mathcal{P}\cup\mathcal{Q})$,
it holds that
\[
h_{\mathcal{P},\,\mathcal{Q}}(\mathbf{s}),\;h_{\mathcal{P},\,\mathcal{Q}}^{*}(\mathbf{s})=(1+o(1))\cdot\frac{|\mathcal{P}|}{|\mathcal{P}|+|\mathcal{Q}|}\;.
\]
\end{thm}

Up to this point, we followed a standard argument from the study of
metastability. The main difficulty in this approach lies on the proof
of these two theorems, especially when the dynamics is non-reversible.
Our main achievement in this study is developing a novel and simple
argument proving these two theorems. This will be explained in Section
\ref{sec4.5}.

\subsection{\label{sec4.4}Proof of the main results}

Before proceeding to the proof of Theorems \ref{t_Cap} and \ref{t_eqp},
we assume that these theorems hold and then prove the main results.
\begin{proof}[Proof of Theorem \ref{t_EK}]
 We fix $\mathbf{a}\in\mathcal{S}$ and $\mathcal{Q}\subseteq\breve{\mathbf{a}}$,
so that the estimate \eqref{e_Etau2} on $\mathbb{E}_{\mathbf{a}}[\tau_{\mathcal{Q}}]$
holds. Applying Theorems \ref{t_Cap} and \ref{t_eqp} to this estimate,
we obtain
\[
\mathbb{E}_{\mathbf{a}}[\tau_{\mathcal{Q}}]=\frac{1}{(1+o(1))\cdot\frac{1\cdot|\mathcal{Q}|}{\kappa(1+|\mathcal{Q}|)}e^{-\Gamma\beta}}\Big[\frac{1}{q}+\frac{1}{q}\cdot\frac{q-1-|\mathcal{Q}|}{1+|\mathcal{Q}|}+o(1)\Big]=(1+o(1))\cdot\frac{\kappa}{|\mathcal{Q}|}e^{\Gamma\beta}\;.
\]
Inserting $\mathcal{Q}=\breve{\mathbf{a}}$ or $\{\mathbf{b}\}$ completes
the proof.
\end{proof}
\begin{proof}[Proof of Theorem \ref{t_MC}]
 Applying Theorems \ref{t_mu}, \ref{t_Cap} and \ref{t_eqp} to
the expression \eqref{e_MCMR} of the jump rate $r_{Y_{\beta}}(\cdot,\,\cdot)$
yields that, for all $\mathbf{a},\,\mathbf{b}\in\mathcal{S}$, 
\[
r_{Y_{\beta}}(\mathbf{a},\,\mathbf{b})=e^{\Gamma\beta}\cdot\frac{(1+o(1))\cdot\frac{1\cdot(q-1)}{\kappa(1+(q-1))}e^{-\Gamma\beta}}{\frac{1}{q}+o(1)}\cdot\Big(\frac{1}{q-1}+o(1)\Big)=\frac{1}{\kappa}+o(1)\;.
\]
This completes the proof.
\end{proof}

\subsection{\label{sec4.5}New strategy based on $H^{1}$-approximation}

We now explain our new strategy to prove Theorems \ref{t_Cap} and
\ref{t_eqp}. 

\subsubsection*{$H^{1}$-approximation of equilibrium potentials}

In view of the expression \eqref{e_Capdef}, the natural first step
to estimate the capacity is to find good approximations of the equilibrium
potentials $h_{\mathcal{P},\,\mathcal{Q}}$ and $h_{\mathcal{P},\,\mathcal{Q}}^{*}$.
We noted at the end of the introduction that the Dirichlet form $D_{\beta}(f)$
can be regarded as the square of the $H^{1}$-norm, and thus the approximation
should be carried out in the sense of this norm.
\begin{prop}
\label{p_eqpappr}Let $\mathcal{P}$ and $\mathcal{Q}$ be two disjoint
and non-empty subsets of $\mathcal{S}$. Then, there exist $\widetilde{h}=\widetilde{h}_{\mathcal{P},\,\mathcal{Q}}$
and $\widetilde{h}^{*}=\widetilde{h}_{\mathcal{P},\,\mathcal{Q}}^{*}$
such that all of the following properties hold. 
\begin{enumerate}
\item Two functions $\widetilde{h}$ and $\widetilde{h}^{*}$ approximate
$h_{\mathcal{P},\,\mathcal{Q}}$ and $h_{\mathcal{P},\,\mathcal{Q}}^{*}$
in the sense that 
\begin{equation}
D_{\beta}(h_{\mathcal{P},\,\mathcal{Q}}-\widetilde{h}),\;D_{\beta}(h_{\mathcal{P},\,\mathcal{Q}}^{*}-\widetilde{h}^{*})=o(e^{-\beta\Gamma})\;.\label{e_eqpappr1}
\end{equation}
\item We have
\begin{equation}
D_{\beta}(\widetilde{h}),\;D_{\beta}(\widetilde{h}^{*})=(1+o(1))\cdot c_{0}(\mathcal{P},\,\mathcal{Q})e^{-\Gamma\beta}\;,\label{e_eqpappr2}
\end{equation}
where $c_{0}(\mathcal{P},\,\mathcal{Q})$ is the constant that appears
in Theorem \ref{t_Cap}.
\item The values of $\widetilde{h}$ and $\widetilde{h}^{*}$ on the set
$\mathcal{S}$ of ground states are given as
\begin{equation}
\widetilde{h}(\mathbf{s})=\widetilde{h}^{*}(\mathbf{s})=\begin{cases}
1 & \text{if }\mathbf{s}\in\mathcal{P}\;,\\
0 & \text{if }\mathbf{s}\in\mathcal{Q}\;,\\
\frac{|\mathcal{P}|}{|\mathcal{P}|+|\mathcal{Q}|} & \text{if }\mathbf{s}\in\mathcal{S}\setminus(\mathcal{P}\cup\mathcal{Q})\;.
\end{cases}\label{e_eqpappr3}
\end{equation}
\end{enumerate}
\end{prop}

The proof of this proposition, i.e., the construction of $\widetilde{h}$
and $\widetilde{h}^{*}$ satisfying all properties above will be given
in Sections \ref{sec7} and \ref{sec9} for the MH and cyclic dynamics,
respectively. We note that the properties \eqref{e_eqpappr2} and
\eqref{e_eqpappr3} are inspired from Theorems \ref{t_Cap} and \ref{t_eqp},
respectively.

We now conclude this section by explaining the general strategy to
derive Theorems \ref{t_Cap} and \ref{t_eqp} from Proposition \ref{p_eqpappr}.
In the remainder of the current section, we fix two disjoint and non-empty
subsets $\mathcal{P}$ and $\mathcal{Q}$ of $\mathcal{S}$ and abbreviate
$h=h_{\mathcal{P},\,\mathcal{Q}}$ and $h^{*}=h_{\mathcal{P},\,\mathcal{Q}}^{*}$.
Let us first look at Theorem \ref{t_Cap}.
\begin{proof}[Proof of Theorem \ref{t_Cap}]
 In principle, the estimate of $\mathrm{Cap}_{\beta}(\mathcal{P},\,\mathcal{Q})=D_{\beta}(h)$
should follow from that of $D_{\beta}(h-\widetilde{h})$ and $D_{\beta}(\widetilde{h})$
obtained in \eqref{e_eqpappr1} and \eqref{e_eqpappr2}, respectively,
since $D_{\beta}(\cdot)^{1/2}$ defines a seminorm. Indeed, by the
triangle inequality, we have 
\[
D_{\beta}(\widetilde{h})^{1/2}-D_{\beta}(h-\widetilde{h})^{1/2}\le D_{\beta}(h)^{1/2}\le D_{\beta}(\widetilde{h})^{1/2}+D_{\beta}(h-\widetilde{h})^{1/2}\;,
\]
and thus by \eqref{e_eqpappr1} and \eqref{e_eqpappr2}, we obtain
\[
D_{\beta}(h)=(1+o(1))\cdot c_{0}(\mathcal{P},\,\mathcal{Q})e^{-\Gamma\beta}\;.
\]
This completes the proof since $\mathrm{Cap}_{\beta}(\mathcal{P},\,\mathcal{Q})=D_{\beta}(h)$.
\end{proof}
\begin{rem}
It should be emphasized here that condition (3) of Proposition \ref{p_eqpappr}
is not used in the proof of the capacity estimate. It will only be
used in the proof of Theorem \ref{t_eqp}.
\end{rem}

Next, we explain the robust proof of Theorem \ref{t_eqp}. It should
be remarked that there was no robust method known in the literature
to estimate the equilibrium potential.
\begin{proof}[Proof of Theorem \ref{t_eqp}]
 We only prove the theorem for $h=h_{\mathcal{P},\,\mathcal{Q}}$
since the proof for $h^{*}$ is identical. Let us fix $\mathbf{s}\in\mathcal{S}\setminus(\mathcal{P\cup\mathcal{Q}})$.
Write 
\[
\delta(\mathbf{s}):=h(\mathbf{s})-\widetilde{h}(\mathbf{s})=h(\mathbf{s})-\Big(\frac{|\mathcal{P}|}{|\mathcal{P}|+|\mathcal{Q}|}+o(1)\Big)\;,
\]
so that we wish to prove $\delta(\mathbf{s})=o(1)$. Note that there
is nothing to prove if $\delta(\mathbf{s})=0$. Otherwise, write 
\[
G(\sigma)=\frac{1}{\delta(\mathbf{s})}\big[h(\sigma)-\widetilde{h}(\sigma)\big]\;\;\;\;;\;\sigma\in\mathcal{X}\;.
\]
By \eqref{e_eqpappr1}, we have 
\[
D_{\beta}(G)=\frac{1}{\delta(\mathbf{s})^{2}}D_{\beta}(h-\widetilde{h})=\frac{1}{\delta(\mathbf{s})^{2}}o(e^{-\beta\Gamma})\;.
\]
Therefore, it suffices to prove that there exists a constant $C>0$
independent of $\beta$ such that $D_{\beta}(G)\ge Ce^{-\beta\Gamma}$.
This follows from the fact that $G(\mathbf{s})=1$, $G(\sigma)=0$
for all $\sigma\in\mathcal{P\cup\mathcal{Q}}$, and the next lemma. 
\end{proof}
\begin{lem}
\label{l_DFC}Let $\mathbf{s}\in\mathcal{S}\setminus(\mathcal{P}\cup\mathcal{Q})$.
Then, there exists a constant $C>0$ independent of $\beta$ such
that 
\begin{equation}
D_{\beta}(F)\ge Ce^{-\beta\Gamma}\label{e_DF}
\end{equation}
for all $F:\mathcal{X}\rightarrow\mathbb{R}$ which satisfy $F(\mathbf{s})=1$
and $F(\sigma)=0$ for all $\sigma\in\mathcal{P\cup\mathcal{Q}}$. 
\end{lem}

Now, we finally discuss the proof of Lemma \ref{l_DFC} which is the
last ingredient in our robust argument. First, we present few standard
strategies regarding the proof of Lemma \ref{l_DFC}.
\begin{enumerate}
\item[\textbf{(1)}]  If $\sigma_{\beta}(\cdot)$ is reversible (e.g., the MH dynamics),
by the Dirichlet principle for reversible Markov chains (cf. \cite{BdenH Meta}),
we have $D_{\beta}(F)\ge\mathrm{Cap}_{\beta}(\mathbf{s},\,\mathcal{P}\cup\mathcal{Q})$.
Thus, the lemma is a direct consequence of our capacity estimate,
i.e., Theorem \ref{t_Cap}.\smallskip{}
\item[\textbf{(2)}]  If $\sigma_{\beta}(\cdot)$ is non-reversible, we can define the
symmetrized dynamics $\sigma_{\beta}^{s}(\cdot)$ with jump rate $\frac{1}{2}(r_{\beta}+r_{\beta}^{*})(\cdot,\,\cdot)$.
This dynamics is reversible with respect to $\mu_{\beta}$. Denote
by $\mathrm{Cap}_{\beta}^{s}(\cdot,\,\cdot)$ the capacity with respect
to $\sigma_{\beta}^{s}(\cdot)$. Since the Dirichlet form associated
with the process $\sigma_{\beta}^{s}(\cdot)$ is again $D_{\beta}(\cdot)$,
by the Dirichlet principle for reversible Markov chains, we have
\begin{equation}
D_{\beta}(F)\ge\mathrm{Cap}_{\beta}^{s}(\mathbf{s},\,\mathcal{P}\cup\mathcal{Q})\;.\label{e_DFs}
\end{equation}
Hence, it suffices to estimate $\mathrm{Cap}_{\beta}^{s}(\mathbf{s},\,\mathcal{P}\cup\mathcal{Q})$.\smallskip{}
\\
\textbf{(2-1)} Since the capacity estimate of reversible dynamics
cannot be more difficult than non-reversible ones, we can readily
derive a result similar to Theorem \ref{t_Cap} for $\mathrm{Cap}_{\beta}^{s}(\cdot,\,\cdot)$
(of course, the constant $c_{0}(\mathcal{P},\,\mathcal{Q})$ in the
right-hand side must be modified). Such a result along with \eqref{e_DFs}
directly implies \eqref{e_DF}.\smallskip{}
\\
\textbf{(2-2)} On the other hand if, for all $\beta>0$, the process
$\sigma_{\beta}(\cdot)$ satisfies the so-called \emph{sector condition}
(cf. \cite[Section 2]{GL}) with a certain constant $C_{0}>0$ independent
of $\beta$, then by \cite[Lemma 2.6]{GL},
\[
C_{0}\mathrm{Cap}_{\beta}^{s}(\mathbf{s},\,\mathcal{P}\cup\mathcal{Q})\ge\mathrm{Cap}_{\beta}(\mathbf{s},\,\mathcal{P}\cup\mathcal{Q})\;.
\]
Thus, in this case Theorem \ref{t_Cap} along with \eqref{e_DFs}
immediately implies \eqref{e_DF}.
\end{enumerate}
In our model, it is possible to use strategy \textbf{(2-2)} since
an elementary computation reveals that the cyclic dynamics satisfies
the sector condition with constant $C_{0}=q^{2}$. However, instead
of explaining this tedious detail, we provide another simple proof
computing $D_{\beta}(F)$ directly, which is also suitable for our
model.
\begin{proof}[Proof of Lemma \ref{l_DFC}]
 Let us take $\mathbf{a}\in\mathcal{P}\cup\mathcal{Q}$. By Theorem
\ref{t_EB}, there exists a path $\omega=(\omega_{n})_{n=0}^{N}$
from $\mathbf{a}$ to $\mathbf{s}$ such that $H(\omega_{n},\,\omega_{n+1})\le\Gamma$
for all $n\in\llbracket0,\,N-1\rrbracket$\footnote{We define  $\llbracket a,\,b\rrbracket:=[a,\,b]\cap\mathbb{Z}$.}.
Note that this path, and thus $N$, is independent of the inverse
temperature $\beta$. 

Since $r_{\beta}(\omega_{n},\,\omega_{n+1})>0$ and $H(\omega_{n},\,\omega_{n+1})\le\Gamma$,
it holds from Theorem \ref{t_mu} and \eqref{e_Hcdt} that
\[
\mu_{\beta}(\omega_{n})r_{\beta}(\omega_{n},\,\omega_{n+1})=\frac{1}{Z_{\beta}}e^{-\beta H(\omega_{n},\,\omega_{n+1})}\ge ce^{-\beta\Gamma}
\]
for some $c>0$. Therefore, by \eqref{e_eqpappr1} and the formula
\eqref{e_Diri} for the Dirichlet form, it holds that
\[
D_{\beta}(F)\ge\frac{1}{2}\sum_{n=0}^{N-1}\mu_{\beta}(\omega_{n})r_{\beta}(\omega_{n},\,\omega_{n+1})[F(\omega_{n+1})-F(\omega_{n})]^{2}\ge\frac{c}{2}e^{-\beta\Gamma}\sum_{n=0}^{N-1}[F(\omega_{n+1})-F(\omega_{n})]^{2}\;.
\]
Since $F(\omega_{N})=F(\mathbf{s})=1$ and $F(\omega_{0})=F(\mathbf{a})=0$,
by the Cauchy--Schwarz inequality, the last summation is bounded
from below by $N^{-1}$. This completes the proof. 
\end{proof}

\subsection{Concluding remark}

The idea of finding proxies of equilibrium potentials is not novel.
Such proxies already played a significant role in the previous potential-theoretic
approach to metastability as well, since they provide approximations
of the optimizers of the Dirichlet and Thomson principles. The point
here is that the new method suggested above looks at these proxies
in a completely different point of view, and in turn provides much
more straightforward way to handle non-reversible situations and also
the equilibrium potential. We refer to Remark \ref{r_analogue} for
more detail.

\section{\label{sec5}Other Non-reversible Models}

Before proceeding to the analysis of the energy landscape and construction
of the test functions, we explain some other interesting \textbf{\emph{non-reversible}}
dynamics associated with the Ising and Potts models that fall into
the framework of Section \ref{sec4}. Proofs of the EK law and the
Markov chain model reduction for these models are similar with that
of the cyclic model, and thus we will omit repetition of the proof.

\subsection{\label{sec5.1}Generalized cyclic dynamics on the Potts model }

In the cyclic dynamics defined in Section \ref{sec2.2}, a spin can
only be rotated in the increasing order $(1\rightarrow2\rightarrow\cdots)$.
One can also imagine a model for which the spin can rotate in the
opposite direction as well. To introduce this model, let us define
an operator $\tau_{x}^{-1}:\mathcal{X}\rightarrow\mathcal{X}$ by

\[
(\tau_{x}^{-1}\sigma)(y)=\begin{cases}
\sigma(x)-1 & \text{if }y=x\;,\\
\sigma(y) & \text{otherwise}\;,
\end{cases}
\]
where we use the convention $1-1=q$. Then, for $p\in[0,\,1]$, define
the jump rate $r_{\beta,\,p}(\cdot,\,\cdot)$ on $\mathcal{X}\times\mathcal{X}$
as

\[
r_{\beta,\,p}(\sigma,\,\zeta)=\begin{cases}
pR_{\beta}(\sigma;\,x) & \text{if }\zeta=\tau_{x}\sigma\;,\;x\in\Lambda\;,\\
(1-p)R_{\beta}(\sigma;\,x) & \text{if }\zeta=\tau_{x}^{-1}\sigma\;,\;x\in\Lambda\;,\\
0 & \text{otherwise}\;,
\end{cases}
\]
where $R_{\beta}(\sigma;\,x)$ is defined in \eqref{e_Rbeta}. We
denote by \textit{$p$-cyclic dynamics} the continuous-time Markov
process with jump rate $r_{\beta,\,p}(\cdot,\,\cdot)$. Note that
the original cyclic dynamics corresponds to the $1$-cyclic dynamics
(i.e., $p=1$). 
\begin{prop}
\label{p_muinv}For all $p\in[0,\,1]$, the unique invariant measure
of the $p$-cyclic dynamics is the Gibbs measure $\mu_{\beta}$.
\end{prop}

\begin{proof}
Since it is clear that the $p$-cyclic dynamics is irreducible, a
unique invariant measure exists. In order to show that this is indeed
$\mu_{\beta}$, it suffices to check that for all $\sigma\in\mathcal{X}$,
\begin{equation}
\sum_{x\in\Lambda}\mu_{\beta}(\sigma)r_{\beta,\,p}(\sigma,\,\tau_{x}^{\pm1}\sigma)=\sum_{x\in\Lambda}\mu_{\beta}(\tau_{x}^{\pm1}\sigma)r_{\beta,\,p}(\tau_{x}^{\pm1}\sigma,\,\sigma)\;.\label{e_muinv}
\end{equation}
By a direct computation based on the explicit formula for $\mu_{\beta}$
and $r_{\beta,\,p}$, we can check that both sides equal
\[
\frac{1}{Z_{\beta}}\sum_{x\in\Lambda}\exp\big\{-\beta\max_{a\in S}H(\sigma^{x,\,a})\big\}\;,
\]
and therefore the proof is completed.
\end{proof}
\textbf{\emph{Suppose from now on that $q\ge3$, so that the $p$-cyclic
dynamics is not the MH dynamics.}} Then, one can readily check that
the $p$-cyclic dynamics is reversible if and only if $p=1/2$, and
the symmetrization (cf. strategy \textbf{(2-1)} after Lemma \ref{l_DFC})
of any $p$-cyclic dynamics is the $1/2$-cyclic dynamics. All results
for the cyclic dynamics can be extended to the $p$-cyclic dynamics
as follows.
\begin{thm}
\label{t_pcyc}For all $p\in[0,\,1]$, the following statements hold
for the $p$-cyclic dynamics.
\begin{enumerate}
\item The energy barrier is $\Gamma=2K+4$.
\item (EK law) There exist constants $\kappa_{i}=\kappa_{i}(K,\,L,\,p)>0$,
$i\in\llbracket1,\,q-1\rrbracket$, such that
\begin{equation}
\lim_{\beta\to\infty}e^{-\Gamma\beta}\mathbb{E}_{\mathbf{a}}[\tau_{\breve{\mathbf{a}}}]=\frac{1}{\kappa_{1}^{-1}+\cdots+\kappa_{q-1}^{-1}}\;\;\;\text{and}\;\;\;\lim_{\beta\to\infty}e^{-\Gamma\beta}\mathbb{E}_{\mathbf{a}}[\tau_{\mathbf{b}}]=\kappa_{|b-a|}\label{e_EKpcyc}
\end{equation}
for all $a,\,b\in S$. The constants $(\kappa_{i})_{i=1}^{q-1}$ satisfy
$\kappa_{i}=\kappa_{q-i}$ for each $i\in\llbracket1,\,q-1\rrbracket$
and 
\begin{equation}
\kappa_{i}(K,\,L,\,p)=\nu_{0}\frac{(1+\hat{p})(1+\cdots+\hat{p}^{i-1})(1+\cdots+\hat{p}^{q-i-1})}{4(1+\cdots+\hat{p}^{q-1})}+o_{K}(1)\;,\label{e_kappaiest}
\end{equation}
 where $\nu_{0}$ is the constant defined in \eqref{e_nu0} and $\hat{p}=\frac{\min\{p,\,1-p\}}{\max\{p,\,1-p\}}$.
\item (Markov chain model reduction) Theorems \ref{t_neg} and \ref{t_MC}
hold with limiting Markov chain $Y(\cdot)$ defined by the jump rate
$r_{Y}(\mathbf{a},\,\mathbf{b})=\kappa_{|b-a|}^{-1}$.
\end{enumerate}
\end{thm}

\begin{rem}[Comparison of transition rates]
 By an elementary computation, we can verify that the fraction in
the right-hand side of \eqref{e_kappaiest} is increasing in $p\in[0,\,1/2]$
and decreasing in $p\in[1/2,\,1]$. Therefore, in view of the EK law
\eqref{e_EKpcyc}, we can observe that the speed of transition becomes
slower if $p$ approaches to $1/2$, i.e., to reversibility. We can
draw the same conclusion when we investigate the Markov chain model
reduction. \textbf{\emph{This verifies a widely-believed fact that
a non-reversible dynamics runs faster than the symmetrized reversible
dynamics.}}
\end{rem}

\subsection{\label{sec5.2}Directed dynamics on the Ising model}

One can readily observe that the $p$-cyclic dynamics defined above
is identical to the reversible MH dynamics if $q=2$. Therefore, so
far we did not consider non-reversible dynamics associated with the
Ising model. In this subsection, we introduce such a non-reversible
dynamics, first constructed in \cite{GB}, for which the analysis
of metastable behavior can be done in a similar manner. 

Suppose from now on that $q=2$, so that $S=\{1,\,2\}$. For each
edge $e=\{x,\,y\}$ in $\Lambda$ (i.e., with $x\sim y$), write 
\[
H_{e}(\sigma):=\mathbf{1}\{\sigma(x)\ne\sigma(y)\}\;,
\]
so that by definition we can write $H(\sigma)=\sum_{e:\,\text{edge of }\Lambda}H_{e}(\sigma)$.

For $\sigma\in\mathcal{X}$, we denote by $\sigma^{x}$, $x\in\Lambda$,
the configuration obtained from $\sigma$ by flipping the spin at
site $x$. Then for $p\in[0,\,1]$, the \emph{$p$-directed dynamics}
$(\sigma_{\beta,\,p}^{\mathrm{dir}}(t))_{t\ge0}$ is defined as a
continuous-time Markov process with jump rate $r_{\beta,\,p}^{\mathrm{dir}}(\cdot,\,\cdot)$
defined by 
\[
r_{\beta,\,p}^{\mathrm{dir}}(\sigma,\,\zeta)=\begin{cases}
\exp\big\{-\beta\Delta_{p}(\sigma,\,\zeta)\big\} & \text{if }\zeta=\sigma^{x}\;,\;x\in\Lambda\;,\\
0 & \text{otherwise}\;,
\end{cases}
\]
where
\begin{align*}
\Delta_{p}(\sigma,\,\sigma^{x}):= & \;p\big[H_{e_{x}}(\sigma^{x})-H_{e_{x}}(\sigma)+H_{n_{x}}(\sigma^{x})-H_{n_{x}}(\sigma)\big]\\
 & \;+(1-p)\big[H_{w_{x}}(\sigma^{x})-H_{w_{x}}(\sigma)+H_{s_{x}}(\sigma^{x})-H_{s_{x}}(\sigma)\big]\;.
\end{align*}
Here, $e_{x}$, $n_{x}$, $w_{x}$ and $s_{x}$ denote four edges
emanating from $x$ towards eastern, northern, western and southern
directions, respectively. 

Note that $\Delta_{1/2}(\sigma,\,\sigma^{x})=\frac{1}{2}[H(\sigma^{x})-H(\sigma)]$
and therefore for $p=\frac{1}{2}$, the dynamics is reduced to the
well-known reversible heat-bath Glauber dynamics. On the other hand,
for $p\neq\frac{1}{2}$, the dynamics puts different masses on the
north/east edges and on the south/west edges, and thereby exhibits
a \textbf{\emph{spatial non-reversibility}}. It is verified in \cite[Section 3.3]{GB}
that the Ising measure $\mu_{\beta}(\cdot)$ is the unique invariant
measure of the $p$-directed dynamics for all $p\in[0,\,1]$. 

We can also analyze metastability of the $p$-directed model, and
thus deduce the following results.
\begin{thm}
For all $p\in[0,\,1]$, the following statements hold for the $p$-directed
model.
\begin{enumerate}
\item The energy barrier is $\Gamma=2K+4\min\{p,\,1-p\}$.
\item (EK law) There exists a constant $\kappa_{0}=\kappa_{0}(K,\,L,\,p)>0$
such that 
\[
\lim_{\beta\to\infty}e^{-\Gamma\beta}\mathbb{E}_{\mathbf{1}}[\tau_{\mathbf{2}}]=\lim_{\beta\to\infty}e^{-\Gamma\beta}\mathbb{E}_{\mathbf{2}}[\tau_{\mathbf{1}}]=\kappa_{0}\;.
\]
The constant $\kappa_{0}$ satisfies $\kappa_{0}(K,\,L,\,p)=\nu_{0}+o_{K}(1)$.
\item (Markov chain model reduction) Theorems \ref{t_neg} and \ref{t_MC}
hold with limiting Markov chain $Y(\cdot)$ defined by the jump rate
$r_{Y}(\mathbf{1},\,\mathbf{2})=r_{Y}(\mathbf{2},\,\mathbf{1})=\kappa_{0}^{-1}$.
\end{enumerate}
\end{thm}

We remark that the analysis of the $p$-directed dynamics is valid
only under periodic boundary conditions.

\section{\label{sec6}Energy Landscape Analysis: MH Dynamics}

In this section, we investigate the energy landscape subjected to
the MH dynamics, and in the next section, we construct the test functions
to complete the proof of Proposition \ref{p_eqpappr} for the MH dynamics.
\begin{notation*}
In Sections \ref{sec6} and \ref{sec7} where we consider the MH dynamics,
we use Greek letters $\eta$ and $\xi$ to denote the spin configurations.
On the other hand, in Sections \ref{sec8} and \ref{sec9} where we
consider the cyclic dynamics, we alternatively use Greek letters $\sigma$
and $\zeta$ to denote the spin configurations.
\end{notation*}
\begin{rem}
The following remarks are important in the upcoming sections.
\begin{enumerate}
\item In Sections \ref{sec6} and \ref{sec7}, we always consider the \emph{MH
dynamics}\textit{. }\textit{\emph{In particular, we have $\Gamma=2K+2$
and we know from }}\cite[Theorem 2.1]{NZ} that this $\Gamma$ is
indeed the energy barrier, i.e., Theorem \ref{t_EB} holds. 
\item In order to focus on the explanation of the saddle structure, which
is quite complicated, we postpone detailed combinatorial proofs of
the technical results to Appendix \ref{secB}.
\end{enumerate}
\end{rem}

\subsection{Neighborhoods}

In the analysis of the energy landscape, the following notions are
importantly used. 
\begin{defn}[Neighborhood of configurations]
\label{d_nhd} For $\eta\in\mathcal{X}$, we define the neighborhoods
of $\sigma$ as 
\[
\mathcal{N}(\eta):=\{\xi\in\mathcal{X}:\Phi(\eta,\,\xi)<\Gamma\}\;\;\;\;\text{and}\;\;\;\;\widehat{\mathcal{N}}(\eta):=\{\xi\in\mathcal{X}:\Phi(\eta,\,\xi)\le\Gamma\}\;.
\]
Then, for $\mathcal{P}\subseteq\mathcal{X}$, we define 
\begin{align*}
\mathcal{N}(\mathcal{P}) & :=\bigcup_{\eta\in\mathcal{P}}\mathcal{N}(\eta)\;\;\;\;\text{and}\;\;\;\;\widehat{\mathcal{N}}(\mathcal{P}):=\bigcup_{\eta\in\mathcal{P}}\widehat{\mathcal{N}}(\eta)\;.
\end{align*}
\end{defn}

\begin{rem}
\label{r_nhd}The following statements hold for neighborhoods. 
\begin{enumerate}
\item By \eqref{e_commheight2}, the set $\mathcal{N}(\eta)$ (resp. $\widehat{\mathcal{N}}(\eta)$)
does not contain $\sigma$ provided that $H(\eta)\ge\Gamma$ (resp.
$H(\eta)>\Gamma$). 
\item By \eqref{e_MHheight} and \eqref{e_commheight}, each configuration
$\xi\in\mathcal{N}(\eta)$ (resp. $\widehat{\mathcal{N}}(\eta)$)
satisfies $H(\xi)<\Gamma$ (resp. $H(\xi)\le\Gamma$). 
\end{enumerate}
\end{rem}

With this notation, the fact that $\Gamma$ is the energy barrier
can be reformulated in the following manner. We say that a path $\omega$
is a $c$\emph{-}\textit{\emph{path}}\textbf{ }if $\Phi_{\omega}\le c$. 
\begin{prop}
\label{p_nhd}For any $\mathbf{a},\,\mathbf{b}\in\mathcal{S}$, we
have $\mathcal{N}(\mathbf{a})\cap\mathcal{N}(\mathbf{b})=\emptyset$
and $\widehat{\mathcal{N}}(\mathbf{a})=\widehat{\mathcal{N}}(\mathbf{b})$. 
\end{prop}

\begin{proof}
By \cite[Theorem 2.1-(i)]{NZ}, any path connecting $\mathbf{a}$
and $\mathbf{b}$ has height at least $\Gamma$, and there indeed
exists a $\Gamma$-path from $\mathbf{a}$ to $\mathbf{b}$. The first
assertion indicates by contradiction that $\mathcal{N}(\mathbf{a})\cap\mathcal{N}(\mathbf{b})=\emptyset$,
while the second one implies that $\widehat{\mathcal{N}}(\mathbf{a})=\widehat{\mathcal{N}}(\mathbf{b})$. 
\end{proof}
We will now investigate what happens on the boundary of the set $\widehat{\mathcal{N}}(\eta)$. 
\begin{lem}
\label{l_Nhatbdry}Let $\eta\in\mathcal{X}$ and let $\xi_{1}\in\widehat{\mathcal{N}}(\eta)$,
$\xi_{2}\notin\widehat{\mathcal{N}}(\eta)$, and $\xi_{1}\sim\xi_{2}$.
Then, $H(\xi_{2})>\Gamma$. 
\end{lem}

\begin{proof}
Since $\xi_{1}\in\mathcal{\widehat{N}}(\eta)$, there exists a path
$\omega:\eta\rightarrow\xi_{1}$ such that $\Phi_{\omega}\le\Gamma$.
Suppose to the contrary that $H(\xi_{2})\le\Gamma$. Then, we can
concatenate the directed edge $(\xi_{1},\,\xi_{2})$ at the end of
the path $\omega$ to form a new path $\widetilde{\omega}:\eta\rightarrow\xi_{2}$
satisfying $\Phi_{\widetilde{\omega}}\le\Gamma$. This contradicts
the assumption that $\xi_{2}\notin\widehat{\mathcal{N}}(\eta)$.
\end{proof}
For $\mathcal{P}\subseteq\mathcal{X}$, a path $\omega=(\omega_{n})_{n=0}^{N}$
is called\textbf{ }a path in\textbf{\emph{ $\mathcal{P}$}} if $\omega_{n}\in\mathcal{P}$
for all $n\in\llbracket0,\,N\rrbracket$. The following notion is
also important in our investigation. 
\begin{defn}[Restricted neighborhood]
\label{d_rnhd} Let $\mathcal{Q}\subseteq\mathcal{X}$. For $\eta\in\mathcal{X}\setminus\mathcal{Q}$,
we define
\[
\widehat{\mathcal{N}}(\eta;\,\mathcal{Q}):=\{\xi\in\mathcal{X}:\exists\omega:\eta\rightarrow\xi\text{ in }\mathcal{Q}^{c}\text{ such that }\Phi_{\omega}\le\Gamma\}\;.
\]
For $\mathcal{P}\subseteq\mathcal{X}$ disjoint with $\mathcal{Q}$,
we define $\widehat{\mathcal{N}}(\mathcal{P};\,\mathcal{Q}):=\bigcup_{\eta\in\mathcal{P}}\widehat{\mathcal{N}}(\eta;\,\mathcal{Q})$.
\end{defn}

By definition, the following lemma is immediate.
\begin{lem}
\label{l_hatNPempty}For all $\mathcal{P}\subseteq\mathcal{X}$, it
holds that
\[
\widehat{\mathcal{N}}(\mathcal{P})=\widehat{\mathcal{N}}(\mathcal{P};\,\emptyset)\;.
\]
\end{lem}

\subsection{\label{sec6.2}Canonical configurations}

First, we introduce the special form of configurations which will
be called \emph{canonical configurations}. These configurations are
inspired by the expansion algorithm given in \cite[Proposition 2.3]{NZ}.
We note that we shall consistently refer to Figure \ref{fig6.1} for
illustrations of the configurations introduced in the following definition.
In the figures of this article, each square corresponds to a site
(or vertex) of the lattice $\Lambda$.

\begin{figure}
\includegraphics[width=13cm]{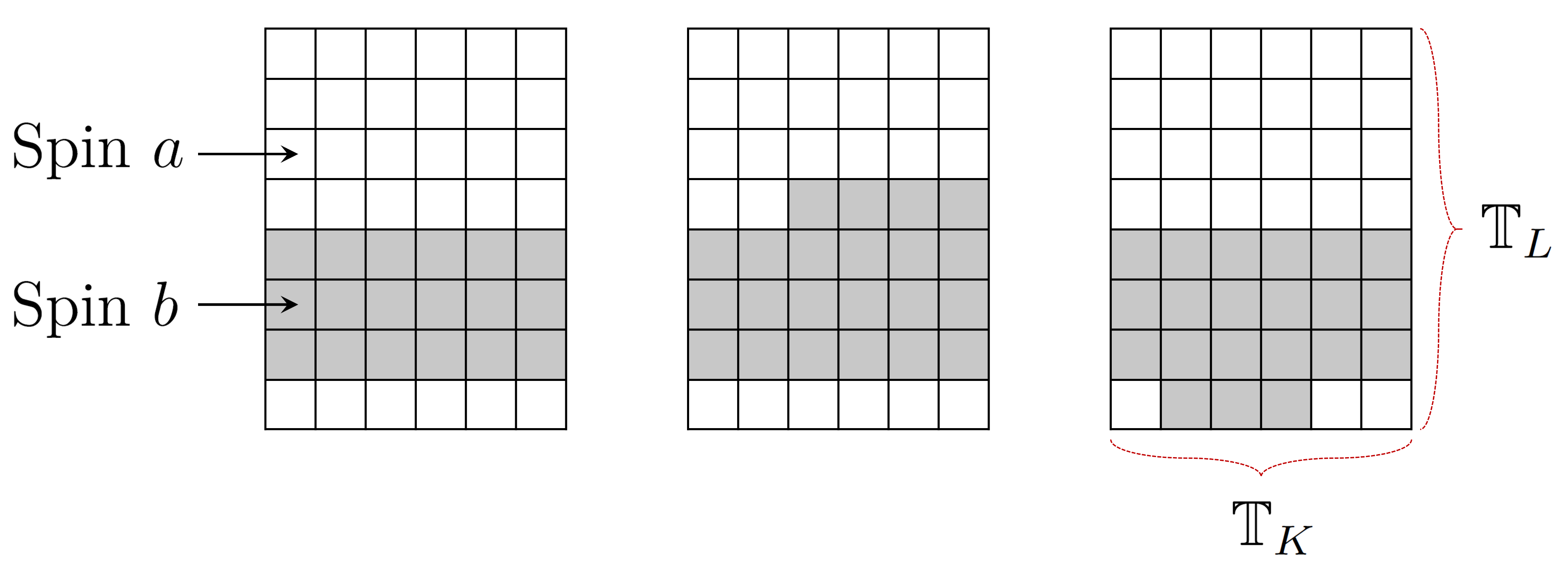}\caption{\label{fig6.1}\textbf{Canonical configurations for the case $(K,\,L)=(6,\,8)$.}
These configurations illustrate $\xi_{2,\,3}^{a,\,b}$ (left), $\xi_{2,\,3;\,3,\,4}^{a,\,b,\,+}$
(middle) and $\xi_{2,\,3;\,2,\,3}^{a,\,b,\,-}$ (right), respectively
(see Definition \ref{d_precanreg}). Notice that our labeling of coordinates
follows the two-dimensional coordinate system, instead of the matrix
system.}
\end{figure}

\begin{defn}[Pre-canonical and pre-regular configurations]
\label{d_precanreg} We fix $a,\,b\in S$ and construct pre-canonical
and pre-regular configurations between two ground states $\mathbf{a}$
and $\mathbf{b}$. We remark that $k$ and $\ell$ are used throughout
to represent elements of $\mathbb{T}_{K}$ and $\mathbb{T}_{L}$,
respectively. In addition, $v$ and $h$ are used to denote vertical
and horizontal lengths, respectively.
\begin{itemize}
\item For $\ell\in\mathbb{T}_{L}$ and $v\in\llbracket0,\,L\rrbracket$,
we denote by $\xi_{\ell,\,v}^{a,\,b}\in\mathcal{X}$ the spin configuration
whose spins are $b$ on the sites in
\[
\mathbb{T}_{K}\times\{\ell+n\in\mathbb{T}_{L}:n\in\llbracket0,\,v-1\rrbracket\subset\mathbb{Z}\}
\]
and $a$ on the remainder. Hence, we have $\xi_{\ell,\,0}^{a,\,b}=\mathbf{a}$
and $\xi_{\ell,\,L}^{a,\,b}=\mathbf{b}$ for all $\ell\in\mathbb{T}_{L}$.
The configurations $\xi_{\ell,\,v}^{a,\,b}$ are called \textit{pre-regular
configuration}s.
\item For $\ell\in\mathbb{T}_{L}$, $v\in\llbracket0,\,L-1\rrbracket$,
$k\in\mathbb{T}_{K}$ and $h\in\llbracket0,\,K\rrbracket$, we denote
by $\xi_{\ell,\,v;\,k,\,h}^{a,\,b,\,+}\in\mathcal{X}$ the configuration
whose spins are $b$ on the sites in
\[
\big\{ x\in\Lambda:\xi_{\ell,\,v}^{a,\,b}(x)=b\big\}\cup\big[\{k+n\in\mathbb{T}_{K}:n\in\llbracket0,\,h-1\rrbracket\subset\mathbb{Z}\}\times\{\ell+v\}\big]
\]
and $a$ on the remainder. Similarly, we denote by $\xi_{\ell,\,v;\,k,\,h}^{a,\,b,\,-}\in\mathcal{X}$
whose spins are $b$ on the sites in
\[
\big\{ x\in\Lambda:\xi_{\ell,\,v}^{a,\,b}(x)=b\big\}\cup\big[\{k+n\in\mathbb{T}_{K}:n\in\llbracket0,\,h-1\rrbracket\subset\mathbb{Z}\}\times\{\ell-1\}\big]
\]
and $a$ on the remainder. Namely, we obtain $\xi_{\ell,\,v;\,k,\,h}^{a,\,b,\,+}$
(resp. $\xi_{\ell,\,v;\,k,\,h}^{a,\,b,\,-}$) from $\xi_{\ell,\,v}^{a,\,b}$
by attaching a protuberance of spins $b$ of size $h$ at its upper
(resp. lower) side of the cluster of spin $b$ starting from the $k$-th
location. It is clear that $\xi_{\ell,\,v;\,k,\,0}^{a,\,b,\,+}=\xi_{\ell,\,v;\,k,\,0}^{a,\,b,\,-}=\xi_{\ell,\,v}^{a,\,b}$,
$\xi_{\ell,\,v;\,k,\,K}^{a,\,b,\,+}=\xi_{\ell,\,v+1}^{a,\,b}$ and
$\xi_{\ell,\,v;\,k,\,K}^{a,\,b,\,-}=\xi_{\ell-1,\,v+1}^{a,\,b}$.
\end{itemize}
\end{defn}

Now, we are ready to define the canonical configurations of $\mathcal{X}$.
\begin{defn}[Canonical and regular configurations]
\label{d_canreg} The definition of canonical configurations differs
between the cases of $K<L$ and $K=L$. Indeed, this is the main reason
why the prefactor of the EK law given in Theorem \ref{t_EK} differs
between these two cases. 
\begin{itemize}
\item \textbf{(Case $K<L$)} For $a,\,b\in S$, we define the collection
$\mathcal{C}^{a,\,b}$ of \textit{canonical configurations} between
$\mathbf{a}$ and $\mathbf{b}$ as
\[
\mathcal{C}^{a,\,b}:=\bigcup_{\ell\in\mathbb{T}_{L}}\bigcup_{v\in\llbracket0,\,L\rrbracket}\{\xi_{\ell,\,v}^{a,\,b}\}\cup\bigcup_{\ell\in\mathbb{T}_{L}}\bigcup_{v\in\llbracket0,\,L-1\rrbracket}\bigcup_{k\in\mathbb{T}_{K}}\bigcup_{h\in\llbracket1,\,K-1\rrbracket}\{\xi_{\ell,\,v;\,k,\,h}^{a,\,b,\,+},\,\xi_{\ell,\,v;\,k,\,h}^{a,\,b,\,-}\}\;.
\]
One can easily see that the right-hand side is a decomposition of
$\mathcal{C}^{a,\,b}$. Then, we define the collection $\mathcal{C}$
of \textit{canonical configurations} as 
\begin{equation}
\mathcal{C}:=\bigcup_{a,\,b\in S}\mathcal{C}^{a,\,b}\;.\label{e_C}
\end{equation}
Similarly, we define
\[
\mathcal{R}_{v}^{a,\,b}:=\bigcup_{\ell\in\mathbb{T}_{L}}\{\xi_{\ell,\,v}^{a,\,b}\}\;\;\;\;\text{and}\;\;\;\;\mathcal{Q}_{v}^{a,\,b}:=\bigcup_{\ell\in\mathbb{T}_{L}}\bigcup_{k\in\mathbb{T}_{K}}\bigcup_{h\in\llbracket1,\,K-1\rrbracket}\{\xi_{\ell,\,v;\,k,\,h}^{a,\,b,\,\pm}\}\;,
\]
and then define $\mathcal{R}_{v}$ and $\mathcal{Q}_{v}$ as in \eqref{e_C}.
A configuration belonging to $\mathcal{R}_{v}$ for some $v\in\llbracket0,\,L\rrbracket$
is called a \textit{regular configuration}. 
\item \textbf{(Case $K=L$)} For $a,\,b\in S$, we temporarily denote by
$\widetilde{\mathcal{C}}^{a,\,b}$, $\widetilde{\mathcal{R}}_{v}^{a,\,b}$
and $\widetilde{\mathcal{Q}}_{v}^{a,\,b}$ the collections $\mathcal{C}^{a,\,b}$,
$\mathcal{R}_{v}^{a,\,b}$ and $\mathcal{Q}_{v}^{a,\,b}$ defined
in the previous case of $K<L$. Define a transpose operator $\Theta:\mathcal{X}\rightarrow\mathcal{X}$
by
\begin{equation}
(\Theta(\sigma))(k,\,\ell)=\sigma(\ell,\,k)\;\;\;\;;\;k\in\mathbb{T}_{K}\text{ and }\ell\in\mathbb{T}_{L}\;.\label{e_transpose}
\end{equation}
Then, we define $\mathcal{A}=\widetilde{\mathcal{A}}\cup\Theta(\mathcal{A})$
where $\mathcal{A}=\mathcal{C}^{a,\,b}$, $\mathcal{R}_{v}^{a,\,b}$
or $\mathcal{Q}_{v}^{a,\,b}$. The sets $\mathcal{C}$, $\mathcal{R}_{v}$
and $\mathcal{Q}_{v}$ are defined as in \eqref{e_C}. These modified
definitions for the case of $K=L$ reflect the fact that the transitions
can occur in both horizontal and vertical directions of the lattice.
\end{itemize}
\end{defn}

We can readily verify from the definition of the Hamiltonian $H$
that $H(\eta)\le\Gamma$ for all $\eta\in\mathcal{C}^{a,\,b}$, and
moreover for $a,\,b\in S$,
\[
H(\eta)=\begin{cases}
\Gamma-2 & \text{if }\eta\in\mathcal{R}_{v}^{a,\,b}\text{ for }v\in\llbracket1,\,L-1\rrbracket\;,\\
\Gamma & \text{if }\eta\in\mathcal{Q}_{v}^{a,\,b}\text{ for }v\in\llbracket1,\,L-2\rrbracket\;.
\end{cases}
\]
We stress that the definition of canonical configurations is not symmetric
between rows and columns if $K<L$, as they play different roles.
Indeed, the roles are fundamentally different in the metastable transitions,
and that leads to the fact that the energy barrier $\Gamma$ depends
only on $K$, not on $L$.

The following notation will be used in the construction of canonical
paths. 
\begin{notation}
\label{n_frakS}Suppose that $N\ge2$ is a positive integer.
\begin{itemize}
\item Define $\mathfrak{S}_{N}$ as the collection of connected subsets
of $\mathbb{T}_{N}$, including the empty set.
\item For $P,\,P'\in\mathfrak{S}_{N}$, we write $P\prec P'$ if $P\subset P'$
and $|P'|=|P|+1$.
\item A sequence $(P_{m})_{m=0}^{N}$ of sets in $\mathfrak{S}_{N}$ is
called an \textit{increasing sequence} if it satisfies 
\[
\emptyset=P_{0}\prec P_{1}\prec\cdots\prec P_{N}=\mathbb{T}_{N}\;,
\]
so that $|P_{m}|=m$ for all $m\in\llbracket0,\,N\rrbracket$.
\end{itemize}
\end{notation}

\begin{figure}
\includegraphics[width=14cm]{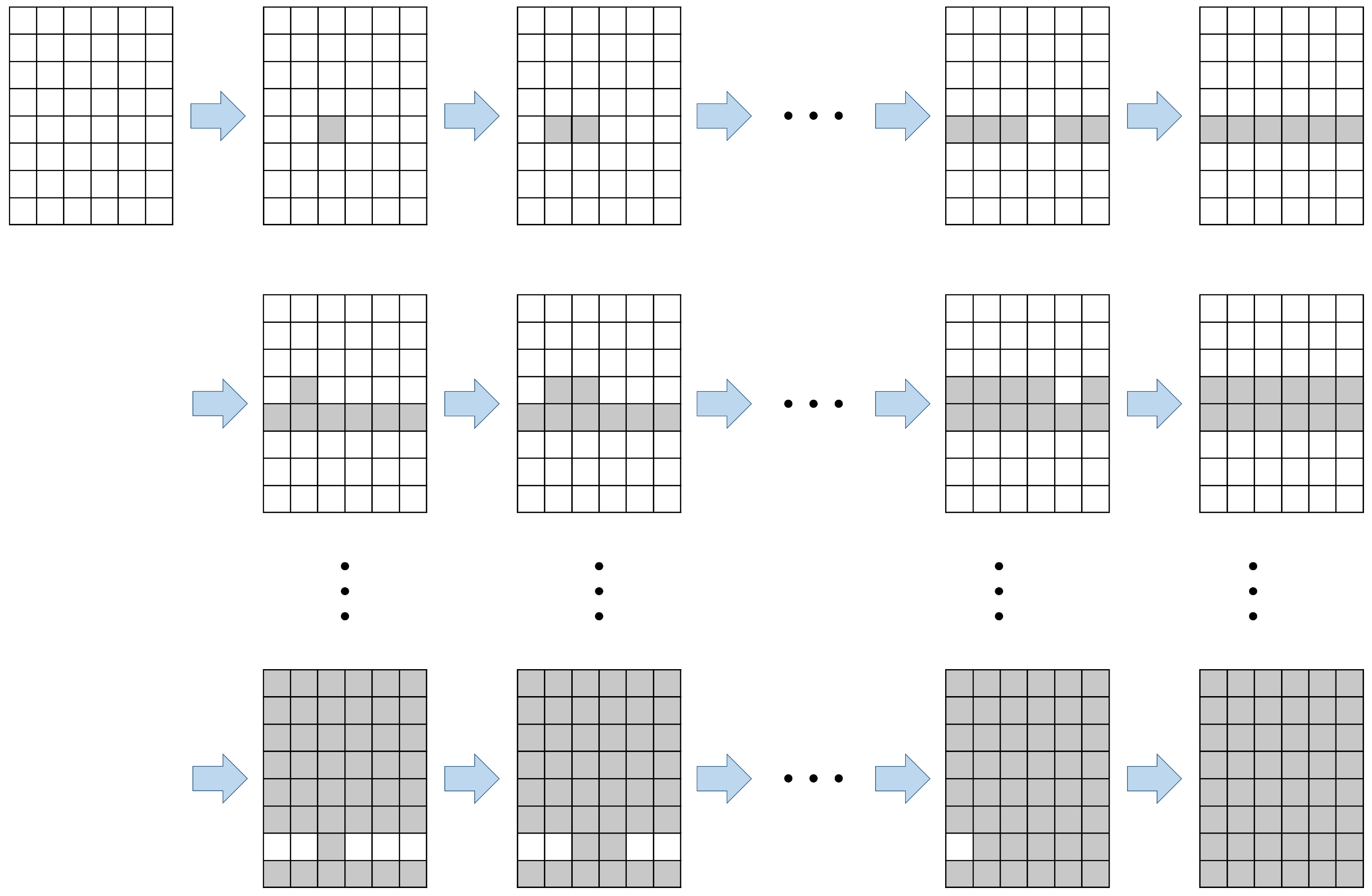}\caption{\label{fig6.2}\textbf{Example of a canonical path for the case of
$(K,\,L)=(6,\,8)$.}}
\end{figure}

Now, we construct the so-called canonical paths between two ground
states.
\begin{defn}[Canonical paths]
\label{d_canpath} $ $
\begin{enumerate}
\item We first introduce a \textit{standard sequence} of subsets of $\Lambda$
connecting the empty set $\emptyset$ and the full set $\Lambda$.
\begin{enumerate}
\item First, for $P,\,P'\in\mathfrak{S}_{L}$ with $P\prec P'$, we say
that a sequence $(A_{h})_{h=0}^{K}$ of subsets of $\Lambda$ is a
standard sequence connecting $\mathbb{T}_{K}\times P$ and $\mathbb{T}_{K}\times P'$
if there exists an increasing sequence $(Q_{h})_{h=0}^{K}$ in $\mathfrak{S}_{K}$
such that 
\[
A_{h}=(\mathbb{T}_{K}\times P)\cup\big[Q_{h}\times(P'\setminus P)\big]\;\;\;\;;\;h\in\llbracket0,\,K\rrbracket\;.
\]
\item Now, a sequence $(A_{m})_{m=0}^{KL}$ of subsets of $\Lambda$ is
a standard sequence connecting $\emptyset$ and $\Lambda$ if there
exists an increasing sequence $(P_{v})_{v=0}^{L}$ in $\mathfrak{S}_{L}$
such that $A_{Kv}=\mathbb{T}_{K}\times P_{v}$ for all $v\in\llbracket0,\,L\rrbracket$,
and furthermore for each $v\in\llbracket0,\,L-1\rrbracket$ the subsequence
$(A_{h})_{h=Kv}^{K(v+1)}$ is a standard sequence connecting $\mathbb{T}_{K}\times P_{v}$
and $\mathbb{T}_{K}\times P_{v+1}$.
\end{enumerate}
\item For $a,\,b\in S$, a sequence of configurations $\omega=(\omega_{m})_{m=0}^{KL}$
in $\mathcal{X}$ is called a \textit{pre-canonical path} (between
$\mathbf{a}$ and $\mathbf{b}$) if there exists a standard sequence
$(A_{m})_{m=0}^{KL}$ connecting $\emptyset$ and $\Lambda$ such
that 
\[
\omega_{m}(x)=\begin{cases}
a & \text{if }x\notin A_{m}\;,\\
b & \text{if }x\in A_{m}\;.
\end{cases}
\]
It is easy to verify that indeed, $\omega$ is a path connecting $\omega_{0}=\mathbf{a}$
and $\omega_{KL}=\mathbf{b}$.
\item Moreover, $(\omega_{m})_{m=0}^{KL}$ in $\mathcal{X}$ is called a
\textit{canonical path }connecting $\mathbf{a}$ and $\mathbf{b}$
(i.e., $\omega_{0}=\mathbf{a}$ and $\omega_{KL}=\mathbf{b}$) if
there exists a pre-canonical path $(\widetilde{\omega}_{m})_{m=0}^{KL}$
such that 
\begin{enumerate}
\item \textbf{(Case $K<L$)} $\omega_{m}=\widetilde{\omega}_{m}$ for all
$m\in\llbracket0,\,KL\rrbracket$,
\item \textbf{(Case $K=L$)} $\omega_{m}=\widetilde{\omega}_{m}$ for all
$m\in\llbracket0,\,KL\rrbracket$ or $\omega_{m}=\Theta(\widetilde{\omega}_{m})$
for all $m\in\llbracket0,\,KL\rrbracket$.
\end{enumerate}
\end{enumerate}
\end{defn}

An example of a canonical path is given in Figure \ref{fig6.2}. It
is direct from the construction that a canonical path consists of
canonical configurations only. In particular, the following properties
are immediate from the construction.
\begin{lem}
\label{l_canpath}For any canonical path $(\omega_{m})_{m=0}^{KL}$
connecting $\mathbf{a}$ and $\mathbf{b}$, the following statements
hold.
\begin{enumerate}
\item For all $v\in\llbracket0,\,L\rrbracket$, we have $\omega_{Kv}\in\mathcal{R}_{v}^{a,\,b}$.
\item For all $v\in\llbracket0,\,L-1\rrbracket$ and $h\in\llbracket1,\,K-1\rrbracket$,
we have $\omega_{Kv+h}\in\mathcal{Q}_{v}^{a,\,b}$.
\item It holds that $\Phi_{\omega}=\Gamma$.
\end{enumerate}
\end{lem}

In view of part (3) of the previous lemma, a canonical path between
$\mathbf{a}$ and $\mathbf{b}$ is an optimal path that achieves the
communication height $\Gamma$.

\subsection{\label{sec6.3}Typical configurations}

Let us now define typical configurations which are extensions of canonical
configurations and indeed the main ingredient in the explanation of
the saddle structure. The definition is complicated, but its meaning
will become clear later. We refer to Figure \ref{fig6.3} for an illustration
of typical configurations.

\begin{figure}
\includegraphics[width=13cm]{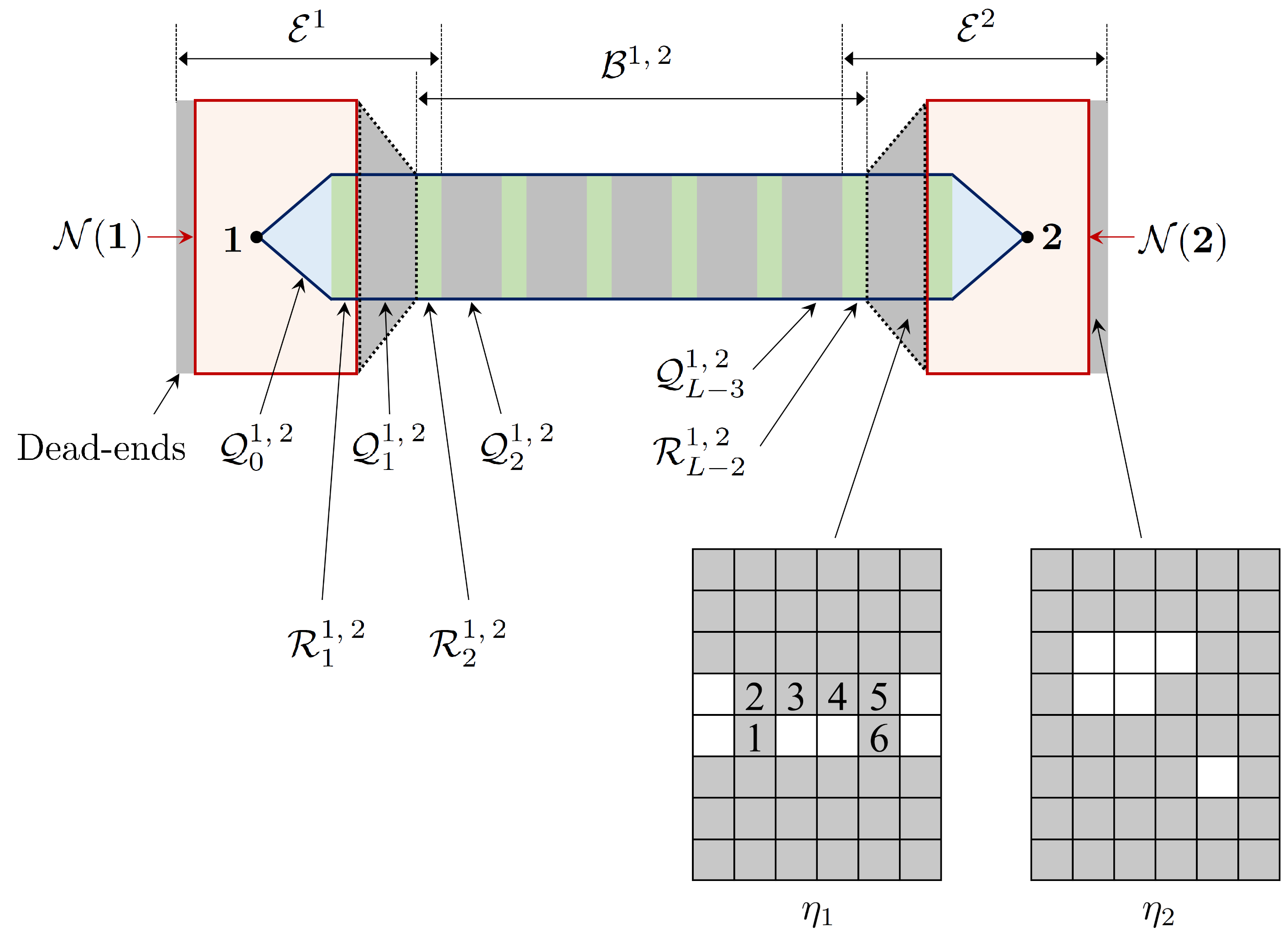}\caption{\label{fig6.3}\textbf{Typical configurations for the Ising model.
}This figure illustrates the structure of $\widehat{\mathcal{N}}(\mathcal{S})$
in the case of $q=2$. We can notice from the figure that as verified
in Proposition \ref{p_typprop}, $\widehat{\mathcal{N}}(\mathcal{S})=\mathcal{E}^{1}\cup\mathcal{B}^{1,\,2}\cup\mathcal{E}^{2}$,
$\mathcal{R}_{2}^{1,\,2}=\mathcal{E}^{1}\cap\mathcal{B}^{1,\,2}$
and $\mathcal{R}_{L-2}^{1,\,2}=\mathcal{B}^{1,\,2}\cap\mathcal{E}^{2}$.
Regions consisting of configurations with energy $\Gamma=2K+2$ are
colored gray. The hexagonal region at the center of the figure enclosed
by the blue line denotes the set $\mathcal{C}=\mathcal{C}^{1,\,2}$
of canonical configurations between $\mathbf{1}$ and $\mathbf{2}$.
Since the MH dynamics cannot escape from $\widehat{\mathcal{N}}(\mathcal{S})$
during the transition from $\mathbf{1}$ to $\mathbf{2}$ (with dominating
probability), the dynamics must path through the $\mathcal{B}^{1,\,2}$-part
of the canonical configurations. The set $\mathcal{E}^{2}$ of edge
typical configurations near $\mathbf{2}$ consists of four regions.
The first one is the neighborhood $\mathcal{N}(\mathbf{2})$ denoted
by the red-enclosed box. The second one is the region consisting of
configurations with energy $\Gamma$ which is connected to $\mathcal{R}_{L-2}^{1,\,2}$
via a $\Gamma$-path in $\mathcal{X}\setminus\mathcal{B}_{\Gamma}$.
This region will be denoted by $\mathcal{Z}^{2,\,1}$ in Section \ref{sec6.4}
(cf. Figure \ref{fig6.4}). An example of a configuration belonging
to this region is $\eta_{1}$. In particular, we can connect $\eta_{1}$
with a configuration in $\mathcal{R}_{L-2}^{1,\,2}$ via a $\Gamma$-path
by updating six gray boxes in the order indicated in the figure (cf.
Proposition \ref{p_Zabtree}). The third region consists of dead-ends
(cf. \eqref{e_Dadef}). An example of a dead-end is $\eta_{2}$, which
has energy $14=2\times6+2=2K+2$. The fourth region is $\mathcal{R}_{L-2}^{1,\,2}$.
A similar decomposition holds for $\mathcal{E}^{1}$ as well.}
\end{figure}

\begin{defn}[Typical configurations]
\label{d_typ}
\begin{itemize}
\item For $a,\,b\in S$, define the collection of \textit{bulk typical configurations}
(between $\mathbf{a}$ and $\mathbf{b}$) by
\begin{equation}
\mathcal{B}^{a,\,b}=\bigcup_{v\in\llbracket2,\,L-2\rrbracket}\mathcal{R}_{v}^{a,\,b}\cup\bigcup_{v\in\llbracket2,\,L-3\rrbracket}\mathcal{Q}_{v}^{a,\,b}\;.\label{e_Bab}
\end{equation}
One can easily notice that $\mathcal{B}^{a,\,b}=\mathcal{B}^{b,\,a}$.
Then, we define $\mathcal{B}=\bigcup_{a,\,b\in S}\mathcal{B}^{a,\,b}$. 
\item Define 
\[
\mathcal{B}_{\Gamma}^{a,\,b}=\bigcup_{v\in\llbracket2,\,L-3\rrbracket}\mathcal{Q}_{v}^{a,\,b}=\{\eta\in\mathcal{B}^{a,\,b}:H(\eta)=\Gamma\}\;.
\]
Then, we define
\[
\mathcal{B}_{\Gamma}:=\bigcup_{a,\,b\in S}\mathcal{B}_{\Gamma}^{a,\,b}=\bigcup_{a,\,b\in S}\bigcup_{v\in\llbracket2,\,L-3\rrbracket}\mathcal{Q}_{v}^{a,\,b}\;.
\]
\item For $a\in S$, the collection of \textit{edge typical configurations}
with respect to $\mathbf{a}$ is defined as
\begin{equation}
\mathcal{E}^{a}=\widehat{\mathcal{N}}(\mathbf{a};\,\mathcal{B}_{\Gamma})\;.\label{e_Ea}
\end{equation}
Finally, we set $\mathcal{E}=\bigcup_{a\in S}\mathcal{E}^{a}$. 
\end{itemize}
A configuration belonging to $\mathcal{B}\cup\mathcal{E}$ is called
a \textit{typical configuration}.
\end{defn}

At first glance, the definition of typical configurations may look
weird. However, the meaning of this will become clear after investigating
their properties. In particular, in Proposition \ref{p_typprop},
we shall demonstrate that $\mathcal{E}^{a}$, $a\in S$ are mutually
disjoint, and that typical configurations are exactly the configurations
which can be reached from ground states via $\Gamma$-paths. Hence,
they form the collection of configurations relevant to the transitions
between the ground states in $\mathcal{S}$.
\begin{rem}[Tube of typical trajectories]
 Later in Proposition \ref{p_typprop}, we will show that
\[
\mathcal{E}\cup\mathcal{B}=\widehat{\mathcal{N}}(\mathcal{S})\;.
\]
From the viewpoint of the pathwise approach to metastability, $\widehat{\mathcal{N}}(\mathcal{S})$
is exactly the \textit{tube of typical trajectories} \cite{OS rev,OS non-rev}
of the metastable transitions between the ground states in $\mathcal{S}$.
In this sense, we demonstrate in Proposition \ref{p_typprop} that
$\mathcal{B}\cup\mathcal{E}$ is exactly the tube of typical trajectories
of the metastable transitions. We also remark here that this is also
the case in the cyclic dynamics as well; see Proposition \ref{p_typpropcyc}.
\end{rem}

\begin{rem}[Classification of typical configurations]
 We decompose typical configurations into the bulk and edge ones,
since the patterns of transitions therein are qualitatively different.
We refer to Figure \ref{fig6.3} for a detailed explanation. Roughly
speaking, to make a transition from $\mathbf{a}$ to $\mathbf{b}$
(without touching the energy level higher than $\Gamma$, i.e., without
escaping from $\widehat{\mathcal{N}}(\mathcal{S})$), the dynamics
has to pass through $\mathcal{E}^{a}$ first to arrive at $\mathcal{B}^{a,\,b}$.
Then, it goes through $\mathcal{B}^{a,\,b}$ along the canonical path
to arrive at $\mathcal{E}^{b}$, and then finally it reaches $\mathbf{b}$.
The behavior of the dynamics on $\mathcal{E}^{a}$ and $\mathcal{E}^{b}$
is somewhat complicated and is explained by an auxiliary Markov chain
defined in Definition \ref{d_EAMc}, which can be regarded as a random
walk on the space of sub-trees of a ladder graph (cf. Proposition
\ref{p_Zabtree}). 
\end{rem}

The next proposition demonstrates the relations between bulk and edge
typical configurations rigorously.
\begin{prop}
\label{p_typprop}The following properties hold for the typical configurations.
\begin{enumerate}
\item For spins $a,\,b\in S$, we have $\mathcal{E}^{a}\cap\mathcal{E}^{b}=\emptyset$.
\item For spins $a,\,b\in S$, we have $\mathcal{E}^{a}\cap\mathcal{B}^{a,\,b}=\mathcal{R}_{2}^{a,\,b}$.
\item For three spins $a,\,b,\,c\in S$, we have $\mathcal{E}^{a}\cap\mathcal{B}^{b,\,c}=\emptyset$.
\item We have $\mathcal{E}\cup\mathcal{B}=\widehat{\mathcal{N}}(\mathcal{S})$.
\end{enumerate}
\end{prop}

We prove this proposition in Section \ref{secB.2}. The bulk typical
configurations have a clear structure and we do not need a further
investigation on this set. On the other hand, the structure of edge
typical configurations is quite complicated and possesses one of the
main difficulties of the current problem. This structure will be explained
in the remainder of the current section. 

\subsection{\label{sec6.4}Classification of edge typical configurations}

We fix a spin $a\in S$ and focus on the classification of the configurations
in the set $\mathcal{E}^{a}$. We refer to Figure \ref{fig6.4} for
an illustration of the definition given here.

\begin{figure}
\includegraphics[width=14.5cm]{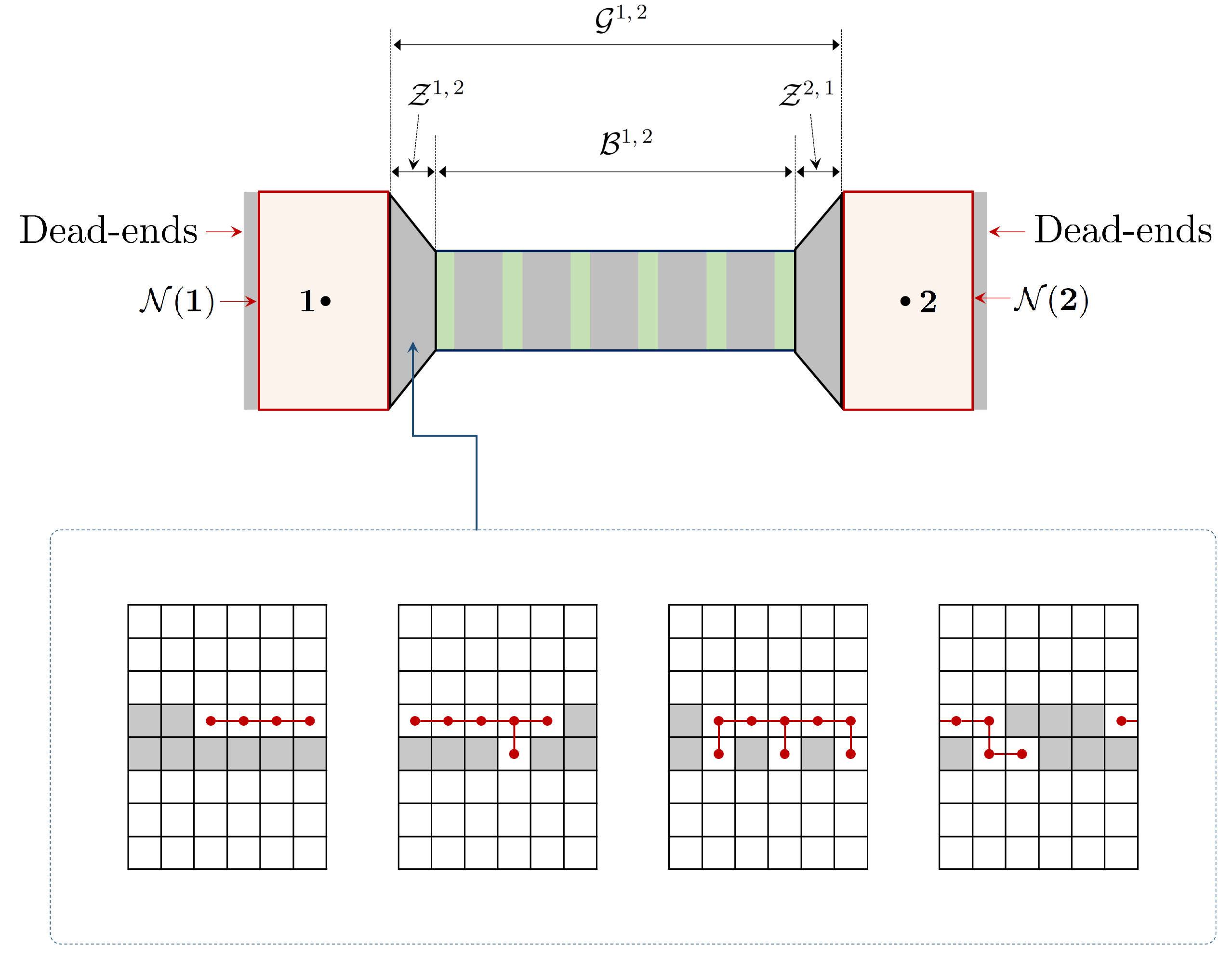}\caption{\label{fig6.4}\textbf{Edge typical configurations.} The configurations
belonging to gray regions have energy $\Gamma$. Green regions denote
the sets of the form $\mathcal{R}_{v}^{1,\,2}$, $v\in\llbracket2,\,L-2\rrbracket$.
We can notice from this figure that the set $\mathcal{N}(\mathbf{1})$
and the bulk typical configurations in $\mathcal{B}^{1,\,2}$ are
connected through the configurations in $\mathcal{Z}^{1,\,2}$. Below
the energy landscape, we give four examples of highly non-trivial
configurations belonging to $\mathcal{Z}^{1,\,2}$. Indeed, as indicated
in the figures, the sites with spin $1$ within the strip form sub-trees.
These configurations have energy $\Gamma$ and do not belong to the
dead-ends since each one of them can be connected to $\mathcal{R}_{2}^{1,\,2}$
by a $\Gamma$-path in $\mathcal{X}\setminus\mathcal{N}(\mathbf{1})$.}
\end{figure}

For another spin $b\ne a$, we define a set $\mathcal{Z}^{a,\,b}$
as
\begin{align}
\mathcal{Z}^{a,\,b}:=\{\eta\in\mathcal{X}: & \;\exists\text{a path }(\omega_{n})_{n=0}^{N}\text{ in }\mathcal{X}\setminus\mathcal{B}_{\Gamma}\text{ such that }\nonumber \\
 & \;\omega_{0}\in\mathcal{R}_{2}^{a,\,b}\;,\;\omega_{N}=\eta\text{ and }H(\omega_{n})=\Gamma\text{ for all }n\in\llbracket1,\,N\rrbracket\}\;.\label{e_Zabdef}
\end{align}
Note that $\mathcal{Z}^{a,\,b}\neq\mathcal{Z}^{b,\,a}$. Write
\[
\mathcal{Z}^{a}:=\bigcup_{b\in S\setminus\{a\}}\mathcal{Z}^{a,\,b}\;\;\;\;\text{and}\;\;\;\;\mathcal{R}^{a}:=\bigcup_{b\in S\setminus\{a\}}\mathcal{R}_{2}^{a,\,b}\;.
\]
Then, by \eqref{e_Zabdef} and Proposition \ref{p_typprop} we have
that
\begin{equation}
\mathcal{Z}^{a}\subset\mathcal{E}^{a}\;\;\;\;\text{and}\;\;\;\;\mathcal{R}^{a}\subset\mathcal{E}^{a}\;.\label{e_ZaRaEa}
\end{equation}

\subsubsection*{Properties of $\mathcal{Z}^{a,\,b}$}

For a configuration $\eta\in\mathcal{X}$ and a spin $a\in S$, we
denote by $P_{a}(\eta)$ the set of sites of spin $a$ with respect
to $\sigma$, i.e.,
\begin{equation}
P_{a}(\eta):=\{x\in\Lambda:\eta(x)=a\}\;.\label{e_Padef}
\end{equation}
The next proposition verifies that each configuration in $\mathcal{Z}^{a,\,b}$
can be matched with \emph{a sub-tree of a ladder graph} (cf. Figure
\ref{fig6.4}).
\begin{prop}
\label{p_Zabtree}For $a,\,b\in S$ and $\eta\in\mathcal{X}$, it
holds that $\eta\in\mathcal{Z}^{a,\,b}$ if and only if the following
three conditions hold simultaneously.
\begin{itemize}
\item \textbf{\textup{{[}Z1{]}}} For all $x\in\Lambda$, we have $\eta(x)\in\{a,\,b\}$.
\item \textbf{\textup{{[}Z2{]}}} There exists $\ell\in\mathbb{T}_{L}$ such
that $P_{b}(\eta)\subseteq\mathbb{T}_{K}\times\{\ell,\,\ell+1\}$
(or $\{\ell,\,\ell+1\}\times\mathbb{T}_{K}$ if $K=L$).
\item \textbf{\textup{{[}Z3{]}}} Define $G_{a}(\eta):=P_{a}(\eta)\cap[\mathbb{T}_{K}\times\{\ell,\,\ell+1\}]$
(or $P_{a}(\eta)\cap[\{\ell,\,\ell+1\}\times\mathbb{T}_{K}]$ if $K=L$).
Then, the set $G_{a}(\eta)$ consists of vertices of a sub-tree\footnote{We regard that the empty set is not a tree.}
of the ladder graph $\mathbb{T}_{K}\times\{\ell,\,\ell+1\}$ (or $\{\ell,\,\ell+1\}\times\mathbb{T}_{K}$
if $K=L$).
\end{itemize}
\end{prop}

We prove this statement in Section \ref{secB.3}.
\begin{defn}
\label{d_Zlabdef}Fix $a,\,b\in S$. For each $\xi_{\ell,\,2}^{a,\,b}\in\mathcal{R}_{2}^{a,\,b}$,
$\ell\in\mathbb{T}_{L}$ (and also $\Theta(\xi_{\ell,\,2}^{a,\,b})\in\mathcal{R}_{2}^{a,\,b}$
if $K=L$), Proposition \ref{p_Zabtree} implies that there exists
a subset of configurations in $\mathcal{Z}^{a,\,b}$ connected to
$\xi_{\ell,\,2}^{a,\,b}$ with tree structures on $\mathbb{T}_{K}\times\{\ell,\,\ell+1\}$
(or $\{\ell,\,\ell+1\}\times\mathbb{T}_{K}$ for the case of $\Theta(\xi_{\ell,\,2}^{a,\,b})\in\mathcal{R}_{2}^{a,\,b}$).
We will denote this subset by $\mathcal{Z}_{\ell}^{a,\,b}\subset\mathcal{Z}^{a,\,b}$
(or $\Theta(\mathcal{Z}_{\ell}^{a,\,b})\subset\mathcal{Z}^{a,\,b}$
for the case of $\Theta(\xi_{\ell,\,2}^{a,\,b})\in\mathcal{R}_{2}^{a,\,b}$).
Thus, we have the following decomposition of $\mathcal{Z}^{a,\,b}$:
\begin{equation}
\mathcal{Z}^{a,\,b}=\begin{cases}
\bigcup_{\ell\in\mathbb{T}_{L}}\mathcal{Z}_{\ell}^{a,\,b} & \text{if }K<L\;,\\
\bigcup_{\ell\in\mathbb{T}_{L}}\mathcal{Z}_{\ell}^{a,\,b}\cup\bigcup_{\ell\in\mathbb{T}_{L}}\Theta(\mathcal{Z}_{\ell}^{a,\,b}) & \text{if }K=L\;.
\end{cases}\label{e_Zabdecomp}
\end{equation}
Moreover, it is clear that each copy of $\mathcal{Z}_{\ell}^{a,\,b}$
(or $\Theta(\mathcal{Z}_{\ell}^{a,\,b})$ if $K=L$) has the same
structure. This structure will be further investigated in the next
subsection.
\end{defn}

\subsubsection*{Decomposition of $\mathcal{E}^{a}$}

We write 
\begin{equation}
\mathcal{D}^{a}:=\widehat{\mathcal{N}}(\mathbf{a};\,\mathcal{Z}^{a})\;,\label{e_Dadef}
\end{equation}
where $\mathcal{D}$ stands for dead-ends. Indeed, the set $\mathcal{D}^{a}\setminus\mathcal{N}(\mathbf{a})$
consists of the dead-end configurations with energy $\Gamma$ (cf.
caption of Figure \ref{fig6.3}). We now classify the edge typical
configurations.
\begin{prop}
\label{p_Eadecomp}For $a\in S$, we have the following decomposition
of $\mathcal{E}^{a}$:
\begin{equation}
\mathcal{E}^{a}=\mathcal{D}^{a}\cup\mathcal{Z}^{a}\cup\mathcal{R}^{a}\;.\label{e_Eadecomp}
\end{equation}
\end{prop}

This proposition is also proved in Section \ref{secB.3}.

\subsection{\label{sec6.5}Graph structure of edge typical configurations}

We consistently refer to Figure \ref{fig6.5} for an illustration
of the notions constructed in this subsection.

\begin{figure}
\includegraphics[width=14.5cm]{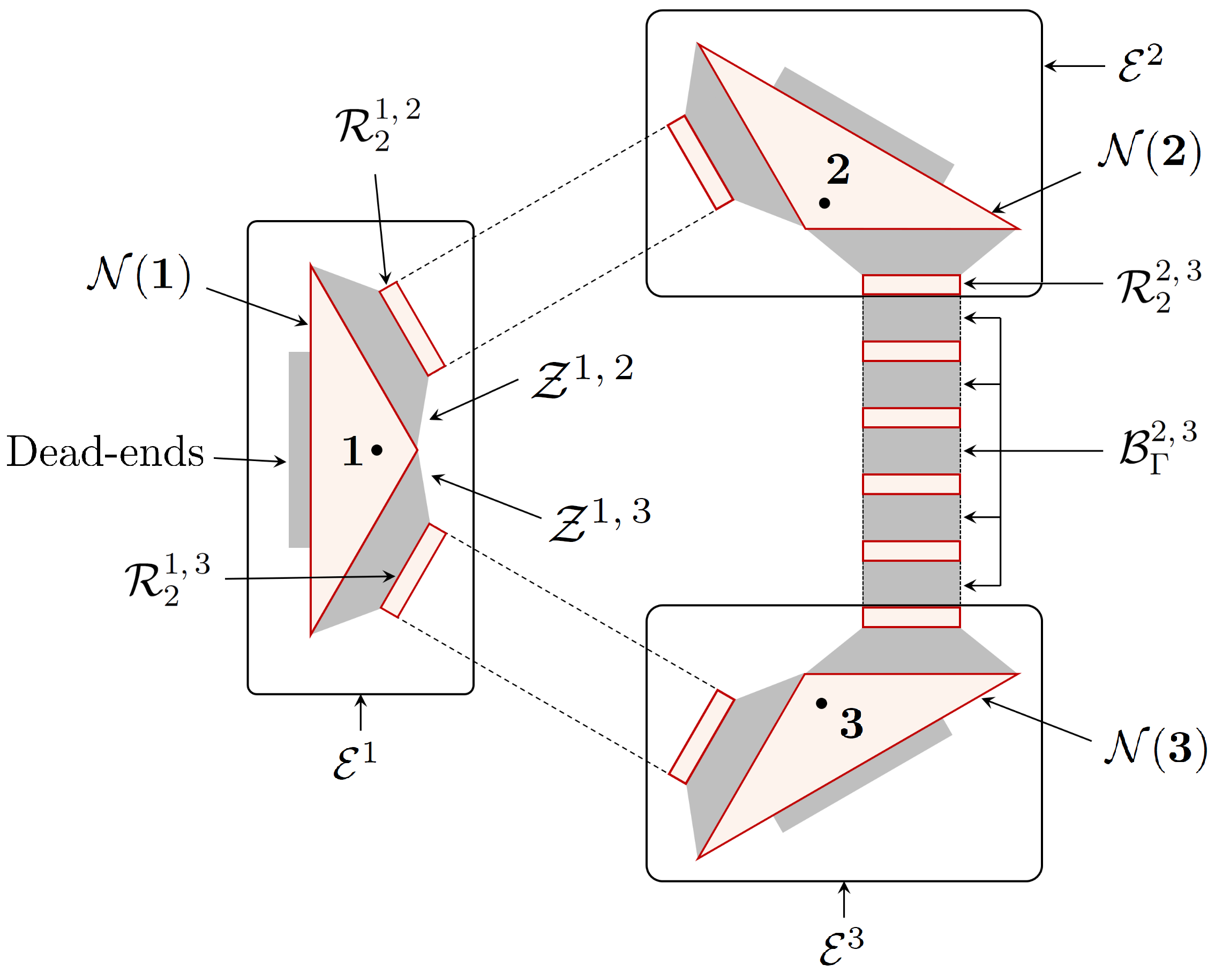}\caption{\label{fig6.5}\textbf{Structure of edge typical configurations.}
Let $S=\{1,\,2,\,3\}$ and $K<L$. Note that the set $\mathcal{E}^{1}$
consists of six parts: $\mathcal{N}(\mathbf{1})$, the dead-ends attached
to $\mathcal{N}(\mathbf{1})$ (cf. caption of Figure \ref{fig6.3}),
$\mathcal{R}_{2}^{1,\,2}$, $\mathcal{Z}^{1,\,2}$ (gray region of
equilateral-trapezoid shape), and finally two more similar collections
$\mathcal{R}_{2}^{1,\,3}$ and $\mathcal{Z}^{1,\,3}$.}
\end{figure}

Fix a spin $a\in S$. The decomposition given in Proposition \ref{p_Eadecomp}
can be alternatively expressed as
\begin{equation}
\mathcal{E}^{a}=\mathcal{D}^{a}\cup\bigcup_{b\in S\setminus\{a\}}[\mathcal{Z}^{a,\,b}\cup\mathcal{R}_{2}^{a,\,b}]\;.\label{e_Eadecomp2}
\end{equation}
Moreover, according to \eqref{e_Zabdecomp}, $\mathcal{E}^{a}$ can
be further expressed as
\begin{equation}
\begin{cases}
\mathcal{D}^{a}\cup\bigcup_{b\in S\setminus\{a\}}\bigcup_{\ell\in\mathbb{T}_{L}}[\mathcal{Z}_{\ell}^{a,\,b}\cup\{\xi_{\ell,\,2}^{a,\,b}\}] & \text{if }K<L\;,\\
\mathcal{D}^{a}\cup\bigcup_{b\in S\setminus\{a\}}\bigcup_{\ell\in\mathbb{T}_{L}}[\mathcal{Z}_{\ell}^{a,\,b}\cup\{\xi_{\ell,\,2}^{a,\,b}\}]\cup\bigcup_{b\in S\setminus\{a\}}\bigcup_{\ell\in\mathbb{T}_{L}}[\Theta(\mathcal{Z}_{\ell}^{a,\,b})\cup\{\Theta(\xi_{\ell,\,2}^{a,\,b})\}] & \text{if }K=L\;.
\end{cases}\label{e_Eadecomp3}
\end{equation}
We now construct a graph and a Markov chain describing the behavior
of the MH dynamics on $\mathcal{E}^{a}$.
\begin{defn}
\label{d_EAMc}We introduce a graph structure and a Markov chain thereon.
Fix $b\ne a$, $\ell\in\mathbb{T}_{L}$ and consider the set $\mathcal{N}(\mathbf{a})\cup\mathcal{Z}_{\ell}^{a,\,b}\cup\{\xi_{\ell,\,2}^{a,\,b}\}$.
\begin{itemize}
\item \textbf{(Graph) }The vertex set $\mathscr{V}_{\ell}^{a,\,b}$ is given
by 
\begin{equation}
\mathscr{V}_{\ell}^{a,\,b}:=\{\mathbf{a}\}\cup\mathcal{Z}_{\ell}^{a,\,b}\cup\{\xi_{\ell,\,2}^{a,\,b}\}\;.\label{e_Vab}
\end{equation}
Then, the edge set $\mathscr{E}_{\ell}^{a,\,b}$ is defined as follows:
$\{\eta,\,\eta'\}\in\mathscr{E}_{\ell}^{a,\,b}$ for $\eta,\,\eta'\in\mathscr{V}_{\ell}^{a,\,b}$
if and only if either $\eta,\,\eta'\ne\mathbf{a}$ and $\eta\sim\eta'$,
or $\eta\in\mathcal{Z}_{\ell}^{a,\,b}$, $\eta'=\mathbf{a}$ and $\eta\sim\xi$
for some $\xi\in\mathcal{N}(\mathbf{a})$. By definition, we can understand
$\mathbf{a}\in\mathscr{V}_{\ell}^{a,\,b}$ as a collapsed state representing
the neighborhood $\mathcal{N}(\mathbf{a})$. By symmetry, the graph
structure $\mathscr{G}_{\ell}^{a,\,b}:=(\mathscr{V}_{\ell}^{a,\,b},\,\mathscr{E}_{\ell}^{a,\,b})$
does not depend on $a\in S$, $b\ne a$ and $\ell\in\mathbb{T}_{L}$
(or even for the collection $\{\mathbf{a}\}\cup\Theta(\mathcal{Z}_{\ell}^{a,\,b})\cup\{\Theta(\xi_{\ell,\,2}^{a,\,b})\}$
if $K=L$). Thus, if there is no risk of confusion, we may omit the
subscript and superscript so that we simply write $\mathscr{G}=(\mathscr{V},\,\mathscr{E})$.
\item \textbf{(Markov chain)} Define $\mathfrak{r}:\mathscr{V}\times\mathscr{V}\rightarrow[0,\,\infty)$
by setting 
\begin{equation}
\mathfrak{r}(\eta,\,\eta'):=\begin{cases}
1 & \text{if }\eta,\,\eta'\ne\mathbf{a}\;,\\
|\{\xi\in\mathcal{N}(\mathbf{a}):\xi\sim\eta'\}| & \text{if }\eta=\mathbf{a}\text{ and }\eta'\in\mathcal{Z}^{a,\,b}\;,\\
|\{\xi\in\mathcal{N}(\mathbf{a}):\xi\sim\eta\}| & \text{if }\eta\in\mathcal{Z}^{a,\,b}\text{ and }\eta'=\mathbf{a}\;,
\end{cases}\label{e_EaMC}
\end{equation}
for $\{\eta,\,\eta'\}\in\mathscr{E}$, and by setting $\mathfrak{r}(\eta,\,\eta')=0$
otherwise. Define $(\mathfrak{Z}(t))_{t\ge0}$ as the continuous-time
Markov chain on $\mathscr{V}$ with rate $\mathfrak{r}(\cdot,\,\cdot)$.
Since the rate function $\mathfrak{r}(\cdot,\,\cdot)$ is symmetric
by its definition, the Markov chain $\mathfrak{Z}(\cdot)$ is reversible
with respect to its invariant distribution, which is the uniform distribution
on $\mathscr{V}$.
\end{itemize}
\end{defn}

We next argue that the process $\mathfrak{Z}(\cdot)$ defined above
approximates the MH dynamics on the subset $\mathcal{N}(\mathbf{a})\cup\mathcal{Z}_{\ell}^{a,\,b}\cup\{\xi_{\ell,\,2}^{a,\,b}\}$
for each selection $b\in S\setminus\{a\}$ and $\ell\in\mathbb{T}_{L}$.
\begin{prop}
\label{p_proj}For each $b\ne a$ and $\ell\in\mathbb{T}_{L}$, define
a projection map $\Pi_{\ell}^{a,\,b}:\mathcal{N}(\mathbf{a})\cup\mathcal{Z}_{\ell}^{a,\,b}\cup\{\xi_{\ell,\,2}^{a,\,b}\}\rightarrow\mathscr{V}$
by
\[
\Pi_{\ell}^{a,\,b}(\eta)=\begin{cases}
\mathbf{a} & \text{if }\eta\in\mathcal{N}(\mathbf{a})\;,\\
\eta & \text{if }\eta\in\mathcal{Z}_{\ell}^{a,\,b}\cup\{\xi_{\ell,\,2}^{a,\,b}\}\;.
\end{cases}
\]
\begin{enumerate}
\item For $\eta_{1},\,\eta_{2}\ne\mathbf{a}$, we have
\begin{equation}
\frac{1}{q}e^{-\Gamma\beta}\mathfrak{r}(\Pi_{\ell}^{a,\,b}(\eta_{1}),\,\Pi_{\ell}^{a,\,b}(\eta_{2}))=(1+o(1))\times\mu_{\beta}(\eta_{1})r_{\beta}(\eta_{1},\,\eta_{2})\;.\label{e_proj1}
\end{equation}
\item For $\eta_{1}\ne\mathbf{a}$, we have
\begin{equation}
\frac{1}{q}e^{-\Gamma\beta}\mathfrak{r}(\Pi_{\ell}^{a,\,b}(\eta_{1}),\,\mathbf{a})=(1+o(1))\times\sum_{\xi\in\mathcal{N}(\mathbf{a})}\mu_{\beta}(\eta_{1})r_{\beta}(\eta_{1},\,\xi)\;.\label{e_proj2}
\end{equation}
\end{enumerate}
In the case of $K=L$, we define $\Theta(\Pi_{\ell}^{a,\,b}):\mathcal{N}(\mathbf{a})\cup\Theta(\mathcal{Z}_{\ell}^{a,\,b})\cup\{\Theta(\xi_{\ell,\,2}^{a,\,b})\}\rightarrow\mathscr{V}$
in a similar way. Then, the same results hold as well.
\end{prop}

We prove this approximation statement in Section \ref{secB.4}.

In view of the previous proposition, we can construct a test function
on $\mathcal{E}^{a}$ in terms of the process $\mathfrak{Z}(\cdot)$.
To this end, we have to investigate the potential-theoretic objects
of the process $\mathfrak{Z}(\cdot)$. Denote by $\mathfrak{f}_{\cdot,\,\cdot}(\cdot)$,
$\mathfrak{cap}(\cdot,\,\cdot)$ and $\mathfrak{D}(\cdot)$ the equilibrium
potential, capacity and Dirichlet form with respect to the Markov
chain $\mathfrak{Z}(\cdot)$, respectively (we define them in the
same way as in Section \ref{sec4.1}). We also denote by $\mathfrak{L}$
the infinitesimal generator associated with the process $\mathfrak{Z}(\cdot)$
which acts on functions $F:\mathfrak{\mathscr{V}}\rightarrow\mathbb{R}$
as
\begin{equation}
(\mathfrak{L}F)(\eta)=\sum_{\eta'\in\mathscr{V}:\,\{\eta,\,\eta'\}\in\mathscr{E}}\mathfrak{r}(\eta,\,\eta')\{F(\eta')-F(\eta)\}\;.\label{e_genZ}
\end{equation}
Then, the constant $\mathfrak{e}_{0}$ that appears in \eqref{e_edef}
for the MH dynamics is defined by

\begin{equation}
\mathfrak{e}_{0}:=\frac{1}{|\mathscr{V}|\cdot\mathfrak{cap}(\mathbf{a},\,\xi_{\ell,\,2}^{a,\,b})}\;.\label{e_e0def}
\end{equation}
By the symmetry of the model, this constant $\mathfrak{e}_{0}$ does
not depend on $a$, $b$ or $\ell$. We emphasize that this constant
appears in the capacity estimate because the equilibrium potential
\begin{equation}
\mathfrak{f}:=\mathfrak{f}_{\mathbf{a},\,\xi_{\ell,2}^{a,b}}\label{e_eqpotMH}
\end{equation}
approximates that of the MH dynamics in $\mathcal{E}^{a}$ (cf. Proposition
\ref{p_proj}). Indeed, this function plays a significant role in
the construction of the test function in Section \ref{sec7}.

We conclude this section by showing that $\mathfrak{e}_{0}$ is indeed
small, so that asymptotically (in $K$) this constant has a negligible
effect on the constant $\kappa$ that appears in the EK law. 
\begin{prop}
\label{p_e0est}It holds that $\mathfrak{e}_{0}\le1$.
\end{prop}

Again, we give a nice proof in Section \ref{secB.4}.
\begin{rem}
\label{r_e0est}In fact, with a more refined argument, we can verify
that there exist two constants $C_{1},\,C_{2}>0$ such that $\frac{C_{1}}{K}\le\mathfrak{e}_{0}\le\frac{C_{2}}{K}$.
However, this level of precision is unnecessary in our logical structure
and thus we omit the proof.
\end{rem}

\section{\label{sec7}Test Functions: MH Dynamics}

Throughout the section, we will again consider the MH dynamics and
fix two disjoint, non-empty subsets $\mathcal{P}$ and $\mathcal{Q}$
of $\mathcal{S}$. The purpose of this section is to construct the
test function $\widetilde{h}=\widetilde{h}_{\mathcal{P},\,\mathcal{Q}}:\mathcal{X}\to\mathbb{R}$
and verify three conditions \eqref{e_eqpappr1}, \eqref{e_eqpappr2}
and \eqref{e_eqpappr3} for this test function. This will provide
the proof of Proposition \ref{p_eqpappr} and thus conclude our analysis
given in Section \ref{sec4} for the MH dynamics. Note that in this
case, the dynamics is reversible and thus we do not need to construct
$\widetilde{h}^{*}$.
\begin{notation}
\label{n_abc}Throughout this section, we always use the alphabets
$a,\,b,\,c$ to represent the spins subjected to $\mathbf{a}\in\mathcal{P}$,
$\mathbf{b}\in\mathcal{Q}$ and $\mathbf{c}\in\mathcal{S}\setminus(\mathcal{P}\cup\mathcal{Q})$,
respectively. If necessary, we also use $a',\,b',\,c'$ to represent
$\mathbf{a'}\in\mathcal{P}$, $\mathbf{b'}\in\mathcal{Q}$ and $\mathbf{c'}\in\mathcal{S}\setminus(\mathcal{P}\cup\mathcal{Q})$.
\end{notation}

\subsection{\label{sec7.1}Construction of $\widetilde{h}$}

In this subsection, we construct the desired test function $\widetilde{h}:\mathcal{X}\to\mathbb{R}$.
Before the explicit construction, we briefly explain the nature behind
such a description. On edge typical configurations, we let $\widetilde{h}$
be a proper rescaled function of $\mathfrak{f}$, which is the equilibrium
potential of the auxiliary dynamics defined in \eqref{e_eqpotMH}.
This construction is natural because Proposition \ref{p_proj} ensures
us that the auxiliary process successfully characterizes the behavior
of the original process on the collection $\mathcal{E}$ of edge typical
configurations. Next, on bulk typical configurations, we define $\widetilde{h}$
as a proper rescaling of the equilibrium potential of the symmetric
simple random walk on an one-dimensional line. This is because the
dynamics on the collection of bulk typical configurations is approximately
one-dimensional with symmetric simple transition rates due to the
simple geometry.
\begin{defn}[Test function]
\label{d_testf} We construct $\widetilde{h}:\mathcal{X}\rightarrow[0,\,1]$
on $\mathcal{E}$, $\mathcal{B}$ and $\mathcal{X}\setminus(\mathcal{E}\cup\mathcal{B})$
separately. Recall the constants defined in \eqref{e_bdef} and \eqref{e_edef}.
Note that we always use Notation \ref{n_abc}.
\begin{enumerate}
\item \textbf{Construction on $\mathcal{E}=\bigcup_{\mathbf{a}\in\mathcal{P}}\mathcal{E}^{a}\cup\bigcup_{\mathbf{b}\in\mathcal{Q}}\mathcal{E}^{b}\cup\bigcup_{\mathbf{c}\in\mathcal{S}\setminus(\mathcal{P}\cup\mathcal{Q})}\mathcal{E}^{c}$.}
\begin{itemize}
\item For $\eta\in\mathcal{E}^{a}$, we recall the decomposition \eqref{e_Eadecomp2}
of $\mathcal{E}^{a}$. We define
\begin{equation}
\widetilde{h}(\eta):=\begin{cases}
1 & \text{if }\eta\in\mathcal{D}^{a}\text{ or }\eta\in\mathcal{Z}^{a,\,a'}\cup\mathcal{R}_{2}^{a,\,a'}\text{ for }a'\ne a\;,\\
1-\frac{\mathfrak{e}}{\kappa}(1-\mathfrak{f}(\eta)) & \text{if }\eta\in\mathcal{Z}^{a,\,b}\cup\mathcal{R}_{2}^{a,\,b}\;,\\
1-\frac{|\mathcal{Q}|}{|\mathcal{P}|+|\mathcal{Q}|}\frac{\mathfrak{e}}{\kappa}(1-\mathfrak{f}(\eta)) & \text{if }\eta\in\mathcal{Z}^{a,\,c}\cup\mathcal{R}_{2}^{a,\,c}\;.
\end{cases}\label{e_testf1}
\end{equation}
\item For $\eta\in\mathcal{E}^{b}$, define
\begin{equation}
\widetilde{h}(\eta):=\begin{cases}
0 & \text{if }\eta\in\mathcal{D}^{b}\text{ or }\eta\in\mathcal{Z}^{b,\,b'}\cup\mathcal{R}_{2}^{b,\,b'}\text{ for }b'\ne b\;,\\
\frac{\mathfrak{e}}{\kappa}(1-\mathfrak{f}(\eta)) & \text{if }\eta\in\mathcal{Z}^{b,\,a}\cup\mathcal{R}_{2}^{b,\,a}\;,\\
\frac{|\mathcal{P}|}{|\mathcal{P}|+|\mathcal{Q}|}\frac{\mathfrak{e}}{\kappa}(1-\mathfrak{f}(\eta)) & \text{if }\eta\in\mathcal{Z}^{b,\,c}\cup\mathcal{R}_{2}^{b,\,c}\;.
\end{cases}\label{e_testf2}
\end{equation}
\item For $\eta\in\mathcal{E}^{c}$, define
\begin{equation}
\widetilde{h}(\eta):=\begin{cases}
\frac{|\mathcal{P}|}{|\mathcal{P}|+|\mathcal{Q}|} & \text{if }\eta\in\mathcal{D}^{c}\text{ or }\eta\in\mathcal{Z}^{c,\,c'}\cup\mathcal{R}_{2}^{c,\,c'}\text{ for }c'\ne c\;,\\
\frac{|\mathcal{P}|}{|\mathcal{P}|+|\mathcal{Q}|}+\frac{|\mathcal{Q}|}{|\mathcal{P}|+|\mathcal{Q}|}\frac{\mathfrak{e}}{\kappa}(1-\mathfrak{f}(\eta)) & \text{if }\eta\in\mathcal{Z}^{c,\,a}\cup\mathcal{R}_{2}^{c,\,a}\;,\\
\frac{|\mathcal{P}|}{|\mathcal{P}|+|\mathcal{Q}|}-\frac{|\mathcal{P}|}{|\mathcal{P}|+|\mathcal{Q}|}\frac{\mathfrak{e}}{\kappa}(1-\mathfrak{f}(\eta)) & \text{if }\eta\in\mathcal{Z}^{c,\,b}\cup\mathcal{R}_{2}^{c,\,b}\;.
\end{cases}\label{e_testf3}
\end{equation}
\end{itemize}
\item \textbf{Construction on $\mathcal{B}$.} Recall the decomposition
\eqref{e_Bab}.
\begin{itemize}
\item We set $\widetilde{h}\equiv1$ on $\mathcal{B}^{a,\,a'}$, $\widetilde{h}\equiv0$
on $\mathcal{B}^{b,\,b'}$ and $\widetilde{h}\equiv\frac{|\mathcal{P}|}{|\mathcal{P}|+|\mathcal{Q}|}$
on $\mathcal{B}^{c,\,c'}$.
\item For $\eta\in\mathcal{R}_{v}^{a,\,b}$ with $v\in\llbracket2,\,L-2\rrbracket$,
we set
\begin{equation}
\widetilde{h}(\eta):=\frac{1}{\kappa}\Big[\frac{L-2-v}{L-4}\mathfrak{b}+\mathfrak{e}\Big]\;.\label{e_testf4}
\end{equation}
For $\eta\in\mathcal{Q}_{v}^{a,\,b}$ with $v\in\llbracket2,\,L-3\rrbracket$,
we can write $\eta=\xi_{\ell,\,v;\,k,\,h}^{a,\,b,\,\pm}$ (or $\Theta(\xi_{\ell,\,v;\,k,\,h}^{a,\,b,\,\pm})$
if $K=L$) for some $\ell\in\mathbb{T}_{L}$, $k\in\mathbb{T}_{K}$
and $h\in\llbracket1,\,K-1\rrbracket$. For such $\eta$, we set
\begin{equation}
\widetilde{h}(\eta):=\frac{1}{\kappa}\Big[\frac{(K+2)(L-2-v)-(h+1)}{(K+2)(L-4)}\mathfrak{b}+\mathfrak{e}\Big]\;.\label{e_testf5}
\end{equation}
\item For $\eta\in\mathcal{R}_{v}^{a,\,c}$, $v\in\llbracket2,\,L-2\rrbracket$,
\[
\widetilde{h}(\eta):=\frac{|\mathcal{P}|}{|\mathcal{P}|+|\mathcal{Q}|}+\frac{|\mathcal{Q}|}{|\mathcal{P}|+|\mathcal{Q}|}\frac{1}{\kappa}\Big[\frac{L-2-v}{L-4}\mathfrak{b}+\mathfrak{e}\Big]\;.
\]
For $\eta\in\mathcal{Q}_{v}^{a,\,c}$, $v\in\llbracket2,\,L-3\rrbracket$,
similarly $\eta=\xi_{\ell,\,v;\,k,\,h}^{a,\,c,\,\pm}$ (or $\Theta(\xi_{\ell,\,v;\,k,\,h}^{a,\,c,\,\pm})$
if $K=L$) for some $\ell$, $k$ and $h$. For such $\eta$, we set
\[
\widetilde{h}(\eta):=\frac{|\mathcal{P}|}{|\mathcal{P}|+|\mathcal{Q}|}+\frac{|\mathcal{Q}|}{|\mathcal{P}|+|\mathcal{Q}|}\frac{1}{\kappa}\Big[\frac{(K+2)(L-2-v)-(h+1)}{(K+2)(L-4)}\mathfrak{b}+\mathfrak{e}\Big]\;.
\]
\item For $\eta\in\mathcal{R}_{v}^{c,\,b}$, $v\in\llbracket2,\,L-2\rrbracket$,
\[
\widetilde{h}(\eta):=\frac{|\mathcal{P}|}{|\mathcal{P}|+|\mathcal{Q}|}\frac{1}{\kappa}\Big[\frac{L-2-v}{L-4}\mathfrak{b}+\mathfrak{e}\Big]\;.
\]
For $\eta\in\mathcal{Q}_{v}^{c,\,b}$, $v\in\llbracket2,\,L-3\rrbracket$,
again, $\eta=\xi_{\ell,\,v;\,k,\,h}^{c,\,b,\,\pm}$ (or $\Theta(\xi_{\ell,\,v;\,k,\,h}^{c,\,b,\,\pm})$
if $K=L$) for some $\ell$, $k$ and $h$. We set
\[
\widetilde{h}(\eta):=\frac{|\mathcal{P}|}{|\mathcal{P}|+|\mathcal{Q}|}\frac{1}{\kappa}\Big[\frac{(K+2)(L-2-v)-(h+1)}{(K+2)(L-4)}\mathfrak{b}+\mathfrak{e}\Big]\;.
\]
\end{itemize}
\item \textbf{Construction on $\mathcal{X}\setminus(\mathcal{E}\cup\mathcal{B})$.}
We define $\widetilde{h}\equiv1$ on this set.
\end{enumerate}
\end{defn}

\begin{rem}
Since $\mathcal{E}\cap\mathcal{B}=\mathcal{R}_{2}$ by Proposition
\ref{p_typprop}-(2), we need to check that the constructions of $\widetilde{h}$
on $\mathcal{E}$ and on $\mathcal{B}$ agree with each other on this
intersection. This is immediate from our definition.
\end{rem}

\begin{rem}
\label{r_deadconst}According to Definition \ref{d_testf}, it is
inferred that
\begin{itemize}
\item $\widetilde{h}(\eta)=1$ for all $\eta\in\mathcal{D}^{a}$, $\mathbf{a}\in\mathcal{P}$,
\item $\widetilde{h}(\eta)=0$ for all $\eta\in\mathcal{D}^{b}$, $\mathbf{b}\in\mathcal{Q}$,
\item $\widetilde{h}(\eta)=\frac{|\mathcal{P}|}{|\mathcal{P}|+|\mathcal{Q}|}$
for all $\eta\in\mathcal{D}^{c}$, $\mathbf{c}\in\mathcal{S}\setminus(\mathcal{P}\cup\mathcal{Q})$.
\end{itemize}
\end{rem}

\begin{rem}
\label{r_eqpappr3}It is clear from our definition that the test function
$\widetilde{h}$ fulfills the requirement \eqref{e_eqpappr3}. Hence,
it remains to check \eqref{e_eqpappr1} and \eqref{e_eqpappr2}. 
\end{rem}

\subsection{\label{sec7.2}Dirichlet energy of the test function}

We first check \eqref{e_eqpappr2}.
\begin{prop}
\label{p_testf}The function $\widetilde{h}$ constructed in Definition
\ref{d_testf} satisfies \eqref{e_eqpappr2}, i.e., 
\[
D_{\beta}(\widetilde{h})=(1+o(1))\cdot\frac{|\mathcal{P}||\mathcal{Q}|}{\kappa(|\mathcal{P}|+|\mathcal{Q}|)}e^{-\Gamma\beta}\;.
\]
\end{prop}

\begin{proof}
Let us divide the Dirichlet form $D_{\beta}(\widetilde{h})$ into
three parts as 
\begin{equation}
\Big[\sum_{\{\eta,\,\xi\}\subset\mathcal{X}\setminus(\mathcal{E}\cup\mathcal{B})}+\sum_{\eta\in\mathcal{E}\cup\mathcal{B},\,\xi\in\mathcal{X}\setminus(\mathcal{E}\cup\mathcal{B})}+\sum_{\{\eta,\,\xi\}\subset\mathcal{E}\cup\mathcal{B}}\Big]\mu_{\beta}(\eta)r_{\beta}(\eta,\,\xi)\{\widetilde{h}(\xi)-\widetilde{h}(\eta)\}^{2}\;,\label{e_decDiri}
\end{equation}
where all summations are carried out for two connected configurations
$\eta$ and $\xi$, i.e., $\eta\sim\xi$. Note that the first summation
is $0$ by part (3) of Definition \ref{d_testf}.

For the second summation, we recall from part (4) of Proposition \ref{p_typprop}
that $\mathcal{E}\cup\mathcal{B}=\widehat{\mathcal{N}}(\mathcal{S})$.
Hence, by Lemma \ref{l_Nhatbdry}, we have $H(\xi)\ge\Gamma+1$. Since
$H(\eta)\le\Gamma$, by Theorem \ref{t_mu} and \eqref{e_cdtMH},
\[
\mu_{\beta}(\eta)r_{\beta}(\eta,\,\xi)=\mu_{\beta}(\xi)=\frac{1}{Z_{\beta}}e^{-\beta H(\xi)}=O(e^{-(\Gamma+1)\beta})\;.
\]
Since $\widetilde{h}:\mathcal{X}\rightarrow[0,\,1]$, we can conclude
that the second summation of \eqref{e_decDiri} is $O(e^{-(\Gamma+1)\beta})$.

It remains to estimate the third summation of \eqref{e_decDiri},
which is indeed the main constituent of the Dirichlet form. For each
subset $\mathcal{P}\subseteq\mathcal{X}$, we write
\begin{equation}
E(\mathcal{P})=\big\{\{\sigma,\,\zeta\}\subseteq\mathcal{P}:\sigma\sim\zeta\big\}\;.\label{e_EP}
\end{equation}
Then, we can observe from Proposition \ref{p_typprop} that we can
decompose $E(\mathcal{E}\cup\mathcal{B})$ as
\begin{equation}
E(\mathcal{E}\cup\mathcal{B})=E(\mathcal{B})\cup E(\mathcal{E})\;,\label{e_decE}
\end{equation}
and thus we can rewrite the first summation of \eqref{e_decDiri}
as
\begin{equation}
\Big[\sum_{\{\eta,\,\xi\}\in E(\mathcal{B})}+\sum_{\{\eta,\,\xi\}\in E(\mathcal{E})}\Big]\mu_{\beta}(\eta)r_{\beta}(\eta,\,\xi)\{\widetilde{h}(\xi)-\widetilde{h}(\eta)\}^{2}\;.\label{e_mce1}
\end{equation}
Now, we divide the estimate of these summations into two cases: $K<L$
or $K=L$.\medskip{}

\noindent \textbf{(Case $K<L$)} We start by calculating the first
summation of \eqref{e_mce1}. Note that $\mathcal{B}$ can be decomposed
as (cf. Notation \ref{n_abc})
\[
\bigcup_{\mathbf{a}\ne\mathbf{a'}}\mathcal{B}^{a,\,a'}\cup\bigcup_{\mathbf{b}\ne\mathbf{b'}}\mathcal{B}^{b,\,b'}\cup\bigcup_{\mathbf{c}\ne\mathbf{c'}}\mathcal{B}^{c,\,c'}\cup\bigcup_{\mathbf{a},\,\mathbf{b}}\mathcal{B}^{a,\,b}\cup\bigcup_{\mathbf{a},\,\mathbf{c}}\mathcal{B}^{a,\,c}\cup\bigcup_{\mathbf{b},\,\mathbf{c}}\mathcal{B}^{c,\,b}\;.
\]
Moreover, note that the summations with respect to $\bigcup_{\mathbf{a}\ne\mathbf{a'}}\mathcal{B}^{a,\,a'}$,
$\bigcup_{\mathbf{b}\ne\mathbf{b'}}\mathcal{B}^{b,\,b'}$ and $\bigcup_{\mathbf{c}\ne\mathbf{c'}}\mathcal{B}^{c,\,c'}$
vanish by part (2) of Definition \ref{d_testf}. Thus, we can divide
the first summation of \eqref{e_mce1} as 
\begin{equation}
\sum_{\mathbf{a},\,\mathbf{b}}\sum_{\{\eta,\,\xi\}\in E(\mathcal{B}^{a,b})}+\sum_{\mathbf{a},\,\mathbf{c}}\sum_{\{\eta,\,\xi\}\in E(\mathcal{B}^{a,c})}+\sum_{\mathbf{b},\,\mathbf{c}}\sum_{\{\eta,\,\xi\}\in E(\mathcal{B}^{c,b})}\;.\label{e_mce1.5}
\end{equation}
Since 
\begin{equation}
E(\mathcal{B}^{a,\,b})=\bigcup_{v\in\llbracket2,\,L-3\rrbracket}E(\mathcal{R}_{v}^{a,\,b}\cup\mathcal{Q}_{v}^{a,\,b}\cup\mathcal{R}_{v+1}^{a,\,b})\;,\label{e_EBab}
\end{equation}
the first double summation of \eqref{e_mce1.5} equals
\[
\sum_{\mathbf{a},\,\mathbf{b}}\sum_{v\in\llbracket2,\,L-3\rrbracket}\sum_{\{\eta,\,\xi\}\in E(\mathcal{R}_{v}^{a,b}\cup\mathcal{Q}_{v}^{a,b}\cup\mathcal{R}_{v+1}^{a,b})}\mu_{\beta}(\eta)r_{\beta}(\eta,\,\xi)\{\widetilde{h}(\xi)-\widetilde{h}(\eta)\}^{2}\;.
\]
Fixing $\mathbf{a}\in\mathcal{P}$ and $\mathbf{b}\in\mathcal{Q}$,
which is possible because of symmetry, the last display can be written
as 
\begin{equation}
|\mathcal{P}||\mathcal{Q}|\sum_{v\in\llbracket2,\,L-3\rrbracket}\sum_{\{\eta,\,\xi\}\in E(\mathcal{R}_{v}^{a,b}\cup\mathcal{Q}_{v}^{a,b}\cup\mathcal{R}_{v+1}^{a,b})}\mu_{\beta}(\eta)r_{\beta}(\eta,\,\xi)\{\widetilde{h}(\xi)-\widetilde{h}(\eta)\}^{2}\;.\label{e_mce2}
\end{equation}
The last summation in $\{\eta,\,\xi\}$ can be written as $\sum_{\ell\in\mathbb{T}_{L}}\sum_{k\in\mathbb{T}_{K}}$
of 
\begin{align*}
 & \mu_{\beta}(\xi_{\ell,\,v}^{a,\,b})r_{\beta}(\xi_{\ell,\,v}^{a,\,b},\,\xi_{\ell,\,v;\,k,\,1}^{a,\,b,\,\pm})\{\widetilde{h}(\xi_{\ell,\,v;\,k,\,1}^{a,\,b,\,\pm})-\widetilde{h}(\xi_{\ell,\,v}^{a,\,b})\}^{2}\\
 & +\sum_{h\in\llbracket1,\,K-2\rrbracket}\mu_{\beta}(\xi_{\ell,\,v;\,k,\,h}^{a,\,b,\,\pm})r_{\beta}(\xi_{\ell,\,v;\,k,\,h}^{a,\,b,\,\pm},\,\xi_{\ell,\,v;\,k,\,h+1}^{a,\,b,\,\pm})\{\widetilde{h}(\xi_{\ell,\,v;\,k,\,h+1}^{a,\,b,\,\pm})-\widetilde{h}(\xi_{\ell,\,v;\,k,\,h}^{a,\,b,\,\pm})\}^{2}\\
 & +\sum_{h\in\llbracket1,\,K-2\rrbracket}\mu_{\beta}(\xi_{\ell,\,v;\,k,\,h}^{a,\,b,\,\pm})r_{\beta}(\xi_{\ell,\,v;\,k,\,h}^{a,\,b,\,\pm},\,\xi_{\ell,\,v;\,k-1,\,h+1}^{a,\,b,\,\pm})\{\widetilde{h}(\xi_{\ell,\,v;\,k-1,\,h+1}^{a,\,b,\,\pm})-\widetilde{h}(\xi_{\ell,\,v;\,k,\,h}^{a,\,b,\,\pm})\}^{2}\\
 & +\mu_{\beta}(\xi_{\ell,\,v;\,k,\,K-1}^{a,\,b,\,+})r_{\beta}(\xi_{\ell,\,v;\,k,\,K-1}^{a,\,b,\,+},\,\xi_{\ell,\,v+1}^{a,\,b})\{\widetilde{h}(\xi_{\ell,\,v+1}^{a,\,b})-\widetilde{h}(\xi_{\ell,\,v;\,k,\,K-1}^{a,\,b,\,+})\}^{2}\\
 & +\mu_{\beta}(\xi_{\ell,\,v;\,k,\,K-1}^{a,\,b,\,-})r_{\beta}(\xi_{\ell,\,v;\,k,\,K-1}^{a,\,b,\,-},\,\xi_{\ell-1,\,v+1}^{a,\,b})\{\widetilde{h}(\xi_{\ell-1,\,v+1}^{a,\,b})-\widetilde{h}(\xi_{\ell,\,v;\,k,\,K-1}^{a,\,b,\,-})\}^{2}\;.
\end{align*}
By \eqref{e_cdtMH}, \eqref{e_testf4} and \eqref{e_testf5}, this
equals $2\sum_{\ell\in\mathbb{T}_{L}}\sum_{k\in\mathbb{T}_{K}}$ (where
$2$ is multiplied since $+$ and $-$ in $\pm$ give us the same
computation) of 
\begin{align*}
 & \frac{e^{-\Gamma\beta}}{Z_{\beta}}\cdot\frac{4\mathfrak{b}^{2}}{[\kappa(K+2)(L-4)]^{2}}+\sum_{h\in\llbracket1,\,K-2\rrbracket}\frac{e^{-\Gamma\beta}}{Z_{\beta}}\cdot\frac{\mathfrak{b}^{2}}{[\kappa(K+2)(L-4)]^{2}}\\
 & +\sum_{h\in\llbracket1,\,K-2\rrbracket}\frac{e^{-\Gamma\beta}}{Z_{\beta}}\cdot\frac{\mathfrak{b}^{2}}{[\kappa(K+2)(L-4)]^{2}}+\frac{e^{-\Gamma\beta}}{Z_{\beta}}\cdot\frac{4\mathfrak{b}^{2}}{[\kappa(K+2)(L-4)]^{2}}\;.
\end{align*}
By Theorem \ref{t_mu}, this is further simplified as
\[
(1+o(1))\cdot\frac{e^{-\Gamma\beta}}{q}\cdot\frac{(2K+4)\mathfrak{b}^{2}}{(K+2)^{2}(L-4)^{2}\kappa^{2}}=(1+o(1))\cdot\frac{2\mathfrak{b}^{2}}{q(K+2)(L-4)^{2}\kappa^{2}}e^{-\Gamma\beta}\;.
\]
Therefore by \eqref{e_bdef}, we can conclude that 
\begin{align*}
 & \sum_{\mathbf{a},\,\mathbf{b}}\sum_{\{\eta,\,\xi\}\in E(\mathcal{B}^{a,b})}\mu_{\beta}(\eta)r_{\beta}(\eta,\,\xi)\{\widetilde{h}(\xi)-\widetilde{h}(\eta)\}^{2}\\
 & =(1+o(1))\cdot|\mathcal{P}||\mathcal{Q}|\cdot\sum_{v\in\llbracket2,\,L-3\rrbracket}2\sum_{\ell\in\mathbb{T}_{L}}\sum_{k\in\mathbb{T}_{K}}\frac{2\mathfrak{b}^{2}}{q(K+2)(L-4)^{2}\kappa^{2}}e^{-\Gamma\beta}\\
 & =(1+o(1))\cdot|\mathcal{P}||\mathcal{Q}|\frac{4KL(L-4)\mathfrak{b}^{2}}{q(K+2)(L-4)^{2}\kappa^{2}}e^{-\Gamma\beta}=(1+o(1))\cdot\frac{|\mathcal{P}||\mathcal{Q}|\mathfrak{b}}{q\kappa^{2}}e^{-\Gamma\beta}\;.
\end{align*}
The second and third summations of \eqref{e_mce1.5} can be dealt
with in the same way (but with suitable constants multiplied), so
that the second and the third summation becomes
\[
(1+o(1))\cdot\frac{|\mathcal{Q}|^{2}|\mathcal{P}|(q-|\mathcal{P}|-|\mathcal{Q}|)\mathfrak{b}}{(|\mathcal{P}|+|\mathcal{Q}|)^{2}q\kappa^{2}}e^{-\Gamma\beta}\;\;\;\;\text{and}\;\;\;\;(1+o(1))\cdot\frac{|\mathcal{P}|^{2}|\mathcal{Q}|(q-|\mathcal{P}|-|\mathcal{Q}|)\mathfrak{b}}{(|\mathcal{P}|+|\mathcal{Q}|)^{2}q\kappa^{2}}e^{-\Gamma\beta}\;,
\]
respectively. Summing up the last two displays, we obtain
\begin{equation}
\sum_{\{\eta,\,\xi\}\in E(\mathcal{B})}\mu_{\beta}(\eta)r_{\beta}(\eta,\,\xi)\{\widetilde{h}(\xi)-\widetilde{h}(\eta)\}^{2}=(1+o(1))\cdot\frac{|\mathcal{P}||\mathcal{Q}|\mathfrak{b}}{(|\mathcal{P}|+|\mathcal{Q}|)\kappa^{2}}e^{-\Gamma\beta}\;.\label{e_mce3}
\end{equation}

Next, we calculate the second summation of \eqref{e_mce1}. We may
decompose $E(\mathcal{E})$ as
\[
E(\mathcal{E})=\bigcup_{\mathbf{a}}E(\mathcal{E}^{a})\cup\bigcup_{\mathbf{b}}E(\mathcal{E}^{b})\cup\bigcup_{\mathbf{c}}E(\mathcal{E}^{c})\;.
\]
First, consider the summation in $\bigcup_{\mathbf{a}}E(\mathcal{E}^{a})$.
By \eqref{e_Eadecomp2}, Definition \eqref{d_testf}-(1) and Remark
\ref{r_deadconst}, we may rewrite it as
\begin{align*}
 & \sum_{\mathbf{a},\,\mathbf{b}}\sum_{\ell\in\mathbb{T}_{L}}\sum_{\{\eta_{1},\,\eta_{2}\}\subset\mathcal{Z}_{\ell}^{a,b}\cup\{\xi_{\ell,2}^{a,b}\}}\mu_{\beta}(\eta_{1})r_{\beta}(\eta_{1},\,\eta_{2})\{\widetilde{h}(\eta_{2})-\widetilde{h}(\eta_{1})\}^{2}\\
 & +\sum_{\mathbf{a},\,\mathbf{b}}\sum_{\ell\in\mathbb{T}_{L}}\sum_{\eta_{1}\in\mathcal{Z}^{a,b}\cup\{\xi_{\ell,2}^{a,b}\}}\sum_{\xi\in\mathcal{N}(\mathbf{a})}\mu_{\beta}(\eta_{1})r_{\beta}(\eta_{1},\,\xi)\{\widetilde{h}(\xi)-\widetilde{h}(\eta_{1})\}^{2}\\
 & +\sum_{\mathbf{a},\,\mathbf{c}}\sum_{\ell\in\mathbb{T}_{L}}\sum_{\{\eta_{1},\,\eta_{2}\}\subset\mathcal{Z}_{\ell}^{a,c}\cup\{\xi_{\ell,2}^{a,c}\}}\mu_{\beta}(\eta_{1})r_{\beta}(\eta_{1},\,\eta_{2})\{\widetilde{h}(\eta_{2})-\widetilde{h}(\eta_{1})\}^{2}\\
 & +\sum_{\mathbf{a},\,\mathbf{c}}\sum_{\ell\in\mathbb{T}_{L}}\sum_{\eta_{1}\in\mathcal{Z}^{a,c}\cup\{\xi_{\ell,2}^{a,c}\}}\sum_{\xi\in\mathcal{N}(\mathbf{a})}\mu_{\beta}(\eta_{1})r_{\beta}(\eta_{1},\,\xi)\{\widetilde{h}(\xi)-\widetilde{h}(\eta_{1})\}^{2}\;.
\end{align*}
By Proposition \ref{p_proj}, taking into advantage the model symmetry,
this equals
\begin{align}
 & |\mathcal{P}||\mathcal{Q}|\cdot L(1+o(1))\cdot\Big[\sum_{\{\eta_{1},\,\eta_{2}\}\subset\mathcal{Z}_{\ell}^{a,b}\cup\{\xi_{\ell,2}^{a,b}\}}+\sum_{\eta_{1}\in\mathcal{Z}^{a,b}\cup\{\xi_{\ell,2}^{a,b}\},\,\eta_{2}=\mathbf{a}}\Big]\nonumber \\
 & +|\mathcal{P}|(q-|\mathcal{P}|-|\mathcal{Q}|)\cdot L(1+o(1))\cdot\Big[\sum_{\{\eta_{1},\,\eta_{2}\}\subset\mathcal{Z}_{\ell}^{a,c}\cup\{\xi_{\ell,2}^{a,c}\}}+\sum_{\eta_{1}\in\mathcal{Z}^{a,c}\cup\{\xi_{\ell,2}^{a,c}\},\,\eta_{2}=\mathbf{a}}\Big]\label{e_mce4}
\end{align}
applied to the summand $\frac{1}{q}e^{-\Gamma\beta}\mathfrak{r}(\eta_{1},\,\eta_{2})\{\widetilde{h}(\eta_{2})-\widetilde{h}(\eta_{1})\}^{2}$.
By \eqref{e_testf1}, this becomes
\begin{align*}
 & |\mathcal{P}||\mathcal{Q}|\cdot L(1+o(1))\cdot\frac{\mathfrak{e}^{2}}{\kappa^{2}}\sum_{\{\eta_{1},\,\eta_{2}\}\in\mathscr{E}}\\
 & +|\mathcal{P}|(q-|\mathcal{P}|-|\mathcal{Q}|)\cdot L(1+o(1))\cdot\frac{|\mathcal{Q}|^{2}}{(|\mathcal{P}|+|\mathcal{Q}|)^{2}}\frac{\mathfrak{e}^{2}}{\kappa^{2}}\sum_{\{\eta_{1},\,\eta_{2}\}\in\mathscr{E}}
\end{align*}
applied to the summand $\frac{1}{q}e^{-\Gamma\beta}\mathfrak{r}(\eta_{1},\,\eta_{2})\{\mathfrak{f}(\eta_{2})-\mathfrak{f}(\eta_{1})\}^{2}$.
Noting the definition of capacities (cf. \eqref{e_Capdef}), this
is equal to
\begin{align*}
 & \Big[1+\frac{(q-|\mathcal{P}|-|\mathcal{Q}|)|\mathcal{Q}|}{(|\mathcal{P}|+|\mathcal{Q}|)^{2}}\Big]\cdot|\mathcal{P}||\mathcal{Q}|\cdot L(1+o(1))\cdot\frac{e^{-\Gamma\beta}\mathfrak{e}^{2}}{\kappa^{2}q}|\mathscr{V}|\mathfrak{cap}(\mathbf{a},\,\mathcal{R}_{2}^{a,\,b})\\
 & =\Big[1+\frac{(q-|\mathcal{P}|-|\mathcal{Q}|)|\mathcal{Q}|}{(|\mathcal{P}|+|\mathcal{Q}|)^{2}}\Big]\cdot\frac{|\mathcal{P}||\mathcal{Q}|\mathfrak{e}(1+o(1))}{\kappa^{2}q}e^{-\Gamma\beta}\;,
\end{align*}
where in the equality we used \eqref{e_edef} and \eqref{e_e0def}.
Employing the same arguments to the summations in $\bigcup_{\mathbf{b}}E(\mathcal{E}^{b})$
and $\bigcup_{\mathbf{c}}E(\mathcal{E}^{c})$ in the second summation
of \eqref{e_mce1}, we obtain that the summations in $\bigcup_{\mathbf{b}}E(\mathcal{E}^{b})$
equals
\[
\Big[1+\frac{|\mathcal{P}|(q-|\mathcal{P}|-|\mathcal{Q}|)}{(|\mathcal{P}|+|\mathcal{Q}|)^{2}}\Big]\cdot\frac{|\mathcal{P}||\mathcal{Q}|\mathfrak{e}(1+o(1))}{\kappa^{2}q}e^{-\Gamma\beta}\;,
\]
and the summation in $\bigcup_{\mathbf{c}}E(\mathcal{E}^{c})$ becomes
\[
\Big[\frac{(q-|\mathcal{P}|-|\mathcal{Q}|)|\mathcal{Q}|}{(|\mathcal{P}|+|\mathcal{Q}|)^{2}}+\frac{|\mathcal{P}|(q-|\mathcal{P}|-|\mathcal{Q}|)}{(|\mathcal{P}|+|\mathcal{Q}|)^{2}}\Big]\cdot\frac{|\mathcal{P}||\mathcal{Q}|\mathfrak{e}(1+o(1))}{\kappa^{2}q}e^{-\Gamma\beta}\;.
\]
Therefore, summing up the last three displays, we conclude that
\begin{equation}
\sum_{\{\eta,\,\xi\}\in E(\mathcal{E})}\mu_{\beta}(\eta)r_{\beta}(\eta,\,\xi)\{\widetilde{h}(\xi)-\widetilde{h}(\eta)\}^{2}=\frac{2|\mathcal{P}||\mathcal{Q}|}{|\mathcal{P}|+|\mathcal{Q}|}\cdot\frac{\mathfrak{e}(1+o(1))}{\kappa^{2}}e^{-\Gamma\beta}\;.\label{e_mce6}
\end{equation}
Therefore, by \eqref{e_mce1}, \eqref{e_mce3} and \eqref{e_mce6},
we conclude that the third summation of \eqref{e_decDiri} equals
$1+o(1)$ times
\[
\frac{|\mathcal{P}||\mathcal{Q}|}{|\mathcal{P}|+|\mathcal{Q}|}\cdot\frac{\mathfrak{b}+2\mathfrak{e}}{\kappa^{2}}e^{-\Gamma\beta}=\frac{|\mathcal{P}||\mathcal{Q}|}{\kappa(|\mathcal{P}|+|\mathcal{Q}|)}e^{-\Gamma\beta}\;,
\]
as desired.\medskip{}

\noindent \textbf{(Case $K=L$)} This case is analogous to the previous
case. The only difference is the fact that \eqref{e_mce1} must be
counted twice. This fact is reflected in the definition of the constants
$\mathfrak{b}$ and $\mathfrak{e}$ in \eqref{e_bdef} and \eqref{e_edef}
through the constant $\nu_{0}$.
\end{proof}
\begin{rem}
\label{r_testf}The estimates \eqref{e_mce3} and \eqref{e_mce6}
are the reason that we call $\mathfrak{b}$ and $\mathfrak{e}$ the
bulk and edge constants, respectively.
\end{rem}

\subsection{\label{sec7.3}Proof of $H^{1}$-approximation}

In this subsection, we prove \eqref{e_eqpappr1} to complete the proof
of Proposition \ref{p_eqpappr} for the MH dynamics. To prove \eqref{e_eqpappr1},
we first expand
\begin{align}
D_{\beta}(h_{\mathcal{P},\,\mathcal{Q}}-\widetilde{h}) & =\langle h_{\mathcal{P},\,\mathcal{Q}}-\widetilde{h},\,-\mathcal{L}_{\beta}(h_{\mathcal{P},\,\mathcal{Q}}-\widetilde{h})\rangle_{\mu_{\beta}}\nonumber \\
 & =D_{\beta}(h_{\mathcal{P},\,\mathcal{Q}})+D_{\beta}(\widetilde{h})-\langle h_{\mathcal{P},\,\mathcal{Q}},\,-\mathcal{L}_{\beta}\widetilde{h}\rangle_{\mu_{\beta}}-\langle\widetilde{h},\,-\mathcal{L}_{\beta}h_{\mathcal{P},\,\mathcal{Q}}\rangle_{\mu_{\beta}}\;.\label{e_decDbhh}
\end{align}
Since $\widetilde{h}\equiv1$ on $\mathcal{P}$, $\widetilde{h}\equiv0$
on $\mathcal{Q}$ (cf. Remark \ref{r_eqpappr3}), and since $\mathcal{L}_{\beta}h_{\mathcal{P},\,\mathcal{Q}}\equiv0$
on $(\mathcal{P}\cup\mathcal{Q})^{c}$ (cf. \eqref{e_eqpsol}), we
have 
\begin{align*}
\langle\widetilde{h},\,-\mathcal{L}_{\beta}h_{\mathcal{P},\,\mathcal{Q}}\rangle_{\mu_{\beta}} & =\sum_{\mathbf{a}\in\mathcal{P}}\widetilde{h}(\mathbf{a})(-\mathcal{L}_{\beta}h_{\mathcal{P},\,\mathcal{Q}})(\mathbf{a})\mu_{\beta}(\mathbf{a})\\
 & =\sum_{\mathbf{a}\in\mathcal{P}}h_{\mathcal{P},\,\mathcal{Q}}(\mathbf{a})(-\mathcal{L}_{\beta}h_{\mathcal{P},\,\mathcal{Q}})(\mathbf{a})\mu_{\beta}(\mathbf{a})=D_{\beta}(h_{\mathcal{P},\,\mathcal{Q}})\;.
\end{align*}
In the second equality, we used that $h_{\mathcal{P},\,\mathcal{Q}}\equiv\widetilde{h}\equiv1$
on $\mathcal{P}$. Inserting this into \eqref{e_decDbhh} and expanding
the first inner product in the right-hand side, we obtain
\begin{equation}
D_{\beta}(h_{\mathcal{P},\,\mathcal{Q}}-\widetilde{h})=D_{\beta}(\widetilde{h})-\sum_{\eta\in\mathcal{X}}h_{\mathcal{P},\,\mathcal{Q}}(\eta)(-\mathcal{L}_{\beta}\widetilde{h})(\eta)\mu_{\beta}(\eta)\;.\label{e_Dbhh}
\end{equation}
By the definition of $\mathcal{L}_{\beta}$ (cf. \eqref{e_gen}) and
Proposition \ref{p_testf}, it suffices to prove that 
\begin{equation}
\sum_{\eta\in\mathcal{X}}h_{\mathcal{P},\,\mathcal{Q}}(\eta)\sum_{\xi\in\mathcal{X}}\mu_{\beta}(\eta)r_{\beta}(\eta,\,\xi)[\widetilde{h}(\eta)-\widetilde{h}(\xi)]=(1+o(1))\cdot\frac{|\mathcal{P}||\mathcal{Q}|}{\kappa(|\mathcal{P}|+|\mathcal{Q}|)}e^{-\Gamma\beta}\;.\label{e_WTS}
\end{equation}
Before proving this identity, we present a simple lemma which states
that the equilibrium potential $h_{\mathcal{P},\,\mathcal{Q}}$ is
asymptotically constant on each neighborhood $\mathcal{N}(\mathbf{s})$
for $\mathbf{s}\in\mathcal{S}$.
\begin{lem}
\label{l_eqpot}For all $\mathbf{s}\in\mathcal{S}$, we have $\max_{\eta\in\mathcal{N}(\mathbf{s})}\big|h_{\mathcal{P},\,\mathcal{Q}}(\eta)-h_{\mathcal{P},\,\mathcal{Q}}(\mathbf{s})\big|=o(1)$.
\end{lem}

\begin{proof}
We first recall from \cite[Theorem 3.2-(iii)]{NZB} that, for all
$\mathbf{s}\in\mathcal{S}$, we have 
\begin{equation}
\max_{\eta\in\mathcal{N}(\mathbf{s})}\mathbb{P}_{\eta}[\tau_{\mathcal{X}\setminus\mathcal{N}(\mathbf{s})}<\tau_{\mathbf{s}}]=o(1)\;.\label{e_eqpot1}
\end{equation}
We first assume that $\mathbf{s}\in\mathcal{Q}$, so that $h_{\mathcal{P},\,\mathcal{Q}}(\mathbf{s})=0$.
For $\eta\in\mathcal{N}(\mathbf{s})$, we can bound
\[
|h_{\mathcal{P},\,\mathcal{Q}}(\eta)-h_{\mathcal{P},\,\mathcal{Q}}(\mathbf{s})|=h_{\mathcal{P},\,\mathcal{Q}}(\eta)=\mathbb{P}_{\eta}[\tau_{\mathcal{P}}<\tau_{\mathcal{Q}}]\le\mathbb{P}_{\eta}[\tau_{\mathcal{X}\setminus\mathcal{N}(\mathbf{s})}<\tau_{\mathbf{s}}]\;.
\]
Thus, the proof is completed by \eqref{e_eqpot1}. We can handle the
case $\mathbf{s}\in\mathcal{P}$ by an entirely same manner. 

Let us finally consider the case $\mathbf{s}\in\mathcal{S}\setminus(\mathcal{P}\cup\mathcal{Q})$.
By the Markov property, for $\eta\in\mathcal{N}(\mathbf{s})$, 
\[
h_{\mathcal{P},\,\mathcal{Q}}(\eta)=\mathbb{P}_{\eta}[\tau_{\mathcal{P}}<\tau_{\mathcal{Q}}]=\mathbb{P}_{\eta}[\tau_{\mathbf{s}}<\tau_{\mathcal{X}\setminus\mathcal{N}(\mathbf{s})}]\mathbb{P}_{\mathbf{s}}[\tau_{\mathcal{P}}<\tau_{\mathcal{Q}}]+\mathbb{P}_{\eta}[\tau_{\mathbf{s}}>\tau_{\mathcal{X}\setminus\mathcal{N}(\mathbf{s})},\;\tau_{\mathcal{P}}<\tau_{\mathcal{Q}}]\;.
\]
From this, we obtain
\[
|h_{\mathcal{P},\,\mathcal{Q}}(\eta)-h_{\mathcal{P},\,\mathcal{Q}}(\mathbf{s})|\le2\mathbb{P}_{\eta}[\tau_{\mathbf{s}}>\tau_{\mathcal{X}\setminus\mathcal{N}(\mathbf{s})}]
\]
and thus the proof of this case is also completed by \eqref{e_eqpot1}.
\end{proof}
\begin{rem}
\label{r_eqpot}We can also prove this lemma by using a capacity bound
on the equilibrium potential (cf. \cite[Lemma 8.4]{BdenH Meta}),
and then applying the Dirichlet--Thomson principle.
\end{rem}

Now, we begin to prove \eqref{e_WTS}. Write 
\begin{equation}
\phi(\eta):=\sum_{\xi\in\mathcal{X}}\mu_{\beta}(\eta)r_{\beta}(\eta,\,\xi)[\widetilde{h}(\eta)-\widetilde{h}(\xi)]\;,\label{e_phi}
\end{equation}
so that our claim \eqref{e_WTS} can be simply rewritten as
\begin{equation}
\sum_{\eta\in\mathcal{X}}h_{\mathcal{P},\,\mathcal{Q}}(\eta)\phi(\eta)=(1+o(1))\cdot\frac{|\mathcal{P}||\mathcal{Q}|}{\kappa(|\mathcal{P}|+|\mathcal{Q}|)}e^{-\Gamma\beta}\;.\label{e_WTS2}
\end{equation}
First, we deduce that $\phi$ is negligible outside $\widehat{\mathcal{N}}(\mathcal{S})$.
\begin{lem}
\label{l_fl1}For all $\eta\in\mathcal{X}\setminus\widehat{\mathcal{N}}(\mathcal{S})$,
we have $\phi(\eta)=o(e^{-\Gamma\beta})$.
\end{lem}

\begin{proof}
Since $\widetilde{h}$ is defined as constant outside $\widehat{\mathcal{N}}(\mathcal{S})$,
we may only consider $\eta\in\mathcal{X}\setminus\widehat{\mathcal{N}}(\mathcal{S})$
which is connected to at least one configuration in $\widehat{\mathcal{N}}(\mathcal{S})$.
Then, by Lemma \ref{l_Nhatbdry}, we have $H(\eta)\ge\Gamma+1$, and
thus by \eqref{e_cdtMH}, it holds for $\xi\in\widehat{\mathcal{N}}(\mathcal{S})$
with $\eta\sim\xi$ that
\[
0\le\mu_{\beta}(\eta)r_{\beta}(\eta,\,\xi)=\min\{\mu_{\beta}(\eta),\,\mu_{\beta}(\xi)\}=\mu_{\beta}(\eta)=O(e^{-\beta(\Gamma+1)})\;.
\]
This completes the proof since $\widetilde{h}:\mathcal{X}\rightarrow[0,\,1]$.
\end{proof}
Next, we deal with $\phi(\eta)$ where $\eta\in\widehat{\mathcal{N}}(\mathcal{S})$.
We decompose $\phi(\eta)=\phi_{1}(\eta)+\phi_{2}(\eta)$ where
\begin{align*}
\phi_{1}(\eta) & =\sum_{\xi\in\widehat{\mathcal{N}}(\mathcal{S})}\mu_{\beta}(\eta)r_{\beta}(\eta,\,\xi)[\widetilde{h}(\eta)-\widetilde{h}(\xi)]\;,\\
\phi_{2}(\eta) & =\sum_{\xi\notin\widehat{\mathcal{N}}(\mathcal{S})}\mu_{\beta}(\eta)r_{\beta}(\eta,\,\xi)[\widetilde{h}(\eta)-\widetilde{h}(\xi)]\;.
\end{align*}
The first claim is that $\phi_{2}(\eta)$ is negligible for all $\eta\in\widehat{\mathcal{N}}(\mathcal{S})$.
\begin{lem}
\label{l_fl2}Suppose that $\eta\in\widehat{\mathcal{N}}(\mathcal{S})$.
Then, $\phi_{2}(\eta)=o(e^{-\Gamma\beta})$.
\end{lem}

\begin{proof}
If $\xi\notin\widehat{\mathcal{N}}(\mathcal{S})$ with $\eta\sim\xi$,
then
\[
\mu_{\beta}(\eta)r_{\beta}(\eta,\,\xi)=\min\{\mu_{\beta}(\eta),\,\mu_{\beta}(\xi)\}=\mu_{\beta}(\xi)=O(e^{-\beta(\Gamma+1)})
\]
by \eqref{e_cdtMH}, Proposition \ref{p_typprop}-(4), Lemma \ref{l_Nhatbdry}
and Theorem \ref{t_mu}. Thus, since $0\le\widetilde{h}\le1$,
\[
|\phi_{2}(\eta)|\le\sum_{\xi\notin\widehat{\mathcal{N}}(\mathcal{S}):\,\xi\sim\eta}\mu_{\beta}(\eta)r_{\beta}(\eta,\,\xi)=O(e^{-\beta(\Gamma+1)})\;,
\]
which is $o(e^{-\beta\Gamma})$ as desired.
\end{proof}
Summing up, instead of \eqref{e_WTS2}, it suffices to show that 
\begin{equation}
\sum_{\eta\in\widehat{\mathcal{N}}(\mathcal{S})}h_{\mathcal{P},\,\mathcal{Q}}(\eta)\phi_{1}(\eta)=(1+o(1))\cdot\frac{|\mathcal{P}||\mathcal{Q}|}{\kappa(|\mathcal{P}|+|\mathcal{Q}|)}e^{-\Gamma\beta}\;.\label{e_WTS3}
\end{equation}
Let us now investigate $\phi_{1}$. We start with the set $\mathcal{B}\setminus\mathcal{E}$.
\begin{lem}
\label{l_fl3}It holds that \textbf{$\phi_{1}(\eta)=0$} for all $\eta\in\mathcal{B}\setminus\mathcal{E}$.
\end{lem}

\begin{proof}
By construction of the function $\widetilde{h}$ on $\mathcal{B}$,
it suffices to deal with the cases $\eta\in\mathcal{B}^{a,\,b}\setminus\mathcal{E}$,
$\eta\in\mathcal{B}^{a,\,c}\setminus\mathcal{E}$ and $\eta\in\mathcal{B}^{c,\,b}\setminus\mathcal{E}$
where $\mathbf{a}\in\mathcal{P}$, $\mathbf{b}\in\mathcal{Q}$ and
$\mathbf{c}\in\mathcal{S}\setminus(\mathcal{P}\cup\mathcal{Q})$.
We start with the case $\eta\in\mathcal{B}^{a,\,b}\setminus\mathcal{E}$.

We first consider the case $K<L$. If $\eta=\xi_{\ell,\,v}^{a,\,b}$
for some $\ell\in\mathbb{T}_{L}$ and $v\in\llbracket3,\,L-3\rrbracket$,
by simple inspection, $\phi_{1}(\eta)$ equals
\[
\sum_{k\in\mathbb{T}_{K}}\sum_{\xi\in\{\xi_{\ell,v;k,1}^{a,b,+},\,\xi_{\ell,v;k,1}^{a,b,-},\,\xi_{\ell,v-1;k,K-1}^{a,b,+},\,\xi_{\ell+1,v-1;k,K-1}^{a,b,-}\}}\mu_{\beta}(\eta)r_{\beta}(\eta,\,\xi)[\widetilde{h}(\eta)-\widetilde{h}(\xi)]\;.
\]
Substituting the exact values from Definition \ref{d_testf} and noting
\eqref{e_cdtMH}, this becomes
\[
\frac{e^{-\beta\Gamma}}{Z_{\beta}}\sum_{k\in\mathbb{T}_{K}}\Big[\frac{2\mathfrak{b}}{\kappa(K+2)(L-4)}+\frac{2\mathfrak{b}}{\kappa(K+2)(L-4)}-\frac{2\mathfrak{b}}{\kappa(K+2)(L-4)}-\frac{2\mathfrak{b}}{\kappa(K+2)(L-4)}\Big]\;,
\]
which is zero. If $\eta=\xi_{\ell,\,v;\,k,\,h}^{a,\,b,\,+}$ for some
$\ell\in\mathbb{T}_{L}$, $v\in\llbracket2,\,L-3\rrbracket$, $k\in\mathbb{T}_{K}$
and $h\in\llbracket1,\,K-1\rrbracket$, then $\phi_{1}(\eta)$ equals
\begin{align*}
 & \sum_{\xi\in\{\xi_{\ell,v;k,h+1}^{a,b,+},\,\xi_{\ell,v;k-1,h+1}^{a,b,+},\,\xi_{\ell,v;k,h-1}^{a,b,+},\,\xi_{\ell,v;k+1,h-1}^{a,b,+}\}}\mu_{\beta}(\eta)r_{\beta}(\eta,\,\xi)[\widetilde{h}(\eta)-\widetilde{h}(\xi)]\\
 & =\frac{e^{-\Gamma\beta}}{Z_{\beta}}\Big[\frac{\mathfrak{b}}{\kappa(K+2)(L-4)}+\frac{\mathfrak{b}}{\kappa(K+2)(L-4)}-\frac{\mathfrak{b}}{\kappa(K+2)(L-4)}-\frac{\mathfrak{b}}{\kappa(K+2)(L-4)}\Big]\;,
\end{align*}
which is again zero. The case of $\eta=\xi_{\ell,\,v;\,k,\,h}^{a,\,b,\,-}$
can be handled similarly. Therefore, we conclude the case $\eta\in\mathcal{B}^{a,\,b}\setminus\mathcal{E}$
under the assumption $K<L$. If $K=L$, then there are twice more
possibilities obtained by transposing the above configurations. However,
this case can be dealt with identically as above.

Finally, note that the structure of $\widetilde{h}$ on $\mathcal{B}^{a,\,c}$
and $\mathcal{B}^{c,\,b}$ are the same as the structure on $\mathcal{B}^{a,\,b}$,
except for linear transformations. Thus, we can repeat the same calculations
to obtain the same result that $\phi_{1}(\eta)=0$.
\end{proof}
Next, we show that $\phi_{1}$ is also zero on $\mathcal{R}_{2}$
and $\mathcal{R}_{L-2}$.
\begin{lem}
\label{l_fl4}It holds that \textbf{$\phi_{1}(\eta)=0$} for all $\eta\in\mathcal{R}_{2}\cup\mathcal{R}_{L-2}$.
\end{lem}

\begin{proof}
We have $\mathcal{R}_{2}=\mathcal{R}_{L-2}$ by Definition \ref{d_canreg}
and thus we only focus on $\mathcal{R}_{2}$. By Definition \ref{d_testf},
we only need to consider $\eta\in\mathcal{R}_{2}^{a,\,b}\cup\mathcal{R}_{2}^{a,\,c}\cup\mathcal{R}_{2}^{c,\,b}$
for $\mathbf{a}\in\mathcal{P}$, $\mathbf{b}\in\mathcal{Q}$ and $\mathbf{c}\in\mathcal{S}\setminus(\mathcal{P}\cup\mathcal{Q})$.
First consider $\eta\in\mathcal{R}_{2}^{a,\,b}$. Without loss of
generality, we assume $\eta=\xi_{\ell,\,2}^{a,\,b}$, since we can
deal with the case of $\eta=\Theta(\xi_{\ell,\,2}^{a,\,b})$ in the
same way. Recall that $\mathfrak{f}=\mathfrak{f}_{\mathbf{a},\,\xi_{\ell,2}^{a,b}}$
from \eqref{e_eqpotMH} and recall the generator $\mathfrak{L}$ from
\eqref{e_genZ}. Then, since the uniform measure on $\mathscr{V}$
is the invariant measure for the process $\mathfrak{Z}(\cdot)$, by
the property of capacities (e.g., \cite[(7.1.39)]{BdenH Meta}), we
can write 
\[
\frac{1}{|\mathscr{V}|}\sum_{\xi\in\mathscr{V}\setminus\{\xi_{\ell,2}^{a,b}\}}\mathfrak{r}(\xi_{\ell,\,2}^{a,\,b},\,\xi)[\mathfrak{f}(\xi_{\ell,\,2}^{a,\,b})-\mathfrak{f}(\xi)]=-\frac{1}{|\mathscr{V}|}(\mathfrak{L}\mathfrak{f})(\xi_{\ell,\,2}^{a,\,b})=-\mathfrak{cap}(\mathbf{a},\,\xi_{\ell,\,2}^{a,\,b})\;.
\]
On the other hand, by the definition of $\widetilde{h}$, we can write
\[
\sum_{\xi\in\mathcal{E}^{a}}\mu_{\beta}(\xi_{\ell,\,2}^{a,\,b})r_{\beta}(\xi_{\ell,\,2}^{a,\,b},\,\xi)[\widetilde{h}(\xi_{\ell,\,2}^{a,\,b})-\widetilde{h}(\xi)]=\frac{1}{Z_{\beta}}e^{-\beta\Gamma}\frac{\mathfrak{e}}{\kappa}\sum_{\xi\in\mathscr{V}\setminus\{\xi_{\ell,2}^{a,b}\}}\mathfrak{r}(\xi_{\ell,\,2}^{a,\,b},\,\xi)\{\mathfrak{f}(\xi_{\ell,\,2}^{a,\,b})-\mathfrak{f}(\xi)\}\;.
\]
Summing up the computations above, we get
\begin{equation}
\sum_{\xi\in\mathcal{E}^{a}}\mu_{\beta}(\xi_{\ell,\,2}^{a,\,b})r_{\beta}(\xi_{\ell,\,2}^{a,\,b},\,\xi)[\widetilde{h}(\xi_{\ell,\,2}^{a,\,b})-\widetilde{h}(\xi)]=-\frac{1}{Z_{\beta}}e^{-\beta\Gamma}\frac{\mathfrak{e}}{\kappa}\times|\mathscr{V}|\mathfrak{cap}(\mathbf{a},\,\xi_{\ell,\,2}^{a,\,b})=-\nu_{0}\frac{e^{-\beta\Gamma}}{Z_{\beta}\kappa L}\;,\label{e_fl4.1}
\end{equation}
where the second identity is a consequence of the definitions of $\mathfrak{e}$
and $\mathfrak{e}_{0}$ given in \eqref{e_edef} and \eqref{e_e0def},
respectively. On the other hand, by the definition of $\mathfrak{b}$,
\begin{equation}
\sum_{\xi\in\mathcal{B}^{a,b}}\mu_{\beta}(\xi_{\ell,\,2}^{a,\,b})r_{\beta}(\xi_{\ell,\,2}^{a,\,b},\,\xi)[\widetilde{h}(\xi_{\ell,\,2}^{a,\,b})-\widetilde{h}(\xi)]=\frac{1}{Z_{\beta}}e^{-\beta\Gamma}\times2K\times\frac{2\mathfrak{b}}{\kappa(K+2)(L-4)}=\nu_{0}\frac{e^{-\beta\Gamma}}{Z_{\beta}\kappa L}\;.\label{e_fl4.2}
\end{equation}
By adding \eqref{e_fl4.1} and \eqref{e_fl4.2}, we obtain
\[
\phi_{1}(\xi_{\ell,\,2}^{a,\,b})=-\nu_{0}\frac{e^{-\beta\Gamma}}{Z_{\beta}\kappa L}+\nu_{0}\frac{e^{-\beta\Gamma}}{Z_{\beta}\kappa L}=0\;.
\]

The same computations can be done with the remaining cases $\eta\in\mathcal{R}_{2}^{a,\,c}$
and $\eta\in\mathcal{R}_{2}^{c,\,b}$, just by multiplying constants
to each term. Thus, we do not repeat the proof.
\end{proof}
Next, we show that $\phi_{1}(\eta)=0$ for $\eta\in\mathcal{Z}^{a,\,b}\cup\mathcal{Z}^{a,\,c}\cup\mathcal{Z}^{c,\,b}$,
$\mathbf{a}\in\mathcal{P}$, $\mathbf{b}\in\mathcal{Q}$ and $\mathbf{c}\in\mathcal{S}\setminus(\mathcal{P}\cup\mathcal{Q})$.
\begin{lem}
\label{l_fl5}It holds that \textbf{$\phi_{1}(\eta)=0$} for all $\eta\in\mathcal{Z}^{a,\,b}\cup\mathcal{Z}^{a,\,c}\cup\mathcal{Z}^{c,\,b}$.
\end{lem}

\begin{proof}
We consider only the case $\eta\in\mathcal{Z}^{a,\,b}$, since the
structure is identical in the other cases (with constants multiplied).
Recall $\mathfrak{L}$ from \eqref{e_genZ}. Then for each $\eta\in\mathcal{Z}^{a,\,b}$,
we can write 
\begin{equation}
\phi_{1}(\eta)=\frac{1}{Z_{\beta}}e^{-\Gamma\beta}\frac{\mathfrak{e}}{\kappa}\times(\mathfrak{L}\mathfrak{f})(\eta)\;.\label{e_fl5}
\end{equation}
By the elementary property of equilibrium potentials (e.g., \cite[(7.1.21)]{BdenH Meta}),
we can easily conclude from \eqref{e_fl5} that $\phi_{1}(\eta)=0$.
\end{proof}
All it remains is to consider the configurations in $\mathcal{D}^{s}$
for all $s\in S$. This is the content of the following two lemmas.
\begin{lem}
\label{l_fl6}For any $\mathbf{s}\in\mathcal{S}$ and $\eta\in\mathcal{D}^{s}\setminus\mathcal{N}(\mathbf{s})$,
we have $\phi_{1}(\eta)=0$.
\end{lem}

\begin{proof}
Recall from Definition \ref{d_testf}-(1) that $\widetilde{h}$ is
defined to be constant on $\mathcal{D}^{s}$, and thus for all $\eta\in\mathcal{D}^{s}\setminus\mathcal{N}(\mathbf{s})$,
\[
\phi_{1}(\eta)=\sum_{\xi\in\widehat{\mathcal{N}}(\mathcal{S}):\,\xi\sim\eta}\mu_{\beta}(\eta)r_{\beta}(\eta,\,\xi)[\widetilde{h}(\eta)-\widetilde{h}(\xi)]=0\;.
\]
This concludes the proof.
\end{proof}
\begin{lem}
\label{l_fl7}For $\mathbf{a}\in\mathcal{P}$, $\mathbf{b}\in\mathcal{Q}$
and $\mathbf{c}\in\mathcal{S}\setminus(\mathcal{P}\cup\mathcal{Q})$,
it holds that
\begin{align}
\sum_{\eta\in\mathcal{N}(\mathbf{a})}\phi_{1}(\eta) & =(1+o(1))\cdot\frac{|\mathcal{Q}|}{\kappa(|\mathcal{P}|+|\mathcal{Q}|)}e^{-\Gamma\beta}\;,\label{efl1}\\
\sum_{\eta\in\mathcal{N}(\mathbf{b})}\phi_{1}(\eta) & =-(1+o(1))\cdot\frac{|\mathcal{P}|}{\kappa(|\mathcal{P}|+|\mathcal{Q}|)}e^{-\Gamma\beta}\;,\label{efl2}\\
\sum_{\eta\in\mathcal{N}(\mathbf{c})}\phi_{1}(\eta) & =o(e^{-\Gamma\beta})\;.\label{efl3}
\end{align}
Moreover, there exists $C>0$ independent of $\beta$ such that for
all $\eta\in\mathcal{N}(\mathbf{a})\cup\mathcal{N}(\mathbf{b})\cup\mathcal{N}(\mathbf{c})$,
$|\phi_{1}(\eta)|\le Ce^{-\Gamma\beta}$.
\end{lem}

\begin{proof}
To start, we prove the first identity which is
\begin{equation}
\sum_{\eta\in\mathcal{N}(\mathbf{a})}\sum_{\xi\in\widehat{\mathcal{N}}(\mathcal{S}):\,\xi\sim\eta}\mu_{\beta}(\eta)r_{\beta}(\eta,\,\xi)[\widetilde{h}(\eta)-\widetilde{h}(\xi)]=(1+o(1))\cdot\frac{|\mathcal{Q}||}{\kappa(|\mathcal{P}|+|\mathcal{Q}|)}e^{-\Gamma\beta}\;.\label{e_fl7}
\end{equation}
The left-hand side can be written as
\begin{align*}
 & -\sum_{\eta\in\mathcal{N}(\mathbf{a})}\sum_{\mathbf{b}\in\mathcal{Q}}|\mathcal{R}_{2}^{a,\,b}|\sum_{\xi\in\mathcal{Z}_{\ell}^{a,b}:\,\xi\sim\eta}\frac{e^{-\Gamma\beta}}{Z_{\beta}}\frac{\mathfrak{e}}{\kappa}[\mathfrak{f}(\xi)-\mathfrak{f}(\mathbf{a})]\\
 & -\sum_{\eta\in\mathcal{N}(\mathbf{a})}\sum_{\mathbf{c}\in\mathcal{S}\setminus(\mathcal{P}\cup\mathcal{Q})}|\mathcal{R}_{2}^{a,\,c}|\sum_{\xi\in\mathcal{Z}_{\ell}^{a,c}:\,\xi\sim\eta}\frac{e^{-\Gamma\beta}}{Z_{\beta}}\frac{|\mathcal{Q}|}{|\mathcal{P}|+|\mathcal{Q}|}\frac{\mathfrak{e}}{\kappa}[\mathfrak{f}(\xi)-\mathfrak{f}(\mathbf{a})]\;.
\end{align*}
This can be rewritten as 
\begin{align*}
 & -\frac{e^{-\Gamma\beta}}{Z_{\beta}}\frac{\mathfrak{e}}{\kappa}\sum_{\mathbf{b}\in\mathcal{Q}}|\mathcal{R}_{2}^{a,\,b}|\sum_{\xi\in\mathcal{Z}_{\ell}^{a,b}:\,\{\xi,\,\mathbf{a}\}\in\mathscr{E}}\mathfrak{r}(\mathbf{a},\,\xi)[\mathfrak{f}(\xi)-\mathfrak{f}(\mathbf{a})]\\
 & -\frac{e^{-\Gamma\beta}}{Z_{\beta}}\frac{|\mathcal{Q}|}{|\mathcal{P}|+|\mathcal{Q}|}\frac{\mathfrak{e}}{\kappa}\sum_{\mathbf{c}\in\mathcal{S}\setminus(\mathcal{P}\cup\mathcal{Q})}|\mathcal{R}_{2}^{a,\,c}|\sum_{\xi\in\mathcal{Z}_{\ell}^{a,c}:\,\{\xi,\,\mathbf{a}\}\in\mathscr{E}}\mathfrak{r}(\mathbf{a},\,\xi)[\mathfrak{f}(\xi)-\mathfrak{f}(\mathbf{a})]\;.
\end{align*}
By the property of capacities (e.g., \cite[(7.1.39)]{BdenH Meta}),
the last display equals
\begin{align*}
 & -\frac{e^{-\Gamma\beta}}{Z_{\beta}}\frac{\mathfrak{e}}{\kappa}\Big[\sum_{\mathbf{b}\in\mathcal{Q}}|\mathcal{R}_{2}^{a,\,b}|\times(\mathfrak{L}\mathfrak{f})(\mathbf{a})+\frac{|\mathcal{Q}|}{|\mathcal{P}|+|\mathcal{Q}|}\sum_{\mathbf{c}\in\mathcal{S}\setminus(\mathcal{P}\cup\mathcal{Q})}|\mathcal{R}_{2}^{a,\,c}|\times(\mathfrak{L}\mathfrak{f})(\mathbf{a})\Big]\\
 & =\frac{e^{-\Gamma\beta}}{Z_{\beta}}\frac{\mathfrak{e}}{\kappa}\Big[|\mathcal{Q}|\times\frac{1}{\mathfrak{e}}+(q-|\mathcal{P}|-|\mathcal{Q}|)\times\frac{|\mathcal{Q}|}{|\mathcal{P}|+|\mathcal{Q}|}\frac{1}{\mathfrak{e}}\Big]=\frac{qe^{-\Gamma\beta}}{Z_{\beta}\kappa}\frac{|\mathcal{Q}|}{|\mathcal{P}|+|\mathcal{Q}|}\;.
\end{align*}
By Theorem \ref{t_mu}, we can verify \eqref{e_fl7}. The second and
third identities of the lemma can be proved in the exact same way
and thus we omit the detail.

Finally, take $\eta\in\mathcal{N}(\mathbf{a})\cup\mathcal{N}(\mathbf{b})\cup\mathcal{N}(\mathbf{c})$.
By Definition \ref{d_testf}, if $\xi\in\widehat{\mathcal{N}}(\mathcal{S})$
with $\xi\sim\eta$, we know that $\widetilde{h}(\eta)\ne\widetilde{h}(\xi)$
only if $\min\{H(\eta),\,H(\xi)\}\ge\Gamma$. This implies that
\[
|\phi_{1}(\eta)|\le\sum_{\xi\in\widehat{\mathcal{N}}(\mathcal{S}):\,\xi\sim\eta}\mu_{\beta}(\eta)r_{\beta}(\eta,\,\xi)|\widetilde{h}(\eta)-\widetilde{h}(\xi)|\le\frac{C}{Z_{\beta}}e^{-\Gamma\beta}\;.
\]
The inequality holds by \eqref{e_cdtMH} and the fact that $0\le\widetilde{h}\le1$.
Since $Z_{\beta}=q+o(1)$ by Theorem \ref{t_mu}, we conclude the
proof.
\end{proof}
Now, we have all ingredients to prove \eqref{e_eqpappr1}. Thus, we
are ready to complete the proof of Proposition \ref{p_eqpappr} for
the MH dynamics.
\begin{proof}[Proof of Proposition \ref{p_eqpappr} for the MH dynamics]
 By Remark \ref{r_eqpappr3} and Proposition \ref{p_testf}, it suffices
to verify \eqref{e_eqpappr1}. As we have discussed earlier, proving
\eqref{e_eqpappr1} is reduced to proving \eqref{e_WTS3}. This has
been verified in Lemmas \ref{l_eqpot} and \ref{l_fl3}-\ref{l_fl7}.
More precisely, by Lemmas \ref{l_fl3}-\ref{l_fl6}, we obtain
\[
\sum_{\eta\in\widehat{\mathcal{N}}(\mathcal{S})}h_{\mathcal{P},\,\mathcal{Q}}(\eta)\phi_{1}(\eta)=\sum_{\mathbf{s}\in\mathcal{S}}\sum_{\eta\in\mathcal{N}(\mathbf{s})}h_{\mathcal{P},\,\mathcal{Q}}(\eta)\phi_{1}(\eta)=o(e^{-\Gamma\beta})+\sum_{\mathbf{s}\in\mathcal{S}}\Big[h_{\mathcal{P},\,\mathcal{Q}}(\mathbf{s})\sum_{\eta\in\mathcal{N}(\mathbf{s})}\phi_{1}(\eta)\Big]\;,
\]
where the second identity follows from Lemma \ref{l_eqpot} and the
last part of Lemma \ref{l_fl7}. Inserting \eqref{efl1}, \eqref{efl2}
and \eqref{efl3} at the right-hand side, we obtain \eqref{e_WTS3}
and thus the proof is completed.
\end{proof}

\section{\label{sec8}Energy Landscape Analysis: Cyclic Dynamics}

In the current and the following sections, we always assume that the
process is the cyclic dynamics with $q\ge3$.

\subsection{\label{sec8.1}Energy barrier: proof of Theorem \ref{t_EB}}

We denote by $\Gamma=\Gamma^{\mathrm{cyc}}$ the energy barrier of
the cyclic dynamics. In this subsection, we prove Theorem \ref{t_EB},
i.e., $\Gamma=2K+4$. 

We first need to modify the canonical path of the MH dynamics defined
in Definition \ref{d_canpath} to get a corresponding object with
respect to the cyclic dynamics, since the spin updates for the cyclic
dynamics are more restrictive compared to that for the MH dynamics.
To this end, we fix two spins $a,\,b\in S$ and define the canonical
paths from $\mathbf{a}\in\mathcal{S}$ to $\mathbf{b}\in\mathcal{S}$.
\begin{defn}[Canonical paths of cyclic dynamics]
\label{d_canpathcyc} We denote by $r\in\llbracket1,\,q-1\rrbracket$
the unique integer which makes $b-a-r$ a multiple of $q$\footnote{Indeed, $r=b-a$ if $a<b$ and $r=q+b-a$ if $a>b$.}.
Then, a sequence of configurations $(\omega_{n})_{n=0}^{rKL}$ is
called a \textit{canonical path }from $\mathbf{a}$ to $\mathbf{b}$
if there exists a backbone canonical path $(\widetilde{\omega}_{m})_{m=0}^{KL}$
of the MH dynamics from $\mathbf{a}$ to $\mathbf{b}$ such that for
each $m\in\llbracket0,\,KL-1\rrbracket$ and $i\in\llbracket0,\,r\rrbracket$,
\[
\omega_{rm+i}=\tau_{x_{m}}^{i}\widetilde{\omega}_{m}\;,
\]
where $\tau_{x}^{i}\sigma$ denotes the configuration obtained from
$\sigma$ by applying $\tau_{x}$ for $i$ times, and where $x_{m}\in\Lambda$
is the site on which the spin update $\widetilde{\omega}_{m}\to\widetilde{\omega}_{m+1}$
occurs in the MH dynamics (i.e., $\widetilde{\omega}_{m+1}=\widetilde{\omega}_{m}^{x_{m},\,b}$).
\end{defn}

\begin{rem}
We note that the sequence $(\omega_{n})_{n=0}^{rKL}$ is well defined.
To this end, we need to verify that the definition agrees on each
$\omega_{rm}$ for $m\in\llbracket1,\,KL-1\rrbracket$. Indeed, this
is the case since
\[
\omega_{rm}=\widetilde{\omega}_{m}=\tau_{x_{m-1}}^{r}\widetilde{\omega}_{m-1}\;,
\]
where the second identity holds since $\widetilde{\omega}_{m}=\widetilde{\omega}_{m-1}^{x_{m-1},\,b}$
and $r$ is exactly the number of $\tau_{x_{m-1}}$ needed to be applied
to $\widetilde{\omega}_{m-1}$ to obtain $\widetilde{\omega}_{m}$.
\end{rem}

Now, we verify that a canonical path is indeed a path in the sense
of the cyclic dynamics, and that it indeed achieves the desired energy
level $2K+4$.
\begin{lem}
\label{l_canpathcyc}Suppose that $\omega=(\omega_{n})_{n=0}^{rKL}$
is a canonical path from $\mathbf{a}$ to $\mathbf{b}$, where $r\in\llbracket1,\,q-1\rrbracket$
is as defined in Definition \ref{d_canpathcyc}.
\begin{enumerate}
\item The sequence $\omega$ is indeed a path.
\item Recall \eqref{e_height}. Then, it holds that $\Phi_{\omega}=2K+4$.
\end{enumerate}
\end{lem}

\begin{proof}
We recall the notation of Definition \ref{d_canpathcyc} so that $(\widetilde{\omega}_{m})_{m=0}^{KL}$
denotes the backbone MH-canonical path associated with the path $\omega$
with $\widetilde{\omega}_{m+1}=\widetilde{\omega}_{m}^{x_{m},\,b}$.
\medskip{}

\noindent (1) By definition, for each $m\in\llbracket0,\,KL-1\rrbracket$
and $i\in\llbracket0,\,r-1\rrbracket$, we can write $\omega_{rm+i}=\tau_{x_{m}}^{i}\widetilde{\omega}_{m}$
and $\omega_{rm+i+1}=\tau_{x_{m}}^{i+1}\widetilde{\omega}_{m}$. Thus,
we have $\omega_{rm+i+1}=\tau_{x_{m}}\omega_{rm+i}$ which implies
$r_{\beta}(\omega_{rm+i},\,\omega_{rm+i+1})>0$. This concludes the
proof of part (1).\medskip{}

\noindent (2) We continue the notation of part (1), so that $\omega$
is induced from an MH-canonical path $(\widetilde{\omega}_{m})_{m=0}^{KL}$.
Recall from \eqref{e_Hsigmazeta} and \eqref{e_height} that 
\begin{align}
\Phi_{\omega} & =\max_{n\in\llbracket0,\,rKL-1\rrbracket}H(\omega_{n},\,\omega_{n+1})\nonumber \\
 & =\max_{m\in\llbracket0,\,KL-1\rrbracket}\max_{i\in\llbracket0,\,r-1\rrbracket}H(\omega_{rm+i},\,\omega_{rm+i+1})=\max_{m\in\llbracket0,\,KL-1\rrbracket}\max_{c\in S}H(\widetilde{\omega}_{m}^{x_{m},\,c})\;.\label{e_Phiomega}
\end{align}
Then, by an elementary computation, we can explicitly calculate $\max_{c\in S}H(\widetilde{\omega}_{m}^{x_{m},\,c})$.
Namely, 
\begin{equation}
\max_{c\in S}H(\widetilde{\omega}_{m}^{x_{m},\,c})=\begin{cases}
4 & \text{if }m=0\text{ or }KL-1\;,\\
2m+5 & \text{if }m\in\llbracket1,\,K-2\rrbracket\;,\\
2K+2 & \text{if }m=K-1\text{ or }K(L-1)\;,\\
2K+3 & \text{if }m=Kv\text{ or }K(v+1)-1\;,\\
2K+4 & \text{if }m\in\llbracket Kv+1,\,K(v+1)-2\rrbracket\;,\\
2KL-2m+3 & \text{if }m\in\llbracket K(L-1)+1,\,KL-2\rrbracket\;,
\end{cases}\label{e_canpathcyc}
\end{equation}
where $v\in\llbracket1,\,L-2\rrbracket$. Here, the condition $q\ge3$
is used in the fact that there always exists a third spin which is
neither $1$ nor $2$. Therefore, by \eqref{e_Phiomega}, we have
$\Phi_{\omega}=2K+4$ and thus conclude the proof.
\end{proof}
Now, we are ready to present the upper bound of the energy barrier.
\begin{prop}
\label{p_EBub}For all $\mathbf{a},\,\mathbf{b}\in\mathcal{S}$, we
have $\Phi(\mathbf{a},\,\mathbf{b})\le2K+4$. 
\end{prop}

\begin{proof}
The statement is direct from part (2) of Lemma \ref{l_canpathcyc}.
\end{proof}
A combinatorial proof of the next result is postponed to Section \ref{secD.1}.
\begin{prop}
\label{p_EBlb}For any path $\omega=(\omega_{n})_{n=0}^{N}$ from
$\mathbf{a}$ to $\breve{\mathbf{a}}$, we have $\Phi_{\omega}\ge2K+4$. 
\end{prop}

Finally, we are ready to prove Theorem \ref{t_EB}.
\begin{proof}[Proof of Theorem \ref{t_EB}]
 Let $\mathbf{a},\,\mathbf{b}\in\mathcal{S}$. Proposition \ref{p_EBub}
implies that $\Phi(\mathbf{a},\,\breve{\mathbf{a}})\le\Phi(\mathbf{a},\,\mathbf{b})\le2K+4$.
On the other hand, Proposition \ref{p_EBlb} implies that $\Phi(\mathbf{a},\,\breve{\mathbf{a}})\ge2K+4$.
Combining this two, we obtain that $\Phi(\mathbf{a},\,\breve{\mathbf{a}})=\Phi(\mathbf{a},\,\mathbf{b})=2K+4$. 
\end{proof}

\subsection{Neighborhoods}

In this subsection, we explain the neighborhoods for the cyclic dynamics.
Before proceeding to define them, we first check the following lemma
which was obvious for the MH dynamics but not quite so for the cyclic
dynamics.
\begin{lem}
\label{l_Phiissym}For all $\sigma,\,\zeta\in\mathcal{X}$, we have
$\Phi(\sigma,\,\zeta)=\Phi(\zeta,\,\sigma)$.
\end{lem}

\begin{proof}
Let $\omega=(\omega_{n})_{n=0}^{N}:\sigma\rightarrow\zeta$ be an
optimal path satisfying $\Phi_{\omega}=\Phi(\sigma,\,\zeta)$. Note
that the reversed sequence $(\omega_{N},\,\omega_{N-1},\,\dots,\,\omega_{0})$
is no longer a path with respect to the cyclic dynamics. However,
since $\omega_{n+1}=\tau_{x}\omega_{n}$ for some $x\in\Lambda$,
we can write $\omega_{n}=\tau_{x}^{q-1}\omega_{n+1}$ where $\tau_{x}^{i}\sigma$
denotes the configuration obtained from $\sigma$ by applying $\tau_{x}$
for $i$ times. Thus, we can replace $(\omega_{n+1},\,\omega_{n})$
with 
\begin{equation}
(\omega_{n+1},\,\tau_{x}\omega_{n+1},\,\tau_{x}^{2}\omega_{n+1},\,\dots,\tau_{x}^{q-1}\omega_{n+1})\;.\label{e_Phiissym}
\end{equation}
Replacing each $(\omega_{n+1},\,\omega_{n})$ in $(\omega_{N},\,\omega_{N-1},\,\dots,\,\omega_{0})$
in this manner, we get a path $\widetilde{\omega}$ (with respect
to the cyclic dynamics) from $\zeta$ to $\sigma$. Since
\[
H(\omega_{n+1},\,\tau_{x}\omega_{n+1})=\cdots=H(\tau_{x}^{q-2}\omega_{n+1},\,\tau_{x}^{q-1}\omega_{n+1})=H(\omega_{n},\,\omega_{n+1})\;,
\]
we immediately have that $\Phi_{\widetilde{\omega}}=\Phi_{\omega}=\Phi(\sigma,\,\zeta)$.
Thus we have $\Phi(\zeta,\,\sigma)\le\Phi(\sigma,\,\zeta)$. Replacing
the roles of $\sigma$ and $\zeta$, we also obtain the reversed inequality
$\Phi(\zeta,\,\sigma)\ge\Phi(\sigma,\,\zeta)$ and thus we complete
the proof. 
\end{proof}
Next, we define the neighborhood $\mathcal{N}(\sigma),\,\widehat{\mathcal{N}}(\sigma),\,\mathcal{N}(\mathcal{P})$
and $\widehat{\mathcal{N}}(\mathcal{P})$ in the same manner with
Definition \ref{d_nhd}. Then, since we also know that $\Gamma$ is
the energy barrier, Proposition \ref{p_nhd} also holds for the cyclic
dynamics. On the other hand, we now have to modify Lemma \ref{l_Nhatbdry}
in the following way. 
\begin{lem}
\label{l_Nhatbdrycyc}Let $\sigma\in\mathcal{X}$ and let $\zeta_{1}\in\widehat{\mathcal{N}}(\sigma)$,
$\zeta_{2}\notin\widehat{\mathcal{N}}(\sigma)$ and $\zeta_{1}\sim\zeta_{2}$.
Then, $H(\zeta_{1},\,\zeta_{2})>\Gamma$. 
\end{lem}

\begin{proof}
Suppose to the contrary that $H(\zeta_{1},\,\zeta_{2})\le\Gamma$.
Since $\zeta_{1}\sim\zeta_{2}$, we have either $r_{\beta}(\zeta_{1},\,\zeta_{2})>0$
or $r_{\beta}(\zeta_{2},\,\zeta_{1})>0$. For the former case, we
first find a path $\omega:\sigma\rightarrow\zeta_{1}$ such that $\Phi_{\omega}\le\Gamma$
(such a path exists since $\zeta_{1}\in\widehat{\mathcal{N}}(\sigma)$).
Then, we can concatenate the directed edge $(\zeta_{1},\,\zeta_{2})$
at the end of the path $\omega$ to form a path $\omega':\sigma\rightarrow\zeta_{2}$
satisfying $\Phi_{\omega'}\le\Gamma$. This contradicts $\zeta_{2}\notin\mathcal{\widehat{N}}(\sigma)$.

On the other hand, for the latter case, i.e., the case $r_{\beta}(\zeta_{2},\,\zeta_{1})>0$,
we find a path $\omega:\zeta_{1}\rightarrow\sigma$ such that $\Phi_{\omega}\le\Gamma$,
where such a path exists since $\Phi(\zeta_{1},\,\sigma)=\Phi(\sigma,\,\zeta_{1})\le\Gamma$
thanks to Lemma \ref{l_Phiissym}. Then, we concatenate the directed
edge $(\zeta_{2},\,\zeta_{1})$ at the front of the path $\omega$
to form a path $\omega'':\zeta_{2}\rightarrow\sigma$ satisfying $\Phi_{\omega''}\le\Gamma$.
Thus, we have $\Phi(\zeta_{2},\,\sigma)\le\Gamma$ and again by Lemma
\ref{l_Phiissym} we obtain $\Phi(\sigma,\,\zeta_{2})\le\Gamma$.
This again contradicts $\zeta_{2}\notin\mathcal{\widehat{N}}(\sigma)$.
\end{proof}
In turn, we modify the definition of the restricted neighborhoods
given in Definition \ref{d_rnhd}. 
\begin{defn}[Restricted neighborhood of cyclic dynamics]
\label{d_rnhdcyc} Let $\mathcal{Q}\subseteq\mathcal{X}$. For $\sigma\in\mathcal{X}\setminus\mathcal{Q}$,
we define $\widehat{\mathcal{N}}(\sigma;\,\mathcal{Q})$ as the collection
of all $\zeta$ with $H(\zeta)\le\Gamma$ such that there exist $\omega_{0},\,\omega_{1},\,\dots,\,\omega_{N}\in\mathcal{X}\setminus\mathcal{Q}$
satisfying $\omega_{0}=\sigma$, $\omega_{N}=\zeta$ and 
\[
\omega_{n}\sim\omega_{n+1}\text{ with }H(\omega_{n},\,\omega_{n+1})\le\Gamma\;\;\;\;\text{for all }n\in\llbracket0,\,N-1\rrbracket\;.
\]
For $\mathcal{P}\subseteq\mathcal{X}$ disjoint with $\mathcal{Q}$,
we define $\widehat{\mathcal{N}}(\mathcal{P};\,\mathcal{Q}):=\bigcup_{\sigma\in\mathcal{P}}\widehat{\mathcal{N}}(\sigma;\,\mathcal{Q})$.
\end{defn}

Then, as we did in the analysis of the MH dynamics, we formulate the
following lemma.
\begin{lem}
\label{l_hatNPemptycyc}For all $\mathcal{P}\subseteq\mathcal{X}$,
we have
\[
\widehat{\mathcal{N}}(\mathcal{P})=\widehat{\mathcal{N}}(\mathcal{P};\,\emptyset)\;.
\]
\end{lem}

\begin{proof}
One can notice that this identity is not as trivial as in the MH dynamics,
since in Definition \ref{d_rnhdcyc} $\omega_{n}$ and $\omega_{n+1}$
satisfy $\omega_{n}\sim\omega_{n+1}$ instead of $r_{\beta}(\omega_{n},\,\omega_{n+1})>0$.
Naturally, the proof presented here overcomes this difference.

We fix $\sigma\in\mathcal{P}$ and prove that $\widehat{\mathcal{N}}(\sigma)=\widehat{\mathcal{N}}(\sigma;\,\emptyset)$.
Since $r_{\beta}(\zeta_{1},\,\zeta_{2})>0$ implies $\zeta_{1}\sim\zeta_{2}$
for any $\zeta_{1},\,\zeta_{2}\in\mathcal{X}$, and since $H(\zeta)\le\Gamma$
for all $\zeta\in\widehat{\mathcal{N}}(\sigma)$, it is immediate
that $\widehat{\mathcal{N}}(\sigma)\subseteq\widehat{\mathcal{N}}(\sigma;\,\emptyset)$.
Thus, it suffices to demonstrate that
\[
\widehat{\mathcal{N}}(\sigma)\supseteq\widehat{\mathcal{N}}(\sigma;\,\emptyset)\;.
\]
Take an arbitrary $\zeta\in\widehat{\mathcal{N}}(\sigma;\,\emptyset)$,
so that $H(\zeta)\le\Gamma$ and there exist $\omega_{0},\,\dots,\,\omega_{N}\in\mathcal{X}$
which satisfy $\omega_{0}=\sigma$, $\omega_{N}=\zeta$, $\omega_{n}\sim\omega_{n+1}$
and $H(\omega_{n},\,\omega_{n+1})\le\Gamma$. We claim that for every
$n\in\llbracket0,\,N\rrbracket$, there exists a path $\widetilde{\omega}^{(n)}:\omega_{0}\to\omega_{n}$
such that $\Phi_{\tilde{\omega}^{(n)}}\le\Gamma$. This claim concludes
the proof by simply substituting $n=N$ and noting that $\omega_{0}=\sigma$
and $\omega_{N}=\zeta$, which would then imply that $\zeta\in\widehat{\mathcal{N}}(\sigma)$
as desired.

It remains to prove the claim. We inductively construct a path $\widetilde{\omega}^{(n)}:\sigma=\omega_{0}\to\omega_{n}$
as follows. First, set $\widetilde{\omega}_{0}=\omega_{0}=\sigma$
which immediately proves the case $n=0$. Then, for each $n\in\llbracket0,\,N-1\rrbracket$
suppose that $\widetilde{\omega}^{(n)}:\omega_{0}\to\omega_{n}$ is
constructed. Note that $r_{\beta}(\omega_{n},\,\omega_{n+1})>0$ or
$r_{\beta}(\omega_{n+1},\,\omega_{n})>0$. If the former case holds,
then we simply define $\widetilde{\omega}^{(n+1)}$ by concatenating
$\omega'=(\omega_{n},\,\omega_{n+1})$ at the end of $\widetilde{\omega}^{(n)}$.
If the latter case holds, we concatenate a detour path $\omega'$
of length $q-1$, as explained in \eqref{e_Phiissym}, at the end
of $\widetilde{\omega}^{(n)}$. In any case, $\widetilde{\omega}^{(n+1)}$
is indeed a path from $\omega_{0}$ to $\omega_{n+1}$, and moreover
\[
\Phi_{\widetilde{\omega}^{(n+1)}}=\max\{\Phi_{\widetilde{\omega}^{(n)}},\,\Phi_{\omega'}\}\le\Gamma\;,
\]
where the inequality holds by the induction hypothesis and the fact
that $\Phi_{\omega'}=H(\omega_{n},\,\omega_{n+1})\le\Gamma$ in both
cases. Therefore, we conclude the proof of the lemma.
\end{proof}

\subsection{\label{sec8.3}Orbits and typical configurations}

\begin{figure}
\includegraphics[width=13.5cm]{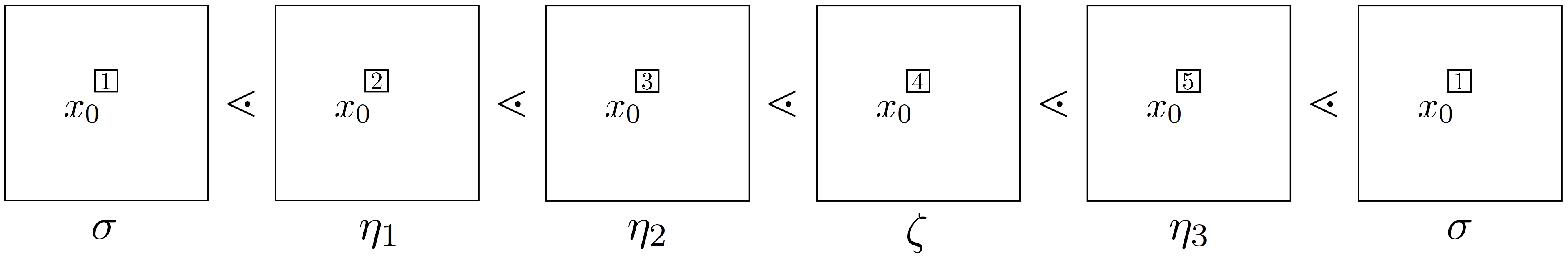}\caption{\label{fig8.1}\textbf{Example of an orbit.} Here, $q=5$. Suppose
that $\zeta=\tau_{x_{0}}^{3}\sigma$ with $\sigma(x_{0})=1$ and $\zeta(x_{0})=4$
as in the figure. Then, the orbit $\mathfrak{O}_{x_{0}}(\sigma)=\mathfrak{O}(\sigma,\,\zeta)$
consists of the presented five configurations. Note that $\sigma\lessdot\eta_{1}\lessdot\eta_{2}\lessdot\zeta\lessdot\eta_{3}\lessdot\sigma$.}
\end{figure}

\begin{defn}[Orbits]
 We refer to Figure \ref{fig8.1}.
\begin{enumerate}
\item For $\sigma\in\mathcal{X}$ and $x\in\Lambda$, the \textit{orbit}
$\mathfrak{O}_{x}(\sigma)$ consists of $q$ configurations which
have same spin values with $\sigma$ on $\Lambda\setminus\{x\}$,
namely,
\[
\mathfrak{O}_{x}(\sigma):=\{\sigma,\,\tau_{x}\sigma,\,\dots,\,\tau_{x}^{q-1}\sigma\}\;.
\]
Note that for all $\zeta_{1},\,\zeta_{2}\in\mathfrak{O}_{x}(\sigma)$,
we have $r_{\beta}^{\mathrm{MH}}(\zeta_{1},\,\zeta_{2})>0$ where
$r_{\beta}^{\mathrm{MH}}(\cdot,\,\cdot)$ denotes the transition rate
of the MH dynamics.
\item Fix $\sigma,\,\zeta\in\mathcal{X}$ such that $\zeta=\tau_{x_{0}}^{i}\sigma$
for some $x_{0}\in\Lambda$ and $i\in\llbracket1,\,q-1\rrbracket$
(i.e., such that $r_{\beta}^{\mathrm{MH}}(\sigma,\,\zeta)>0$). The
\textit{orbit} $\mathfrak{O}(\sigma,\,\zeta)$ containing $\sigma$
and $\zeta$ is defined as
\[
\mathfrak{O}(\sigma,\,\zeta):=\mathfrak{O}_{x_{0}}(\sigma)=\{\sigma,\,\tau_{x_{0}}\sigma,\,\dots,\,\tau_{x_{0}}^{q-1}\sigma\}\;.
\]
It is clear that $\sigma,\,\zeta\in\mathfrak{O}(\sigma,\,\zeta)$,
and that $\mathfrak{O}(\sigma,\,\zeta)=\mathfrak{O}(\zeta,\,\sigma)$.
\item Following the above notation, for $\eta\in\mathfrak{O}(\sigma,\,\zeta)\setminus\{\sigma,\,\zeta\}$
so that $\eta=\tau_{x_{0}}^{j}\sigma$ with $j\in\llbracket1,\,q-1\rrbracket\setminus\{i\}$,
we write $\sigma\lessdot\eta\lessdot\zeta$ if $j\in\llbracket1,\,i-1\rrbracket$
and $\zeta\lessdot\eta\lessdot\sigma$ if $j\in\llbracket i+1,\,q-1\rrbracket$.
Intuitively, $\sigma\lessdot\eta\lessdot\zeta$ (resp. $\zeta\lessdot\eta\lessdot\sigma$)
if one meets $\eta$ during the series of $i$ updates $\sigma\to\zeta$
(resp. $q-i$ updates $\zeta\to\sigma$).
\item For a subset $\mathcal{P}\subseteq\mathcal{X}$, we define
\begin{equation}
\mathfrak{O}(\mathcal{P}):=\bigcup_{\sigma,\,\zeta\in\mathcal{P}:\,r_{\beta}^{\mathrm{MH}}(\sigma,\,\zeta)>0}\mathfrak{O}(\sigma,\,\zeta)\;.\label{e_OP}
\end{equation}
\end{enumerate}
\end{defn}

The notion of orbits is necessary since all spin updates in an orbit
$\mathfrak{O}_{x}(\sigma)$ attain the same level of height. More
precisely, noting that $\mathfrak{O}_{x}(\sigma)=\{\sigma,\,\tau_{x}\sigma,\,\dots,\,\tau_{x}^{q-1}\sigma\}$,
we have for all $i\in\llbracket0,\,q-1\rrbracket$ that (cf. \eqref{e_height})
\begin{equation}
H(\tau_{x}^{i}\sigma,\,\tau_{x}^{i+1}\sigma)=\max_{j\in\llbracket0,\,q-1\rrbracket}H(\tau_{x}^{j}\sigma)=\max_{\mathfrak{O}_{x}(\sigma)}H\;.\label{e_OrbitH}
\end{equation}
Now, we introduce typical configurations subjected to the cyclic dynamics.
\begin{defn}[Typical configurations of cyclic dynamics]
 Recall the sets $\mathcal{R}_{v}^{a,\,b}$ and $\mathcal{Q}_{v}^{a,\,b}$
defined in Definition \ref{d_canreg}.
\begin{enumerate}
\item For each regular configuration $\sigma\in\mathcal{R}_{v}^{a,\,b}$,
we write 
\begin{equation}
\overline{\sigma}:=\bigcup_{x\in\Lambda}\mathfrak{O}_{x}(\sigma)\;.\label{e_sigmabar}
\end{equation}
Then for $v\in\llbracket2,\,L-2\rrbracket$, we define
\[
\overline{\mathcal{R}}_{v}^{a,\,b}:=\bigcup_{\sigma\in\mathcal{R}_{v}^{a,b}}\overline{\sigma}\;\;\;\;\text{and}\;\;\;\;\overline{\mathcal{R}}_{v}:=\bigcup_{a,\,b\in S}\overline{\mathcal{R}}_{v}^{a,\,b}\;.
\]
\item For $v\in\llbracket2,\,L-3\rrbracket$, we define
\[
\overline{\mathcal{Q}}_{v}^{a,\,b}:=\mathfrak{O}(\mathcal{Q}_{v}^{a,\,b})\;\;\;\;\text{and}\;\;\;\;\overline{\mathcal{Q}}_{v}:=\bigcup_{a,\,b\in S}\overline{\mathcal{Q}}_{v}^{a,\,b}\;.
\]
\item The set of (cyclic) \emph{bulk typical configurations} between two
ground states $\mathbf{a}$ and $\mathbf{b}$ is defined by 
\[
\overline{\mathcal{B}}^{a,\,b}:=\bigcup_{v\in\llbracket2,\,L-3\rrbracket}\overline{\mathcal{Q}}_{v}^{a,\,b}\cup\bigcup_{v\in\llbracket2,\,L-2\rrbracket}\overline{\mathcal{R}}_{v}^{a,\,b}\;.
\]
We write
\[
\overline{\mathcal{B}}_{\Gamma}^{a,\,b}:=\Big[\bigcup_{v\in\llbracket2,\,L-3\rrbracket}\overline{\mathcal{Q}}_{v}^{a,\,b}\Big]\setminus\Big[\bigcup_{v\in\llbracket2,\,L-2\rrbracket}\overline{\mathcal{R}}_{v}^{a,\,b}\Big]\;.
\]
Then, we naturally define $\overline{\mathcal{B}}:=\bigcup_{a,\,b\in S}\overline{\mathcal{B}}^{a,\,b}$
and $\overline{\mathcal{B}}_{\Gamma}:=\bigcup_{a,\,b\in S}\overline{\mathcal{B}}_{\Gamma}^{a,\,b}$.
\item For $a\in S$, we define the set of (cyclic) \emph{edge typical configurations}
associated with spin $a$ as $\overline{\mathcal{E}}^{a}:=\widehat{\mathcal{N}}(\mathbf{a};\,\overline{\mathcal{B}}_{\Gamma})$.
Then, we set $\overline{\mathcal{E}}:=\bigcup_{a\in S}\overline{\mathcal{E}}^{a}$.
\end{enumerate}
\end{defn}

\begin{notation*}
For each regular configuration $\xi_{\ell,\,v}^{a,\,b}\in\mathcal{R}_{v}^{a,\,b}$,
we simply denote by $\overline{\xi}_{\ell,\,v}^{a,\,b}$ the collection
$\overline{\xi_{\ell,\,v}^{a,\,b}}$ (cf. \eqref{e_sigmabar}).
\end{notation*}
With this new definition of the typical configurations, we can construct
the saddle structure for the cyclic dynamics as we did in Proposition
\ref{p_typprop} for the MH dynamics. 
\begin{prop}
\label{p_typpropcyc}The following properties hold.
\begin{enumerate}
\item For spins $a,\,b\in S$, we have $\overline{\mathcal{E}}^{a}\cap\overline{\mathcal{E}}^{b}=\emptyset$.
\item For spins $a,\,b\in S$, we have $\overline{\mathcal{E}}^{a}\cap\overline{\mathcal{B}}^{a,\,b}=\overline{\mathcal{R}}_{2}^{a,\,b}$.
\item For three spins $a,\,b,\,c\in S$, we have $\overline{\mathcal{E}}^{a}\cap\overline{\mathcal{B}}^{b,\,c}=\emptyset$.
\item It holds that $\overline{\mathcal{E}}\cup\overline{\mathcal{B}}=\widehat{\mathcal{N}}(\mathcal{S})$. 
\end{enumerate}
\end{prop}

A detailed proof of this proposition is given in Section \ref{secD.2}.

\subsection{Graph structure of edge typical configurations}

Recall that in the investigation of the edge typical configurations
of the MH dynamics, the graph structure (cf. Section \ref{sec6.5})
played a significant role. In this subsection, we explain the corresponding
results for the cyclic dynamics. 

Recall the sets $\mathcal{Z}^{a,\,b}$ and $\mathcal{Z}_{\ell}^{a,\,b}$
from \eqref{e_Zabdef} and Definition \ref{d_Zlabdef}, respectively,
and define 
\[
\overline{\mathcal{Z}}^{a,\,b}:=\mathfrak{O}(\mathcal{Z}^{a,\,b})\;\;\;\text{and}\;\;\;\overline{\mathcal{Z}}_{\ell}^{a,\,b}:=\mathfrak{O}(\mathcal{Z}_{\ell}^{a,\,b})
\]
so that, by the characterization of the set $\mathcal{Z}^{a,\,b}$
in Proposition \ref{p_Zabtree}, we can decompose the set $\overline{\mathcal{Z}}^{a,\,b}$
into (similarly with \eqref{e_Zabdecomp}) 
\begin{equation}
\overline{\mathcal{Z}}^{a,\,b}=\begin{cases}
\bigcup_{\ell\in\mathbb{T}_{L}}\overline{\mathcal{Z}}_{\ell}^{a,\,b} & \text{if }K<L\;,\\
\bigcup_{\ell\in\mathbb{T}_{L}}\overline{\mathcal{Z}}_{\ell}^{a,\,b}\cup\bigcup_{\ell\in\mathbb{T}_{L}}\Theta(\overline{\mathcal{Z}}_{\ell}^{a,\,b}) & \text{if }K=L\;.
\end{cases}\label{e_Zbarabdecomp}
\end{equation}
Now, we define 
\begin{equation}
\overline{\mathcal{Z}}^{a}:=\bigcup_{b\in S\setminus\{a\}}\overline{\mathcal{Z}}^{a,\,b}\;,\;\;\;\;\overline{\mathcal{R}}^{a}:=\bigcup_{b\in S\setminus\{a\}}\overline{\mathcal{R}}_{2}^{a,\,b}\;\;\;\;\text{and}\;\;\;\;\overline{\mathcal{D}}^{a}:=\widehat{\mathcal{N}}(\mathbf{a};\,\overline{\mathcal{Z}}^{a})\;,\label{e_Dbaradef}
\end{equation}
so that we have the following representations of $\overline{\mathcal{E}}^{a}$:
\begin{align*}
\overline{\mathcal{E}}^{a} & =\overline{\mathcal{D}}^{a}\cup\overline{\mathcal{Z}}^{a}\cup\overline{\mathcal{R}}^{a}\\
 & =\begin{cases}
\overline{\mathcal{D}}^{a}\cup\bigcup_{b\in S\setminus\{a\}}\bigcup_{\ell\in\mathbb{T}_{L}}[\overline{\mathcal{Z}}_{\ell}^{a,\,b}\cup\overline{\xi}_{\ell,\,2}^{a,\,b}] & \text{if }K<L\;,\\
\overline{\mathcal{D}}^{a}\cup\bigcup_{b\in S\setminus\{a\}}\bigcup_{\ell\in\mathbb{T}_{L}}[\overline{\mathcal{Z}}_{\ell}^{a,\,b}\cup\overline{\xi}_{\ell,\,2}^{a,\,b}]\cup\bigcup_{b\in S\setminus\{a\}}\bigcup_{\ell\in\mathbb{T}_{L}}[\Theta(\overline{\mathcal{Z}}_{\ell}^{a,\,b})\cup\Theta(\overline{\xi}_{\ell,\,2}^{a,\,b})] & \text{if }K=L\;.
\end{cases}
\end{align*}
A crucial difference here with the structure explained in Section
\ref{sec6} is that \textbf{\emph{$\overline{\mathcal{Z}}^{a,\,b}$
is no longer disjoint with both $\overline{\mathcal{D}}^{a}$ and
$\overline{\mathcal{R}}^{a}$.}} In turn, here we should define the
graph structure on the $\overline{\mathcal{Z}}^{a,\,b}$ itself and
may exclude the sets $\overline{\mathcal{D}}^{a}$ and $\overline{\mathcal{R}}^{a}$.
\begin{defn}
\label{d_ZbarMc}We fix $a,\,b\in S$ and $\ell\in\mathbb{T}_{L}$.
Then, we introduce a graph structure and a Markov chain on $\overline{\mathcal{Z}}_{\ell}^{a,\,b}$. 
\begin{itemize}
\item \textbf{(Graph) }We first assign a graph structure (cf. \eqref{e_EP})
\[
\overline{\mathscr{G}}_{\ell}^{a,\,b}=(\overline{\mathscr{V}}_{\ell}^{a,\,b},\,\overline{\mathscr{E}}_{\ell}^{a,\,b}):=(\overline{\mathcal{Z}}_{\ell}^{a,\,b},\,E(\overline{\mathcal{Z}}_{\ell}^{a,\,b}))\;.
\]
By symmetry, the graph structure $\overline{\mathscr{G}}_{\ell}^{a,\,b}$
does not depend on $\ell\in\mathbb{T}_{L}$ (or the transpose operator
$\Theta$ if $K=L$), while \textbf{\emph{it indeed depends on $a,\,b\in S$}}
(actually on the value of $b-a$ modulo $q$) unlike in Definition
\ref{d_EAMc}. Thus, we omit the subscript $\ell$ and simply write
$\overline{\mathscr{G}}^{a,\,b}=(\overline{\mathscr{V}}^{a,\,b},\,\overline{\mathscr{E}}^{a,\,b})$
when no risk of confusion arises. 
\item \textbf{(Target sets)} We define two disjoint subsets of $\overline{\mathscr{V}}^{a,\,b}=\overline{\mathcal{Z}}_{\ell}^{a,\,b}$
as
\[
\overline{\mathscr{A}}^{a,\,b}=\overline{\mathscr{A}}_{\ell}^{a,\,b}:=\overline{\mathcal{Z}}_{\ell}^{a,\,b}\cap\mathcal{N}(\mathbf{a})\;\;\;\;\text{and}\;\;\;\;\overline{\mathscr{B}}^{a,\,b}=\overline{\mathscr{B}}_{\ell}^{a,\,b}:=\overline{\mathcal{Z}}_{\ell}^{a,\,b}\cap\overline{\xi}_{\ell,\,2}^{a,\,b}\;.
\]
\item \textbf{(Markov chain)} We define a symmetric rate function $\overline{\mathfrak{r}}^{a,\,b}:\overline{\mathscr{V}}^{a,\,b}\times\overline{\mathscr{V}}^{a,\,b}\rightarrow[0,\,\infty)$
in such a way that 
\[
\overline{\mathfrak{r}}^{a,\,b}(\sigma,\,\sigma')=\begin{cases}
1 & \text{if }\{\sigma,\,\sigma'\}\in\overline{\mathscr{E}}^{a,\,b}\text{ and }r_{\beta}(\sigma,\,\sigma')>0\;,\\
0 & \text{otherwise}\;.
\end{cases}
\]
Then, define $(\overline{\mathfrak{Z}}^{a,\,b}(t))_{t\ge0}$ as the
continuous-time Markov chain on $\overline{\mathscr{V}}^{a,\,b}$
with rate $\overline{\mathfrak{r}}^{a,\,b}(\cdot,\,\cdot)$. The Markov
chain $\overline{\mathfrak{Z}}^{a,\,b}(\cdot)$ is invariant under
the uniform distribution on $\overline{\mathscr{V}}^{a,\,b}$, but
it is not reversible with respect to it. We denote by $\overline{\mathfrak{Z}}^{*,\,a,\,b}(\cdot)$
the adjoint Markov process of $\overline{\mathfrak{Z}}^{a,\,b}(\cdot)$.
Denote by $\overline{\mathfrak{f}}_{\cdot,\,\cdot}^{a,\,b}(\cdot)$,
$\overline{\mathfrak{cap}}^{a,\,b}(\cdot,\,\cdot)$ and $\overline{\mathfrak{D}}^{a,\,b}$
the equilibrium potential, capacity and Dirichlet form with respect
to the Markov chain $\overline{\mathfrak{Z}}^{a,\,b}(\cdot)$, respectively.
In addition, we denote by $\overline{\mathfrak{f}}_{\cdot,\,\cdot}^{*,\,a,\,b}(\cdot)$
the equilibrium potential with respect to the adjoint process $\overline{\mathfrak{Z}}^{*,\,a,\,b}(\cdot)$.
Finally, we abbreviate
\[
\overline{\mathfrak{f}}^{a,\,b}:=\overline{\mathfrak{f}}_{\overline{\mathscr{A}}^{a,b},\,\overline{\mathscr{B}}^{a,b}}^{a,\,b}\;\;\;\;\text{and}\;\;\;\;\overline{\mathfrak{f}}^{*,\,a,\,b}=\overline{\mathfrak{f}}_{\overline{\mathscr{A}}^{a,b},\,\overline{\mathscr{B}}^{a,b}}^{*,\,a,\,b}\;.
\]
\end{itemize}
\end{defn}

\subsubsection*{Edge constant}

We are now ready to define the edge constant $\mathfrak{e}_{0}$ that
appears in \eqref{e_edef} for the cyclic dynamics. We first define 

\begin{equation}
\overline{\mathfrak{e}}_{0}^{a,\,b}:=\frac{1}{|\overline{\mathscr{V}}^{a,\,b}|\cdot\overline{\mathfrak{cap}}^{a,\,b}(\overline{\mathscr{A}}^{a,\,b},\,\overline{\mathscr{B}}^{a,\,b})}\;.\label{e_e0barabdef}
\end{equation}
The next proposition not only provides the bound on this constant,
but also proves that this constant, somewhat surprisingly, does not
depend on the choices of $a$ and $b$. 
\begin{prop}
\label{p_e0barabest} The value of $\overline{\mathfrak{e}}_{0}^{a,\,b}$
does not depend on $a,\,b\in S$, and moreover we have $\overline{\mathfrak{e}}_{0}^{a,\,b}<2$.
\end{prop}

A proof of this proposition is given in Section \ref{secD.2}. Thanks
to this proposition, we can finally define the constant $\mathfrak{e}_{0}$
that appears in \eqref{e_edef} for the cyclic dynamics:
\begin{equation}
\mathfrak{e}_{0}:=\overline{\mathfrak{e}}_{0}^{a,\,b}\;\;\;\;\text{for any }a,\,b\in S\;.\label{e_e0cycdef}
\end{equation}

\section{\label{sec9}Test Functions: Cyclic Dynamics}

Fix two non-empty, disjoint $\mathcal{P}$ and $\mathcal{Q}$ of $\mathcal{S}$.
The purpose of this section is to construct two test functions $\widetilde{h},\,\widetilde{h}^{*}:\mathcal{X}\to\mathbb{R}$
that appear in Proposition \ref{p_eqpappr}, and then to prove that
these two test functions satisfy three conditions \eqref{e_eqpappr1},
\eqref{e_eqpappr2} and \eqref{e_eqpappr3}. We again use Notation
\ref{n_abc} throughout this section.
\begin{rem}
\label{r_analogue}Of course, the test functions of the cyclic dynamics
are different to (in fact, slightly more complex than) those of the
MH dynamics. However, most parts of the proofs of three conditions
\eqref{e_eqpappr1}, \eqref{e_eqpappr2} and \eqref{e_eqpappr3} are
quite similar or identical, and thus we omit the details for such
cases. \emph{This reveals that the proof of the non-reversible model
based on our strategy does not possess additional difficulty compared
to the proof of the reversible model.}
\end{rem}

\subsection{\label{sec9.1}Test functions}

In this subsection, we construct test functions $\widetilde{h}$ and
$\widetilde{h}^{*}$ which are approximations of the equilibrium potentials
$h_{\mathcal{P},\,\mathcal{Q}}$ and $h_{\mathcal{P},\,\mathcal{Q}}^{*}$,
respectively. The following definition is an analogue of Definition
\ref{d_testf}.
\begin{defn}[Test functions]
\label{d_testfcyc} We construct $\widetilde{h},\,\widetilde{h}^{*}:\mathcal{X}\rightarrow\mathbb{R}$.
Notation \ref{n_abc} is valid throughout. We first assume that $K<L$.
\begin{enumerate}
\item \textbf{Construction on $\overline{\mathcal{E}}$. }We first define
the function $\widetilde{h}$.
\begin{itemize}
\item For $\sigma\in\overline{\mathcal{E}}^{a}$, we define
\begin{equation}
\widetilde{h}(\sigma):=\begin{cases}
1 & \text{if }\sigma\in\overline{\mathcal{D}}^{a}\text{ or }\sigma\in\overline{\mathcal{Z}}^{a,\,a'}\cup\overline{\mathcal{R}}_{2}^{a,\,a'}\text{ for }a'\ne a\;,\\
1-\frac{\mathfrak{e}}{\kappa}(1-\overline{\mathfrak{f}}^{a,\,b}(\sigma)) & \text{if }\sigma\in\overline{\mathcal{Z}}^{a,\,b}\;,\\
1-\frac{|\mathcal{Q}|}{|\mathcal{P}|+|\mathcal{Q}|}\frac{\mathfrak{e}}{\kappa}(1-\overline{\mathfrak{f}}^{a,\,c}(\sigma)) & \text{if }\sigma\in\overline{\mathcal{Z}}^{a,\,c}\;.
\end{cases}\label{e_testfcyc1}
\end{equation}
\item For $\sigma\in\overline{\mathcal{E}}^{b}$, we define
\[
\widetilde{h}(\sigma):=\begin{cases}
0 & \text{if }\sigma\in\overline{\mathcal{D}}^{b}\text{ or }\sigma\in\overline{\mathcal{Z}}^{b,\,b'}\cup\overline{\mathcal{R}}_{2}^{b,\,b'}\text{ for }b'\ne b\;,\\
\frac{\mathfrak{e}}{\kappa}(1-\overline{\mathfrak{f}}^{b,\,a}(\sigma)) & \text{if }\sigma\in\overline{\mathcal{Z}}^{b,\,a}\;,\\
\frac{|\mathcal{P}|}{|\mathcal{P}|+|\mathcal{Q}|}\frac{\mathfrak{e}}{\kappa}(1-\overline{\mathfrak{f}}^{b,\,c}(\sigma)) & \text{if }\sigma\in\overline{\mathcal{Z}}^{b,\,c}\;.
\end{cases}
\]
\item For $\sigma\in\overline{\mathcal{E}}^{c}$, we define
\[
\widetilde{h}(\sigma):=\begin{cases}
\frac{|\mathcal{P}|}{|\mathcal{P}|+|\mathcal{Q}|} & \text{if }\sigma\in\overline{\mathcal{D}}^{c}\text{ or }\sigma\in\overline{\mathcal{Z}}^{c,\,c'}\cup\overline{\mathcal{R}}_{2}^{c,\,c'}\text{ for }c'\ne c\;,\\
\frac{|\mathcal{P}|}{|\mathcal{P}|+|\mathcal{Q}|}+\frac{|\mathcal{Q}|}{|\mathcal{P}|+|\mathcal{Q}|}\frac{\mathfrak{e}}{\kappa}(1-\overline{\mathfrak{f}}^{c,\,a}(\sigma)) & \text{if }\sigma\in\overline{\mathcal{Z}}^{c,\,a}\;,\\
\frac{|\mathcal{P}|}{|\mathcal{P}|+|\mathcal{Q}|}-\frac{|\mathcal{P}|}{|\mathcal{P}|+|\mathcal{Q}|}\frac{\mathfrak{e}}{\kappa}(1-\overline{\mathfrak{f}}^{c,\,b}(\sigma)) & \text{if }\sigma\in\overline{\mathcal{Z}}^{c,\,b}\;.
\end{cases}
\]
\end{itemize}
The function $\widetilde{h}^{*}$ is defined in the same way, except
that we substitute $\overline{\mathfrak{f}}^{*,\,s,\,s'}$ in place
of $\overline{\mathfrak{f}}^{s,\,s'}$ for all $s,\,s'\in S$ above.
\item \textbf{Construction on $\overline{\mathcal{B}}$. }We first set 
\[
\widetilde{h}\equiv\widetilde{h}^{*}\equiv\begin{cases}
1 & \text{on }\overline{\mathcal{B}}^{a,\,a'}\;,\\
0 & \text{on }\overline{\mathcal{B}}^{b,\,b'}\;,\\
\frac{|\mathcal{P}|}{|\mathcal{P}|+|\mathcal{Q}|} & \text{on }\overline{\mathcal{B}}^{c,\,c'}\;.
\end{cases}
\]
Next, we consider the construction on $\overline{\mathcal{B}}^{a,\,b}$
which is carried out by \textbf{\emph{interpolating the value from
$\frac{\mathfrak{b}+\mathfrak{e}}{\kappa}$ at $\overline{\mathcal{R}}_{2}^{a,\,b}$
to $\frac{\mathfrak{e}}{\kappa}$ at $\overline{\mathcal{R}}_{L-2}^{a,\,b}$}}
in the following way.
\begin{itemize}
\item For $\sigma\in\overline{\mathcal{R}}_{v}^{a,\,b}$ with $v\in\llbracket2,\,L-2\rrbracket$,
we set 
\[
\widetilde{h}(\sigma)=\widetilde{h}^{*}(\sigma)=\frac{1}{\kappa}\Big[\frac{L-2-v}{L-4}\mathfrak{b}+\mathfrak{e}\Big]\;.
\]
\item For $\sigma\in\mathcal{Q}_{v}^{a,\,b}$ with $v\in\llbracket2,\,L-3\rrbracket$,
where $\sigma=\xi_{\ell,\,v;\,k,\,h}^{a,\,b,\,\pm}$ for some $\ell\in\mathbb{T}_{L}$,
$k\in\mathbb{T}_{K}$ and $h\in\llbracket1,\,K-1\rrbracket$, we set
\[
\widetilde{h}(\sigma)=\widetilde{h}^{*}(\sigma)=\frac{1}{\kappa}\Big[\frac{(K-2)(L-2-v)-(h-1)}{(K-2)(L-4)}\mathfrak{b}+\mathfrak{e}\Big]\;.
\]
\item For $\sigma\in\overline{\mathcal{Q}}_{v}^{a,\,b}\setminus\mathcal{Q}_{v}^{a,\,b}$
with $v\in\llbracket2,\,L-3\rrbracket$, we have $\sigma\in\mathfrak{O}(\xi_{\ell,\,v;\,k,\,h}^{a,\,b,\,\pm},\,\xi_{\ell,\,v;\,k,\,h+1}^{a,\,b,\,\pm})$
or $\sigma\in\mathfrak{O}(\xi_{\ell,\,v;\,k,\,h}^{a,\,b,\,\pm},\,\xi_{\ell,\,v;\,k-1,\,h+1}^{a,\,b,\,\pm})$
for some $\ell\in\mathbb{T}_{L}$, $k\in\mathbb{T}_{K}$ and $h\in\llbracket1,\,K-2\rrbracket$.
For the former case, we set
\begin{equation}
\widetilde{h}(\sigma):=\begin{cases}
\widetilde{h}(\xi_{\ell,\,v;\,k,\,h+1}^{a,\,b,\,\pm}) & \text{if }\xi_{\ell,\,v;\,k,\,h}^{a,\,b,\,\pm}\lessdot\sigma\lessdot\xi_{\ell,\,v;\,k,\,h+1}^{a,\,b,\,\pm}\;,\\
\widetilde{h}(\xi_{\ell,\,v;\,k,\,h}^{a,\,b,\,\pm}) & \text{if }\xi_{\ell,\,v;\,k,\,h+1}^{a,\,b,\,\pm}\lessdot\sigma\lessdot\xi_{\ell,\,v;\,k,\,h}^{a,\,b,\,\pm}\;.
\end{cases}\label{e_testfcyc2}
\end{equation}
The value $\widetilde{h}^{*}(\sigma)$ is defined in the reversed
way. For the latter case, i.e., the case of $\sigma\in\mathfrak{O}(\xi_{\ell,\,v;\,k,\,h}^{a,\,b,\,\pm},\,\xi_{\ell,\,v;\,k-1,\,h+1}^{a,\,b,\,\pm})$,
the functions are constructed in the same manner.
\end{itemize}
Similarly, constructions on $\overline{\mathcal{B}}^{a,\,c}$ and
$\overline{\mathcal{B}}^{c,\,b}$ can be done by interpolating from
$\frac{|\mathcal{P}|}{|\mathcal{P}|+|\mathcal{Q}|}+\frac{|\mathcal{Q}|}{|\mathcal{P}|+|\mathcal{Q}|}\frac{\mathfrak{b}+\mathfrak{e}}{\kappa}$
to $\frac{|\mathcal{P}|}{|\mathcal{P}|+|\mathcal{Q}|}+\frac{|\mathcal{Q}|}{|\mathcal{P}|+|\mathcal{Q}|}\frac{\mathfrak{e}}{\kappa}$,
and from $\frac{|\mathcal{P}|}{|\mathcal{P}|+|\mathcal{Q}|}\frac{\mathfrak{b}+\mathfrak{e}}{\kappa}$
to $\frac{|\mathcal{P}|}{|\mathcal{P}|+|\mathcal{Q}|}\frac{\mathfrak{e}}{\kappa}$,
respectively.
\item \textbf{Construction on $\mathcal{X}\setminus(\overline{\mathcal{E}}\cup\overline{\mathcal{B}})$.}
We set $\widetilde{h}\equiv\widetilde{h}^{*}\equiv1$ thereon.
\end{enumerate}
For the case of $K=L$, for all $\sigma\in\overline{\mathcal{E}}\cup\overline{\mathcal{B}}$
considered in parts (1) and (2) above, we additionally set $\widetilde{h}(\Theta(\sigma)):=\widetilde{h}(\sigma)$
and $\widetilde{h}^{*}(\Theta(\sigma)):=\widetilde{h}^{*}(\sigma)$.
\end{defn}

\begin{rem}
\label{r_eqpappr3cyc}It is immediate from the definition that \eqref{e_eqpappr3}
holds for the test functions $\widetilde{h}$ and $\widetilde{h}^{*}$.
\end{rem}

It remains to prove conditions \eqref{e_eqpappr1} and \eqref{e_eqpappr2}.
We only verify these two conditions for $\widetilde{h}$, since the
verification for $\widetilde{h}^{*}$ is essentially identical (done
by reversing the order of orbits). Let us first check \eqref{e_eqpappr2}
for $\widetilde{h}$.
\begin{prop}
\label{p_testfcyc}The function $\widetilde{h}$ satisfies \eqref{e_eqpappr2},
i.e., 
\[
D_{\beta}(\widetilde{h})=(1+o(1))\cdot\frac{|\mathcal{P}||\mathcal{Q}|}{\kappa(|\mathcal{P}|+|\mathcal{Q}|)}e^{-\Gamma\beta}\;.
\]
\end{prop}

\begin{proof}
We decompose 
\[
D_{\beta}(\widetilde{h})=\frac{1}{2}\Big[\sum_{\sigma,\,\zeta\in\mathcal{X}\setminus(\overline{\mathcal{E}}\cup\overline{\mathcal{B}})}+\sum_{\substack{\sigma\in\overline{\mathcal{E}}\cup\overline{\mathcal{B}},\,\zeta\in\mathcal{X}\setminus(\overline{\mathcal{E}}\cup\overline{\mathcal{B}})\text{ or}\\
\zeta\in\overline{\mathcal{E}}\cup\overline{\mathcal{B}},\,\sigma\in\mathcal{X}\setminus(\overline{\mathcal{E}}\cup\overline{\mathcal{B}})
}
}+\sum_{\sigma,\,\zeta\in\overline{\mathcal{E}}\cup\overline{\mathcal{B}}}\Big]\mu_{\beta}(\sigma)r_{\beta}(\sigma,\,\zeta)[\widetilde{h}(\zeta)-\widetilde{h}(\sigma)]^{2}\;,
\]
where all summations are carried over pairs $\sigma,\,\zeta$ such
that $\sigma\sim\zeta$. Since $\widetilde{h}$ is defined as constant
on $\mathcal{X}\setminus(\overline{\mathcal{E}}\cup\overline{\mathcal{B}})$,
the first summation is $0$. For the second summation, by \eqref{e_Hcdt},
Lemma \ref{l_Nhatbdrycyc} and Proposition \ref{p_typpropcyc}-(4),
we have $\mu_{\beta}(\sigma)r_{\beta}(\sigma,\,\zeta)=O(e^{-(\Gamma+1)\beta})=o(e^{-\Gamma\beta})$
and thus this summation is negligible. It only remains to prove 
\[
\frac{1}{2}\sum_{\sigma,\,\zeta\in\overline{\mathcal{E}}\cup\overline{\mathcal{B}}}\mu_{\beta}(\sigma)r_{\beta}(\sigma,\,\zeta)[\widetilde{h}(\zeta)-\widetilde{h}(\sigma)]^{2}=(1+o(1))\cdot\frac{|\mathcal{P}||\mathcal{Q}|}{\kappa(|\mathcal{P}|+|\mathcal{Q}|)}e^{-\Gamma\beta}\;.
\]
The proof of this estimate is essentially the same with the corresponding
part in the proof of Proposition \ref{p_testf}, and we omit the detail.
\end{proof}

\subsection{Proof of $H^{1}$-approximation}

In this section, we prove \eqref{e_eqpappr1} for $\widetilde{h}$
and thus complete the proof of Proposition \ref{p_eqpappr}. The storyline
is identical to the one given in Section \ref{sec7.3}. Note that
the computation carried out in \eqref{e_Dbhh} is valid without reversibility
and thus we have
\[
D_{\beta}(h_{\mathcal{P},\,\mathcal{Q}}-\widetilde{h})=D_{\beta}(\widetilde{h})-\sum_{\sigma\in\mathcal{X}}h_{\mathcal{P},\,\mathcal{Q}}(\sigma)(-\mathcal{L}_{\beta}\widetilde{h})(\sigma)\mu_{\beta}(\sigma)\;.
\]
Hence, as in \eqref{e_WTS}, it suffices to prove that 
\begin{equation}
\sum_{\sigma\in\mathcal{X}}h_{\mathcal{P},\,\mathcal{Q}}(\sigma)\sum_{\zeta\in\mathcal{X}}\mu_{\beta}(\sigma)r_{\beta}(\sigma,\,\zeta)[\widetilde{h}(\sigma)-\widetilde{h}(\zeta)]=(1+o(1))\cdot\frac{|\mathcal{P}||\mathcal{Q}|}{\kappa(|\mathcal{P}|+|\mathcal{Q}|)}e^{-\Gamma\beta}\;.\label{e_WTScyc}
\end{equation}
To this end, we first demonstrate that the equilibrium potential $h_{\mathcal{P},\,\mathcal{Q}}$
is nearly constant on each neighborhood $\mathcal{N}(\sigma)$ for
any $\sigma\in\mathcal{X}$, which is an analogue of Lemma \ref{l_eqpot}.
\begin{lem}
\label{l_eqpotcyc}For all $\sigma\in\mathcal{X}$, it holds that
$\max_{\zeta\in\mathcal{N}(\sigma)}\big|h_{\mathcal{P},\,\mathcal{Q}}(\zeta)-h_{\mathcal{P},\,\mathcal{Q}}(\sigma)\big|=o(1)$.
\end{lem}

Note that the argument given in Lemma \ref{l_eqpot} relies on the
reversibility only because of the cited result \cite[Theorem 3.2-(iii)]{NZB}.
Thus, in this lemma, it suffices to replace this reference with \cite[Propositions 3.10 and 3.18-(3)]{CNSoh}
which is a non-reversible generalization of \cite[Theorem 3.2-(iii)]{NZB}.
In addition, we emphasize that the proof outlined in Remark \ref{r_eqpot}
can also be applied to this lemma as well. 

We now start to estimate the left-hand side of \eqref{e_WTScyc}.
As done in Section \ref{sec7.3}, write 
\begin{equation}
\phi(\sigma):=\sum_{\zeta\in\mathcal{X}}\mu_{\beta}(\sigma)r_{\beta}(\sigma,\,\zeta)[\widetilde{h}(\sigma)-\widetilde{h}(\zeta)]\;,\label{e_psi}
\end{equation}
so that we can rewrite \eqref{e_WTScyc} as
\begin{equation}
\sum_{\sigma\in\mathcal{X}}h_{\mathcal{P},\,\mathcal{Q}}(\sigma)\phi(\sigma)=(1+o(1))\cdot\frac{|\mathcal{P}||\mathcal{Q}|}{\kappa(|\mathcal{P}|+|\mathcal{Q}|)}e^{-\Gamma\beta}\;.\label{e_WTScyc2}
\end{equation}
The following lemma can be proved in the same manner with Lemma \ref{l_fl1};
it suffices to use Lemma \ref{l_Nhatbdrycyc} instead of Lemma \ref{l_Nhatbdry}.
\begin{lem}
\label{l_flcyc1}For all $\sigma\in\mathcal{X}\setminus\widehat{\mathcal{N}}(\mathcal{S})$,
we have $\phi(\sigma)=o(e^{-\Gamma\beta})$.
\end{lem}

Next, we decompose $\phi=\phi_{1}+\phi_{2}$ where for $\sigma\in\widehat{\mathcal{N}}(\mathcal{S})$,
\begin{align*}
\phi_{1}(\sigma) & :=\sum_{\zeta\in\widehat{\mathcal{N}}(\mathcal{S})}\mu_{\beta}(\sigma)r_{\beta}(\sigma,\,\zeta)[\widetilde{h}(\sigma)-\widetilde{h}(\zeta)]\;,\\
\phi_{2}(\sigma) & :=\sum_{\zeta\notin\widehat{\mathcal{N}}(\mathcal{S})}\mu_{\beta}(\sigma)r_{\beta}(\sigma,\,\zeta)[\widetilde{h}(\sigma)-\widetilde{h}(\zeta)]\;.
\end{align*}
Next, we can show that $\phi_{2}$ is negligible as in Lemma \ref{l_fl2}. 
\begin{lem}
\label{l_flcyc2}For all $\sigma\in\widehat{\mathcal{N}}(\mathcal{S})$,
we have $\phi_{2}(\sigma)=o(e^{-\Gamma\beta})$.
\end{lem}

This can be proved in the same manner with Lemma \ref{l_fl2}. In
addition to the replacement of Lemma \ref{l_Nhatbdry} with Lemma
\ref{l_Nhatbdrycyc}, we need to replace Proposition \ref{p_typprop}-(4)
with Proposition \ref{p_typpropcyc}-(4). Summing up, instead of \eqref{e_WTScyc2},
it suffices to show that 
\begin{equation}
\sum_{\sigma\in\widehat{\mathcal{N}}(\mathcal{S})}h_{\mathcal{P},\,\mathcal{Q}}(\sigma)\phi_{1}(\sigma)=(1+o(1))\cdot\frac{|\mathcal{P}||\mathcal{Q}|}{\kappa(|\mathcal{P}|+|\mathcal{Q}|)}e^{-\Gamma\beta}\;.\label{e_WTScyc3}
\end{equation}
Next, we investigate $\phi_{1}$ in several lemmas. Let us first consider
$\phi_{1}$ on the bulk typical configurations. 
\begin{lem}
\label{l_flcyc3}For each $v\in\llbracket2,\,L-3\rrbracket$ and $a,\,b\in S$,
it holds that \textbf{$\phi_{1}\equiv0$} on $\overline{\mathcal{Q}}_{v}^{a,\,b}\setminus(\overline{\mathcal{R}}_{v}^{a,\,b}\cup\overline{\mathcal{R}}_{v+1}^{a,\,b})$. 
\end{lem}

\begin{proof}
We only prove that $\phi_{1}(\sigma)=0$ for all $\sigma\in\mathfrak{O}(\xi_{\ell,\,v;\,k,\,h}^{a,\,b,\,+},\,\xi_{\ell,\,v;\,k,\,h+1}^{a,\,b,\,+})$
with $v\in\llbracket2,\,L-3\rrbracket$ and $h\in\llbracket1,\,K-2\rrbracket$,
since the other orbits can be handled in the same way. Write $\xi_{\ell,\,v;\,k,\,h+1}^{a,\,b,\,+}=\tau_{x}^{m}\xi_{\ell,\,v;\,k,\,h}^{a,\,b,\,+}$
for some $x\in\Lambda$ and $m\in\llbracket1,\,q-1\rrbracket$\footnote{Actually, $m=b-a$ or $m=q+b-a$.},
so that we can write 
\[
\sigma=\tau_{x}^{i}\xi_{\ell,\,v;\,k,\,h}^{a,\,b,\,+}\;\;\;\;\text{for some }i\in\llbracket0,\,q-1\rrbracket\;.
\]

\noindent \textbf{(Case 1:} $i\neq0,\,m$\textbf{)} We immediately
have $\widetilde{h}(\tau_{x}\sigma)=\widetilde{h}(\sigma)$ by definition
and thus 
\[
\phi_{1}(\sigma)=\mu_{\beta}(\sigma)r_{\beta}(\sigma,\,\tau_{x}\sigma)[\widetilde{h}(\sigma)-\widetilde{h}(\tau_{x}\sigma)]=0\;,
\]
where $\phi_{1}(\sigma)$ consists of only one term since we have
$\tau_{y}\sigma\notin\widehat{\mathcal{N}}(\mathcal{S})$ for all
$y\in\Lambda\setminus x$.

\noindent \textbf{(Case 2: $\sigma=\xi_{\ell,\,v;\,k,\,h}^{a,\,b,\,+}$,
so that $h\ne1$ since $\sigma\notin\overline{\mathcal{R}}_{v}^{a,\,b}$)}
In this case, there are four updates (that does not exceed the energy
level $\Gamma$) from $\sigma$ along the orbits $\mathfrak{O}(\xi_{\ell,\,v;\,k,\,h}^{a,\,b,\,+},\,\xi_{\ell,\,v;\,k,\,h+1}^{a,\,b,\,+})$,
$\mathfrak{O}(\xi_{\ell,\,v;\,k,\,h}^{a,\,b,\,+},\,\xi_{\ell,\,v;\,k-1,\,h+1}^{a,\,b,\,+})$,
$\mathfrak{O}(\xi_{\ell,\,v;\,k,\,h-1}^{a,\,b,\,+},\,\xi_{\ell,\,v;\,k,\,h}^{a,\,b,\,+})$
and $\mathfrak{O}(\xi_{\ell,\,v;\,k+1,\,h-1}^{a,\,b,\,+},\,\xi_{\ell,\,v;\,k,\,h}^{a,\,b,\,+})$.
By a direct computation with the definition of $\widetilde{h}$, we
can check that for all these updates, which we denote as $\tau_{x_{1}},\,\tau_{x_{2}},\,\tau_{x_{3}},\,\tau_{x_{4}}$
respectively, we have 
\begin{equation}
\mu_{\beta}(\sigma)r_{\beta}(\sigma,\,\tau_{x_{i}}\sigma)[\widetilde{h}(\sigma)-\widetilde{h}(\tau_{x_{i}}\sigma)]=\pm\frac{e^{-\Gamma\beta}}{Z_{\beta}}\cdot\frac{\mathfrak{b}}{\kappa(K-2)(L-4)}\;\;\;\;;\;i\in\llbracket1,\,4\rrbracket\;,\label{e_flcyc3}
\end{equation}
where the sign is plus for $i\in\{1,\,2\}$ and minus for $i\in\{3,\,4\}$.
Therefore,
\begin{align*}
\phi_{1}(\sigma) & =\sum_{i=1}^{4}\mu_{\beta}(\sigma)r_{\beta}(\sigma,\,\tau_{x_{i}}\sigma)[\widetilde{h}(\sigma)-\widetilde{h}(\tau_{x_{i}}\sigma)]\\
 & =\frac{e^{-\Gamma\beta}}{Z_{\beta}}\cdot\frac{\mathfrak{b}}{\kappa(K-2)(L-4)}\cdot(1+1-1-1)=0\;.
\end{align*}

\noindent \textbf{(Case 3: $\sigma=\xi_{\ell,\,v;\,k,\,h+1}^{a,\,b,\,+}$,
so that $h\ne K-2$ since $\sigma\notin\overline{\mathcal{R}}_{v+1}^{a,\,b}$)}
This case can be handled in the same manner with \textbf{(Case 2)}.
\end{proof}
Next, we handle $\overline{\mathcal{R}}_{v}^{a,\,b}$ for $a,\,b\in S$
and $v\in\llbracket2,\,L-2\rrbracket$.
\begin{lem}
\label{l_flcyc4}For each $v\in\llbracket2,\,L-2\rrbracket$ and $a,\,b\in S$,
it holds that
\begin{equation}
\sum_{\sigma\in\overline{\mathcal{R}}_{v}^{a,b}}\phi_{1}(\sigma)=0\;.\label{e_flcyc4}
\end{equation}
Moreover, there exists a constant $C>0$ such that $|\phi_{1}(\sigma)|\le Ce^{-\Gamma\beta}$
for all $\sigma\in\overline{\mathcal{R}}_{v}^{a,\,b}$.
\end{lem}

\begin{proof}
Let us first assume that $v\neq2,\,L-2$. Recall that $\overline{\mathcal{R}}_{v}^{a,\,b}=\bigcup_{\sigma\in\mathcal{R}_{v}^{a,b}}\overline{\sigma}$.
By the definition of $\widetilde{h}$, it holds that the function
$\phi_{1}$ vanishes outside 
\[
\bigcup_{\ell\in\mathbb{T}_{L}}\bigcup_{k\in\mathbb{T}_{K}}\{\xi_{\ell,\,v;\,k,\,1}^{a,\,b,\,+},\,\xi_{\ell,\,v;\,k,\,1}^{a,\,b,\,-},\,\xi_{\ell,\,v-1;\,k,\,K-1}^{a,\,b,\,+},\,\xi_{\ell+1,\,v-1;\,k,\,K-1}^{a,\,b,\,-}\}\;.
\]
Thus, to prove \eqref{e_flcyc4}, it suffices to verify that for all
$\ell\in\mathbb{T}_{L}$ and $k\in\mathbb{T}_{K}$,
\begin{equation}
\phi_{1}(\xi_{\ell,\,v;\,k,\,1}^{a,\,b,\,+})+\phi_{1}(\xi_{\ell,\,v;\,k,\,1}^{a,\,b,\,-})+\phi_{1}(\xi_{\ell,\,v-1;\,k,\,K-1}^{a,\,b,\,+})+\phi_{1}(\xi_{\ell+1,\,v-1;\,k,\,K-1}^{a,\,b,\,-})=0\;.\label{e_flcyc4.2}
\end{equation}
By definition of $\widetilde{h}$ on the bulk typical configurations,
we have
\begin{align*}
\phi_{1}(\xi_{\ell,\,v;\,k,\,1}^{a,\,b,\,+})=\phi_{1}(\xi_{\ell,\,v;\,k,\,1}^{a,\,b,\,-}) & =\frac{e^{-\Gamma\beta}}{Z_{\beta}}\cdot\frac{\mathfrak{b}}{\kappa(K-2)(L-4)}\;,\\
\phi_{1}(\xi_{\ell,\,v-1;\,k,\,K-1}^{a,\,b,\,+})=\phi_{1}(\xi_{\ell+1,\,v-1;\,k,\,K-1}^{a,\,b,\,-}) & =-\frac{e^{-\Gamma\beta}}{Z_{\beta}}\cdot\frac{\mathfrak{b}}{\kappa(K-2)(L-4)}\;.
\end{align*}
Thus, \eqref{e_flcyc4.2} is a direct consequence of this computation.
Moreover, we can also verify the second statement from this computation
as well via Theorem \ref{t_mu} and \eqref{e_bdef}.

The proof of cases $v=2,\,L-2$ is a notational modification of that
of Lemma \ref{l_fl4}, and thus we omit the details.
\end{proof}
Now, we consider the function $\phi_{1}$ at the edge typical configurations.
\begin{lem}
\label{l_flcyc5} We have \textbf{$\phi_{1}(\sigma)=0$} for
\begin{enumerate}
\item $\sigma\in(\overline{\mathcal{Z}}^{a,\,b}\cup\overline{\mathcal{Z}}^{a,\,c}\cup\overline{\mathcal{Z}}^{c,\,b})\setminus\mathcal{N}(\mathcal{S})$
(cf. Notation \ref{n_abc}) and
\item $\sigma\in\overline{\mathcal{D}}^{s}\setminus\mathcal{N}(\mathbf{s})$
with $s\in S$.
\end{enumerate}
\end{lem}

\begin{proof}
Part (1) can be proved in the same way as we proved Lemma \ref{l_fl5}.
On the other hand, part (2) directly follows from the fact that $\widetilde{h}$
is defined as constant on $\overline{\mathcal{D}}^{s}$ (cf. Definition
\ref{d_testfcyc}-(1)).
\end{proof}
It remains to investigate $\phi_{1}$ on the sets $\mathcal{N}(\mathbf{s})$,
$\mathbf{s}\in\mathcal{S}$.
\begin{lem}
\label{l_flcyc6}For $\mathbf{a},\,\mathbf{b},\,\mathbf{c}$ chosen
according to Notation \ref{n_abc}, we have
\begin{align}
\sum_{\sigma\in\mathcal{N}(\mathbf{a})}\phi_{1}(\sigma) & =(1+o(1))\cdot\frac{|\mathcal{Q}|}{\kappa(|\mathcal{P}|+|\mathcal{Q}|)}e^{-\Gamma\beta}\;,\label{e_fc1}\\
\sum_{\sigma\in\mathcal{N}(\mathbf{b})}\phi_{1}(\sigma) & =-(1+o(1))\cdot\frac{|\mathcal{P}|}{\kappa(|\mathcal{P}|+|\mathcal{Q}|)}e^{-\Gamma\beta}\;,\label{e_fc2}\\
\sum_{\sigma\in\mathcal{N}(\mathbf{c})}\phi_{1}(\sigma) & =o(e^{-\Gamma\beta})\;.\label{e_fc3}
\end{align}
Moreover, there exists a positive constant $C$ so that for all $\sigma\in\mathcal{N}(\mathcal{S})$,
$|\phi_{1}(\sigma)|\le Ce^{-\Gamma\beta}$.
\end{lem}

\begin{proof}
The proof of this lemma is identical to that of Lemma \ref{l_fl7}
and we omit the detail. 
\end{proof}
Finally, we are ready to prove Proposition \ref{p_eqpappr} for the
cyclic dynamics. 
\begin{proof}[Proof of Proposition \ref{p_eqpappr} for the cyclic dynamics]
 By Remark \ref{r_eqpappr3cyc} and Proposition \ref{p_testfcyc},
if suffices to verify \eqref{e_eqpappr1}. We explained above that
\eqref{e_eqpappr1} follows if we can prove \eqref{e_WTScyc3}. By
Lemmas \ref{l_flcyc3} and \ref{l_flcyc5}, we have

\begin{align*}
\sum_{\sigma\in\widehat{\mathcal{N}}(\mathcal{S})}h_{\mathcal{P},\,\mathcal{Q}}(\sigma)\phi_{1}(\sigma) & =o(e^{-\Gamma\beta})+\sum_{s\in S}\sum_{\sigma\in\mathcal{N}(\mathbf{s})}h_{\mathcal{P},\,\mathcal{Q}}(\sigma)\phi_{1}(\sigma)\\
 & =o(e^{-\Gamma\beta})+\sum_{s\in S}\Big[h_{\mathcal{P},\,\mathcal{Q}}(\mathbf{s})\sum_{\sigma\in\mathcal{N}(\mathbf{s})}\phi_{1}(\sigma)\Big]\;,
\end{align*}
where the first identity follows from Lemma \ref{l_eqpotcyc} and
the last part of Lemma \ref{l_flcyc4}, and where the second identity
follows from Lemma \ref{l_eqpotcyc} and the last part of Lemma \ref{l_flcyc6}.
Now, inserting \eqref{e_fc1}, \eqref{e_fc2} and \eqref{e_fc3} completes
the proof.
\end{proof}

\appendix

\section{\label{secA}Decomposition Lemma}

Before proceeding to the proofs of results stated in Sections \ref{sec6}
and \ref{sec8}, we provide a decomposition of neighborhoods which
is crucially used in later discussions. Let $\mathcal{P}_{1}$, $\mathcal{P}_{2}$
and $\mathcal{Q}$ be pairwise disjoint subsets of $\mathcal{X}$.
Then, by the definition of neighborhoods, we immediately have

\begin{equation}
\widehat{\mathcal{N}}(\mathcal{P}_{1}\cup\mathcal{P}_{2};\,\mathcal{Q})=\widehat{\mathcal{N}}(\mathcal{P}_{1};\,\mathcal{Q})\cup\widehat{\mathcal{N}}(\mathcal{P}_{2};\,\mathcal{Q})\;.\label{e_set1}
\end{equation}

The following lemma is a refinement of this decomposition. Note that
the proof works for both the MH and cyclic dynamics.
\begin{lem}
\label{l_nhddc}Suppose that $\mathcal{P}_{1}$, $\mathcal{P}_{2}$
and $\mathcal{Q}$ are pairwise disjoint subsets of $\mathcal{X}$.
Then, it holds that 
\[
\widehat{\mathcal{N}}(\mathcal{P}_{1}\cup\mathcal{P}_{2};\,\mathcal{Q})=\widehat{\mathcal{N}}(\mathcal{P}_{1};\,\mathcal{P}_{2}\cup\mathcal{Q})\cup\widehat{\mathcal{N}}(\mathcal{P}_{2};\,\mathcal{P}_{1}\cup\mathcal{Q})\;.
\]
In particular, if $\mathcal{Q}=\emptyset$ then we have
\[
\widehat{\mathcal{N}}(\mathcal{P}_{1}\cup\mathcal{P}_{2})=\widehat{\mathcal{N}}(\mathcal{P}_{1};\,\mathcal{P}_{2})\cup\widehat{\mathcal{N}}(\mathcal{P}_{2};\,\mathcal{P}_{1})\;.
\]
\end{lem}

\begin{proof}
First, we consider the first identity. By \eqref{e_set1} and the
definition of restricted neighborhoods, we immediately have that
\begin{equation}
\widehat{\mathcal{N}}(\mathcal{P}_{1}\cup\mathcal{P}_{2};\,\mathcal{Q})\supseteq\widehat{\mathcal{N}}(\mathcal{P}_{1};\,\mathcal{P}_{2}\cup\mathcal{Q})\cup\widehat{\mathcal{N}}(\mathcal{P}_{2};\,\mathcal{P}_{1}\cup\mathcal{Q})\;.\label{e_set2}
\end{equation}
To prove the reversed inclusion, we assume to the contrary that there
exists $\sigma\in\mathcal{X}$ such that 
\begin{equation}
\sigma\in\widehat{\mathcal{N}}(\mathcal{P}_{1}\cup\mathcal{P}_{2};\,\mathcal{Q})\setminus\big[\widehat{\mathcal{N}}(\mathcal{P}_{1};\,\mathcal{P}_{2}\cup\mathcal{Q})\cup\widehat{\mathcal{N}}(\mathcal{P}_{2};\,\mathcal{P}_{1}\cup\mathcal{Q})\big]\;.\label{e_set3}
\end{equation}
By \eqref{e_set1}, we may assume without loss of generality that
\begin{equation}
\sigma\in\widehat{\mathcal{N}}(\mathcal{P}_{1};\,\mathcal{Q})\setminus\big[\widehat{\mathcal{N}}(\mathcal{P}_{1};\,\mathcal{P}_{2}\cup\mathcal{Q})\cup\widehat{\mathcal{N}}(\mathcal{P}_{2};\,\mathcal{P}_{1}\cup\mathcal{Q})\big]\;.\label{e_assm}
\end{equation}
Note that $\sigma\in\widehat{\mathcal{N}}(\mathcal{P}_{1};\,\mathcal{Q})$
implies automatically that $H(\sigma)\le\Gamma$. Then, the fact that
$\sigma\notin\widehat{\mathcal{N}}(\mathcal{P}_{1};\,\mathcal{P}_{2}\cup\mathcal{Q})\cup\widehat{\mathcal{N}}(\mathcal{P}_{2};\,\mathcal{P}_{1}\cup\mathcal{Q})$
implies that we readily have $\sigma\notin\mathcal{P}_{1}\cup\mathcal{P}_{2}$.

Now since $\sigma\in\widehat{\mathcal{N}}(\mathcal{P}_{1};\,\mathcal{Q})$,
we can find $(\omega_{n})_{n=0}^{N}$ in $\mathcal{X}\setminus\mathcal{Q}$
with $N\ge1$ such that $\omega_{0}\in\mathcal{P}_{1}$, $\omega_{N}=\sigma$,
$\omega_{n}\sim\omega_{n+1}$ and $H(\omega_{n},\,\omega_{n+1})\le\Gamma$.
Let us assume that $N$ is the smallest of all such sequences. If
$\omega_{n}\notin\mathcal{P}_{2}$ for all $n\in\llbracket1,\,N-1\rrbracket$,
then $(\omega_{n})_{n=0}^{N}\subseteq\mathcal{X}\setminus(\mathcal{P}_{2}\cup\mathcal{Q})$
so that we get a contradiction since we must have $\sigma\notin\widehat{\mathcal{N}}(\mathcal{P}_{1};\,\mathcal{P}_{2}\cup\mathcal{Q})$.
Therefore, we can find $n_{0}\in\llbracket1,\,N-1\rrbracket$ such
that $\omega_{n_{0}}\in\mathcal{P}_{2}$. Since we have $\omega_{n}\notin\mathcal{P}_{1}$
for all $n\in\llbracket1,\,N\rrbracket$ by the minimality of the
length $N$, we can notice that $(\omega_{n})_{n=n_{0}}^{N}\subseteq(\mathcal{P}_{1}\cup\mathcal{Q})^{c}$
and thus $\sigma\in\widehat{\mathcal{N}}(\mathcal{P}_{2};\,\mathcal{P}_{1}\cup\mathcal{Q})$.
Thus, we get a contradiction to \eqref{e_assm}. Hence, we have proved
the reversed inclusion of \eqref{e_set2} and the proof of the first
identity is completed.\medskip{}

Finally, the second identity of the lemma is immediate from Lemmas
\ref{l_Nhatbdry} and \ref{l_Nhatbdrycyc}.
\end{proof}

\section{\label{secB}Proof of Results in Section \ref{sec6}}

\subsection{\label{secB.1}Preliminaries}

First, we give in Proposition \ref{p_lowE} a characterization of
configurations which have energy less than $\Gamma$. For $a\in S$,
we write 
\begin{equation}
\|\eta\|_{a}=\sum_{x\in\Lambda}\mathbf{1}\{\eta(x)=a\}\;,\label{e_spinnum}
\end{equation}
which denotes the number of sites in $\eta$ with spin $a$. We recall
some notation from \cite[Section 2.1]{NZ}.

\begin{figure}
\includegraphics[width=11.5cm]{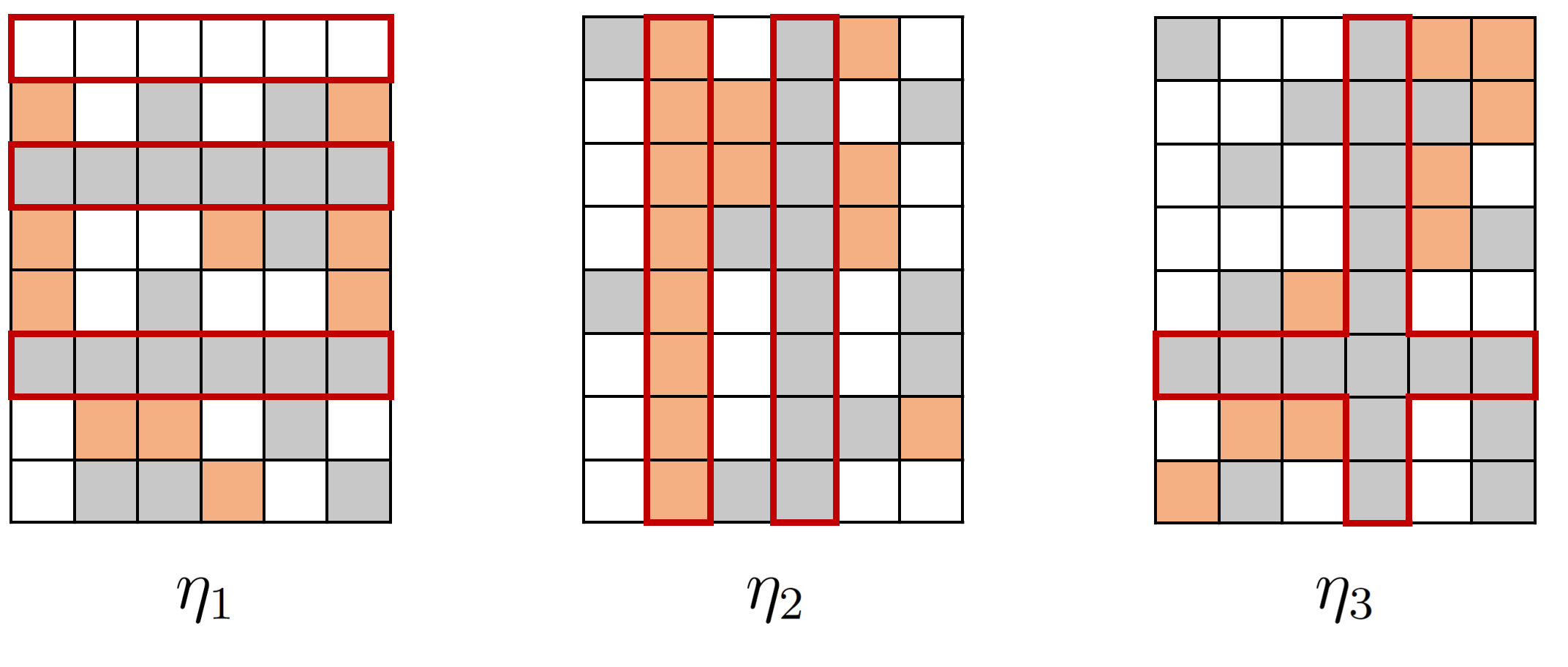}\caption{\label{figB.1}\textbf{Figures for Notation \ref{n_bridge}.} Configurations
$\eta_{1}$, $\eta_{2}$ and $\eta_{3}$ have horizontal bridges,
vertical bridges and a cross, respectively. Let $S=\{1,\,2,\,3\}$
and let white, gray and orange boxes denote sites with spin $1$,
$2$ and $3$, respectively. Then e.g., for configuration $\eta_{1}$,
it holds that $B_{1}(\eta_{1})=1$, $B_{2}(\eta_{1})=2$, $B_{3}(\eta_{1})=0$,
$\Delta H_{r_{1}}(\eta_{1})=5$ and $\Delta H_{c_{2}}(\eta_{1})=6$.
Note that one should be careful about the periodic boundary condition
when computing $\Delta H_{r_{1}}(\eta_{1})$ and $\Delta H_{c_{2}}(\eta_{1})$.}
\end{figure}

\begin{notation}
\label{n_bridge}We refer to Figure \ref{figB.1} for an illustration
of the notions introduced below.
\begin{itemize}
\item For a configuration $\eta\in\mathcal{X}$, a \textit{bridge}, which
is a \textit{horizontal }or \textit{vertical bridge}, is a row or
column, respectively, in which all spins are the same. If a bridge
consists of spin $a\in S$, we call this bridge an $a$-bridge. Then,
we denote by $B_{a}(\eta)$ the number of $a$-bridges with respect
to $\eta$.
\item In a configuration $\eta$, a \textit{cross} is the union of a horizontal
bridge and a vertical bridge. A cross consisting of spin $a\in S$
is called an $a$-cross. Moreover, $\eta$ is called \textit{cross-free}
if it does not have a cross.
\item We denote by $r_{1},\,\dots,\,r_{L}$ the rows and $c_{1},\,\dots,\,c_{K}$
the columns of $\Lambda=\mathbb{T}_{K}\times\mathbb{T}_{L}$, starting
from the lowest and the leftmost one, respectively. For $v\in\llbracket1,\,L\rrbracket$,
$h\in\llbracket1,\,K\rrbracket$ and $\eta\in\mathcal{X}$, we define
\[
\Delta H_{r_{v}}(\eta)=\sum_{\{x,\,y\}\subset r_{v}:\,x\sim y}\mathbf{1}\{\eta(x)\ne\eta(y)\}\;\;\;\;\text{and}\;\;\;\;\Delta H_{c_{h}}(\eta)=\sum_{\{x,\,y\}\subset c_{h}:\,x\sim y}\mathbf{1}\{\eta(x)\ne\eta(y)\}
\]
so that 
\begin{equation}
H(\eta)=\sum_{v\in\llbracket1,\,L\rrbracket}\Delta H_{r_{v}}(\eta)+\sum_{h\in\llbracket1,\,K\rrbracket}\Delta H_{c_{h}}(\eta)\;.\label{e_decH}
\end{equation}
An edge belonging to a row (resp. column) is called a horizontal (resp.
vertical) edge.
\end{itemize}
\end{notation}

The following lower bound for the Hamiltonian is useful.
\begin{lem}
\label{l_Hlb}It holds that 
\[
H(\eta)\ge2\Big[K+L-\sum_{a\in S}B_{a}(\eta)\Big]\;.
\]
\end{lem}

\begin{proof}
It follows directly from \eqref{e_decH} since $\Delta H_{r_{v}}(\eta)\ge2$
(resp. $\Delta H_{c_{h}}(\eta)\ge2)$ if $r_{v}$ (resp. $c_{h}$)
is not a bridge.
\end{proof}
Now, we classify the configurations with low energy.
\begin{prop}
\label{p_lowE}Suppose that $\eta\in\mathcal{X}$ satisfies $H(\eta)<\Gamma$.
Then, $\eta$ satisfies exactly one of the following types.
\begin{itemize}
\item \textbf{\textup{(L1)}} There exist $a,\,b\in S$ and $v\in\llbracket2,\,L-2\rrbracket$
such that $\eta\in\mathcal{R}_{v}^{a,\,b}$. Here, $\mathcal{N}(\eta)$
is a singleton, i.e., $\mathcal{N}(\eta)=\{\eta\}$.
\item \textbf{\textup{(L2)}} There exist $a,\,b\in S$ such that $\eta\in\mathcal{R}_{1}^{a,\,b}$.
In this case, $\mathcal{N}(\eta)=\mathcal{N}(\mathbf{a})$.
\item \textbf{\textup{(L3)}} For some $a\in S$, $\eta$ has an $a$-cross.
Then, $\mathcal{N}(\eta)=\mathcal{N}(\mathbf{a})$ and
\begin{equation}
\sum_{b\ne a}\Vert\eta\Vert_{b}\le\frac{H(\eta)^{2}}{16}\le\frac{(2K+1)^{2}}{16}\;.\label{e_lowE}
\end{equation}
\end{itemize}
\end{prop}

\begin{proof}
Fix $\eta\in\mathcal{X}$ with $H(\eta)<\Gamma=2K+2$. It is obvious
that no configuration can be of more than one type. Thus, we first
show that $\eta$ falls into exactly one among three categories, and
then we will prove that configurations of each category satisfy each
succeeding statement in the list.

By Lemma \ref{l_Hlb}, $\eta$ has at least $L$ bridges. We take
one of them and let this be an $a$-bridge for some $a\in S$. Now,
we consider three cases separately.\medskip{}

\noindent \textbf{(Case 1: $\eta$ has a horizontal $a$-bridge without
a vertical $a$-bridge)} Since $\Delta H_{c_{h}}(\eta)\ge2$ for all
$h\in\llbracket1,\,K\rrbracket$ and since $H(\eta)\le\Gamma-1=2K+1$,
we get from \eqref{e_decH} that 
\[
\sum_{v\in\llbracket1,\,L\rrbracket}\Delta H_{r_{v}}(\eta)\le1\;.
\]
Since we cannot have $\Delta H_{r_{v}}(\eta)=1$, we must have $\Delta H_{r_{v}}(\eta)=0$
for all $v\in\llbracket1,\,L\rrbracket$. This implies that $\eta\in\mathcal{R}_{v}^{a,\,b}$
for some $b\ne a$ and $v\in\llbracket0,\,L-1\rrbracket$ ($v=L$
is excluded because $\eta$ must have an $a$-bridge). We consider
three sub-cases separately.
\begin{itemize}
\item If $v\in\llbracket2,\,L-2\rrbracket$, then $\eta$ satisfies \textbf{(L1)}.
\item If $v=0$, then $\eta=\mathbf{a}$ and thus $\eta$ clearly has an
$a$-cross, so that $\eta$ satisfies \textbf{(L3)}.
\item If $v=1$ or $v=L-1$, then $\eta$ satisfies \textbf{(L2)} (if $v=L-1$
then $\eta\in\mathcal{R}_{L-1}^{a,\,b}=\mathcal{R}_{1}^{b,\,a}$).
\end{itemize}
\medskip{}
\textbf{(Case 2: $\eta$ has a vertical $a$-bridge without a horizontal
$a$-bridge)} Since $\Delta H_{r_{v}}(\eta)\ge2$ for all $v\in\llbracket1,\,L\rrbracket$,
we get from \eqref{e_decH} that $2K+2>2L$ and hence we must have
$K=L$. The rest of the proof is now identical to \textbf{(Case 1)};
it suffices to switch the role of columns and rows.\medskip{}

\noindent \textbf{(Case 3: $\eta$ has an $a$-cross)} Here, $\eta$
readily satisfies \textbf{(L3)}.\medskip{}

Now, we verify the conditions of each category. If $\eta$ satisfies
\textbf{(L1)}, then it is clear that $\mathcal{N}(\eta)$ is a singleton
since any configuration obtained from $\eta$ by flipping a spin has
energy greater than or equal to $\Gamma$. If $\eta$ satisfies \textbf{(L2)},
then a part of a canonical path connecting $\mathbf{a}$ and $\eta$
is a $(\Gamma-2)$-path, and thus $\eta\in\mathcal{N}(\mathbf{a})$.

Finally, suppose that $\eta$ satisfies \textbf{(L3)}. For this case,
without loss of generality we may assume that $\mathbb{T}_{K}\times\{1\}$
and $\{1\}\times\mathbb{T}_{L}$ form the $a$-cross. Then, we update
each spin to $a$ in $\llbracket2,\,K\rrbracket\times\llbracket2,\,L\rrbracket$
in the ascending lexicographic order. The presence of $a$-bridges
assures us that the Hamiltonian does not increase during the entire
procedure. In the end we reach $\mathbf{a}$, and the maximal energy
along the path is attained at the initial configuration, which is
$H(\eta)<\Gamma$. Hence, we can conclude that $\eta\in\mathcal{N}(\mathbf{a})$.

Next, we verify the inequality \eqref{e_lowE}. Define $\widetilde{\eta}\in\mathcal{X}$
as the configuration obtained from $\eta$ by replacing all non-$a$
spins by some fixed $b_{0}\in S\setminus\{a\}$, i.e., for $x\in\Lambda$,
\[
\widetilde{\eta}(x)=\begin{cases}
a & \text{if }\eta(x)=a\;,\\
b_{0} & \text{if }\eta(x)\ne a\;.
\end{cases}
\]
Then, it is straightforward from \eqref{e_decH} that $H(\widetilde{\eta})\le H(\eta)$.
Moreover, as $\widetilde{\eta}$ also has an $a$-cross, we may regard
clusters of spins $b_{0}$ in $\widetilde{\eta}$ as an object in
$\llbracket2,\,K\rrbracket\times\llbracket2,\,L\rrbracket\subset\mathbb{Z}^{2}$.
Thus, we can apply the well-known isoperimetric lemma \cite[Corollary 2.5]{AC}
to deduce
\[
\|\widetilde{\eta}\|_{b_{0}}\le\frac{H(\widetilde{\eta})^{2}}{16}\le\frac{H(\eta)^{2}}{16}\le\frac{(2K+1)^{2}}{16}\;,
\]
where the last inequality follows from the assumption that $H(\eta)\le\Gamma-1=2K+1$.
We also remark that the equality of the first inequality holds when
the only cluster of spin $b_{0}$ in $\widetilde{\eta}$ forms a square.
Since $\sum_{b\ne a}\Vert\eta\Vert_{b}=\|\widetilde{\eta}\|_{b_{0}}$,
\eqref{e_lowE} is now proved and we conclude the proof.
\end{proof}
To conclude this subsection, we record a lemma regarding depths of
valleys other than the metastable ones. This lemma is not needed in
the current study since the metastable configurations are already
identified in \cite{NZ}; nevertheless, we choose to include it here
since it is essential in the characterization of the three-dimensional
metastable valleys in our companion article \cite{KS 3D}.
\begin{lem}
\label{l_depth}Let $\eta\in\mathcal{X}$ and $a\in S$. For any standard
sequence $(A_{m})_{m=0}^{KL}$ of sets connecting $\emptyset$ and
$\Lambda$ and for $m\in\llbracket0,\,KL\rrbracket$, we define $\omega_{m}\in\mathcal{X}$
in such a manner that 
\[
\omega_{m}(x)=\begin{cases}
a & \text{if }x\in A_{m}\;,\\
\eta(x) & \text{if }x\in\Lambda\setminus A_{m}\;.
\end{cases}
\]
Then, we have that $H(\omega_{m})\le H(\eta)+\Gamma$ for all $m\in\llbracket0,\,KL\rrbracket$.
\end{lem}

\begin{proof}
As $\omega_{KL}=\mathbf{a}$, the lemma is immediate for $m=KL$.
Thus, we prove this for $m\in\llbracket0,\,KL-1\rrbracket$.

Suppose that $m\in\llbracket K\ell,\,K(\ell+1)-1\rrbracket$ for some
$\ell\in\llbracket0,\,L-1\rrbracket$ and write $\xi=\omega_{m}$.
Suppose that $A_{K\ell}=\mathbb{T}_{K}\times P_{\ell}$, $A_{K(\ell+1)}=\mathbb{T}_{K}\times P_{\ell+1}$
and $P_{\ell+1}\setminus P_{\ell}=\{\ell_{0}\}$. By \eqref{e_decH},
we can write 
\begin{equation}
H(\xi)-H(\eta)=\sum_{v\in\llbracket1,\,L\rrbracket}[\Delta H_{r_{v}}(\xi)-\Delta H_{r_{v}}(\eta)]+\sum_{h\in\llbracket1,\,K\rrbracket}[\Delta H_{c_{h}}(\xi)-\Delta H_{c_{h}}(\eta)]\;.\label{e_depth1}
\end{equation}
We start by considering the first summation at the right-hand side.
For $v\in P_{\ell}$, we have $\Delta H_{r_{v}}(\xi)=0$ and thus
the summand is non-positive. For $v\in\mathcal{X}\setminus P_{\ell+1}$,
the configurations $\eta$ and $\xi$ have the same $v$-th row and
thus the summand is $0$. Finally, for $v=\ell_{0}$, observe that
$\xi$ is obtained from $\eta$ by changing spins at consecutive sites
to $b$, and thus we can readily conclude that the energy has been
increased by at most $2$. Summing up, we obtain
\begin{equation}
\sum_{v\in\llbracket1,\,L\rrbracket}[\Delta H_{r_{v}}(\xi)-\Delta H_{r_{v}}(\eta)]\le2\;.\label{e_depth2}
\end{equation}
By the same reasoning with the case of $v=\ell_{0}$ above, we can
also observe that each summand of the second summation of \eqref{e_depth1}
is at most $2$ and hence 
\begin{equation}
\sum_{h\in\llbracket1,\,K\rrbracket}[\Delta H_{c_{h}}(\xi)-\Delta H_{c_{h}}(\eta)]\le2K\;.\label{e_depth3}
\end{equation}
By \eqref{e_depth1}, \eqref{e_depth2} and \eqref{e_depth3}, we
get $H(\xi)-H(\eta)\le2K+2=\Gamma$, and we have proved the lemma.
\end{proof}

\subsection{\label{secB.2}Typical configurations}

In this subsection, we prove Proposition \ref{p_typprop}. We first
give two lemmas.
\begin{lem}
\label{l_typ1}For $a,\,b\in S$, suppose that $\eta_{1}\in\mathcal{B}^{a,\,b}$
and $\eta_{2}\in\mathcal{X}$ satisfy $\eta_{1}\sim\eta_{2}$ and
$H(\eta_{2})\le\Gamma$. Then, the following statements hold.
\begin{enumerate}
\item If $\eta_{1}\in\mathcal{R}_{v}^{a,\,b}$ for $v\in\llbracket2,\,L-2\rrbracket$,
then $\eta_{2}\in\mathcal{Q}_{v-1}^{a,\,b}\cup\mathcal{Q}_{v}^{a,\,b}$.
In particular, if $v\in\llbracket3,\,L-3\rrbracket$ then $\eta_{2}\in\mathcal{B}_{\Gamma}^{a,\,b}$.
\item If $\eta_{1}\in\mathcal{B}_{\Gamma}^{a,\,b}$, then $\eta_{2}\in\mathcal{B}^{a,\,b}$.
\end{enumerate}
In particular, it necessarily holds that $\eta_{2}\in\mathcal{C}^{a,\,b}$.
\end{lem}

\begin{proof}
We first assume that $K<L$ and consider two cases separately.\medskip{}

\noindent \textbf{(Case 1: $\eta_{1}=\xi_{\ell,\,v}^{a,\,b}$ for
some $\ell\in\mathbb{T}_{L}$ and $v\in\llbracket2,\,L-2\rrbracket$)}
We can observe from the illustration given in Figure \ref{fig6.1}
that the only way of flipping a spin of $\eta_{1}$ in such a way
that the resulting configuration has energy at most $\Gamma$ is either
to attach a protuberance of spin $b$ to the cluster of spin $b$
of $\eta_{1}$ or to attach a protuberance of spin $a$ to the cluster
of spin $a$ of $\eta_{1}$. This implies that $\eta_{2}$ must be
one of the following forms: 
\[
\xi_{\ell,\,v;\,k,\,1}^{a,\,b,\,\pm}\;,\;\xi_{\ell,\,v-1;\,k,\,K-1}^{a,\,b,\,+}\;,\;\xi_{\ell+1,\,v-1;\,k,\,K-1}^{a,\,b,\,-}\;\;\;\;;\;k\in\mathbb{T}_{K}\;.
\]
Hence, $\eta_{2}\in\mathcal{Q}_{v-1}^{a,\,b}\cup\mathcal{Q}_{v}^{a,\,b}$.
This proves the first assertion of part (1). The second assertion
is straightforward.\medskip{}

\noindent \textbf{(Case 2: $\eta_{1}=\xi_{\ell,\,v;\,k,\,h}^{a,\,b,\,\pm}$
for some $\ell\in\mathbb{T}_{L}$, $v\in\llbracket2,\,L-3\rrbracket$,
$k\in\mathbb{T}_{K}$ and $h\in\llbracket1,\,K-1\rrbracket$)} In
this case, we can also observe from the illustration given in Figure
\ref{fig6.1} that the only way of flipping a spin of $\eta_{1}$
without increasing the energy is to expand or shrink the protuberance
of spin $b$ attached at $\xi_{\ell,\,v}^{a,\,b}$, and therefore
$\eta_{2}=\xi_{\ell,\,v;\,k,\,h-1}^{a,\,b,\,\pm}$, $\xi_{\ell,\,v;\,k,\,h+1}^{a,\,b,\,\pm}$,
$\xi_{\ell,\,v;\,k+1,\,h-1}^{a,\,b,\,\pm}$ or $\xi_{\ell,\,v;\,k-1,\,h+1}^{a,\,b,\,\pm}$.
This proves that $\eta_{2}\in\mathcal{B}^{a,\,b}$ and hence part
(2) is now verified. Thus, we conclude the case of $K<L$.\medskip{}

The case of $K=L$ can be treated in the same manner. We just have
to consider also the configurations transposed by the operator $\Theta$.
\end{proof}
The previous lemma implies the following result.
\begin{lem}
\label{l_typ2}It holds that $\widehat{\mathcal{N}}(\mathcal{B};\,\mathcal{C}\setminus\mathcal{B})=\mathcal{B}$.
\end{lem}

\begin{proof}
Since the Hamiltonian of configurations belonging to $\mathcal{B}$
does not exceed $\Gamma$, it is immediate that 
\[
\widehat{\mathcal{N}}(\mathcal{B};\,\mathcal{C}\setminus\mathcal{B})\supseteq\mathcal{B}\;.
\]
Thus, it suffices to show the opposite inclusion. Since
\[
\widehat{\mathcal{N}}(\mathcal{B};\,\mathcal{C}\setminus\mathcal{B})=\bigcup_{a,\,b\in S}\widehat{\mathcal{N}}(\mathcal{B}^{a,\,b};\,\mathcal{C}\setminus\mathcal{B})\subseteq\bigcup_{a,\,b\in S}\widehat{\mathcal{N}}(\mathcal{B}^{a,\,b};\,\mathcal{C}^{a,\,b}\setminus\mathcal{B}^{a,\,b})\;,
\]
it suffices to show that for $a,\,b\in S$, 
\begin{equation}
\widehat{\mathcal{N}}(\mathcal{B}^{a,\,b};\,\mathcal{C}^{a,\,b}\setminus\mathcal{B}^{a,\,b})\subseteq\mathcal{B}^{a,\,b}\;.\label{e_typ2.1}
\end{equation}
Take $\eta\in\widehat{\mathcal{N}}(\mathcal{B}^{a,\,b};\,\mathcal{C}^{a,\,b}\setminus\mathcal{B}^{a,\,b})$.
Then, we have a $\Gamma$-path $(\omega_{n})_{n=0}^{N}$ in $\mathcal{X}\setminus(\mathcal{C}^{a,\,b}\setminus\mathcal{B}^{a,\,b})=(\mathcal{X}\setminus\mathcal{C}^{a,\,b})\cup\mathcal{B}^{a,\,b}$
from $\mathcal{B}^{a,\,b}$ to $\eta$. Suppose the contrary that
$\eta\notin\mathcal{B}^{a,\,b}$. Then, as $\omega_{0}\in\mathcal{B}^{a,\,b}$,
there exists $n_{0}\in\llbracket0,\,N-1\rrbracket$ such that $\omega_{n_{0}}\in\mathcal{B}^{a,\,b}$
and $\omega_{n_{0}+1}\in\mathcal{X}\setminus\mathcal{B}^{a,\,b}$.
Since $(\omega_{n})_{n=0}^{N}$ is a path in $(\mathcal{X}\setminus\mathcal{C}^{a,\,b})\cup\mathcal{B}^{a,\,b}$,
we get 
\[
\omega_{n_{0}+1}\in(\mathcal{X}\setminus\mathcal{B}^{a,\,b})\cap\big[(\mathcal{X}\setminus\mathcal{C}^{a,\,b})\cup\mathcal{B}^{a,\,b}\big]=\mathcal{X}\setminus\mathcal{C}^{a,\,b}\;.
\]
On the other hand, since $\omega_{n_{0}}\in\mathcal{B}^{a,\,b}$ we
must have $\omega_{n_{0}+1}\in\mathcal{C}^{a,\,b}$ by Lemma \ref{l_typ1}
and thus we get a contradiction. Therefore, we must have $\eta\in\mathcal{B}^{a,\,b}$
and thus we get \eqref{e_typ2.1}.
\end{proof}
Now, we are ready to present a proof of Proposition \ref{p_typprop}.
\begin{proof}[Proof of Proposition \ref{p_typprop}]
 We first prove that

\begin{equation}
\{\mathbf{b}\}\cap\widehat{\mathcal{N}}(\mathbf{a};\,\mathcal{B}_{\Gamma})=\emptyset\;.\label{e_nonint}
\end{equation}
Let us suppose the contrary that $(\omega_{n})_{n=0}^{N}$ is a $\Gamma$-path
from $\mathbf{a}$ to $\mathbf{b}$ in $\mathcal{X}\setminus\mathcal{B}_{\Gamma}$.
We now follow the idea of \cite[Proposition 2.6]{NZ}. Define $u:\llbracket0,\,N\rrbracket\rightarrow\mathbb{R}$
as 
\[
u(n)=B_{b}(\omega_{n})\;\;\;\;;\;n\in\llbracket0,\,N\rrbracket\;,
\]
where $B_{b}(\cdot)$ is defined in Notation \ref{n_bridge}. Then,
we have that 
\begin{equation}
u(0)=0\;,\;u(N)=K+L\text{ and }|u(n+1)-u(n)|\le2\text{ for all }n\in\llbracket0,\,N-1\rrbracket\;.\label{e_nonint1}
\end{equation}
Thus, the following instant $n^{*}$ is well defined:
\begin{equation}
n^{*}=\min\{n\in\llbracket0,\,N-1\rrbracket:u(n),\;u(n+1)\ge2\}\;.\label{e_typprop1.1}
\end{equation}
Note that, since we need to change at least $2K-1$ spins from $\mathbf{a}$
to get $u(t)\ge2,$ we clearly have $n^{*}\ge2K-1$. By \eqref{e_nonint1},
we have $B_{b}(\omega_{n^{*}})=2$ or $3$. We divide the proof into
three cases as in Proposition \ref{p_lowE}.\medskip{}

\noindent \textbf{(Case 1: $\omega_{n^{*}}$ has a horizontal $b$-bridge
without a vertical $b$-bridge) }For this case, if $B_{b}(\omega_{n^{*}})=3$,
we have $B_{b}(\omega_{n^{*}-1})\ge2$ and thus we get a contradiction
to the minimality of $n^{*}$. Thus, we must have $B_{b}(\omega_{n^{*}})=2$.

Since $\omega_{n^{*}}$ cannot have any vertical bridges, we get $\Delta H_{c_{h}}(\omega_{n^{*}})\ge2$
for all $h\in\llbracket1,\,K\rrbracket$. In view of \eqref{e_decH}
and the fact that $H(\omega_{n^{*}})\le\Gamma=2K+2$, we can readily
conclude that the only possibility is $\omega_{n^{*}}\in\mathcal{R}_{2}^{a_{0},\,b}$
or $\mathcal{Q}_{2}^{a_{0},\,b}$ for some $a_{0}\ne b$. The latter
case yields an immediate contradiction since $\mathcal{Q}_{2}^{a_{0},\,b}\subseteq\mathcal{B}_{\Gamma}^{a_{0},\,b}\subseteq\mathcal{B}_{\Gamma}$.
Summing up, we must have $\omega_{n^{*}}\in\mathcal{R}_{2}^{a_{0},\,b}$
for some $a_{0}\ne b$. Since $H(\omega_{n^{*}+1})\le\Gamma$, the
only possibility to get $\omega_{n^{*}+1}$ from $\omega_{n^{*}}$
is either to change a spin $a_{0}$ neighboring the cluster of spin
$b$ to $b$, or to change a spin $b$ neighboring the cluster of
spin $a_{0}$ to $a_{0}$. Both cases yield a contradiction, since
in the former we get $\omega_{n^{*}+1}\in\mathcal{B}_{\Gamma}^{a,\,b}\subseteq\mathcal{B}_{\Gamma}$
while in the latter $u(n^{*}+1)=1$.\medskip{}

\noindent \textbf{(Case 2: $\omega_{n^{*}}$ has a vertical $b$-bridge
without a horizontal $b$-bridge)} This case is similar to \textbf{(Case
1)} and we omit the detail.\medskip{}

\noindent \textbf{(Case 3: $\omega_{n^{*}}$ has a $b$-cross)} In
this case, $\omega_{n^{*}}$ does not have a bridge whose spin is
not $b$ and thus, by \eqref{e_nonint1}, the configuration $\omega_{n^{*}}$
has at most $3$ bridges. Therefore, by Lemma \ref{l_Hlb}, it holds
that 
\[
H(\omega_{n^{*}})\ge2(K+L-3)>\Gamma\;.
\]
This contradicts the assumption that $(\omega_{n})_{n=0}^{N}$ is
a $\Gamma$-path. Therefore, we conclude the proof of \eqref{e_nonint}.\medskip{}

\noindent (1) This is direct from \eqref{e_nonint}, since otherwise
we can construct a $\Gamma$-path from $\mathbf{a}$ to $\mathbf{b}$
avoiding $\mathcal{B}_{\Gamma}$ by reversing and concatenating.\medskip{}

\noindent (2) First, we have $\mathcal{B}^{a,\,b}\supseteq\mathcal{R}_{2}^{a,\,b}$
from the definition of the set $\mathcal{B}^{a,\,b}$. Moreover, since
a canonical path connecting $\mathcal{R}_{2}^{a,\,b}$ and $\mathbf{a}$
is a $\Gamma$-path in $\mathcal{X}\setminus\mathcal{B}_{\Gamma}$,
we also have $\mathcal{E}^{a}\supseteq\mathcal{R}_{2}^{a,\,b}$. This
concludes that 
\begin{equation}
\mathcal{E}^{a}\cap\mathcal{B}^{a,\,b}\supseteq\mathcal{R}_{2}^{a,\,b}\;.\label{e_int1}
\end{equation}
Now, we claim that the reversed inclusion also holds. To this end,
we start by observing that, since $\mathcal{B}_{\Gamma}^{a,\,b}$
and $\mathcal{E}^{a}$ are disjoint from the definition \eqref{e_Ea},
\[
\mathcal{E}^{a}\cap\mathcal{B}^{a,\,b}\subseteq\mathcal{B}^{a,\,b}\setminus\mathcal{B}_{\Gamma}^{a,\,b}=\bigcup_{v\in\llbracket2,\,L-2\rrbracket}\mathcal{R}_{v}^{a,\,b}\;.
\]
For $\eta\in\mathcal{R}_{v}^{a,\,b}$ with $v\in\llbracket3,\,L-3\rrbracket$,
we cannot have a $\Gamma$-path in $\mathcal{X}\setminus\mathcal{B}_{\Gamma}$
from $\mathbf{a}$ to $\eta$ by part (1) of Lemma \ref{l_typ1}.
Thus, such an $\eta$ cannot belong to $\mathcal{E}^{a}$ and we deduce
that 
\begin{equation}
\mathcal{E}^{a}\cap\mathcal{B}^{a,\,b}\subseteq\mathcal{R}_{2}^{a,\,b}\cup\mathcal{R}_{L-2}^{a,\,b}\;.\label{e_int2}
\end{equation}
Note that we have $\mathcal{E}^{b}\supseteq\mathcal{R}_{L-2}^{a,\,b}$
by the same reason with $\mathcal{E}^{a}\supseteq\mathcal{R}_{2}^{a,\,b}$
that we proved before; hence, any configuration $\eta\in\mathcal{R}_{L-2}^{a,\,b}$
cannot belong to $\mathcal{E}^{a}$ by part (1). This observation
and \eqref{e_int2} implies that $\mathcal{E}^{a}\cap\mathcal{B}^{a,\,b}\subseteq\mathcal{R}_{2}^{a,\,b}$.
This along with \eqref{e_int1} completes the proof of part (2). \medskip{}

\noindent (3) We can deduce $\mathcal{E}^{a}\cap\mathcal{B}^{b,\,c}\subseteq\mathcal{R}_{2}^{b,\,c}\cup\mathcal{R}_{L-2}^{b,\,c}$
by the the same logic as we obtained \eqref{e_int2}. Then, since
$\mathcal{R}_{2}^{b,\,c}\subseteq\mathcal{E}^{b}$ and $\mathcal{R}_{L-2}^{b,\,c}\subseteq\mathcal{E}^{c}$,
part (1) implies that $\mathcal{E}^{a}\cap\mathcal{B}^{b,\,c}=\emptyset$.

\medskip{}

\noindent (4) The inclusion $\mathcal{E}\subseteq\widehat{\mathcal{N}}(\mathcal{S})$
is immediate from the definition of $\mathcal{E}$, and the inclusion
$\mathcal{B}\subseteq\widehat{\mathcal{N}}(\mathcal{S})$ is direct
from the fact that any configuration in $\mathcal{B}^{a,\,b}$ is
obtained starting from $\mathbf{a}$ via a canonical path which is
a $\Gamma$-path (cf. Lemma \ref{l_canpath}-(3)). Thus, we get
\begin{equation}
\mathcal{E}\cup\mathcal{B}\subseteq\widehat{\mathcal{N}}(\mathcal{S})\;.\label{e_uni1}
\end{equation}
On the other hand, by Lemma \ref{l_nhddc} with $\mathcal{P}_{1}=\mathcal{C}\setminus\mathcal{B}$,
$\mathcal{P}_{2}=\mathcal{B}$ and $\mathcal{Q}=\emptyset$ and Lemma
\ref{l_hatNPempty}, we can write 
\begin{align}
\widehat{\mathcal{N}}(\mathcal{C}) & =\widehat{\mathcal{N}}(\mathcal{B};\,\mathcal{C}\setminus\mathcal{B})\cup\widehat{\mathcal{N}}(\mathcal{C}\setminus\mathcal{B};\,\mathcal{B})\;.\label{e_uni2}
\end{align}
First, by Lemma \ref{l_typ2}, 
\begin{equation}
\widehat{\mathcal{N}}(\mathcal{B};\,\mathcal{C}\setminus\mathcal{B})=\mathcal{B}\;.\label{e_uni3}
\end{equation}
On the other hand, since any configuration in $\mathcal{C}\setminus\mathcal{B}$
is obtained starting from a ground state by a part of a canonical
path which is a $\Gamma$-path in $\mathcal{X}\setminus\mathcal{B}_{\Gamma}$,
we get 
\begin{equation}
\widehat{\mathcal{N}}(\mathcal{C}\setminus\mathcal{B};\,\mathcal{B})\subseteq\widehat{\mathcal{N}}(\mathcal{C}\setminus\mathcal{B};\,\mathcal{B}_{\Gamma})=\widehat{\mathcal{N}}(\mathcal{S};\,\mathcal{B}_{\Gamma})=\mathcal{E}\;.\label{e_uni4}
\end{equation}
By combining \eqref{e_uni2}, \eqref{e_uni3} and \eqref{e_uni4},
we get $\widehat{\mathcal{N}}(\mathcal{C})\subseteq\mathcal{B}\cup\mathcal{E}$.
Since $\mathcal{S}\subset\mathcal{C}$, the opposite inclusion of
\eqref{e_uni1} holds, and the proof is completed.
\end{proof}

\subsection{\label{secB.3}Edge typical configurations}

Here, we prove Propositions \ref{p_Zabtree} and \ref{p_Eadecomp}.
\begin{proof}[Proof of Proposition \ref{p_Zabtree}]
 To begin with, we consider the case of $K<L$. First, we prove the
only if part. Fix $\eta\in\mathcal{Z}^{a,\,b}$ and take a path $(\omega_{n})_{n=0}^{N}$
in $\mathcal{X}\setminus\mathcal{B}_{\Gamma}$ such that $\omega_{0}=\xi_{\ell,\,2}^{a,\,b}$
for some $\ell\in\mathbb{T}_{L}$, $\omega_{N}=\eta$ and $H(\omega_{n})=\Gamma$
for all $n\in\llbracket1,\,N\rrbracket$. It suffices to prove that
$\omega_{n}$ satisfies all requirements \textbf{{[}Z1{]}}, \textbf{{[}Z2{]}}
and \textbf{{[}Z3{]}} for all $n\in\llbracket1,\,N\rrbracket$. We
prove this by induction on $n$. First, consider the case $n=1$.
Then, since $\omega_{0}=\xi_{\ell,\,2}^{a,\,b}$ and $\omega_{1}\notin\mathcal{B}_{\Gamma}$,
the proof of Lemma \ref{l_typ1} implies that
\[
\omega_{1}\in\bigcup_{k\in\mathbb{T}_{K}}\{\xi_{\ell,\,1;\,k,\,K-1}^{a,\,b,\,+},\,\xi_{\ell+1,\,1;\,k,\,K-1}^{a,\,b,\,-}\}\;.
\]
In any case \textbf{{[}Z1{]}} and \textbf{{[}Z2{]}} are obvious, and
since $G_{a}(\eta)$ is a singleton subset of $\mathbb{T}_{K}\times\{\ell,\,\ell+1\}$,
\textbf{{[}Z3{]}} also holds. These observations conclude the first
step of $n=1$.

Now, suppose that the three conditions hold for $\omega_{n}$, and
then we consider $\omega_{n+1}$. Note that $H(\omega_{n})=H(\omega_{n+1})=\Gamma$.
We denote by $z\in\Lambda$ the site where the spin update happens
from $\omega_{n}$ to $\omega_{n+1}$, which does not change the energy.
By simple inspection, any spin update outside $\mathbb{T}_{K}\times\{\ell,\,\ell+1\}$
strictly increases the energy, and thus $z\in\mathbb{T}_{K}\times\{\ell,\,\ell+1\}$.
Moreover, any spin update to a third spin (other than $a$ and $b$)
strictly increases the energy, since the fact that $G_{a}(\omega_{n})$
is a sub-tree in $\mathbb{T}_{K}\times\{\ell,\,\ell+1\}$ implies
that there is no isolated spin. Therefore, the spin update at site
$z$ must be $b\to a$ or $a\to b$ inside $\mathbb{T}_{K}\times\{\ell,\,\ell+1\}$.
This observation verifies both \textbf{{[}Z1{]}} and \textbf{{[}Z2{]}}
for the new configuration $\omega_{n+1}$.

To check \textbf{{[}Z3{]}} for $\omega_{n+1}$, we divide into two
cases.
\begin{itemize}
\item If the update at $z$ is $b\to a$, then $z\notin G_{a}(\omega_{n})$
and $G_{a}(\omega_{n+1})=G_{a}(\omega_{n})\cup\{z\}$. Since $H(\omega_{n})=H(\omega_{n+1})$,
$z$ must have exactly two neighboring sites with spin $a$ and two
neighboring sites with spin $b$. Since $z$ has exactly one neighbor
outside $\mathbb{T}_{K}\times\{\ell,\,\ell+1\}$ which has spin $a$
by condition \textbf{{[}Z1{]}}, $z$ must have exactly one neighbor
with spin $a$ inside $\mathbb{T}_{K}\times\{\ell,\,\ell+1\}$. This
implies that $G_{a}(\omega_{n+1})=G_{a}(\omega_{n})\cup\{z\}$ is
still a tree.
\item If the update at $z$ is $a\to b$, then $z\in G_{a}(\omega_{n})$
and $G_{a}(\omega_{n+1})=G_{a}(\omega_{n})\setminus\{z\}$. The same
logic as above indicates that $z$ has exactly one neighbor with spin
$a$ inside $\mathbb{T}_{K}\times\{\ell,\,\ell+1\}$. This is equivalent
to saying that $z$ is an external vertex of the tree $G_{a}(\omega_{n})$
and that $\{z\}\subset G_{a}(\omega_{n})$ (with proper inclusion).
Therefore, $G_{a}(\omega_{n+1})=G_{a}(\omega_{n})\setminus\{z\}$
is also a tree. This concludes the proof of the only if part.
\end{itemize}
Now, we prove the if part. Suppose that \textbf{{[}Z1{]}}, \textbf{{[}Z2{]}}
and \textbf{{[}Z3{]}} hold for a configuration $\eta\in\mathcal{X}$.
Then, since $G_{a}(\eta)$ is a tree in $\mathbb{T}_{K}\times\{\ell,\,\ell+1\}$,
we may enumerate
\[
G_{a}(\eta)=\{x_{1},\,x_{2},\,\dots,\,x_{m}\}
\]
in a way that $\{x_{i},\,\dots,\,x_{m}\}$ is still a tree for all
$i\in\llbracket1,\,m\rrbracket$. Indeed, we can just inductively
choose an external vertex $x_{i}$ in each $G_{a}(\eta)\setminus\{x_{1},\,\dots,\,x_{i-1}\}$.
Then, starting from $\eta$, we define a path $(\omega_{n})_{n=0}^{m}$
such that for each $n\in\llbracket1,\,m\rrbracket$, $\omega_{n}$
is obtained from $\omega_{n-1}$ by flipping the spin $a$ on $x_{n}$
to spin $b$ (i.e., $\omega_{n}=\omega_{n-1}^{x_{n},\,b}$). Then,
by the same logic used in the two cases above, we have $H(\omega_{n})=H(\omega_{n-1})$
for $n\in\llbracket1,\,m-1\rrbracket$. On the other hand, the final
update from $\omega_{m-1}$ to $\omega_{m}$ decreases the energy
by $2$ since in $\omega_{m-1}$, $x_{m}$ has three neighboring sites
with spin $b$ and one neighboring site with spin $a$. Moreover,
$G_{a}(\omega_{m})=\emptyset$ which implies that $\omega_{m}=\xi_{\ell,\,2}^{a,\,b}\in\mathcal{R}_{2}^{a,\,b}$.
Therefore, the reversed path $(\omega_{m-n})_{n=0}^{m}$ guarantees
by definition that $\omega_{0}=\eta\in\mathcal{Z}^{a,\,b}$. This
concludes the proof of the case of $K<L$.

Finally, the case of $K=L$ can be dealt with in an identical manner;
the only difference is that we may also have $\omega_{0}=\Theta(\xi_{\ell,\,2}^{a,\,b})$
for some $\ell\in\mathbb{T}_{L}$, so that in \textbf{{[}Z2{]}} we
may have $P_{b}(\eta)\subseteq\{\ell,\,\ell+1\}\times\mathbb{T}_{K}$.
Everything else works in the same way.
\end{proof}
Before proceeding to the proof of Proposition \ref{p_Eadecomp}, we
give an additional property of the set $\mathcal{Z}^{a,\,b}$.
\begin{lem}
\label{l_Zab}Let $a,\,b\in S$ and $\eta\in\mathcal{Z}^{a,\,b}$.
Then, the followings hold. 
\begin{enumerate}
\item There exists a $\Gamma$-path from $\mathbf{a}$ to $\eta$ in $\mathcal{X}\setminus\mathcal{B}_{\Gamma}$.
\item If another $\xi\in\mathcal{X}$ satisfies $\eta\sim\xi$ and $H(\xi)\le\Gamma$,
then we have $\xi\in\mathcal{Z}^{a,\,b}\cup\mathcal{N}(\mathbf{a})\cup\mathcal{R}_{2}^{a,\,b}$.
\end{enumerate}
\end{lem}

\begin{proof}
(1) Without loss of generality, we assume that $K<L$. According to
the notation and statements from Proposition \ref{p_Zabtree}, $G_{a}(\eta)$,
the collection of sites in $\mathbb{T}_{K}\times\{\ell,\,\ell+1\}$
with spin $a$, is a tree. Then, we may enumerate
\[
\big[\mathbb{T}_{K}\times\{\ell,\,\ell+1\}\big]\setminus G_{a}(\eta)=\{y_{1},\,\dots,\,y_{M}\}
\]
in a way that $G_{a}(\eta)\cup\{y_{1},\,\dots,\,y_{i}\}$ is connected
for all $i\in\llbracket1,\,M\rrbracket$. Indeed, we can just inductively
choose a neighboring site to $G_{a}(\eta)\cup\{y_{1},\,\dots,\,y_{i}\}$.
Then, we define a path $(\omega_{n})_{n=0}^{M}$ such that for each
$n\in\llbracket1,\,M\rrbracket$, $\omega_{n}$ is obtained from $\omega_{n-1}$
by flipping the spin $b$ on $y_{n}$ to spin $a$ (i.e., $\omega_{n}=\omega_{n-1}^{y_{n},\,a}$).
By our construction, on each step $n$, the site $y_{n}$ has at least
two neighboring sites with spin $a$ (one in $G_{a}(\eta)\cup\{y_{1},\,\dots,\,y_{n-1}\}$
and another outside $\mathbb{T}_{K}\times\{\ell,\,\ell+1\}$), which
implies that $H(\omega_{n})\le H(\omega_{n-1})$. Moreover, it is
clear that $H(\omega_{M})=\mathbf{a}$. Therefore, $(\omega_{n})_{n=0}^{M}$
is a $H(\eta)$-path from $\eta$ to $\mathbf{a}$. Since $H(\eta)=\Gamma$
and it is clear by construction that $\omega_{n}\notin\mathcal{B}_{\Gamma}$,
we conclude that $(\omega_{n})_{n=0}^{M}$ satisfies the requirements
of part (1).\medskip{}

\noindent (2) Let $(\omega_{n})_{n=0}^{N}$ be a path in $\mathcal{X}\setminus\mathcal{B}_{\Gamma}$
from $\mathcal{R}_{2}^{a,\,b}$ to $\eta$ such that $H(\omega_{n})=\Gamma$
for $n\in\llbracket1,\,N\rrbracket$ (cf. \eqref{e_Zabdef}). Suppose
first that $H(\xi)=\Gamma$. Then by part (3) of Lemma \ref{l_typ1},
it follows that $\xi\notin\mathcal{B}_{\Gamma}$. Hence, by letting
$\omega_{N+1}=\xi$, the path $(\omega_{n})_{n=0}^{N+1}$ in $\mathcal{X}\setminus\mathcal{B}_{\Gamma}$
is from $\mathcal{R}_{2}^{a,\,b}$ to $\xi$ and satisfies $H(\omega_{n})=\Gamma$
for $n\in\llbracket1,\,N+1\rrbracket$. Thus, by \eqref{e_Zabdef},
we get $\xi\in\mathcal{Z}^{a,\,b}$.

Next, we consider the case $H(\xi)<\Gamma$. Using Proposition \ref{p_lowE}
and considering the conditions \textbf{{[}Z1{]}}, \textbf{{[}Z2{]}}
and \textbf{{[}Z3{]}} in Proposition \ref{p_Zabtree} that $\eta$
must satisfy, we can easily deduce that $\xi\in\mathcal{R}_{2}^{a,\,b}$
or $\xi$ has an $a$-cross. In the latter case, again by Proposition
\ref{p_lowE} it holds that $\xi\in\mathcal{N}(\mathbf{a})$. Thus,
we conclude the proof.
\end{proof}
Now, we prove Proposition \ref{p_Eadecomp}.
\begin{proof}[Proof of Proposition \ref{p_Eadecomp}]
 First, we demonstrate that
\begin{equation}
\mathcal{E}^{a}=\widehat{\mathcal{N}}(\mathbf{a};\,\mathcal{Z}^{a}\cup\mathcal{R}^{a}\cup\mathcal{B}_{\Gamma})\cup\mathcal{Z}^{a}\cup\mathcal{R}^{a}\;.\label{e_proof0}
\end{equation}
Recall that $\mathcal{E}^{a}=\widehat{\mathcal{N}}(\mathbf{a};\,\mathcal{B}_{\Gamma})=\widehat{\mathcal{N}}(\mathcal{N}(\mathbf{a});\,\mathcal{B}_{\Gamma})$.
Since $\mathcal{Z}^{a}$ is connected to $\mathcal{N}(\mathbf{a})$
by a $\Gamma$-path outside $\mathcal{B}_{\Gamma}$ by Lemma \ref{l_Zab}-(1),
we have $\widehat{\mathcal{N}}(\mathcal{Z}^{a};\,\mathcal{B}_{\Gamma})=\widehat{\mathcal{N}}(\mathcal{N}(\mathbf{a});\,\mathcal{B}_{\Gamma})$.
Moreover, $\mathcal{R}^{a}$ is connected to $\mathbf{a}$ by a part
of a canonical path outside $\mathcal{B}_{\Gamma}$, so that $\widehat{\mathcal{N}}(\mathcal{R}^{a};\,\mathcal{B}_{\Gamma})=\widehat{\mathcal{N}}(\mathcal{N}(\mathbf{a});\,\mathcal{B}_{\Gamma})$.
Thus, we obtain
\[
\mathcal{E}^{a}=\widehat{\mathcal{N}}\big(\mathcal{N}(\mathbf{a})\cup\mathcal{Z}^{a}\cup\mathcal{R}^{a};\,\mathcal{B}_{\Gamma}\big)\;.
\]
Then, by Lemma \ref{l_nhddc} with $\mathcal{P}_{1}=\mathcal{N}(\mathbf{a})$,
$\mathcal{P}_{2}=\mathcal{Z}^{a}\cup\mathcal{R}^{a}$ and $\mathcal{Q}=\mathcal{B}_{\Gamma}$,
we obtain
\[
\mathcal{E}^{a}=\widehat{\mathcal{N}}(\mathcal{N}(\mathbf{a});\,\mathcal{Z}^{a}\cup\mathcal{R}^{a}\cup\mathcal{B}_{\Gamma})\cup\widehat{\mathcal{N}}(\mathcal{Z}^{a}\cup\mathcal{R}^{a};\,\mathcal{N}(\mathbf{a})\cup\mathcal{B}_{\Gamma})\;.
\]
Since it is clear that $\widehat{\mathcal{N}}(\mathcal{N}(\mathbf{a});\,\mathcal{Z}^{a}\cup\mathcal{R}^{a}\cup\mathcal{B}_{\Gamma})=\widehat{\mathcal{N}}(\mathbf{a};\,\mathcal{Z}^{a}\cup\mathcal{R}^{a}\cup\mathcal{B}_{\Gamma})$,
it suffices to prove that
\begin{equation}
\widehat{\mathcal{N}}(\mathcal{Z}^{a}\cup\mathcal{R}^{a};\,\mathcal{N}(\mathbf{a})\cup\mathcal{B}_{\Gamma})=\mathcal{Z}^{a}\cup\mathcal{R}^{a}\;.\label{e_proof}
\end{equation}
Subjected to the energy barrier $\Gamma$, the first exit from $\mathcal{Z}^{a}$
is to $\mathcal{N}(\mathbf{a})\cup\mathcal{R}^{a}$ by Lemma \ref{l_Zab}-(2),
and the first exit from $\mathcal{R}^{a}$ is to $\mathcal{Z}^{a}\cup\mathcal{B}^{\Gamma}$
by Lemma \ref{l_typ1}-(1) and \eqref{e_Zabdef}. Thus, we can deduce
\eqref{e_proof} by the same logic that we used to prove \eqref{e_typ2.1}.
Thus, we have proved \eqref{e_proof0}.

Now, to conclude the proof of Proposition \ref{p_Eadecomp} it remains
to prove that $\widehat{\mathcal{N}}(\mathbf{a};\,\mathcal{Z}^{a}\cup\mathcal{R}^{a}\cup\mathcal{B}_{\Gamma})=\mathcal{D}^{a}$.
Note that $\mathcal{D}^{a}=\widehat{\mathcal{N}}(\mathbf{a};\,\mathcal{Z}^{a})$.
It is clear that $\widehat{\mathcal{N}}(\mathbf{a};\,\mathcal{Z}^{a}\cup\mathcal{R}^{a}\cup\mathcal{B}_{\Gamma})\subseteq\widehat{\mathcal{N}}(\mathbf{a};\,\mathcal{Z}^{a})$.
Suppose that there exists $\eta\in\widehat{\mathcal{N}}(\mathbf{a};\,\mathcal{Z}^{a})\setminus\widehat{\mathcal{N}}(\mathbf{a};\,\mathcal{Z}^{a}\cup\mathcal{R}^{a}\cup\mathcal{B}_{\Gamma})$.
This implies that not only there exists a $\Gamma$-path $\omega$
from $\mathbf{a}$ to $\eta$ in $\mathcal{X}\setminus\mathcal{Z}^{a}$,
but also all such paths must visit $\mathcal{R}^{a}\cup\mathcal{B}_{\Gamma}$.
Thus, we may define
\[
n_{1}:=\min\{n\ge0:\omega_{n}\in\mathcal{R}^{a}\cup\mathcal{B}_{\Gamma}\}\;.
\]
By Lemma \ref{l_typ1} and \eqref{e_Zabdef}, we readily have three
possibilities:
\[
\omega_{n_{1}-1}\in\mathcal{Z}^{a}\cup\bigcup_{b\in S\setminus\{a\}}\bigcup_{v\in\llbracket3,\,L-3\rrbracket}\mathcal{R}_{v}^{a,\,b}\cup\bigcup_{b\in S\setminus\{a\}}\mathcal{R}_{L-2}^{a,\,b}\;.
\]
The first possibility, $\omega_{n_{1}-1}\in\mathcal{Z}^{a}$, is clearly
impossible by the definition of $\omega$. The second possibility,
$\omega_{n_{1}-1}\in\mathcal{R}_{v}^{a,\,b}$ for some $b\ne a$ and
$v\in\llbracket3,\,L-3\rrbracket$, is also impossible since then
by Lemma \ref{l_typ1}-(1), $\omega_{n_{1}-2}\in\mathcal{B}_{\Gamma}$
which contradicts the minimality of $n_{1}$. Finally, the third possibility,
$\omega_{n_{1}-1}\in\mathcal{R}_{L-2}^{a,\,b}$ for some $b\ne a$,
implies that $\omega_{n_{1}-1}\in\mathcal{E}^{b}$ and this is again
impossible since $\mathcal{E}^{a}\cap\mathcal{E}^{b}=\emptyset$ by
Proposition \ref{p_typprop}-(1). Therefore, all cases reject the
possibility that such $\eta$ exists, and we conclude the proof.
\end{proof}

\subsection{\label{secB.4}Graph structure}

In this subsection, we prove Propositions \ref{p_proj} and \ref{p_e0est}.
\begin{proof}[Proof of Proposition \ref{p_proj}]
 By symmetry, it suffices to prove only for the projection $\Pi_{\ell}^{a,\,b}:\mathcal{N}(\mathbf{a})\cup\mathcal{Z}_{\ell}^{a,\,b}\cup\{\xi_{\ell,\,2}^{a,\,b}\}\rightarrow\mathscr{V}$.
We first consider part (1). Suppose that $\eta_{1},\,\eta_{2}\ne\mathbf{a}$.
If $\eta_{1}\nsim\eta_{2}$, then both sides of \eqref{e_proj1} are
clearly $0$ and there is nothing to prove. If $\eta_{1}\sim\eta_{2}$
so that $\{\eta_{1},\,\eta_{2}\}\in\mathscr{E}$, then by \eqref{e_cdtMH}
and \eqref{e_EaMC}, we can write
\[
\frac{1}{q}e^{-\Gamma\beta}\mathfrak{r}(\Pi_{\ell}^{a,\,b}(\eta_{1}),\,\Pi_{\ell}^{a,\,b}(\eta_{2}))=\frac{1}{q}e^{-\Gamma\beta}=\frac{Z_{\beta}}{q}\cdot\mu_{\beta}(\eta_{1})r_{\beta}(\eta_{1},\,\eta_{2})\;,
\]
since $\min\{\mu_{\beta}(\eta_{1}),\,\mu_{\beta}(\eta_{2})\}=\frac{1}{Z_{\beta}}e^{-\Gamma\beta}$.
By Theorem \ref{t_mu}, the right-hand side of the last display becomes
$(1+o(1))\cdot\mu_{\beta}(\eta_{1})r_{\beta}(\eta_{1},\,\eta_{2})$.

We next consider part (2) so that $\eta_{1}\ne\mathbf{a}$. Similarly,
we may assume $\{\eta_{1},\,\mathbf{a}\}\in\mathscr{E}$ since otherwise
both sides of \eqref{e_proj2} become $0$. Then by \eqref{e_EaMC},
we can write 
\begin{align*}
\frac{1}{q}e^{-\Gamma\beta}\mathfrak{r}(\Pi_{\ell}^{a,\,b}(\eta_{1}),\,\mathbf{a}) & =\frac{1}{q}e^{-\Gamma\beta}|\{\xi\in\mathcal{N}(\mathbf{a}):\xi\sim\eta_{1}\}|\\
 & =|\{\xi\in\mathcal{N}(\mathbf{a}):\xi\sim\eta_{1}\}|\cdot\frac{Z_{\beta}}{q}\cdot\mu_{\beta}(\eta_{1})\;.
\end{align*}
Since $\min\{\mu_{\beta}(\eta_{1}),\,\mu_{\beta}(\xi)\}=\mu_{\beta}(\eta_{1})=\frac{1}{Z_{\beta}}e^{-\Gamma\beta}$
for all $\xi\in\mathcal{N}(\mathbf{a})$, by \eqref{e_cdtMH} and
Theorem \ref{t_mu}, the right-hand side equals
\[
(1+o(1))\cdot\sum_{\xi\in\mathcal{N}(\mathbf{a})}\mu_{\beta}(\eta_{1})r_{\beta}(\eta_{1},\,\xi)\;.
\]
This concludes the proof.
\end{proof}
\begin{proof}[Proof of Proposition \ref{p_e0est}]
 We fix $a,\,b\in S$ and $\ell\in\mathbb{T}_{L}$. Recall \eqref{e_eqpotMH}.
Note that
\[
\mathfrak{cap}(\mathbf{a},\,\xi_{\ell,\,2}^{a,\,b})=\mathfrak{D}(\mathfrak{f})=\sum_{\{x,\,y\}\subset\mathscr{V}}\frac{1}{|\mathscr{V}|}\mathfrak{r}(x,\,y)[\mathfrak{f}(y)-\mathfrak{f}(x)]^{2}\;.
\]
For each $k\in\mathbb{T}_{K}$, consider a path $\omega^{k}=(\omega_{n}^{k})_{n=0}^{K}$
such that
\[
\omega_{0}^{k}=\mathbf{a}\;\;\;\;\text{and}\;\;\;\;\omega_{n}^{k}=\xi_{\ell,\,1;\,k,\,n}^{a,\,b,\,+}\;\;\;\;\text{for }n\in\llbracket1,\,K\rrbracket\;.
\]
This is indeed a path from $\mathbf{a}$ to $\xi_{\ell,\,2}^{a,\,b}$
subjected to the process $\mathfrak{Z}(\cdot)$ since $\omega_{0}^{k}=\xi_{\ell,\,1}^{a,\,b}\in\mathcal{N}(\mathbf{a})$
and $\omega_{K}^{k}=\xi_{\ell,\,1;\,k,\,K}^{a,\,b,\,+}=\xi_{\ell,\,2}^{a,\,b}$.
Therefore, we may calculate along only these $K$ paths:
\[
\mathfrak{cap}(\mathbf{a},\,\xi_{\ell,\,2}^{a,\,b})\ge\sum_{k\in\mathbb{T}_{K}}\sum_{n=0}^{K-1}\frac{1}{|\mathscr{V}|}\mathfrak{r}(\omega_{n}^{k},\,\omega_{n+1}^{k})[\mathfrak{f}(\omega_{n+1}^{k})-\mathfrak{f}(\omega_{n}^{k})]^{2}=\frac{1}{|\mathscr{V}|}\sum_{k\in\mathbb{T}_{K}}\sum_{n=0}^{K-1}[\mathfrak{f}(\omega_{n+1}^{k})-\mathfrak{f}(\omega_{n}^{k})]^{2}\;.
\]
By the Cauchy--Schwarz inequality, we obtain
\[
\mathfrak{cap}(\mathbf{a},\,\xi_{\ell,\,2}^{a,\,b})\ge\frac{1}{|\mathscr{V}|}\sum_{k\in\mathbb{T}_{K}}\frac{1}{K}\Big[\sum_{n=0}^{K-1}[\mathfrak{f}(\omega_{n+1}^{k})-\mathfrak{f}(\omega_{n}^{k})]\Big]^{2}=\frac{1}{|\mathscr{V}|}\sum_{k\in\mathbb{T}_{K}}\frac{1}{K}[\mathfrak{f}(\mathbf{a})-\mathfrak{f}(\xi_{\ell,\,2}^{a,\,b})]^{2}\;.
\]
Since $\mathfrak{f}(\mathbf{a})=1$ and $\mathfrak{f}(\xi_{\ell,\,2}^{a,\,b})=0$,
we deduce that $\mathfrak{cap}(\mathbf{a},\,\xi_{\ell,\,2}^{a,\,b})\ge|\mathscr{V}|^{-1}$
which concludes the proof.
\end{proof}

\subsection{\label{secB.5}Gateway configurations}

In this subsection, we define the notion of \emph{gateway configurations}.
These are not needed in the current study, but they are essential
in the three-dimensional Ising/Potts models in our companion article
\cite{KS 3D}.
\begin{defn}[Gateway configurations]
\label{d_gate} Recall the collection $\mathcal{Z}^{a,\,b}$ defined
in \eqref{e_Zabdef}. Define the gateway $\mathcal{G}^{a,\,b}$ between
$\mathbf{a}$ and $\mathbf{b}$ as
\begin{equation}
\mathcal{G}^{a,\,b}=\mathcal{Z}^{a,\,b}\cup\mathcal{B}^{a,\,b}\cup\mathcal{Z}^{b,\,a}\;,\label{e_gatedef}
\end{equation}
where this expression gives a decomposition of $\mathcal{G}^{a,\,b}$.
Since $\mathcal{B}^{a,\,b}=\mathcal{B}^{b,\,a}$, we have $\mathcal{G}^{a,\,b}=\mathcal{G}^{b,\,a}$.
A configuration belonging to $\mathcal{G}^{a,\,b}$ is called a\textit{
gateway configuration} between $\mathbf{a}$ and $\mathbf{b}$.
\end{defn}

\begin{rem}
\label{r_gate}In the terminology which was first suggested in \cite{MNOS}
and widely used since then, \textit{gate} configurations are the ones
which serve as effective checkpoints of the transition, along with
the property that their energy equals the exact energy barrier (i.e.,
$\Gamma$ in our case). On the other hand, in our sense (cf. Definition
\ref{d_gate}), \textit{gateway} configurations may have energy less
than $\Gamma$; indeed, we have $\mathcal{R}_{v}^{a,\,b}\subset\mathcal{G}^{a,\,b}$
for $v\in\llbracket2,\,L-2\rrbracket$ where each configuration in
$\mathcal{R}_{v}^{a,\,b}$, $v\in\llbracket2,\,L-2\rrbracket$ has
energy $\Gamma-2<\Gamma$. Nevertheless, our definition of gateway
configurations still works properly since the gateway configurations
with energy less than $\Gamma$ must immediately visit the ones with
energy $\Gamma$ before preceding further. This fact is interpreted
in part (1) of Lemma \ref{l_typ1}.
\end{rem}

\begin{lem}
\label{l_gate}For $a,\,b\in S$, suppose that two configurations
$\eta$ and $\xi$ satisfy 
\[
\eta\in\mathcal{G}^{a,\,b}\;,\;\xi\notin\mathcal{G}^{a,\,b}\;,\;\eta\sim\xi\;\text{and}\;H(\xi)\le\Gamma\;.
\]
Then, we have either $\xi\in\mathcal{N}(\mathbf{a})$ and $\eta\in\mathcal{Z}^{a,\,b}$
or $\xi\in\mathcal{N}(\mathbf{b})$ and $\eta\in\mathcal{Z}^{b,\,a}$.
In particular, we cannot have $\eta\in\mathcal{B}^{a,\,b}$.
\end{lem}

\begin{proof}
The proof is straightforward from Lemmas \ref{l_typ1} and \ref{l_Zab}.
\end{proof}

\section{Dirichlet form on gateway configurations}

In this short section, we record a result which is not needed in the
current study, but necessary in our companion article \cite{KS 3D}.
\emph{In this subsection, we assume that $q=2$.}
\begin{prop}
Recall the test function $\widetilde{h}:\mathcal{X}=\{1,\,2\}^{\Lambda}\to\mathbb{R}$
from Definition \ref{d_testf}. Then, it holds that
\[
\sum_{\{\eta,\,\xi\}\subset\mathcal{X}:\,\{\eta,\,\xi\}\cap\mathcal{G}^{1,2}\ne\emptyset}\mu_{\beta}(\eta)r_{\beta}(\eta,\,\xi)\{\widetilde{h}(\xi)-\widetilde{h}(\eta)\}^{2}=\frac{1+o(1)}{2\kappa}e^{-\Gamma\beta}\;.
\]
\end{prop}

\begin{proof}
It has already been verified in the proof of Proposition \ref{p_testf}
that
\begin{equation}
\sum_{\{\eta,\,\xi\}\subset\mathcal{E}\cup\mathcal{B}}\mu_{\beta}(\eta)r_{\beta}(\eta,\,\xi)\{\widetilde{h}(\xi)-\widetilde{h}(\eta)\}^{2}=\frac{1+o(1)}{2\kappa}e^{-\Gamma\beta}\;.\label{eq1}
\end{equation}
Then, by the definition of $\widetilde{h}$ given in Definition \ref{d_testf},
$\widetilde{h}\equiv1$ on $\mathcal{D}^{1}$ and $\widetilde{h}\equiv0$
on $\mathcal{D}^{2}$. Thus, the summation in the left-hand side vanishes
if
\begin{equation}
\{\eta,\,\xi\}\subset\mathcal{D}^{1}\;\;\;\;\text{or}\;\;\;\;\{\eta,\,\xi\}\subset\mathcal{D}^{2}\;.\label{eq2}
\end{equation}
Proposition \ref{p_Eadecomp} and Definition \ref{d_gate} imply the
following decomposition:
\[
\mathcal{E}\cup\mathcal{B}=\mathcal{D}^{1}\cup\mathcal{G}^{1,\,2}\cup\mathcal{D}^{2}\;.
\]
Therefore, by \eqref{eq2}, the left-hand side of \eqref{eq1} equals
\[
\sum_{\{\eta,\,\xi\}\subset\mathcal{X}:\,\{\eta,\,\xi\}\cap\mathcal{G}^{1,2}\ne\emptyset}\mu_{\beta}(\eta)r_{\beta}(\eta,\,\xi)\{\widetilde{h}(\xi)-\widetilde{h}(\eta)\}^{2}\;,
\]
which concludes the proof of the proposition.
\end{proof}

\section{\label{secD}Proof of Results in Section \ref{sec8}}

\subsection{\label{secD.1}Energy barrier}
\begin{proof}[Proof of Proposition \ref{p_EBlb}]
 We fix such a path $\omega$, so that $\omega_{0}=\mathbf{a}$ and
$\omega_{N}=\mathbf{b}$ for some $\mathbf{b}\in\breve{\mathbf{a}}$,
and suppose the contrary that $\Phi_{\omega}\le2K+3$. By the formula
\eqref{e_Phiomega}, we note that
\begin{equation}
\max_{m\in\llbracket0,\,KL\rrbracket}H(\omega_{m})\le\max_{m\in\llbracket0,\,KL-1\rrbracket}\max_{a\in S}H(\omega_{m}^{x_{m},\,a})=\Phi_{\omega}\le2K+3\;,\label{e_2K+3}
\end{equation}
where $x_{m}$ is the location of the spin update from $\omega_{m}$
to $\omega_{m+1}$. Therefore, each configuration $\omega_{m}$ has
energy at most $2K+3$.

We claim that there exists $m\in\llbracket0,\,KL\rrbracket$ such
that $\omega_{m}\in\mathcal{Q}_{v}^{a_{0},\,b}$ for some $a_{0}\ne b$
and $v\in\llbracket2,\,L-3\rrbracket$, and that the size of protuberance
of spin $b$ belongs to $\llbracket2,\,K-2\rrbracket$. First, we
assume that the claim holds and then prove the theorem. If the claim
holds, then we can easily check that $H(\omega_{m})=2K+2$ and for
every $x\in\Lambda$,
\[
\max_{a\in S}H(\omega_{m}^{x,\,a})\in\{2K+4,\,2K+5,\,2K+6\}\;.
\]
Here, it is used that $q\ge3$; this no longer holds if $q=2$. This
implies that 
\[
\Phi_{\omega}\ge\max_{a\in S}H(\omega_{m}^{x_{m},\,a})\ge2K+4\;,
\]
which contradicts the assumption and thus concludes the proof.

It remains to verify the claim. To this end, we take maximal $m\in\llbracket0,\,KL\rrbracket$
such that $\|\omega_{m}\|_{b}=\lfloor\frac{KL}{2}\rfloor+3$ (cf.
\eqref{e_spinnum}). We divide the proof into three cases.\medskip{}

\noindent \textbf{(Case 1: $\omega_{m}$ has a horizontal $b$-bridge
without a vertical $b$-bridge) }For this case, since $\Delta H_{c_{h}}(\omega_{m})\ge2$
for all $h\in\llbracket1,\,K\rrbracket$, by \eqref{e_decH}, the
fact that $H(\omega_{m})\le2K+3$, and the fact that $\Delta H_{r_{v}}(\omega_{m})\ge2$
if $\Delta H_{r_{v}}(\omega_{m})>0$, we deduce that the possibilities
lie among:
\begin{itemize}
\item \textbf{(C1)} $\sum_{v\in\mathbb{T}_{L}}\Delta H_{r_{v}}(\omega_{m})=0$
and $\sum_{h\in\mathbb{T}_{K}}\Delta H_{c_{h}}(\omega_{m})=2K$,
\item \textbf{(C2)} $\sum_{v}\Delta H_{r_{v}}(\omega_{m})=0$ and $\sum_{h}\Delta H_{c_{h}}(\omega_{m})=2K+1$,
\item \textbf{(C3)} $\sum_{v}\Delta H_{r_{v}}(\omega_{m})=0$ and $\sum_{h}\Delta H_{c_{h}}(\omega_{m})=2K+2$,
\item \textbf{(C4)} $\sum_{v}\Delta H_{r_{v}}(\omega_{m})=0$ and $\sum_{h}\Delta H_{c_{h}}(\omega_{m})=2K+3$,
\item \textbf{(C5)} $\sum_{v}\Delta H_{r_{v}}(\omega_{m})=2$ and $\sum_{h}\Delta H_{c_{h}}(\omega_{m})=2K$,
\item \textbf{(C6)} $\sum_{v}\Delta H_{r_{v}}(\omega_{m})=2$ and $\sum_{h}\Delta H_{c_{h}}(\omega_{m})=2K+1$,
\item \textbf{(C7)} $\sum_{v}\Delta H_{r_{v}}(\omega_{m})=3$ and $\sum_{h}\Delta H_{c_{h}}(\omega_{m})=2K$.
\end{itemize}
In sub-cases \textbf{(C1)}, \textbf{(C2)}, \textbf{(C3)} and \textbf{(C4)},
all rows must be bridges so that all columns are the same. This is
impossible since $\|\omega_{m}\|_{b}=\lfloor\frac{KL}{2}\rfloor+3$
cannot be a multiple of $K\ge11$. In sub-case \textbf{(C5)}, it necessarily
holds that $\Delta H_{c_{h}}(\omega_{m})=2$ for all $h\in\llbracket1,\,K\rrbracket$.
Then, it is easy to deduce that $\omega_{m}\in\mathcal{Q}_{v}^{a_{0},\,b}$
for some $a_{0}\ne b$. Moreover, the size of protuberance of $b$
cannot be $1$ or $K-1$ since $\lfloor\frac{KL}{2}\rfloor+3$ cannot
be equivalent to $\pm1$ modulo $K$. Thus, the claim is proved in
this case. In sub-case \textbf{(C6)}, the only possible configuration
is the one which is obtained from a regular configuration by attaching
a protuberance of a third spin. This is impossible by the same logic
that $\lfloor\frac{KL}{2}\rfloor+3$ cannot be equivalent to $0$
or $\pm1$ modulo $K$. Finally, in sub-case \textbf{(C7)}, there
exists $v_{0}\in\mathbb{T}_{L}$ such that $\Delta H_{r_{v_{0}}}(\omega_{m})=3$
and $\Delta H_{r_{v}}(\omega_{m})=0$ for all $v\ne v_{0}$. This
implies that there are three types of spins in the row $r_{v_{0}}$
and all the other rows are bridges, but this contradicts the fact
that $\Delta H_{c_{h}}(\omega_{m})$ must be $2$ for all $h\in\mathbb{T}_{K}$.
Thus, this sub-case is impossible.

Therefore, we conclude the proof of the claim in this case.

\noindent \medskip{}

\noindent \textbf{(Case 2: $\omega_{m}$ has a vertical $b$-bridge
without a horizontal $b$-bridge)} This case can be handled by the
identical manner to \textbf{(Case 1)} and we omit the details.\medskip{}

\noindent \textbf{(Case 3: $\omega_{m}$ has a $b$-cross)} In this
case, as in the proof of Proposition \ref{p_lowE}, define $\widetilde{\sigma}\in\mathcal{X}$
as the configuration obtained from $\omega_{m}$ by replacing all
non-$b$ spins with some fixed $a_{0}\in S\setminus\{b\}$, i.e.,
for $x\in\Lambda$,
\[
\widetilde{\sigma}(x)=\begin{cases}
b & \text{if }\omega_{m}(x)=b\;,\\
a_{0} & \text{if }\omega_{m}(x)\ne b\;,
\end{cases}
\]
so that $H(\widetilde{\sigma})\le H(\omega_{m})$. Then, we can apply
the well-known isoperimetric inequality (e.g. \cite[Corollary 2.5]{AC})
to the $a_{0}$-clusters of $\widetilde{\sigma}$ to deduce
\[
H(\widetilde{\sigma})\ge4\sqrt{\|\widetilde{\sigma}\|_{a_{0}}}=4\sqrt{KL-\lfloor\frac{KL}{2}\rfloor-3}>2K+4\;,
\]
where we used $L\ge K\ge11$ at the last inequality. Thus, we obtain
$H(\omega_{m})>2K+3$ which contradicts \eqref{e_2K+3}. This completes
the proof. 
\end{proof}

\subsection{\label{secD.2}Typical configurations}

From now on, we provide proofs of the characterization of typical
configurations subjected to the cyclic dynamics, which are formulated
in Section \ref{sec8.3}. Since the structure is highly similar to
the typical configurations with respect to the MH dynamics, we will
omit the details if there are no specific novel ideas involved compared
to the corresponding previous versions in Section \ref{sec6} and
Appendix \ref{secB}.

First, we prove Proposition \ref{p_typpropcyc}. We need two lemmas.
\begin{lem}
\label{l_typcyc1}For $a,\,b\in S$, suppose that $\sigma_{1}\in\overline{\mathcal{B}}^{a,\,b}$
and $\sigma_{2}\in\mathcal{X}$ satisfy $\sigma_{1}\sim\sigma_{2}$
and $H(\sigma_{1},\,\sigma_{2})\le\Gamma$. Then, the following statements
hold.
\begin{enumerate}
\item If $\sigma_{1}\in\overline{\mathcal{R}}_{v}^{a,\,b}$ for $v\in\llbracket2,\,L-2\rrbracket$,
then $\sigma_{2}\in\overline{\mathcal{R}}_{v}^{a,\,b}\cup\overline{\mathcal{Q}}_{v-1}^{a,\,b}\cup\overline{\mathcal{Q}}_{v}^{a,\,b}$.
\item If $\sigma_{1}\in\overline{\mathcal{B}}_{\Gamma}^{a,\,b}$, then $\sigma_{2}\in\overline{\mathcal{B}}^{a,\,b}$.
\end{enumerate}
\end{lem}

\begin{proof}
This lemma can be proved in the exact same way as Lemma \ref{l_typ1},
and thus we omit the proof.
\end{proof}
We employ the previous lemma to prove the following crucial lemma.
\begin{lem}
\label{l_typcyc2}It holds that $\widehat{\mathcal{N}}(\overline{\mathcal{B}};\,\overline{\mathcal{E}}\setminus\overline{\mathcal{B}})=\overline{\mathcal{B}}$.
\end{lem}

\begin{proof}
It suffices to show $\widehat{\mathcal{N}}(\overline{\mathcal{B}};\,\overline{\mathcal{E}}\setminus\overline{\mathcal{B}})\subseteq\overline{\mathcal{B}}$.
If we suppose the contrary, then there must be a first exit from the
set $\overline{\mathcal{B}}$ outside $\overline{\mathcal{E}}\setminus\overline{\mathcal{B}}$.
This is impossible, since Lemma \ref{l_typcyc1} implies that we first
exit from $\overline{\mathcal{B}}$ via a configuration in $\overline{\mathcal{R}}_{2}\subset\overline{\mathcal{E}}$.
This concludes the proof.
\end{proof}
\begin{proof}[Proof of Proposition \ref{p_typpropcyc}]
 (1) It suffices to prove that $\{\mathbf{b}\}\cap\widehat{\mathcal{N}}(\mathbf{a};\,\overline{\mathcal{B}}_{\Gamma})=\emptyset$.
Suppose the contrary that $(\omega_{n})_{n=0}^{N}$ is a sequence
from $\mathbf{a}$ to $\mathbf{b}$ in $\mathcal{X}\setminus\overline{\mathcal{B}}_{\Gamma}$
so that $\omega_{n}\sim\omega_{n+1}$ and $H(\omega_{n},\,\omega_{n+1})\le\Gamma$.
Following the proof of Proposition \ref{p_EBlb}, we again claim that
there exists $m\in\llbracket0,\,KL\rrbracket$ such that $\omega_{m}\in\mathcal{Q}_{v}^{a_{0},\,b}$
for some $a_{0}\ne b$ and $v\in\llbracket2,\,L-3\rrbracket$ and
that the size of protuberance of spin $b$ belongs to $\llbracket2,\,K-2\rrbracket$.
Such configuration belongs to $\overline{\mathcal{B}}_{\Gamma}$,
and thus we are able to conclude the proof by contradiction.

As before, we take maximal $m\in\llbracket0,\,KL\rrbracket$ such
that $\|\omega_{m}\|_{b}=\lfloor\frac{KL}{2}\rfloor+3$ (cf. \eqref{e_spinnum}).
The only difference here with the proof of Proposition \ref{p_EBlb}
is that we also allow $\omega_{m}$ to have energy $2K+4$. In turn,
we have four additional sub-cases added to the seven sub-cases in
\textbf{(Case 1)} in the proof of Proposition \ref{p_EBlb}, namely,
\begin{itemize}
\item \textbf{(C8)} $\sum_{v}\Delta H_{r_{v}}(\omega_{m})=0$ and $\sum_{h}\Delta H_{c_{h}}(\omega_{m})=2K+4$,
\item \textbf{(C9)} $\sum_{v}\Delta H_{r_{v}}(\omega_{m})=2$ and $\sum_{h}\Delta H_{c_{h}}(\omega_{m})=2K+2$,
\item \textbf{(C10)} $\sum_{v}\Delta H_{r_{v}}(\omega_{m})=3$ and $\sum_{h}\Delta H_{c_{h}}(\omega_{m})=2K+1$,
\item \textbf{(C11)} $\sum_{v}\Delta H_{r_{v}}(\omega_{m})=4$ and $\sum_{h}\Delta H_{c_{h}}(\omega_{m})=2K$.
\end{itemize}
\begin{figure}
\includegraphics[height=2.8cm]{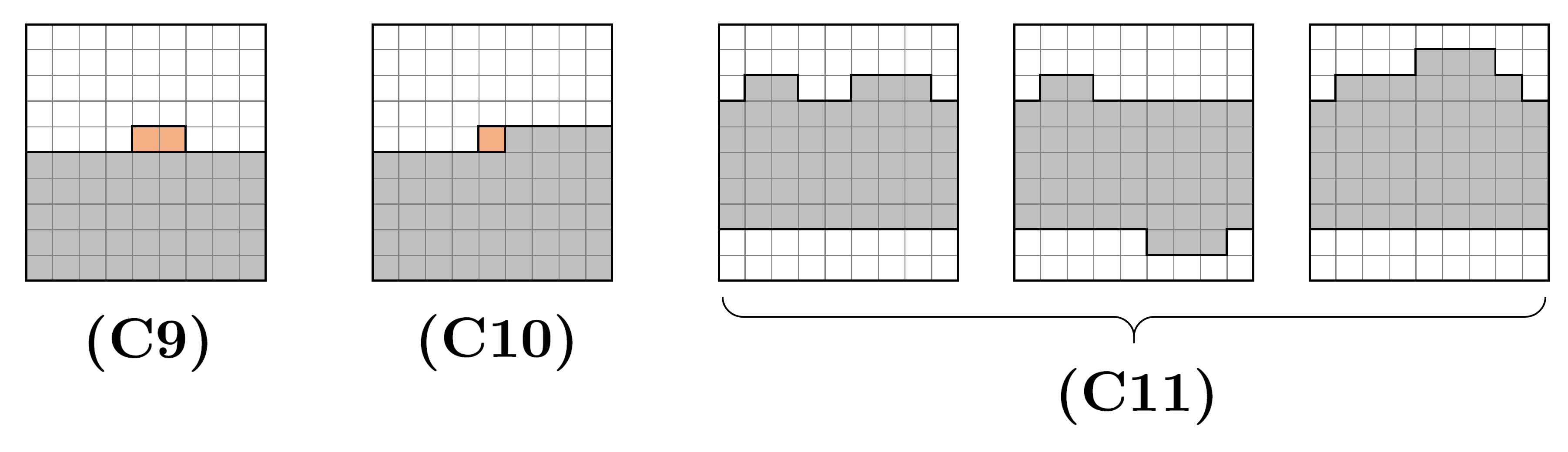}

\caption{\label{figC.1}\textbf{Illustration for the proof of Proposition \ref{p_typpropcyc}.}
The figures represent configurations belonging to sub-cases \textbf{(C9)},\textbf{
(C10)} and \textbf{(C11)}, respectively.}
\end{figure}

Sub-case \textbf{(C8)} is impossible since $\|\omega_{m}\|_{b}$ cannot
be a multiple of $K$. For sub-case \textbf{(C9)}, the only possible
type is illustrated in Figure \ref{figC.1}-left and thus in this
case we must have $\|\omega_{m}\|_{b}=\lfloor\frac{KL}{2}\rfloor+3\equiv0$
or $\pm2$ modulo $K$. This is impossible for $K\ge11$.

For \textbf{(C10)}, the only possible type of configuration is given
in Figure \ref{figC.1}-middle. Then, to obtain the next configuration
$\omega_{m+1}$ without exceeding the energy barrier $\Gamma=2K+4$,
it is mandatory that the spin flip $\omega_{m}\to\omega_{m+1}$ happens
on the single protuberance spin. Since $m$ is maximal, the resulting
spin must be $b$, so that we obtain $\omega_{m+1}\in\overline{\mathcal{B}}_{\Gamma}$
which contradicts our assumption (here we employed the fact that $\|\omega_{m}\|_{b}=\lfloor\frac{KL}{2}\rfloor+3$
cannot be equivalent to $-2$ modulo $K$). Finally, for sub-case
\textbf{(C11)}, the possible types of configurations are given in
Figure \ref{figC.1}-right. Note that any update from a configuration
of one of these types attains energy barrier at least $2K+6$, and
thus we obtain a contradiction. This completes the proof of part (1).
\medskip{}

\noindent (2) By definition, we have $\overline{\mathcal{R}}_{2}^{a,\,b}\subset\overline{\mathcal{B}}^{a,\,b}$.
Moreover, it is verified in Definition \ref{d_canpathcyc} and Lemma
\ref{l_canpathcyc} that a part of a canonical path from $\mathbf{a}$
to $\mathbf{b}$ becomes a $\Gamma$-path from $\mathbf{a}$ to $\mathcal{R}_{2}^{a,\,b}$
in $\mathcal{X}\setminus\overline{\mathcal{B}}_{\Gamma}$, so that
$\mathcal{R}_{2}^{a,\,b}\subset\overline{\mathcal{E}}^{a}$. This
automatically implies that $\overline{\mathcal{R}}_{2}^{a,\,b}\subset\overline{\mathcal{E}}^{a}$.
Hence, we have proved 
\begin{equation}
\overline{\mathcal{E}}^{a}\cap\overline{\mathcal{B}}^{a,\,b}\supseteq\overline{\mathcal{R}}_{2}^{a,\,b}\;.\label{e_typcyc1}
\end{equation}
To prove the other direction, observe that
\[
\overline{\mathcal{E}}^{a}\cap\overline{\mathcal{B}}^{a,\,b}\subseteq\overline{\mathcal{B}}^{a,\,b}\setminus\overline{\mathcal{B}}_{\Gamma}^{a,\,b}=\bigcup_{v\in\llbracket2,\,L-2\rrbracket}\overline{\mathcal{R}}_{v}^{a,\,b}\;.
\]
For $\sigma\in\overline{\mathcal{R}}_{v}^{a,\,b}$ with $v\in\llbracket3,\,L-3\rrbracket$,
we cannot have a $\Gamma$-path in $\mathcal{X}\setminus\overline{\mathcal{B}}_{\Gamma}$
from $\mathbf{a}$ to $\sigma$ by part (1) of Lemma \ref{l_typcyc1}.
Thus, such $\sigma$ cannot belong to $\overline{\mathcal{E}}^{a}$
and we deduce that 
\begin{equation}
\overline{\mathcal{E}}^{a}\cap\overline{\mathcal{B}}^{a,\,b}\subseteq\overline{\mathcal{R}}_{2}^{a,\,b}\cup\overline{\mathcal{R}}_{L-2}^{a,\,b}\;.\label{e_typcyc1.5}
\end{equation}
Since $\overline{\mathcal{R}}_{L-2}^{a,\,b}\subset\overline{\mathcal{E}}^{b}$
and $\overline{\mathcal{E}}^{a}\cap\overline{\mathcal{E}}^{b}=\emptyset$
by part (1), we conclude that 
\begin{equation}
\overline{\mathcal{E}}^{a}\cap\overline{\mathcal{B}}^{a,\,b}\subseteq\overline{\mathcal{R}}_{2}^{a,\,b}\;.\label{e_typcyc2}
\end{equation}
The proof is completed by \eqref{e_typcyc1} and \eqref{e_typcyc2}.\medskip{}

\noindent (3) By the same logic as we obtained \eqref{e_typcyc1.5},
we deduce that $\overline{\mathcal{E}}^{a}\cap\overline{\mathcal{B}}^{b,\,c}\subseteq\overline{\mathcal{R}}_{2}^{b,\,c}\cup\overline{\mathcal{R}}_{L-2}^{b,\,c}$.
Since $\overline{\mathcal{R}}_{2}^{b,\,c}\subset\overline{\mathcal{E}}^{b}$
and $\overline{\mathcal{R}}_{L-2}^{b,\,c}\subset\overline{\mathcal{E}}^{c}$,
part (1) implies that $\overline{\mathcal{E}}^{a}\cap\overline{\mathcal{B}}^{b,\,c}=\emptyset$.

\medskip{}

\noindent (4) It suffices to prove that 
\begin{equation}
\widehat{\mathcal{N}}(\overline{\mathcal{E}}\cup\overline{\mathcal{B}})\subseteq\overline{\mathcal{E}}\cup\overline{\mathcal{B}}\;,\label{e_typcyc3}
\end{equation}
since it is clear that $\overline{\mathcal{E}}\cup\overline{\mathcal{B}}\subseteq\widehat{\mathcal{N}}(\overline{\mathcal{E}}\cup\overline{\mathcal{B}})=\widehat{\mathcal{N}}(\mathcal{S})$.
Indeed, the canonical paths defined in Definition \ref{d_canpathcyc}
assure us that $\widehat{\mathcal{N}}(\overline{\mathcal{B}})=\widehat{\mathcal{N}}(\mathcal{S})$
and by definition and Lemma \ref{l_hatNPemptycyc}, $\widehat{\mathcal{N}}(\overline{\mathcal{E}})\subseteq\widehat{\mathcal{N}}(\mathcal{S})$.
To prove \eqref{e_typcyc3}, by Lemma \ref{l_nhddc} with $\mathcal{P}_{1}=\overline{\mathcal{B}}$,
$\mathcal{P}_{2}=\overline{\mathcal{E}}\setminus\overline{\mathcal{B}}$
and $\mathcal{Q}=\emptyset$ and Lemma \ref{l_hatNPemptycyc}, 
\begin{align*}
\widehat{\mathcal{N}}(\overline{\mathcal{E}}\cup\overline{\mathcal{B}}) & =\widehat{\mathcal{N}}(\overline{\mathcal{B}};\,\overline{\mathcal{E}}\setminus\overline{\mathcal{B}})\cup\widehat{\mathcal{N}}(\overline{\mathcal{E}}\setminus\overline{\mathcal{B}};\,\overline{\mathcal{B}})\;.
\end{align*}
This completes the proof of \eqref{e_typcyc3} by Lemma \ref{l_typcyc2},
since we have $\widehat{\mathcal{N}}(\overline{\mathcal{E}}\setminus\overline{\mathcal{B}};\,\overline{\mathcal{B}})\subseteq\widehat{\mathcal{N}}(\overline{\mathcal{E}};\,\overline{\mathcal{B}}_{\Gamma})=\overline{\mathcal{E}}$.
\end{proof}
Finally, we proceed to prove Proposition \ref{p_e0barabest}. First,
we state a few lemmas. For $a\in S$, we denote by $\mathcal{N}^{\mathrm{MH}}(\mathbf{a})$
the neighborhood of $\mathbf{a}$ in the sense of the MH dynamics
(cf. Definition \ref{d_nhd}).
\begin{lem}
\label{l_NMHNcyc}For every $a\in S$, it holds that $\mathcal{N}^{\mathrm{MH}}(\mathbf{a})\subseteq\mathcal{N}(\mathbf{a})$.
\end{lem}

\begin{proof}
We take $\sigma\in\mathcal{N}^{\mathrm{MH}}(\mathbf{a})$ and prove
that $\sigma\in\mathcal{N}(\mathbf{a})$. There exists a path $\widetilde{\omega}:\mathbf{a}\to\sigma$
subject to the MH dynamics such that $\Phi_{\widetilde{\omega}}\le2K+1$.
Now, we define a path $\omega:\mathbf{a}\to\sigma$ by enlarging each
MH-spin update in $\widetilde{\omega}$ by iterating spin rotations
on the corresponding site. Then, to prove that $\sigma\in\mathcal{N}(\mathbf{a})$,
it suffices to demonstrate that $\Phi_{\omega}\le2K+3$.

Consider each pair $(\widetilde{\omega}_{m},\,\widetilde{\omega}_{m+1})$
with $\widetilde{\omega}_{m+1}=\widetilde{\omega}_{m}^{x_{m},\,b}$
for some $x_{m}\in\Lambda$ and $b\in S$, from which the enlarged
path becomes $(\omega_{n},\,\omega_{n+1},\,\dots,\,\omega_{n+r})$
(so that $\omega_{n}=\widetilde{\omega}_{m}$ and $\omega_{n+r}=\widetilde{\omega}_{m+1}$).
Then, by \eqref{e_Hsigmazeta},
\[
H(\omega_{n+i},\,\omega_{n+i+1})=\max_{c\in S}H(\widetilde{\omega}_{m}^{x_{m},\,c})\;\;\;\;\text{for all }i\in\llbracket0,\,r-1\rrbracket\;.
\]
Since $\widetilde{\omega}_{m}(x_{m})\ne\widetilde{\omega}_{m+1}(x_{m})$
and $\widetilde{\omega}_{m}(x)=\widetilde{\omega}_{m+1}(x)$ for all
$x\ne x_{m}$, it holds that
\[
\min\Big\{\sum_{x\in\Lambda:\,x\sim x_{m}}\mathbf{1}\{\widetilde{\omega}_{m}(x)=\widetilde{\omega}_{m}(x_{m})\},\,\sum_{x\in\Lambda:\,x\sim x_{m}}\mathbf{1}\{\widetilde{\omega}_{m+1}(x)=\widetilde{\omega}_{m+1}(x_{m})\}\Big\}\le2\;.
\]
Indeed, if the minimum is bigger than $2$, then $x_{m}$ must have
at least three neighboring sites with spin $\widetilde{\omega}_{m}(x_{m})$
and also at least three neighboring sites with spin $\widetilde{\omega}_{m+1}(x_{m})$,
which is absurd since $x_{m}$ has exactly four neighboring sites
in $\Lambda$. Thus, we deduce that
\[
\max_{c\in S}H(\widetilde{\omega}_{m}^{x_{m},\,c})\le\max\{H(\widetilde{\omega}_{m}),\,H(\widetilde{\omega}_{m+1})\}+2\;.
\]
This concludes the proof since we have
\[
\Phi_{\omega}=\max_{n\ge0}H(\omega_{n},\,\omega_{n+1})\le\max_{m\ge0}\big[\max\{H(\widetilde{\omega}_{m}),\,H(\widetilde{\omega}_{m+1})\}+2\big]=\Phi_{\widetilde{\omega}}+2\le2K+3\;.
\]
\end{proof}
\begin{lem}
\label{l_e0barabest}Fix $a,\,b\in S$ and $\ell\in\mathbb{T}_{L}$.
The following statements hold.
\begin{enumerate}
\item The sets $\overline{\mathscr{A}}_{\ell}^{a,\,b}$ and $\overline{\mathscr{B}}_{\ell}^{a,\,b}$
belong to $\mathcal{Z}_{\ell}^{a,\,b}$.
\item It holds that
\begin{equation}
\overline{\mathscr{A}}_{\ell}^{a,\,b}=\big\{\sigma\in\mathcal{Z}_{\ell}^{a,\,b}:\exists\zeta\in\mathcal{N}^{\mathrm{MH}}(\mathbf{a})\text{ such that }r_{\beta}^{\mathrm{MH}}(\sigma,\,\zeta)>0\big\}\;,\label{e_Alab}
\end{equation}
\item It holds that
\begin{equation}
\overline{\mathscr{B}}_{\ell}^{a,\,b}=\big\{\sigma\in\mathcal{Z}_{\ell}^{a,\,b}:r_{\beta}^{\mathrm{MH}}(\sigma,\,\xi_{\ell,\,2}^{a,\,b})>0\big\}\;.\label{e_Blab}
\end{equation}
\end{enumerate}
\end{lem}

\begin{proof}
(1) Recall from Definition \ref{d_ZbarMc} that
\[
\overline{\mathscr{A}}_{\ell}^{a,\,b}=\overline{\mathcal{Z}}_{\ell}^{a,\,b}\cap\mathcal{N}(\mathbf{a})\;\;\;\;\text{and}\;\;\;\;\overline{\mathscr{B}}_{\ell}^{a,\,b}=\overline{\mathcal{Z}}_{\ell}^{a,\,b}\cap\overline{\xi}_{\ell,\,2}^{a,\,b}\;.
\]
It suffices to prove that a configuration $\sigma\in\overline{\mathcal{Z}}_{\ell}^{a,\,b}\setminus\mathcal{Z}_{\ell}^{a,\,b}$
cannot belong to $\mathcal{N}(\mathbf{a})\cup\overline{\xi}_{\ell,\,2}^{a,\,b}$.
Indeed, such $\sigma$ has a single spin $c\in S\setminus\{a,\,b\}$
which is obtained by a spin update from a configuration in $\mathcal{Z}_{\ell}^{a,\,b}$,
and thus we can verify that this update increases the energy by $2$
and thus $H(\sigma)=2K+4$. This clearly implies that $\sigma\notin\mathcal{N}(\mathbf{a})$,
since every configuration in $\mathcal{N}(\mathbf{a})$ has energy
at most $2K+3$. Moreover, suppose the contrary that $\sigma\in\overline{\xi}_{\ell,\,2}^{a,\,b}$.
Then, since $\sigma$ has a single spin $c$ and $H(\xi_{\ell,\,2}^{a,\,b})=2K$,
$\sigma$ must be obtained from $\xi_{\ell,\,2}^{a,\,b}$ by an update
from a spin ($a$ or $b$) to $c$ which increases the energy by $4$.
This contradicts the previous observation, and thus we conclude that
$\sigma\notin\overline{\xi}_{\ell,\,2}^{a,\,b}$.\medskip{}

\noindent (2) First, we prove the $\supseteq$-part of \eqref{e_Alab}.
To this end, suppose that $\sigma\in\mathcal{Z}_{\ell}^{a,\,b}$ and
$\zeta\in\mathcal{N}^{\mathrm{MH}}(\mathbf{a})$ satisfy $r_{\beta}^{\mathrm{MH}}(\sigma,\,\zeta)>0$.
By property \textbf{{[}Z1{]}} of Proposition \ref{p_Zabtree}, $\sigma$
consists of spins $a$ and $b$ only. Moreover, the fact that $r_{\beta}^{\mathrm{MH}}(\sigma,\,\zeta)>0$
and $H(\zeta)<H(\sigma)$ implies that the MH-spin flip $\sigma\to\zeta$
occurs on some $x_{0}\in\Lambda$ in either way:
\begin{itemize}
\item $a\to b$ where $a$ has three neighboring spins $b$ and one neighboring
spin $a$,
\item $b\to a$ where $b$ has three neighboring spins $a$ and one neighboring
spin $b$.
\end{itemize}
In any cases, we may obtain $\sigma$ by iterating the spin rotation
$\tau_{x_{0}}$ to $\zeta$. Then, by \eqref{e_Hsigmazeta} and the
presence of three same spins on the nearest neighbors of $x_{0}$,
the height along that path $\zeta\to\sigma$ equals (cf. \eqref{e_OrbitH})
\[
\max_{\mathfrak{O}(\zeta,\,\sigma)}H=2K+3\;.
\]
Next, by Lemma \ref{l_NMHNcyc}, we have $\zeta\in\mathcal{N}(\mathbf{a})$
and thus there exists a $(2K+3)$-path from $\mathbf{a}$ to $\zeta$.
Therefore, concatenating this path before the previously-defined path
$\zeta\to\sigma$, we obtain a $(2K+3)$-path from $\mathbf{a}$ to
$\sigma$. This deduces that $\sigma\in\mathcal{N}(\mathbf{a})$,
and thus we have proved the $\supseteq$-part.

To prove the $\subseteq$-part of \eqref{e_Alab}, fix $\sigma\in\overline{\mathscr{A}}_{\ell}^{a,\,b}=\overline{\mathcal{Z}}_{\ell}^{a,\,b}\cap\mathcal{N}(\mathbf{a})$.
By part (1), we know that $\sigma\in\mathcal{Z}_{\ell}^{a,\,b}\cap\mathcal{N}(\mathbf{a})$,
so that we may apply Proposition \ref{p_Zabtree}. Since $\sigma\in\mathcal{N}(\mathbf{a})$,
there exists $\sigma_{0}\in\mathcal{X}$ such that $r_{\beta}(\sigma_{0},\,\sigma)>0$
and $H(\sigma_{0},\,\sigma)\le2K+3$. Then, it is immediate that $r_{\beta}^{\mathrm{MH}}(\sigma,\,\sigma_{0})>0$.
According to properties \textbf{{[}Z1{]}} and \textbf{{[}Z2{]}} in
Proposition \ref{p_Zabtree}, this can only happen when the MH-spin
flip $\sigma\to\sigma_{0}$ occurs on $x_{0}\in\Lambda$ in either
way:
\begin{itemize}
\item $a\to c$ for some $c\ne a$, where $a$ has three neighboring spins
$b$ and one neighboring spin $a$,
\item $b\to c$ for some $c\ne b$, where $b$ has three neighboring spins
$a$ and one neighboring spin $b$.
\end{itemize}
According to Proposition \ref{p_Zabtree}, the former case implies
that $x_{0}$ is the only site with spin $a$ on the strip $\mathbb{T}_{K}\times\{\ell,\,\ell+1\}$
(or $\{\ell,\,\ell+1\}\times\mathbb{T}_{K}$ if $K=L$), and thus
$\sigma\in\overline{\mathcal{R}}_{2}^{a,\,b}$, which is clearly impossible.
Therefore, the second case must hold.

Now, we define a new configuration $\sigma_{1}$ as
\[
\sigma_{1}(x)=\begin{cases}
\sigma(x) & \text{if }x\ne x_{0}\;,\\
a & \text{if }x=x_{0}\;.
\end{cases}
\]
Then, $r_{\beta}^{\mathrm{MH}}(\sigma,\,\sigma_{1})>0$. Moreover,
since $x_{0}$ has three neighboring spins $a$, $H(\sigma_{1})=2K$.
Thus, by Proposition \ref{p_lowE}, we readily deduce that $\sigma_{1}\in\mathcal{N}^{\mathrm{MH}}(\mathbf{a})$.
This concludes the proof of part (2).\medskip{}

\noindent (3) This can be proved in a nearly same manner to part (2),
and thus we omit the proof.
\end{proof}
Finally, we present a proof of Proposition \ref{p_e0barabest}.
\begin{proof}[Proof of Proposition \ref{p_e0barabest}]
 To prove that the constant $\overline{\mathfrak{e}}_{0}^{a,\,b}$
is independent of $a,\,b\in S$, we first denote by $\widetilde{\mathfrak{Z}}_{\ell}^{a,\,b}(\cdot)$
the trace of the process $\overline{\mathfrak{Z}}_{\ell}^{a,\,b}(\cdot)$
to the subset $\mathcal{Z}_{\ell}^{a,\,b}\subset\overline{\mathcal{Z}}_{\ell}^{a,\,b}$
(cf. Section \ref{sec3.3}). Then, by \cite[Corollary 6.2]{BL TM},
the jump rate $\widetilde{\mathfrak{r}}_{\ell}^{a,\,b}(\cdot,\,\cdot)$
of the process $\widetilde{\mathfrak{Z}}_{\ell}^{a,\,b}(\cdot)$ can
be written as 
\[
\widetilde{\mathfrak{r}}_{\ell}^{a,\,b}(\sigma_{1},\,\sigma_{2})=\overline{\mathfrak{r}}_{\ell}^{a,\,b}(\sigma_{1},\,\sigma_{2})+\sum_{\zeta\in\overline{\mathcal{Z}}_{\ell}^{a,b}\setminus\mathcal{Z}_{\ell}^{a,b}}\overline{\mathfrak{r}}_{\ell}^{a,\,b}(\sigma_{1},\,\zeta)\cdot\mathbf{P}_{\zeta}[\tau_{\mathcal{Z}_{\ell}^{a,b}}=\tau_{\sigma_{2}}]
\]
for all $\sigma_{1},\,\sigma_{2}\in\mathcal{Z}_{\ell}^{a,\,b}$, where
$\mathbf{P}_{\zeta}$ denotes the law of the process $\overline{\mathfrak{Z}}_{\ell}^{a,\,b}(\cdot)$
starting from $\zeta$. Thus, by the cyclic structure of the set $\overline{\mathcal{Z}}_{\ell}^{a,\,b}$,
we can observe that 
\[
\widetilde{\mathfrak{r}}_{\ell}^{a,\,b}(\sigma_{1},\,\sigma_{2})=\mathbf{1}\big\{ r_{\beta}^{\mathrm{MH}}(\sigma_{1},\,\sigma_{2})>0\big\}\;,
\]
and hence the trace process $\widetilde{\mathfrak{Z}}_{\ell}^{a,\,b}(\cdot)$
is equivalent to the Markov chain defined in Definition \ref{d_EAMc}
restricted to $\mathcal{Z}_{\ell}^{a,\,b}$. In particular, it does
not depend on $a,\,b\in S$.

Denote by $\widetilde{\mathfrak{cap}}^{a,\,b}(\cdot,\,\cdot)$ the
capacity with respect to the process $\widetilde{\mathfrak{Z}}_{\ell}^{a,\,b}(\cdot)$
on $\mathcal{Z}_{\ell}^{a,\,b}$. Then, by \cite[(A.10)]{BL TM2},
we can write (recall that $\overline{\mathscr{A}}_{\ell}^{a,\,b}$
and $\overline{\mathscr{B}}_{\ell}^{a,\,b}$ are subsets of $\mathcal{Z}_{\ell}^{a,\,b}$
by Lemma \ref{l_e0barabest}-(1)), 
\begin{equation}
\frac{1}{\overline{\mathfrak{e}}_{0}^{a,\,b}}=|\overline{\mathscr{V}}_{\ell}^{a,\,b}|\cdot\overline{\mathfrak{cap}}^{a,\,b}(\overline{\mathscr{A}}_{\ell}^{a,\,b},\,\overline{\mathscr{B}}_{\ell}^{a,\,b})=\big(|\overline{\mathscr{A}}_{\ell}^{a,\,b}|+|\overline{\mathscr{B}}_{\ell}^{a,\,b}|\big)\cdot\widetilde{\mathfrak{cap}}^{a,\,b}(\overline{\mathscr{A}}_{\ell}^{a,\,b},\,\overline{\mathscr{B}}_{\ell}^{a,\,b})\;.\label{e_e0abtrace}
\end{equation}
By Proposition \ref{p_Zabtree}, Definition \ref{d_Zlabdef} and Lemma
\ref{l_e0barabest}, we notice that the Markov chain $\widetilde{\mathfrak{Z}}_{\ell}^{a,\,b}(\cdot)$
and the sets $\overline{\mathscr{A}}_{\ell}^{a,\,b}$ and $\overline{\mathscr{B}}_{\ell}^{a,\,b}$
have the same structure for all $a,\,b\in S$. Therefore, we conclude
from the expression \eqref{e_e0abtrace} that the constant $\overline{\mathfrak{e}}_{0}^{a,\,b}$
does not depend on $a,\,b\in S$, and we have proved the first statement
of the proposition.

Next, we estimate $\overline{\mathfrak{e}}_{0}^{a,\,b}$. Since it
does not depend on $a,\,b\in S$ by the first statement, we may assume
that $a=1$ and $b=2$. By the same logic with the proof of Proposition
\ref{p_e0est} (given in Section \ref{secB.4}), we write
\[
\overline{\mathfrak{cap}}_{\ell}^{1,\,2}(\overline{\mathscr{A}}_{\ell}^{1,\,2},\,\overline{\mathscr{B}}_{\ell}^{1,\,2})=\overline{\mathfrak{D}}^{1,\,2}(\overline{\mathfrak{f}}^{1,\,2})=\frac{1}{2}\sum_{x,\,y\in\overline{\mathscr{V}}^{1,2}}\frac{1}{|\overline{\mathscr{V}}^{1,\,2}|}\overline{\mathfrak{r}}^{1,\,2}(x,\,y)[\overline{\mathfrak{f}}^{1,\,2}(y)-\overline{\mathfrak{f}}^{1,\,2}(x)]^{2}\;.
\]
For each $k\in\mathbb{T}_{K}$, we consider a path $\omega^{(k)}=(\omega_{n}^{(k)})_{n=0}^{K-2}$
where $\omega_{n}^{(k)}=\xi_{\ell,\,1;\,k,\,n+1}^{1,\,2,\,+}$ for
$n\in\llbracket0,\,K-2\rrbracket$. Note that $\xi_{\ell,\,1;\,k,\,1}^{1,\,2,\,+}\in\overline{\mathscr{A}}_{\ell}^{1,\,2}$
and $\xi_{\ell,\,1;\,k,\,K-1}^{1,\,2,\,+}\in\overline{\mathscr{B}}_{\ell}^{1,\,2}$.
Thus, we have 
\[
\overline{\mathfrak{cap}}_{\ell}^{1,\,2}(\overline{\mathscr{A}}_{\ell}^{1,\,2},\,\overline{\mathscr{B}}_{\ell}^{1,\,2})\ge\frac{1}{2}\sum_{k\in\mathbb{T}_{K}}\sum_{n=0}^{K-3}\frac{1}{|\overline{\mathscr{V}}^{1,\,2}|}[\overline{\mathfrak{f}}^{1,\,2}(\omega_{n+1}^{(k)})-\overline{\mathfrak{f}}^{1,\,2}(\omega_{n}^{(k)})]^{2}\;.
\]
By the Cauchy--Schwarz inequality,
\begin{align*}
\overline{\mathfrak{cap}}_{\ell}^{1,\,2}(\overline{\mathscr{A}}_{\ell}^{1,\,2},\,\overline{\mathscr{B}}_{\ell}^{1,\,2}) & \ge\frac{1}{2|\overline{\mathscr{V}}^{1,\,2}|}\sum_{k\in\mathbb{T}_{K}}\frac{1}{K-2}\Big[\sum_{n=0}^{K-3}[\overline{\mathfrak{f}}^{1,\,2}(\omega_{n+1}^{(k)})-\overline{\mathfrak{f}}^{1,\,2}(\omega_{n}^{(k)})]\Big]^{2}=\frac{K}{2(K-2)|\overline{\mathscr{V}}^{1,\,2}|}\;,
\end{align*}
where the last equality follows from 
\[
\overline{\mathfrak{f}}^{1,\,2}(\omega_{0}^{(k)})=\overline{\mathfrak{f}}^{1,\,2}(\xi_{\ell,\,1;\,k,\,1}^{1,\,2,\,+})=1\;\;\;\;\text{and}\;\;\;\;\overline{\mathfrak{f}}^{1,\,2}(\omega_{K-2}^{(k)})=\overline{\mathfrak{f}}^{1,\,2}(\xi_{\ell,\,1;\,k,\,K-1}^{1,\,2,\,+})=0\;.
\]
This proves that $\overline{\mathfrak{cap}}_{\ell}^{1,\,2}(\overline{\mathscr{A}}_{\ell}^{1,\,2},\,\overline{\mathscr{B}}_{\ell}^{1,\,2})\cdot|\overline{\mathscr{V}}^{1,\,2}|>\frac{1}{2}$,
and the proof is completed.
\end{proof}
\begin{acknowledgement*}
SK was supported by NRF-2019-Fostering Core Leaders of the Future
Basic Science Program/Global Ph.D. Fellowship Program and the National
Research Foundation of Korea (NRF) grant funded by the Korean government
(MSIT) (No. 2022R1F1A106366811, No. 2022R1A5A6000840). IS was supported
by the National Research Foundation of Korea (NRF) grant funded by
the Korean government (MSIT) (No. 2022R1F1A106366811, 2022R1A5A6000840).
\end{acknowledgement*}

\end{document}